\documentclass[12pt]{amsbook}

\usepackage{amssymb, amscd, txfonts}
\usepackage{graphicx}
\usepackage[all]{xy} 
\usepackage[mathscr]{euscript}

\numberwithin{equation}{chapter}
\numberwithin{section}{chapter}
\numberwithin{subsection}{section}

\newtheorem{theorem}{Theorem}[chapter]
\newtheorem{proposition}[theorem]{Proposition}
\newtheorem{lemma}[theorem]{Lemma}
\newtheorem{corollary}[theorem]{Corollary}
\newtheorem{step}{Step}
\newtheorem{step+}{Step}
\newtheorem{step++}{Step}

\theoremstyle{definition}
\newtheorem{definition}[theorem]{Definition}
\newtheorem{example}[theorem]{Example}

\theoremstyle{remark}
\newtheorem{remark}[theorem]{Remark}
\newtheorem{claim}[theorem]{Claim}

\renewcommand{\hom}{\operatorname{Hom}}
\renewcommand{\ker}{\operatorname{Ker}}

\newcommand{\Z}{\mathbb{Z}}
\newcommand{\Q}{\mathbb{Q}}
\newcommand{\R}{\mathbb{R}}
\newcommand{\C}{\mathbb{C}}
\newcommand{\HH}{\mathbb{H}}
\newcommand{\proj}{{\mathbb P}}

\newcommand{\OL}{{\rm O}^{+}(L)}

\newcommand{\OLR}{{\rm O}^{+}(L_{\mathbb{R}})}
\newcommand{\OLC}{{\rm O}(L_{\mathbb{C}})}
\newcommand{\SOLR}{{\rm SO}^{+}(L_{\mathbb{R}})}
\newcommand{\SL}{{\rm SL}}
\newcommand{\D}{\mathcal{D}}
\newcommand{\G}{\Gamma}
\newcommand{\FG}{\mathcal{F}(\Gamma)}
\newcommand{\FGcpt}{\mathcal{F}(\Gamma)^{\Sigma}}
\newcommand{\MG}{M_{\lambda, k}(\Gamma)}
\newcommand{\SG}{{S \!}_{\lambda,k}(\Gamma)}
\newcommand{\volD}{{\rm vol}_{\mathcal{D}}}
\newcommand{\LL}{\mathcal{L}}
\newcommand{\LJ}{\mathcal{L}_{J}}
\newcommand{\E}{\mathcal{E}}
\newcommand{\EJ}{\mathcal{E}_{J}}
\newcommand{\EJp}{\mathcal{E}_{J}^{\perp}}
\newcommand{\El}{\mathcal{E}_{\lambda}}
\newcommand{\Elk}{\mathcal{E}_{\lambda,k}}
\newcommand{\ElJ}{\mathcal{E}_{\lambda}^{J}}
\newcommand{\On}{{\rm O}(n, \mathbb{C})}
\newcommand{\SOn}{{\rm SO}(n, \mathbb{C})}
\newcommand{\ICv}{I_{\mathbb{C}}^{\vee}}
\newcommand{\ICp}{I_{\mathbb{C}}^{\perp}}
\newcommand{\Ip}{I^{\perp}}
\newcommand{\OD}{\mathcal{O}_{\mathcal{D}}}
\newcommand{\DI}{\mathcal{D}_{I}}
\newcommand{\ImZ}{{\rm Im}(Z)}

\newcommand{\VIl}{V(I)_{\lambda}}
\newcommand{\VIlk}{V(I)_{\lambda, k}}
\newcommand{\UIZ}{U(I)_{\mathbb{Z}}}
\newcommand{\UIZZ}{U(I)_{\mathbb{Z}}^{\vee}}
\newcommand{\UIF}{U(I)_{F}}
\newcommand{\UIQ}{U(I)_{\mathbb{Q}}}
\newcommand{\UIR}{U(I)_{\mathbb{R}}}
\newcommand{\UIC}{U(I)_{\mathbb{C}}}

\newcommand{\GIQ}{\Gamma(I)_{\mathbb{Q}}}
\newcommand{\GIR}{\Gamma(I)_{\mathbb{R}}}
\newcommand{\GIF}{\Gamma(I)_{F}}
\newcommand{\GIZ}{\Gamma(I)_{\mathbb{Z}}}
\newcommand{\GIZbar}{\overline{\Gamma(I)}_{\mathbb{Z}}}
\newcommand{\GJF}{\Gamma(J)_{F}}
\newcommand{\GJR}{\Gamma(J)_{\mathbb{R}}}
\newcommand{\GJQ}{\Gamma(J)_{\mathbb{Q}}}
\newcommand{\GJZ}{\Gamma(J)_{\mathbb{Z}}}
\newcommand{\GJZbar}{\overline{\Gamma(J)}_{\mathbb{Z}}}
\newcommand{\GJFbar}{\overline{\Gamma(J)}_{F}}

\newcommand{\GJRbar}{\overline{\Gamma(J)}_{\mathbb{R}}}
\newcommand{\UJF}{U(J)_{F}}
\newcommand{\UJR}{U(J)_{\mathbb{R}}}
\newcommand{\UJQ}{U(J)_{\mathbb{Q}}}
\newcommand{\UJZ}{U(J)_{\mathbb{Z}}}
\newcommand{\UJC}{U(J)_{\mathbb{C}}}
\newcommand{\XI}{\mathcal{X}(I)}
\newcommand{\XIcpt}{\mathcal{X}(I)^{\Sigma}}
\newcommand{\XJ}{\mathcal{X}(J)}
\newcommand{\XJcpt}{\overline{\mathcal{X}(J)}}
\newcommand{\HJ}{\mathbb{H}_{J}}
\newcommand{\VJ}{\mathcal{V}_{J}}

\begin{document}

\title[]{Vector-valued orthogonal modular forms}
\author[]{Shouhei Ma}

\address{Department of Mathematics, Tokyo Institute of Technology, Tokyo 152-8551, Japan}
\email{ma@math.titech.ac.jp}
\subjclass[2020]{11F55, 11F50}
\keywords{orthogonal modular forms, vector-valued modular forms, Siegel operator, 
Fourier-Jacobi expansion, vector-valued Jacobi forms, Petersson metric, vanishing theorem} 

\begin{abstract}
This monograph is devoted to the theory of vector-valued modular forms for orthogonal groups of signature $(2, n)$. 
Our purpose is multi-layered: 
(1) to lay a foundation of the theory of vector-valued orthogonal modular forms; 
(2) to develop some aspects of the theory in more depth such as 
geometry of the Siegel operators, 
filtrations associated to $1$-dimensional cusps, 
decomposition of vector-valued Jacobi forms, 
square integrability etc; and 
(3) as applications derive several types of vanishing theorems for vector-valued modular forms of small weight. 
Our vanishing theorems imply in particular vanishing of holomorphic tensors of degree $<n/2-1$ 
on orthogonal modular varieties, which is optimal as a general bound. 

The fundamental ingredients of the theory are the two Hodge bundles. 
The first is the Hodge line bundle which already appears in the theory of scalar-valued modular forms. 
The second Hodge bundle emerges in the vector-valued theory and plays a central role. 
It corresponds to the non-abelian part ${\rm O}(n, {\R})$ of the maximal compact subgroup of ${\rm O}(2, n)$. 
The main focus of this monograph is centered around the properties and the role of the second Hodge bundle 
in the theory of vector-valued orthogonal modular forms. 
\end{abstract}

\maketitle 

\setcounter{tocdepth}{1}
\tableofcontents

\chapter{Introduction}\label{sec: intro}

In the theory of modular forms of several variables, 
it is natural and also necessary to study vector-valued modular forms. 
One way to account for this is that 
scalar-valued modular forms are concerned only with the $1$-dimensional abelian quotient 
of the maximal compact subgroup $K$ of the Lie group, 
while the contribution from the whole $K$ emerges if we consider vector-valued modular forms. 
In more concrete levels, the significance of vector-valued modular forms 
appears in the study of the cohomology of modular varieties, 
holomorphic tensors on modular varieties, 
and constructions of Galois representations etc. 
The passage from scalar-valued to vector-valued modular forms is an intrinsic non-abelianization step. 

This subject has been well-developed for Siegel modular forms 
since the pioneering work of Freitag, Weissauer and others around the early 1980's (see, e.g., \cite{vdG} for a survey). 
In particular, a lot of detailed study have been done in the case of Siegel modular forms of genus $2$. 

By contrast, despite its potential and expected applications, 
no systematic theory of vector-valued modular forms for orthogonal groups of signature $(2, n)$ seems to have been developed so far. 
Only recently its application to holomorphic tensors on modular varieties started to be investigated (\cite{Ma2}). 
The observation that some aspects of the theory of scalar-valued Siegel modular forms of genus $2$ 
have been generalized to orthogonal modular forms, 
rather than to Siegel modular forms of higher genus, also suggests a promising theory. 

Vector-valued orthogonal modular forms will have applications to the geometry and arithmetic of 
orthogonal modular varieties, and so especially to the moduli spaces of $K3$ surfaces and holomorphic symplectic varieties. 
Moreover, from the geometric viewpoint of $K3$ surfaces, 
vector-valued modular forms on a period domain of (lattice-polarized) $K3$ surfaces 
are considered as holomorphic invariants related to the family  
that can be captured by the variation of the Hodge structures on $H^2(K3)$ 
but typically not by the Hodge line bundle $H^{0}(K_{K3})$ alone. 
For example in this direction, the infinitesimal invariants of normal functions for higher Chow cycles in $CH^{2}(K3, 1)$ 
give vector-valued modular forms with singularities (\S \ref{ssec: higher Chow}). 
This geometric viewpoint offers another motivation to develop the theory of vector-valued orthogonal modular forms. 

The purpose of this monograph is multi-layered: 
\begin{enumerate}
\item to lay a foundation of the theory of vector-valued orthogonal modular forms, 
\item to investigate some aspects of the theory in more depth, and 
\item as applications to establish several types of vanishing theorems for vector-valued modular forms of small weight. 
\end{enumerate}
Our theory is developed in a full generality in the sense that  
we work with general arithmetic groups ${\G}<{\OL}$ for general integral quadratic forms $L$ of signature $(2, n)$. 
The facts that unimodular lattices are rare even up to ${\Q}$-equivalence (unlike the symplectic case) 
and that various types of groups ${\G}$ appear in the moduli examples urge us to work in this generality. 

Our approach is geometric in the sense that we define modular forms as sections of the automorphic vector bundles. 
Trivializations of the automorphic vector bundles, and thus passage from sections of vector bundles to vector-valued functions, 
are provided for each $0$-dimensional cusp. 
This approach is suitable for working with general ${\G}$, without losing connection with the more classical style. 

In the rest of this introduction, we give a summary of the theory developed in this monograph. 

\subsection*{The two Hodge bundles (\S \ref{sec: L and E})}

Let $L$ be an integral quadratic lattice of signature $(2, n)$. 
We assume $n\geq 3$ for simplicity. 
The Hermitian symmetric domain ${\D}={\D}_{L}$ attached to $L$ is defined as an open subset of the isotropic quadric in ${\proj}L_{{\C}}$. 
It parametrizes polarized Hodge structures $0\subset F^2\subset F^1 \subset L_{{\C}}$ of weight $2$ on $L$ 
with $\dim F^2=1$ and $F^1=(F^2)^{\perp}$. 
Over ${\D}$ we have two fundamental Hodge bundles. 
The first is the Hodge line bundle 
\begin{equation*}
{\LL}=\mathcal{O}_{{\proj}L_{{\C}}}(-1)|_{{\D}}, 
\end{equation*}
which geometrically consists of the period lines $F^2$ in the Hodge filtrations. 
In terms of representation theory, ${\LL}$ is the homogeneous line bundle associated to the standard character of 
${\C}^{\ast}\subset {\C}^{\ast}\times {\On}$, 
where ${\C}^{\ast}\times {\On}$ is the reductive part of a standard parabolic subgroup of ${\OLC}\simeq {\rm O}(n+2, {\C})$. 
Invariant sections of powers of ${\LL}$ are scalar-valued modular forms on ${\D}$, 
which have been classically studied. 

The Hodge line bundle ${\LL}$ is naturally embedded in $L_{{\C}}\otimes {\OD}$ as an isotropic sub line bundle. 
The second Hodge bundle is defined as 
\begin{equation*}
{\E}= {\LL}^{\perp}/{\LL}. 
\end{equation*}
Geometrically this vector bundle consists of the middle graded quotients $F^1/F^2$ of the Hodge filtrations. 
In terms of representation theory, ${\E}$ is the homogeneous vector bundle associated to 
the standard representation of ${\On}\subset {\C}^{\ast}\times {\On}$. 
It is this second Hodge bundle ${\E}$ 
that emerges in the theory of vector-valued modular forms on ${\D}$ 
and plays a central role in this monograph. 

While ${\LL}$ is concerned with scalar-valued modular forms, 
${\E}$ is responsible for the higher rank aspect of the theory of vector-valued modular forms. 
While ${\LL}$ provides a polarization, 
${\E}$ is an orthogonal vector bundle, and in particular self-dual (but not trivial). 
Thus ${\LL}$ and ${\E}$ are rather contrastive. 

\subsection*{Vector-valued modular forms (\S \ref{sec: vector-valued})}

Weights of vector-valued modular forms on ${\D}$ are expressed by pairs $(\lambda, k)$, 
where $\lambda=(\lambda_{1}\geq \cdots \geq \lambda_{n} \geq 0)$ is a partition 
which corresponds to an irreducible representation $V_{\lambda}$ of ${\On}$, 
and $k$ is an integer which corresponds to a character of ${\C}^{\ast}$. 
The partition $\lambda$ satisfies ${}^t \lambda_{1}+ {}^t \lambda_{2}\leq n$ 
where ${}^t \lambda$ is the transpose of $\lambda$. 
To such a pair $(\lambda, k)$ we associate the automorphic vector bundle 
\begin{equation*}
{\Elk}={\El}\otimes {\LL}^{\otimes k}, 
\end{equation*}
where ${\El}$ is the vector bundle constructed from ${\E}$ by applying the orthogonal Schur functor associated to $\lambda$. 
Modular forms of weight $(\lambda, k)$ are defined as holomorphic sections of ${\Elk}$ over ${\D}$ 
invariant under a finite-index subgroup ${\G}$ of ${\OL}$ 
(with cusp conditions when $n\leq 2$). 
We denote by ${\MG}$ the space of ${\G}$-modular forms of weight $(\lambda, k)$. 

Sometimes it is more appropriate to work with irreducible representations of ${\rm SO}(n, {\C})$ rather than ${\On}$, 
but in that way we obtain only ${\rm SO}^{+}(L_{{\R}})$-equivariant vector bundles. 
Since in some applications we encounter subgroups ${\G}$ of ${\OL}$ not contained in ${\rm SO}^{+}(L)$, 
we decided to work with ${\On}$ at the outset. 
It is not difficult to switch to ${\rm SO}(n, {\C})$ (see \S \ref{ssec: SO}). 

\subsection*{Fourier expansion (\S \ref{sec: vector-valued})} 

A first basic point is that ${\Elk}$ can be trivialized for each $0$-dimensional cusp of ${\D}$ in a natural way. 
Let $I$ be a rank $1$ primitive isotropic sublattice of $L$, which corresponds to a $0$-dimensional cusp of ${\D}$. 
The quotient lattice $I^{\perp}/I$ is naturally endowed with a hyperbolic quadratic form. 
Then we have isomorphisms 
\begin{equation*}
{\ICv}\otimes{\OD}\to {\LL}, \qquad (I^{\perp}/I)_{{\C}}\otimes{\OD}\to {\E}, 
\end{equation*}
canonically associated to $I$. 
If we write $V(I)_{\lambda,k}=((I^{\perp}/I)_{{\C}})_{\lambda}\otimes (I_{{\C}}^{\vee})^{\otimes k}$, 
these induce an isomorphism 
\begin{equation*}
V(I)_{\lambda,k}\otimes{\OD} \to {\Elk}, 
\end{equation*}
which we call the \textit{$I$-trivialization} of ${\Elk}$. 
Via this trivialization, modular forms of weight $(\lambda, k)$ are identified with 
$V(I)_{\lambda,k}$-valued holomorphic functions $f$ on ${\D}$ satisfying invariance with the factor of automorphy. 
Then, after taking the tube domain realization of ${\D}$ associated to $I$ (\cite{PS}), 
we obtain the Fourier expansion of $f$ of the form 
\begin{equation}\label{eqn: Fourier intro}
f(Z) = \sum_{l\in {\UIZZ}}a(l) \exp(2\pi i(l, Z))), \qquad Z\in {\DI},   
\end{equation}
where ${\DI}$ is the tube domain in $(I^{\perp}/I)_{{\C}}\otimes I_{{\C}}$,  
${\UIZZ}$ is a certain lattice in $(I^{\perp}/I)_{{\Q}}\otimes I_{{\Q}}$, and $a(l)\in V(I)_{\lambda,k}$. 
By the Koecher principle, the index vectors $l$ range only over the intersection of ${\UIZZ}$ with the closure of the positive cone 
(a connected component of the locus of vectors of positive norm). 
We prove that the constant term $a(0)$ always vanishes unless $\lambda = (0), (1^n)$, 
which correspond to the trivial and the determinant characters respectively. 
(In what follows, we write $\lambda = 1, \det$ instead.) 
Therefore the Siegel operators are interesting only at the $1$-dimensional cusps. 
We can speak of rationality of the Fourier coefficients $a(l)$ because ${\VIlk}$ has a natural ${\Q}$-structure. 

In this way, the choice of a $0$-dimensional cusp $I$ determines a passage to a more classical style of defining modular forms. 
Since there is no distinguished $0$-dimensional cusp for a general arithmetic group ${\G}$, 
we need to treat all $0$-dimensional cusps equally. 
Even after the $I$-trivialization, it is more suitable to have $V(I)_{\lambda,k}$ as the \textit{canonical} space of values, 
rather than identifying it with ${\C}^N$ by choosing a basis. 
This approach enables us to develop various later constructions in an intrinsic and coherent way (and so in a full generality) 
without sacrificing the classical style. 

These most basic parts of the theory are developed in \S \ref{sec: L and E} and \S \ref{sec: vector-valued}. 
%
In \S \ref{sec: pullback}, as a functorial aspect of the theory, 
we study pullback and quasi-pullback of vector-valued modular forms to sub orthogonal modular varieties. 
This type of operations are sometimes called the \textit{Witt operators}. 
The consideration of pullbacks leads to an elementary vanishing theorem for $M_{\lambda,k}({\G})$ in $k\leq 0$ (Proposition \ref{prop: vanish k<0}). 
We prove that the quasi-pullback produces \textit{cusp} forms (Proposition \ref{prop: quasi-pullback}), 
generalizing a result of Gritsenko-Hulek-Sankaran \cite{GHS} in the scalar-valued case. 

After these foundational parts, this monograph is developed in the following two directions: 
\begin{enumerate}
\item Geometric treatment of the Siegel operators and the Fourier-Jacobi expansions at $1$-dimensional cusps 
(\S \ref{sec: cano exte} -- \S \ref{sec: VT I}). 
\item Square integrability of modular forms (\S \ref{sec: L2} -- \S \ref{sec: VT II}). 
\end{enumerate}
Both lead, as applications, to vanishing theorems of respective type for modular forms of small weight.

\subsection*{Siegel operator (\S \ref{sec: Siegel})} 

Let $J$ be a rank $2$ primitive isotropic sublattice of $L$. 
This corresponds to a $1$-dimensional cusp ${\HJ}$ of ${\D}$, which is isomorphic to the upper half plane. 
We take a geometric approach for introducing and studying the Siegel operator and the Fourier-Jacobi expansion at the cusp ${\HJ}$,  
by using the partial toroidal compactification over ${\HJ}$. 
The Siegel operator is the restriction to the boundary divisor, and the Fourier-Jacobi expansion is the Taylor expansion along it. 

The Siegel domain realization of ${\D}$ with respect to $J$ (\cite{PS}) is a two-step fibration  
\begin{equation*}
{\D} \stackrel{\pi_{1}}{\to} {\VJ} \stackrel{\pi_{2}}{\to} {\HJ}, 
\end{equation*}
where $\pi_{1}$ is a fibration of upper half planes and $\pi_{2}$ is an affine space bundle. 
Dividing ${\D}$ by a rank $1$ abelian group ${\UJZ}<{\G}$, 
the quotient ${\XJ}={\D}/{\UJZ}$ is a fibration of punctured discs over ${\VJ}$. 
The partial toroidal compactification ${\XJ}\hookrightarrow {\XJcpt}$ is obtained by filling the origins of the punctured discs (\cite{AMRT}). 
Its boundary divisor $\Delta_{J}$ is naturally identified with ${\VJ}$. 
We can extend ${\Elk}$ to a vector bundle over ${\XJcpt}$ via the $I$-trivialization for an arbitrary $0$-dimensional cusp $I\subset J$,  
the result being independent of $I$ (\S \ref{ssec: cano exte 1dim cusp}). 
This is an explicit form of Mumford's canonical extension \cite{Mu} which is suitable for dealing with the Fourier-Jacobi expansion. 
If $f$ is a ${\G}$-modular form of weight $(\lambda, k)$, 
it extends to a holomorphic section of the extended bundle ${\Elk}$ over ${\XJcpt}$. 

Intuitively (and more traditionally), 
the Siegel operator should be an operation of ``restriction to ${\HJ}$" 
which produces vector-valued modular forms of some reduced weight on ${\HJ}$. 
Geometrically this requires some justification because of the complicated structure around the boundary of the Baily-Borel compactification. 
We take a somewhat indirect but more geometrically tractable approach, 
working with the automorphic vector bundle ${\Elk}$ over the partial toroidal compactification ${\XJcpt}$.  

Let ${\LJ}$ be the Hodge line bundle on ${\HJ}$. 
We write $V(J)=(J^{\perp}/J)_{{\C}}$. 
For $\lambda=(\lambda_{1}\geq \cdots \geq \lambda_{n})$ we denote by 
$V(J)_{\lambda'}$ the irreducible representation of ${\rm O}(V(J))\simeq {\rm O}(n-2, {\C})$ 
for the partition $\lambda'=(\lambda_{2}\geq \cdots \geq \lambda_{n-1})$. 

\begin{theorem}[Theorem \ref{thm: Siegel operator}]\label{thm: Siegel operator intro}
Let $\lambda\ne 1, \det$. 
There exists a sub vector bundle ${\E}_{\lambda,k}^{J}$ of ${\Elk}$ such that 
${\E}_{\lambda,k}^{J}|_{\Delta_{J}}\simeq \pi_{2}^{\ast}{\LL}_{J}^{\otimes k+\lambda_{1}}\otimes V(J)_{\lambda'}$ 
and that the restriction of every modular form $f$ of weight $(\lambda, k)$ to $\Delta_{J}$ takes values in 
${\E}_{\lambda,k}^{J}|_{\Delta_{J}}$. 
In particular, there exists a $V(J)_{\lambda'}$-valued cusp form $\Phi_{J}f$ 
of weight $k+\lambda_{1}$ on ${\HJ}$ such that $f|_{\Delta_{J}}=\pi_{2}^{\ast}(\Phi_{J}f)$. 
\end{theorem}

The map 
\begin{equation*}
{\MG} \to S\!_{k+\lambda_{1}}(\Gamma_{J})\otimes V(J)_{\lambda'}, \qquad f\mapsto \Phi_{J}f, 
\end{equation*}
is the Siegel operator at the $J$-cusp, 
where $\Gamma_{J}$ is a suitable subgroup of ${\rm SL}(J)\simeq {\rm SL}(2, {\Z})$. 
If we take the $I$-trivialization for a $0$-dimensional cusp $I\subset J$
and introduce suitable coordinates $(\tau, z, w)$ on the tube domain in which the Siegel domain realization is given by 
$(\tau, z, w)\mapsto (\tau, z) \mapsto \tau$, 
the Siegel operator can be expressed as 
\begin{equation*}
(\Phi_{J}f)(\tau) = \lim_{t\to \infty} f(\tau, 0, it), \qquad \tau\in {\HH}. 
\end{equation*} 
In this way, the naive "restriction to ${\HJ}$" can be geometrically justified at the level of automorphic vector bundles as the combined operation 
\begin{equation*}
\textrm{restrict to} \: \Delta_{J} \: \: + \: \: \textrm{reduce to} \: {\E}_{\lambda,k}^{J} \: \: + \: \: \textrm{descend to} \: {\HJ}. 
\end{equation*}
This a priori tells us the modularity of $\Phi_{J}f$ with its weight. 
When $n=3$, the weight calculation in Theorem \ref{thm: Siegel operator intro} agrees with 
the corresponding result for Siegel modular forms of genus $2$ (\cite{We}, \cite{Ar}). 
The sub vector bundle ${\E}_{\lambda,k}^{J}$ will be taken up in \S \ref{sec: filtration} again 
from the viewpoint of a filtration on ${\Elk}$.

\subsection*{Fourier-Jacobi expansion (\S \ref{sec: FJ})}
 
Next we explain the Fourier-Jacobi expansion at the $J$-cusp. 
Let $\Theta_{J}$ be the conormal bundle of $\Delta_{J}$ in ${\XJcpt}$. 
After certain choices, we have a special generator $\omega_{J}$ of the ideal sheaf of $\Delta_{J}$. 
With this normal coordinate, we can take the Taylor expansion of a modular form 
$f\in {\MG}$ along $\Delta_{J}$ as a section of the extended bundle ${\Elk}$: 
\begin{equation}\label{eqn: FJ intro}
f = \sum_{m\geq 0} \phi_{m} \omega_{J}^{m}. 
\end{equation}
The $m$-th Taylor coefficient $\phi_{m}$, or rather $\phi_{m}\otimes \omega_{J}^{\otimes m}$, 
is essentially a section of the vector bundle ${\Elk}\otimes \Theta_{J}^{\otimes m}$ over $\Delta_{J}$. 
We call \eqref{eqn: FJ intro} the \textit{Fourier-Jacobi expansion} of $f$ at the $J$-cusp, 
and call the section $\phi_{m}\otimes \omega_{J}^{\otimes m}$ of ${\Elk}\otimes \Theta_{J}^{\otimes m}$ 
for $m>0$ the $m$-th \textit{Fourier-Jacobi coefficient} of $f$. 
($\phi_{0}$ is just $f|_{\Delta_{J}}$ considered above.) 
Although the choice of $\omega_{J}$ is needed for defining the Fourier-Jacobi expansion, 
the resulting expansion and the sections of ${\Elk}\otimes \Theta_{J}^{\otimes m}$ 
are independent of this choice, thus canonically determined by $J$ (\S \ref{ssec: geometry FJ}). 
This geometric definition of Fourier-Jacobi expansion, 
whose advantage is its canonicity, agrees with 
the more familiar style of defining Fourier-Jacobi expansion by slicing the Fourier expansion (\S \ref{ssec: FJ expansion}) 
if we take the $(I, \omega_{J})$-trivialization. 

In general, we define \textit{vector-valued Jacobi forms} of weight $(\lambda, k)$ and index $m>0$  
as holomorphic sections of ${\Elk}\otimes \Theta_{J}^{\otimes m}$ over $\Delta_{J}={\VJ}$ 
which is invariant under the integral Jacobi group and satisfies a certain cusp condition 
(Definition \ref{def: Jacobi form}). 
The $m$-th Fourier-Jacobi coefficient of a modular form of weight $(\lambda, k)$ 
is such a vector-valued Jacobi form (Proposition \ref{prop: FJ expansion = Taylor expansion II}). 
In the scalar-valued case, our geometric definition agrees with the classical definition of Jacobi forms (\cite{Sk}, \cite{Gr}) 
after introducing suitable coordinates and trivialization (\S \ref{ssec: classical Jacobi form}).  
When $n=3$, our vector-valued Jacobi forms essentially agree with those considered by Ibukiyama-Kyomura \cite{IK} 
for Siegel modular forms of genus $2$.

\subsection*{Filtrations associated to $1$-dimension cusps (\S \ref{sec: filtration})}

While a $0$-dimensional cusp of ${\D}$ provides a trivialization of ${\Elk}$ which enables the Fourier expansion, 
we will show that a $1$-dimensional cusp introduces a filtration on ${\Elk}$ which is useful when studying the Fourier-Jacobi expansion. 
To start with, we observe that for each $1$-dimensional cusp $J$, 
the second Hodge bundle ${\E}$ has an isotropic sub line bundle ${\EJ}$ canonically determined by $J$. 
This defines the filtration 
\begin{equation*}
0 \subset {\EJ} \subset {\EJp} \subset {\E} 
\end{equation*}
associated to the $J$-cusp, which we call the \textit{$J$-filtration}. 
Its graded quotients are respectively isomorphic to 
\begin{equation*}
{\EJ}\simeq \pi^{\ast}{\LJ}, \qquad {\EJp}/{\EJ}\simeq (J^{\perp}/J)_{{\C}}\otimes {\OD}, \qquad {\E}/{\EJp}\simeq \pi^{\ast}{\LL}_{J}^{-1}, 
\end{equation*}
where $\pi=\pi_{2}\circ \pi_{1}$ is the projection from ${\D}$ to ${\HJ}$. 
The $J$-filtration is translated to a constant filtration on $V(I)\otimes {\OD}$ by the $I$-trivialization 
for every adjacent $0$-dimensional cusp $I\subset J$ (Proposition \ref{prop: J-filtration under I-trivialization}). 

The $J$-filtration on ${\E}$ induces a (decreasing) filtration on a general automorphic vector bundle ${\Elk}$,  
also called the $J$-filtration, whose graded quotient in level $r$ is isomorphic to a direct sum of copies of $\pi^{\ast}{\LL}_{J}^{\otimes k+r}$. 
Representation-theoretic calculations show that the $J$-filtration on ${\Elk}$ has length $\leq 2\lambda_{1}+1$ 
(from level $-\lambda_{1}$ to $\lambda_{1}$), 
and that the sub vector bundle $\mathcal{E}_{\lambda,k}^{J}$ of ${\Elk}$ in Theorem \ref{thm: Siegel operator intro} 
is exactly the last ($=$ level $\lambda_{1}$) sub vector bundle in the $J$-filtration (Proposition \ref{cor: J-filtration and representation}). 
Moreover, we have a duality between the graded quotients in level $r$ and $-r$. 

We give two applications of the $J$-filtration. 
The first is decomposition of vector-valued Jacobi forms. 
We prove that a vector-valued Jacobi form of weight $(\lambda, k)$ decomposes, in a certain sense, into 
a tuple of scalar-valued Jacobi forms of various weights in the range $[ k-\lambda_{1}, k+\lambda_{1}]$ 
(Proposition \ref{prop: decompose Jacobi form}). 
More precisely, what is proved is that certain graded pieces are scalar-valued Jacobi forms, 
so this result does not mean that the theory of vector-valued Jacobi forms reduces to the scalar-valued theory. 
Nevertheless this decomposition theorem enables us to 
derive some basic results for vector-valued Jacobi forms from those for scalar-valued ones. 
For example, we deduce that vector-valued Jacobi forms of weight $(\lambda, k)$ with $k+\lambda_{1}<n/2-1$ vanish 
(Corollary \ref{cor: Jacobi form vanish}). 
In the case of Siegel modular forms of genus $2$ (namely $n=3$), 
the fact that vector-valued Jacobi forms decompose into scalar-valued Jacobi forms 
was first found by Ibukiyama and Kyomura \cite{IK}. 
Their method is different, using differential operators, 
but it might be plausible that their decomposition agrees with that of us.

\subsection*{Vanishing theorem I (\S \ref{sec: VT I})}

It is a classical fact that there is no nonzero scalar-valued modular form of weight $0<k<n/2-1$ on ${\D}$. 
Two proofs of this fact are well-known. 
The first uses vanishing of Jacobi forms (cf.~\cite{Gr}, \cite{Sk}), and the second uses classification of unitary representations. 
We give two generalizations of this classical vanishing theorem to the vector-valued case, 
corresponding to these two approaches. 

Our first vanishing theorem belongs to the Jacobi form approach, 
and is obtained as the second application of the $J$-filtration. 
We assume that $L$ has Witt index $2$, i.e., ${\D}$ has a $1$-dimensional cusp. 
This is always satisfied when $n\geq 5$. 

\begin{theorem}[Theorem \ref{thm: VT I}]\label{thm: VT I intro}
Let $\lambda\ne 1, \det$. 
If $k<\lambda_{1}+n/2-1$, then ${\MG}=0$. 
In particular, ${\MG}=0$ whenever $k<n/2$. 
\end{theorem}

As a consequence, we obtain the following vanishing theorem for holomorphic tensors on the modular variety ${\G}\backslash {\D}$. 

\begin{corollary}[Theorem \ref{thm: hol tensor}]\label{cor: hol tensor intro}
Let $X$ be the regular locus of ${\G}\backslash {\D}$. 
Then we have $H^{0}(X, (\Omega_{X}^{1})^{\otimes k})=0$ for all $0<k<n/2-1$. 
\end{corollary}

Moreover, we obtain a classification of possible types of holomorphic tensors 
of the next few degrees up to $n/2$ (Proposition \ref{prop: hol tensor smallest}). 
The vanishing bound $k<n/2-1$ is optimal as a general bound.

The proof of Theorem \ref{thm: VT I intro} is built on the results of \S \ref{sec: FJ} and \S \ref{sec: filtration}, and proceeds as follows. 
We apply the classical vanishing theorem of scalar-valued Jacobi forms of weight $<n/2-1$ (\cite{Sk}, \cite{Gr}) 
to the first graded quotient of the $J$-filtration on ${\Elk}$. 
This implies that the Fourier-Jacobi coefficients of $f\in {\MG}$  
take values in a certain sub vector bundle of ${\Elk}\otimes \Theta_{J}^{\otimes m}$. 
Passing to the Fourier expansion at $I\subset J$, 
we see that the Fourier coefficients of $f$ are contained in a proper subspace of $V(I)_{\lambda,k}$. 
Finally, running $J$ over all $1$-dimensional cusps containing $I$, we conclude that the Fourier coefficients are zero. 


In the case of Siegel modular forms of genus $2$, 
the idea to use Jacobi forms to deduce a vanishing theorem for vector-valued modular forms  
seems to go back to Ibukiyama (\cite{Ibu0} Section 6). 
Our proof of Theorem \ref{thm: VT I intro} can be regarded as a generalization of his argument. 

In this way, we have the unified viewpoint that 
the Siegel operator is concerned with the last sub vector bundle in the $J$-filtration, 
while the proof of Theorem \ref{thm: VT I intro} makes use of the first graded quotient. 
We expect that a closer look at the intermediate pieces of the $J$-filtration would tell us more.

\subsection*{Square integrability (\S \ref{sec: L2})}

We now turn to our second line of investigation. 
We can explicitly define and calculate an invariant Hermitian metric on ${\E}$ (and on ${\LL}$, which is well-known). 
They are essentially the Hodge metrics. 
They induce an invariant Hermitian metric $(\: , \: )_{\lambda,k}$ on a general automorphic vector bundle ${\Elk}$. 
Apart from the matter of convergence, this defines the Petersson inner product on ${\MG}$: 
\begin{equation*}
(f, g) = \int_{{\G}\backslash {\D}} (f , g)_{\lambda,k} {\volD}, \qquad f, g\in {\MG}, 
\end{equation*}
where ${\volD}$ is the invariant volume form on ${\D}$. 
When $f$ or $g$ is a cusp form, this integral converges as usual. 
Conversely, we prove the following. 
Let 
\begin{equation*}
\bar{\lambda}=(\bar{\lambda}_{1}, \cdots, \bar{\lambda}_{[n/2]}) = 
(\lambda_{1}-\lambda_{n}, \lambda_{2}-\lambda_{n-1}, \cdots, \lambda_{[n/2]}-\lambda_{n+1-[n/2]}) 
\end{equation*}
be the highest weight for ${\rm SO}(n, {\C})$ associated to $\lambda$. 
We write $|\bar{\lambda}|=\sum_{i}\bar{\lambda}_{i}$. 

\begin{theorem}[Theorem \ref{thm: L2}]
Let $\lambda \ne 1, \det$ and assume that $k\geq n+|\bar{\lambda}|-1$. 
Then a modular form $f$ of weight $(\lambda, k)$ is a cusp form if and only if $(f, f)<\infty$. 
\end{theorem}

This holds also for $\lambda=1, \det$ at least when $L$ has Witt index $2$ (Remark \ref{rmk: L2 scalar-valued}). 
In fact, Theorem \ref{thm: L2} contains one more result that 
any modular form of weight $(\lambda, k)$ with $k< n-|\bar{\lambda}|-1$ and $\lambda\ne 1, \det$ is square integrable, 
but this is rather an intermediate step in the proof of our second vanishing theorem. 


\subsection*{Vanishing theorem II (\S \ref{sec: VT II})}

Our study of square integrability is partly motivated by the following vanishing theorem. 
Let ${\rm corank}(\lambda)$ be the maximal index $1\leq i\leq [n/2]$ such that 
$\bar{\lambda}_{1}=\bar{\lambda}_{2}= \cdots = \bar{\lambda}_{i}$. 
Let ${\SG} \subset {\MG}$ be the subspace of cusp forms.  

\begin{theorem}[Theorem \ref{thm: VT II}]\label{thm: VT II intro}
Let $\lambda\ne 1, \det$. 
If $k<n+\lambda_{1}-{\rm corank}(\lambda)-1$, 
there is no nonzero square integrable modular form of weight $(\lambda, k)$.  
In particular, 

(1) ${\SG}=0$ if $k<n+\lambda_{1}-{\rm corank}(\lambda)-1$. 

(2) ${\MG}=0$ if $k<n-|\bar{\lambda}|-1$. 
\end{theorem}

Although $\lambda_{1}+n/2-1<n+\lambda_{1}-{\rm corank}(\lambda)-1$, 
Theorem \ref{thm: VT II intro} does not supersede Theorem \ref{thm: VT I intro} because 
it is about square integrable modular forms. 
It depends on $(\lambda, k)$ which bound in Theorem \ref{thm: VT I intro} or Theorem \ref{thm: VT II intro} (2) is larger. 
The two vanishing theorems are rather complementary. 

The proof of Theorem \ref{thm: VT II intro} is parallel to 
Weissauer's vanishing theorem \cite{We} for Siegel modular forms. 
If we have a square integrable modular form, 
we can construct a unitary highest weight module for ${\SOLR}$ by a standard procedure. 
Then the bound $k<n+\lambda_{1}-{\rm corank}(\lambda)-1$ is derived from  
the classification of unitary highest weight modules (\cite{EP}, \cite{EHW}, \cite{Ja}). 
The more specific conclusions (1), (2) are consequences of the square integrability theorem (Theorem \ref{thm: L2}).

\subsection*{Organization}

The logical dependence between the chapters is as follows. 
A dotted arrow means that the dependence is weak. 

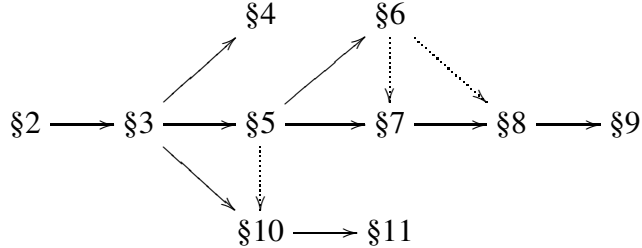
\begin{equation*}
\xymatrix{
 &  & \S \ref{sec: pullback} & \S \ref{sec: Siegel} \ar@{.>}[d] \ar@{.>}[rd] &  & \\ 
\S \ref{sec: L and E} \ar[r] & \S \ref{sec: vector-valued} \ar[r] \ar[ur] \ar[dr] & 
\S \ref{sec: cano exte} \ar[r] \ar[ur] \ar@{.>}[d] & \S \ref{sec: FJ} \ar[r] & \S \ref{sec: filtration} \ar[r] &  \S \ref{sec: VT I} \\ 
 &   &  \S \ref{sec: L2} \ar[r] & \S \ref{sec: VT II}  &   
}
\end{equation*}

\subsection*{Notations}

Let us summarize some standing terminologies and notations.  

(1) 
By a \textit{lattice} we mean a free ${\Z}$-module $L$ of finite rank equipped with 
a nondegenerate symmetric bilinear form $(\cdot, \cdot):L\times L \to {\Z}$. 
(Sometimes we still use the word "lattice" when the bilinear form is only ${\Q}$-valued.) 
The dual lattice $\hom(L, {\Z})$ of $L$ is written as $L^{\vee}$. 
A sublattice $M$ of $L$ is called \textit{primitive} if $L/M$ is free. 
We denote by $M^{\perp}$ the orthogonal complement of $M$ in $L$. 
A sublattice $I$ of $L$ is called \textit{isotropic} if $(I, I)\equiv 0$. 
The lattice $L$ is called an \textit{even lattice} if $(l, l)\in 2{\Z}$ for every $l\in L$.  
The orthogonal group of a lattice $L$ is denoted by ${\rm O}(L)$. 
For $F={\Q}, {\R}, {\C}$ we write $L_{F}=L\otimes_{{\Z}}F$. 
This is a quadratic space over $F$. 
Its orthogonal group is denoted by ${\rm O}(L_F)$. 
The special orthogonal group, namely the subgroup of ${\rm O}(L_F)$ of determinant $1$, is denoted by ${\rm SO}(L_F)$. 
A lattice $L$ in a ${\Q}$-quadratic space $V$ is called a \textit{full lattice} in $V$ if $V=L_{{\Q}}$. 
For a rational number $\alpha\ne 0$ we write $L(\alpha)$ for the $\alpha$-scaling of $L$, 
namely the same underlying ${\Z}$-module with the bilinear form multiplied by $\alpha$. 
In the context of lattices, the symbol $U$ will stand for the integral hyperbolic plane, 
namely the even unimodular lattice of signature $(1, 1)$. 

(2) 
Let $G$ be a group acting on a set $X$ and let $Y$ be a subset of $X$. 
By the \textit{stabilizer} of $Y$ in $G$, we mean the subgroup of $G$ consisting of elements $g$ such that $g(Y)=Y$. 

(3) 
Let $V$ be a nondegenerate quadratic space over $F={\Q}, {\R}, {\C}$. 
Let $I$ be an isotropic line in $V$, and $P(I)$ be the stabilizer of $I$ in ${\rm O}(V)$. 
Then we have the canonical exact sequence 
\begin{equation}\label{eqn: stabilizer isotropic line}
0 \to (I^{\perp}/I)\otimes_{F} I \to P(I) \to {\rm GL}(I) \times {\rm O}(I^{\perp}/I) \to 1. 
\end{equation}
Here $P(I) \to {\rm GL}(I)$ and $P(I) \to {\rm O}(I^{\perp}/I)$ are the natural maps, 
and the map $(I^{\perp}/I)\otimes_{F} I \to P(I)$ sends a vector $m\otimes l$ of $(I^{\perp}/I)\otimes_{F} I$ to 
the isometry $E_{m\otimes l}$ of $V$ defined by 
\begin{equation}\label{eqn: Eichler}
E_{m\otimes l}(v) = v - (\tilde{m}, v)l + (l, v)\tilde{m} - \frac{1}{2}(m, m)(l, v)l, \qquad v\in V. 
\end{equation}
Here $\tilde{m}\in I^{\perp}$ is a lift of $m\in I^{\perp}/I$. 
In particular, when $v\in I^{\perp}$, \eqref{eqn: Eichler} is simplified to  
\begin{equation*}
E_{m\otimes l}(v) = v - (m, v)l. 
\end{equation*}
The isometries $E_{m\otimes l}$ are sometimes called the \textit{Eichler transvections}. 
If we take a basis $e_{1}, \cdots, e_{n}$ of $V$ such that 
$I=\langle e_{1} \rangle$, $I^{\perp}=\langle e_{1}, \cdots, e_{n-1} \rangle$ and 
$(e_{1}, e_{n})=1$, $(e_{i}, e_{n})=0$ for $i>1$, 
then $E_{m\otimes e_{1}}$ is expressed by the matrix 
\begin{equation*}
\begin{pmatrix} 
1 & -m^{\vee} & -(m, m)/2 \\ 0 & I_{n-2} & m \\ 0 & 0 & 1 
\end{pmatrix} 
\end{equation*}
where we regard $m\in \langle e_{2}, \cdots, e_{n-1}\rangle \simeq I^{\perp}/I$. 
The group $(I^{\perp}/I)\otimes_{F} I$ of Eichler transvections is the unipotent radical of $P(I)$. 

(4) 
We will not distinguish between vector bundles and locally free sheaves on a complex manifold $X$. 
The fiber of a vector bundle $\mathcal{F}$ over a point $x\in X$ is denoted by $\mathcal{F}_{x}$ (not the germ of the sheaf). 
A collection of sections of a vector bundle $\mathcal{F}$ is called a \textit{frame} of $\mathcal{F}$ 
when it defines an isomorphism $\mathcal{O}_{X}^{\oplus r}\simeq \mathcal{F}$, 
i.e., it forms a basis in every fiber. 
The dual vector bundle of $\mathcal{F}$ is denoted by $\mathcal{F}^{\vee}$. 

(5)
Let $X$ be a complex manifold and $G$ be a group acting on $X$. 
Let $\mathcal{F}$ be a $G$-equivariant vector bundle on $X$. 
Suppose that $\mathcal{F}$ is endowed with an isomorphism $\iota \colon V\otimes \mathcal{O}_{X}\to \mathcal{F}$ 
for a ${\C}$-linear space $V$. 
Then the \textit{factor of automorphy} of the $G$-action on $\mathcal{F}$ with respect to the trivialization $\iota$ 
is the ${\rm GL}(V)$-valued function on $G\times X$ defined by 
\begin{equation}\label{eqn: f.a.}
j(g, x) = \iota^{-1}_{gx} \circ g \circ \iota_{x} : \: \: 
V\to \mathcal{F}_{x} \to \mathcal{F}_{gx} \to V 
\end{equation}
for $g\in G$, $x\in X$. 
Here the middle map is the equivariant action by $g$. 
If ${\G}$ is a subgroup of $G$, a ${\G}$-invariant section of $\mathcal{F}$ over $X$ is identified via $\iota$ with 
a $V$-valued holomorphic function $f$ on $X$ satisfying 
$f(\gamma x)=j(\gamma, x)f(x)$ for every $\gamma \in {\G}$ and $x\in X$.

(6) 
We write $e(z)=\exp(2\pi i z)$ for $z\in {\C}/{\Z}$. 
We use the symbol ${\HH}$ for the upper half plane $\{ \tau \in {\C} \: | \: {\rm Im}(\tau)>0 \}$.

\vspace{0.3cm}
\noindent
\textit{Acknowledgements.} 
It is our pleasure to thank Eberhard Freitag and Tomoyoshi Ibukiyama 
for sharing with us their perspectives on vector-valued modular forms. 
This work is supported by KAKENHI 21H00971 and 20H00112.


\chapter{The two Hodge bundles}\label{sec: L and E}

In this chapter we study some basic properties of the Hodge bundles ${\LL}$ and ${\E}$. 
In \S \ref{ssec: domain} we recall basic facts on the Hermitian symmetric domains of type IV. 
The Hodge line bundle ${\LL}$ is well-known, and we recall it in \S \ref{ssec: L}. 
In \S \ref{ssec: E} and \S \ref{ssec: trivialize} we study the second Hodge bundle ${\E}$. 
In \S \ref{ssec: accidental} we describe ${\E}$ and ${\LL}$ in the case $n\leq 4$ under the accidental isomorphisms. 

\section{The domain}\label{ssec: domain}

Let $L$ be a lattice of signature $(2, n)$. 
Let $Q=Q_{L}$ be the isotropic quadric in ${\proj}L_{{\C}}$ defined by the equation $(\omega, \omega)=0$ for $\omega\in L_{{\C}}$. 
We express a point of $Q$ as $[\omega]={\C}\omega$. 
The open set of $Q$ defined by the inequality $(\omega, \bar{\omega})>0$ has two connected components. 
They are interchanged by the complex conjugation $\omega\mapsto \bar{\omega}$. 
We choose one of them and denote it by ${\D}={\D}_{L}$. 
This is the Hermitian symmetric domain attached to $L$. 
In Cartan's classification, ${\D}$ is a Hermitian symmetric domain of type IV. 
The isotropic quadric $Q$ is the compact dual of ${\D}$. 
Points of ${\D}$ are in one-to-one correspondence with positive-definite planes in $L_{{\R}}$, 
by associating 
\begin{equation*}
{\D} \ni [\omega] \: \: \mapsto \: \: H_{\omega}=\langle {\rm Re}(\omega), {\rm Im}(\omega) \rangle. 
\end{equation*}
The choice of the component ${\D}$ determines orientation on the positive-definite planes. 
Note that $({\rm Re}(\omega), {\rm Im}(\omega))=0$ and 
$({\rm Re}(\omega), {\rm Re}(\omega))=({\rm Im}(\omega), {\rm Im}(\omega))$ 
by the isotropicity condition $(\omega, \omega)=0$. 

We denote by ${\OLR}$ the index $2$ subgroup of ${\rm O}(L_{{\R}})$ preserving the component ${\D}$. 
Then ${\OLR}$ consists of two connected components, 
the identity component being 
${\rm SO}^{+}(L_{{\R}})={\OLR}\cap {\rm SO}(L_{{\R}})$. 
The stabilizer $K$ of a point $[\omega]\in {\D}$ in ${\OLR}$ is the same as the stabilizer of the oriented plane $H_{\omega}$, 
and is described as  
\begin{equation*}
K={\rm SO}(H_{\omega})\times {\rm O}(H_{\omega}^{\perp})\simeq {\rm SO}(2, {\R})\times {\rm O}(n, {\R}). 
\end{equation*}
This is a maximal compact subgroup of ${\OLR}$. 
We have ${\D}\simeq {\OLR}/K$. 
On the other hand, as explained in \eqref{eqn: stabilizer isotropic line}, 
the stabilizer $P$ of $[\omega]$ in ${\OLC}$ sits in the canonical exact sequence 
\begin{equation}\label{eqn: stabilzier C}
0 \to (\omega^{\perp}/{\C}\omega)\otimes {\C}\omega \to P \to {\rm GL}({\C}\omega)\times {\rm O}(\omega^{\perp}/{\C}\omega) \to 1. 
\end{equation}
The reductive part 
\begin{equation*}
{\rm GL}({\C}\omega)\times {\rm O}(\omega^{\perp}/{\C}\omega) \simeq {\C}^{\ast}\times {\On} 
\end{equation*}
is the complexification of $K$. 

The domain ${\D}$ has two types of rational boundary components (cusps): 
$0$-dimensional and $1$-dimensional cusps. 
The $0$-dimensional cusps correspond to rational isotropic lines in $L_{{\Q}}$, or equivalently, 
rank $1$ primitive isotropic sublattices $I$ of $L$. 
The point $p_{I}=[I_{{\C}}]$ of $Q$ is in the closure of ${\D}$, 
and this is the $0$-dimensional cusp corresponding to $I$. 
The $1$-dimensional cusps correspond to rational isotropic planes in $L_{{\Q}}$, or equivalently, 
rank $2$ primitive isotropic sublattices $J$ of $L$. 
Each such $J$ determines the line ${\proj}J_{{\C}}$ on $Q$. 
If we remove ${\proj}J_{{\R}}$ from ${\proj}J_{{\C}}$, 
then ${\proj}J_{{\C}}-{\proj}J_{{\R}}$ consists of two copies of the upper half plane, 
one in the closure of ${\D}$. 
This component, say ${\HJ}$, is the $1$-dimensional cusp corresponding to $J$. 
A $0$-dimensional cusp $p_{I}$ is in the closure of a $1$-dimensional cusp ${\HJ}$ if and only if $I\subset J$. 

Let ${\OL}={\rm O}(L)\cap {\OLR}$ and ${\G}$ be a finite-index subgroup of ${\OL}$. 
By Baily-Borel \cite{BB}, the quotient space 
\begin{equation*}
\mathcal{F}({\G})^{bb} = {\G} \backslash \left( {\D} \cup \bigcup_{J}{\HJ} \cup \bigcup_{I}p_{I} \right) 
\end{equation*}
has the structure of a normal projective variety of dimension $n$. 
Here the union of ${\D}$ and the cusps is equipped with the so-called Satake topology. 
In particular, the quotient 
\begin{equation*}
{\FG}={\G}\backslash {\D} 
\end{equation*}
is a normal quasi-projective variety. 
The variety $\mathcal{F}({\G})^{bb}$ is called the \textit{Baily-Borel compactification} of ${\FG}$.

\section{The Hodge line bundle}\label{ssec: L}

In this section we recall the first Hodge bundle. 
Let $\mathcal{O}_{Q}(-1)$ be the tautological line bundle over $Q\subset {\proj}L_{{\C}}$. 
The Hodge line bundle over ${\D}$ is defined as 
\begin{equation*}
{\LL} = \mathcal{O}_{Q}(-1)|_{{\D}}. 
\end{equation*}
This is an ${\OLR}$-invariant sub line bundle of $L_{{\C}}\otimes {\OD}$. 
The fiber of ${\LL}$ over $[\omega]\in {\D}$ is the line ${\C}\omega$. 
By definition ${\LL}$ extends over $Q$ naturally, and we sometimes write ${\LL}=\mathcal{O}_{Q}(-1)$ 
when no confusion is likely to occur. 
A holomorphic section of ${\LL}^{\otimes k}$ over ${\D}$ invariant under a finite-index subgroup of ${\OL}$ 
and holomorphic at the cusps (in the sense explained later) is called a (scalar-valued) modular form of weight $k$. 

The stabilizer $K\subset {\OLR}$ of a point $[\omega]\in {\D}$ acts on the fiber ${\LL}_{[\omega]}$ of ${\LL}$  
as the weight $1$ character of ${\rm SO}(2, {\R})\subset K$. 
Therefore, if we denote by $W \simeq {\C}$ the representation space of the weight $1$ character of ${\rm SO}(2, {\R})$, 
we have an ${\OLR}$-equivariant isomorphism 
\begin{equation*}
{\LL}\simeq {\OLR}\times_{K} {\LL}_{[\omega]} \simeq {\OLR}\times_{K} W.  
\end{equation*}
Similarly, the extension $\mathcal{O}_{Q}(-1)$ over $Q$ is the homogeneous line bundle 
corresponding to the weight $1$ character of ${\C}^{\ast}\subset {\C}^{\ast}\times {\On}$. 

A trivialization of ${\LL}$ can be defined for each $0$-dimensional cusp of ${\D}$ as follows. 
Let $I$ be a rank $1$ primitive isotropic sublattice of $L$. 
For later use, it is useful to work over the following enlargement of ${\D}$:    
\begin{equation*}
Q(I) = Q - Q\cap {\proj}{\ICp}. 
\end{equation*}
This is a Zariski open set of $Q$ containing ${\D}$. 
Its complement $Q\cap {\proj}{\ICp}$ 
is the cone over the isotropic quadric in ${\proj}({\Ip}/I)_{{\C}}$ with vertex $[I_{{\C}}]$. 
If $[\omega]\in Q(I)$, the pairing between $I_{{\C}}$ and ${\C}\omega$ is nonzero. 
This defines an isomorphism ${\C}\omega \to I_{{\C}}^{\vee}$. 
Since ${\C}\omega$ is the fiber of ${\LL}=\mathcal{O}_{Q}(-1)$ over $[\omega]$, 
by varying $[\omega]$ we obtain an isomorphism 
\begin{equation}\label{eqn: I-trivialization Q(I)}
{\ICv}\otimes \mathcal{O}_{Q(I)} \to {\LL} 
\end{equation} 
of line bundles on $Q(I)$. 
We call this isomorphism the \textit{$I$-trivialization} of ${\LL}$. 
This is equivariant with respect to the stabilizer of $I_{{\C}}$ in ${\OLC}$. 
Over $Q$ the $I$-trivialization has pole of order $1$ at the divisor $Q\cap {\proj}{\ICp}$, 
and hence extends to an isomorphism 
\begin{equation*}
{\ICv}\otimes \mathcal{O}_{Q} \to {\LL} (Q\cap {\proj}{\ICp}). 
\end{equation*} 

In what follows, we work over ${\D}$. 
We call the restriction of \eqref{eqn: I-trivialization Q(I)} to ${\D}$ the \textit{$I$-trivialization} of ${\LL}$ too. 
If we choose a nonzero vector of ${\ICv}$, it defines a nowhere vanishing section of ${\LL}$ via the $I$-trivialization. 
To be more specific, we choose a vector $l\ne 0 \in I$ and let $s_{l}$ be the section of ${\LL}$ corresponding to the dual vector of $l$. 
This section is determined by the condition that the vector $s_{l}([\omega])\in {\LL}_{[\omega]}={\C}\omega$ has pairing $1$ with $l$. 
The factor of automorphy of the ${\OLR}$-action on ${\LL}$ with respect to the $I$-trivialization is 
a function on ${\OLR}\times {\D}$ which can be written as  
\begin{equation}\label{eqn: f.a. L}
j(g, [\omega]) = \frac{g\cdot s_l([\omega])}{s_l([g\omega])} = \frac{(g\omega, l)}{(\omega, l)}, \qquad 
g\in {\OLR}, \; [\omega] \in {\D}. 
\end{equation}
This gives a more classical style of defining scalar-valued modular forms. 
Note that if $g$ acts trivially on $I_{{\R}}$, then $j(g, [\omega])\equiv 1$.

\section{The second Hodge bundle}\label{ssec: E}

In this section we define the second Hodge bundle. 
We have a natural quadratic form on the vector bundle $L_{{\C}}\otimes{\OD}$. 
By the definition of $Q$, ${\LL}$ is an isotropic sub line bundle of $L_{{\C}}\otimes{\OD}$, 
so we have ${\LL}\subset {\LL}^{\perp}$. 
The second Hodge bundle is defined by 
\begin{equation*}
{\E}={\LL}^{\perp}/{\LL}. 
\end{equation*}
This is an ${\OLR}$-equivariant vector bundle of rank $n$ over ${\D}$. 
The fiber of ${\E}$ over $[\omega]\in{\D}$ is $\omega^{\perp}/{\C}\omega$. 
The quadratic form on $L_{{\C}}\otimes{\OD}$ induces a nondegenerate ${\OLR}$-invariant quadratic form on ${\E}$. 
In other words, ${\E}$ is an orthogonal vector bundle. 
In particular, we have ${\E}^{\vee}\simeq {\E}$. 
Since ${\LL}$ is naturally defined on $Q$, ${\E}$ is also naturally defined on $Q$. 
This is an ${\OLC}$-equivariant vector bundle. 
By abuse of notation, we often use the same notation ${\E}$ for this extended vector bundle. 

The stabilizer $K\subset {\OLR}$ of a point $[\omega]\in {\D}$ acts on the fiber ${\E}_{[\omega]}$ of ${\E}$  
as the standard ${\C}$-representation of ${\rm O}(n, {\R})\subset K$, 
because we have a natural isomorphism $H_{\omega}^{\perp}\otimes_{{\R}}{\C}\simeq \omega^{\perp}/{\C}\omega$. 
Therefore, if we denote by $V={\C}^{n}$ the standard representation space of ${\On}$, 
we have an ${\OLR}$-equivariant isomorphism 
\begin{equation}\label{eqn: E standard rep} 
{\E}\simeq {\OLR}\times_{K}{\E}_{[\omega]} \simeq {\OLR}\times_{K} V.  
\end{equation}
Similarly, the extension of ${\E}$ over $Q$ is the homogeneous vector bundle 
corresponding to the standard representation of ${\On}\subset {\C}^{\ast}\times {\On}$. 

We present some examples where ${\E}$ and ${\LL}$ appear naturally. 

\begin{example}\label{ex: 3rd Hodge}
The ``third'' Hodge bundle $(L_{{\C}}\otimes{\OD})/{\LL}^{\perp}$ is isomorphic to ${\LL}^{-1}$ by the natural pairing with ${\LL}$. 
\end{example}

\begin{example}\label{ex: det}
The determinant line bundle $\det {\E} = \wedge^{n}{\E}$ of ${\E}$ is isomorphic, as an ${\OLR}$-equivariant bundle, 
to the line bundle ${\rm det} \otimes {\OD}$ associated to the determinant character 
$\det \colon {\OLR}\to \{ \pm 1 \}$ of ${\OLR}$. 
Indeed, by Example \ref{ex: 3rd Hodge}, we have the ${\OLR}$-equivariant isomorphism 
\begin{equation*}
\det {\E} \simeq \det(L_{{\C}}\otimes {\OD})\otimes {\LL}\otimes {\LL}^{-1} \simeq \det(L_{{\C}}\otimes {\OD}) \simeq {\rm det} \otimes  {\OD}. 
\end{equation*}
The line bundle ${\rm det} \otimes  {\OD}$ appears in the study of scalar-valued modular forms with determinant character. 
\end{example}

\begin{example}\label{ex: TD}
Let $T_{{\D}}$ and $\Omega_{{\D}}^{1}$ be the tangent and cotangent bundles of ${\D}$ respectively. 
Then we have the canonical isomorphisms 
\begin{equation}\label{eqn: TD}
T_{{\D}} \simeq {\E}\otimes {\LL}^{-1}, \qquad \Omega_{{\D}}^{1} \simeq {\E}\otimes {\LL}. 
\end{equation}
Indeed, by the Euler sequence for ${\proj}L_{{\C}}$, we have 
\begin{equation*}
T_{{\proj}L_{{\C}}} \simeq 
\mathcal{O}_{{\proj}L_{{\C}}}(1)\otimes ( (L_{{\C}}\otimes \mathcal{O}_{{\proj}L_{{\C}}}) / \mathcal{O}_{{\proj}L_{{\C}}}(-1)). 
\end{equation*}
As a sub vector bundle of $T_{{\proj}L_{{\C}}}|_{Q}$, we have 
\begin{equation*}
T_{Q}\simeq \mathcal{O}_{Q}(1)\otimes (\mathcal{O}_{Q}(-1)^{\perp} / \mathcal{O}_{Q}(-1)) = {\LL}^{-1}\otimes {\E}. 
\end{equation*}
The isomorphism for $\Omega_{Q}^{1}$ is obtained by taking the dual. 

Tautologically, the identity of ${\D}$ can be regarded as the period map $[ \omega ] \mapsto {\LL}_{[ \omega ]}$ for 
the universal variation $0\subset {\LL}\subset {\LL}^{\perp} \subset L_{{\C}}\otimes {\OD}$ 
of Hodge structures on ${\D}$. 
Then the isomorphism $T_{{\D}} \simeq {\LL}^{-1}\otimes {\E}$ is nothing but the differential of this tautological period map 
(cf.~\cite{Vo2} \S 10.1). 
By taking the adjunctions of $T_{{\D}} \simeq {\LL}^{-1}\otimes {\E}$, 
we obtain the homomorphisms 
\begin{equation}\label{eqn: differential period map}
{\LL}\otimes T_{{\D}} \stackrel{\simeq}{\to} {\E}, \qquad 
{\E}\otimes T_{{\D}} \to {\LL}^{-1}. 
\end{equation}
These are familiar forms in the context of variation of Hodge structures. 
Here the second homomorphism is given by the pairing on ${\E}$: 
\begin{equation*}
{\E}\otimes T_{{\D}} \simeq {\E} \otimes {\E} \otimes {\LL}^{-1} \to {\LL}^{-1}. 
\end{equation*}
\end{example}

\begin{example}\label{ex: Koszul}
Adjunctions of \eqref{eqn: differential period map} induce the following complex of vector bundles on ${\D}$ (the \textit{Koszul complex}): 
\begin{equation}\label{eqn: Koszul}
{\LL} \to {\E}\otimes \Omega_{{\D}}^{1} \to {\LL}^{-1}\otimes \Omega_{{\D}}^{2}. 
\end{equation}
Here the second homomorphism is the composition 
\begin{equation*}
{\E}\otimes \Omega^{1}_{{\D}} \stackrel{}{\simeq} {\LL}^{-1}\otimes \Omega^{1}_{{\D}}\otimes \Omega^{1}_{{\D}} 
\stackrel{\wedge}{\to} {\LL}^{-1}\otimes \Omega_{{\D}}^{2}. 
\end{equation*}
By \eqref{eqn: TD}, the Koszul complex is identified with the complex 
\begin{equation*}
{\LL}\, \otimes \, ( {\OD} \to {\E}^{\otimes 2} \stackrel{\wedge}{\to} \wedge^{2}{\E}), 
\end{equation*}
where $\mathcal{O}_{{\D}}\to {\E}^{\otimes 2}$ is the embedding defined by the quadratic form on ${\E}$. 
This shows that \eqref{eqn: Koszul} is indeed a complex, 
and its middle cohomology sheaf is isomorphic to 
\begin{equation*}
({\rm Sym}^{2}{\E}/{\OD})\otimes {\LL} \simeq {\E}_{(2)}\otimes {\LL}, 
\end{equation*}
where ${\E}_{(2)}$ is the automorphic vector bundle associated to the representation 
${\rm Sym}^2{\C}^{n}/{\C}$ of ${\On}$ (see \S \ref{ssec: automorphic VB}). 
The Koszul complex will be taken up in \S \ref{ssec: higher Chow}. 
\end{example}

\section{$I$-trivialization of the second Hodge bundle}\label{ssec: trivialize}

In this section we define a trivialization of ${\E}$ associated to each $0$-dimensional cusp. 
This is the starting point of various later constructions. 

Let $I$ be a rank $1$ primitive isotropic sublattice of $L$. 
The quadratic form on $L$ induces a hyperbolic quadratic form on the ${\Z}$-module $I^{\perp}/I$. 
We write $V(I)_{F}=(I^{\perp}/I)\otimes_{{\Z}} F$ for $F={\Q}, {\R}, {\C}$. 
This is a quadratic space over $F$. 
We especially abbreviate $V(I)=V(I)_{{\C}}$. 
We consider the following sub vector bundle of $L_{{\C}}\otimes \mathcal{O}_{Q(I)}$:  
\begin{equation*}
{\Ip}\cap {\LL}^{\perp} = ({\ICp}\otimes \mathcal{O}_{Q(I)})\cap {\LL}^{\perp}. 
\end{equation*}
Th fiber of ${\Ip}\cap {\LL}^{\perp}$ over $[\omega]\in Q(I)$ is the subspace ${\ICp}\cap \omega^{\perp}$ of $L_{{\C}}$. 
The projection ${\LL}^{\perp}\to {\E}$ induces a homomorphism ${\Ip}\cap {\LL}^{\perp}\to {\E}$, 
and the projection ${\ICp}\to V(I)$ induces a homomorphism 
${\Ip}\cap {\LL}^{\perp}\to V(I)\otimes \mathcal{O}_{Q(I)}$. 

\begin{lemma}
The homomorphisms 
${\Ip}\cap {\LL}^{\perp}\to {\E}$ and ${\Ip}\cap {\LL}^{\perp}\to V(I)\otimes \mathcal{O}_{Q(I)}$ 
are isomorphisms. 
Therefore we obtain an isomorphism 
\begin{equation}\label{eqn: I-trivialization E}
V(I)\otimes \mathcal{O}_{Q(I)} \to {\E} 
\end{equation}
of vector bundles on $Q(I)$. 
This is equivariant with respect to the stabilizer of $I_{{\C}}$ in ${\OLC}$, 
and preserves the quadratic forms on both sides. 
\end{lemma}

\begin{proof}
At the fibers over a point $[\omega]\in Q(I)$, 
the two homomorphisms are given by the linear maps 
${\ICp}\cap \omega^{\perp} \to \omega^{\perp}/{\C}\omega$ and 
${\ICp}\cap \omega^{\perp} \to ({\Ip}/I)_{{\C}}$ 
respectively. 
The source and the target have the same dimension ($=n$) for both maps, 
so it suffices to check the injectivity of these two maps. 
This is equivalent to ${\ICp}\cap {\C}\omega = 0$ and $\omega^{\perp}\cap I_{{\C}}=0$ respectively, 
and both follow from the nondegeneracy $(I_{{\C}}, {\C}\omega)\ne 0$ for $[\omega]\in Q(I)$. 

Since both ${\ICp}\cap \omega^{\perp} \to \omega^{\perp}/{\C}\omega$ and ${\ICp}\cap \omega^{\perp} \to ({\Ip}/I)_{{\C}}$ 
preserve the quadratic forms, so does the composition $\omega^{\perp}/{\C}\omega \to ({\Ip}/I)_{{\C}}$. 
Hence \eqref{eqn: I-trivialization E} preserves the quadratic forms.  
The equivariance of \eqref{eqn: I-trivialization E} can be verified similarly. 
\end{proof}

We call the isomorphism \eqref{eqn: I-trivialization E} and its restriction to ${\D}$ the \textit{$I$-trivialization} of ${\E}$. 
This is a trivialization as an orthogonal vector bundle. 
See Claim \ref{claim: I-trivialization boundary} for the boundary behavior of this isomorphism 
at a Zariski open set of the divisor $Q\cap {\proj}{\ICp}$. 

For later use, we calculate the sections of ${\E}$ corresponding to vectors of $V(I)$. 
We choose a vector $l\ne 0$ of $I$ and let $s_{l}$ be the corresponding section of ${\LL}$ as defined in \S \ref{ssec: L}. 

\begin{lemma}\label{lem: basic section E}
Let $v$ be a vector of $V(I)$. 
We define a section of ${\Ip}\cap {\LL}^{\perp}$ by 
\begin{equation*}
s_{v}([\omega]) = \tilde{v} - (\tilde{v}, s_l([\omega]))l, \qquad [\omega]\in Q(I), 
\end{equation*}
where $\tilde{v}\in {\ICp}$ is a lift of $v\in V(I)$ and we regard $s_l([\omega]) \in {\C}\omega \subset L_{{\C}}$. 
Then the image of $s_{v}$ in ${\E}$ is the section of ${\E}$ which corresponds by the $I$-trivialization 
to the constant section of $V(I)\otimes \mathcal{O}_{Q(I)}$ with value $v$. 
\end{lemma}

\begin{proof}
It is straightforward to check that 
$s_v([\omega])$ does not depend on the choice of the lift $\tilde{v}$ and that 
$(s_v([\omega]), \omega)=(s_v([\omega]), l)=0$. 
Thus $s_v$ is indeed a section of ${\Ip}\cap {\LL}^{\perp}$. 
Since $s_v([\omega])\equiv \tilde{v}$ mod $I_{\C}$ as a vector of ${\ICp}$, 
the image of $s_v([\omega])$ in $V(I)$ is $v$. 
This proves our assertion. 
\end{proof}

\section{Accidental isomorphisms}\label{ssec: accidental}

When $n\leq 4$, orthogonal modular varieties are isomorphic to other types of classical modular varieties 
by the so-called accidental isomorphisms. 
In this section we explain how the second Hodge bundle ${\E}$ in $n\leq 4$ is translated under the accidental isomorphism. 
(This is well-known for ${\LL}$; we also include it for completeness.) 
This correspondence is the basis of comparing vector-valued orthogonal modular forms in $n=3, 4$ with 
vector-valued Siegel and Hermitian modular forms respectively. 
We explain the translation from both algebro-geometric and representation-theoretic viewpoints. 
Since the contents of this section will be used only sporadically in the rest of this monograph, 
the reader may skip it for the moment.  

\subsection{Modular curves}\label{sssec: n=1}

When $n=1$, the accidental isomorphism between the real Lie groups is 
${\rm PSL}(2, {\R})\simeq {\rm SO}^{+}(1, 2)$. 
Its complexification is ${\rm PSL}(2, {\C})\simeq {\rm SO}(3, {\C})$. 
This lifts to ${\rm SL}(2, {\C})\simeq {\rm Spin}(3, {\C})$. 
The isomorphism between the compact duals is provided by 
the anti-canonical embedding ${\proj}^1\hookrightarrow {\proj}^2$ of ${\proj}^1$, which maps ${\proj}^1$ to a conic $Q\subset {\proj}^2$. 
This gives an isomorphism between the upper half plane and the type IV domain in $n=1$. 
The line bundle ${\LL}=\mathcal{O}_{Q}(-1)$ on $Q$ is identified with $\mathcal{O}_{{\proj}^1}(-2)$ on ${\proj}^1$. 
This means that orthogonal modular forms of weight $k$ correspond to elliptic modular forms of weight $2k$. 

The reductive part of a standard parabolic subgroup of ${\rm SL}(2, {\C})$ 
is the $1$-dimensional torus $T$ consisting of diagonal matrices 
$\begin{pmatrix}\alpha & 0 \\ 0 & \alpha^{-1} \end{pmatrix}$ 
of determinant $1$. 
The corresponding group in ${\rm PSL}(2, {\C})$ is $T/-1$. 
The weight $2$ character $\alpha \mapsto \alpha^{2}$ of $T$ defines an isomorphism $T/-1 \simeq {\C}^{\ast}$. 
This explains $\mathcal{O}_{Q}(-1)\simeq \mathcal{O}_{{\proj}^1}(-2)$ from representation theory. 

The full orthogonal group ${\rm O}(3, {\C})$ is ${\rm SO}(3, {\C}) \times \{ \pm {\rm id}\}$.   
By Example \ref{ex: det}, the second Hodge bundle ${\E}$ is the line bundle associated to the determinant character 
$\det \colon {\rm O}(3, {\C}) \to \{ \pm 1 \}$. 
This is nontrivial as an ${\rm O}(3, {\C})$-line bundle, but trivial as an ${\rm SO}(3, {\C})$-line bundle. 
Therefore ${\E}$ cannot be detected at the side of ${\rm SL}(2, {\C})$.

\subsection{Hilbert modular surfaces}\label{sssec: n=2}

When $n=2$, the accidental isomorphism between the real Lie groups is 
\begin{equation*}
{\rm SL}(2, {\R}) \times {\rm SL}(2, {\R})/ (-1, -1) \simeq {\rm SO}^{+}(2, 2). 
\end{equation*}
Its complexification is 
\begin{equation*}
{\rm SL}(2, {\C}) \times {\rm SL}(2, {\C})/ (-1, -1) \simeq {\rm SO}(4, {\C}). 
\end{equation*}
This lifts to 
${\rm SL}(2, {\C}) \times {\rm SL}(2, {\C}) \simeq {\rm Spin}(4, {\C})$. 
The isomorphism between the compact duals is provided by the Segre embedding 
${\proj}^1\times {\proj}^1\hookrightarrow {\proj}^3$ of ${\proj}^1\times {\proj}^1$, 
which maps ${\proj}^1\times {\proj}^1$ to a quadric surface $Q\subset {\proj}^3$.  
This gives an isomorphism between the product of two upper half planes and the type IV domain in $n=2$. 
Since the Segre embedding is defined by $\mathcal{O}_{{\proj}^1\times {\proj}^1}(1, 1)$, 
the Hodge line bundle ${\LL}=\mathcal{O}_{Q}(-1)$ on $Q$ is identified with 
$\mathcal{O}_{{\proj}^1\times {\proj}^1}(-1, -1)$ on ${\proj}^1\times {\proj}^1$. 
This means that orthogonal modular forms of weight $k$ correspond to Hilbert modular forms of weight $(k, k)$. 

We explain the representation-theoretic aspect. 
The reductive part of a standard parabolic subgroup of ${\rm SL}(2, {\C}) \times {\rm SL}(2, {\C})$ is 
the $2$-dimensional torus $T_1\times T_2 \simeq {\C}^{\ast}\times {\C}^{\ast}$ 
consisting of pairs $(\alpha, \beta)$ of diagonal matrices in each ${\rm SL}(2, {\C})$. 
The corresponding group in ${\rm SL}(2, {\C}) \times {\rm SL}(2, {\C})/ (-1, -1)$ is $T_1\times T_2 / (-1, -1)$. 
We have natural isomorphisms 
\begin{equation}\label{eqn: Levi part n=2}
T_1\times T_2 / (-1, -1) \simeq {\C}^{\ast}\times {\C}^{\ast}  \simeq {\C}^{\ast}\times {\rm SO}(2, {\C}), 
\end{equation}
where the first isomorphism is induced by 
\begin{equation*}
T_1\times T_2 \to {\C}^{\ast}\times {\C}^{\ast}, \quad  (\alpha, \beta)\mapsto (\alpha\beta, \alpha^{-1}\beta). 
\end{equation*} 
This is the isomorphism between the reductive parts of standard parabolic subgroups of 
${\rm SL}(2, {\C}) \times {\rm SL}(2, {\C})/ (-1, -1)$ and ${\rm SO}(4, {\C})$. 
The pullback of the weight $1$ character of ${\C}^{\ast}\subset {\C}^{\ast}\times {\rm SO}(2, {\C})$ to $T_{1}\times T_2$ by \eqref{eqn: Levi part n=2} 
is the tensor product $\chi_1 \boxtimes \chi_2$ of the weight $1$ characters $\chi_1$, $\chi_2$ of $T_1$, $T_2$. 
This explains $\mathcal{O}_{Q}(-1)\simeq \mathcal{O}_{{\proj}^1\times {\proj}^1}(-1, -1)$ from representation theory. 

The second Hodge bundle ${\E}$ is described as follows. 

\begin{lemma}\label{lem: 2nd Hodge n=2}
We have an ${\rm O}(4, {\C})$-equivariant isomorphism  
\begin{equation}\label{eqn: 2nd Hodge n=2}
{\E} \: \simeq \: \mathcal{O}_{{\proj}^1\times {\proj}^1}(-1, 1) \oplus \mathcal{O}_{{\proj}^1\times {\proj}^1}(1, -1).  
\end{equation}
\end{lemma}

\begin{proof} 
Let $\pi_{i}\colon {\proj}^{1}\times {\proj}^{1}\to {\proj}^{1}$ be the $i$-th projection. 
Then 
\begin{equation*}
\Omega_{{\proj}^1\times {\proj}^1}^{1} \simeq 
\pi_{1}^{\ast}\Omega_{{\proj}^{1}}^{1} \oplus \pi_{2}^{\ast}\Omega_{{\proj}^{1}}^{1} \simeq 
\mathcal{O}_{{\proj}^1\times {\proj}^1}(-2, 0) \oplus \mathcal{O}_{{\proj}^1\times {\proj}^1}(0, -2). 
\end{equation*} 
By \eqref{eqn: TD} and ${\LL}^{-1} \simeq \mathcal{O}_{{\proj}^1\times {\proj}^1}(1, 1)$, we have 
\begin{eqnarray*}
{\E} & \simeq & \Omega_{{\proj}^1\times {\proj}^1}^{1} \otimes \mathcal{O}_{{\proj}^1\times {\proj}^1}(1, 1) \\ 
& \simeq & \mathcal{O}_{{\proj}^1\times {\proj}^1}(-1, 1) \oplus \mathcal{O}_{{\proj}^1\times {\proj}^1}(1, -1).  
\end{eqnarray*}
This proves \eqref{eqn: 2nd Hodge n=2}. 
\end{proof}

Note that ${\rm O}(4, {\C})$ is the semi-product $\frak{S}_{2}\ltimes {\rm SO}(4, {\C})$, 
where $\frak{S}_{2}$ switches the two ${\rm SL}(2, {\C})$. 
This involution switches the two rulings of $Q\simeq {\proj}^1\times {\proj}^1$, 
and acts on the right hand side of \eqref{eqn: 2nd Hodge n=2} by switching the two components. 

At the level of representations, the isomorphism \eqref{eqn: 2nd Hodge n=2} comes from the following correspondence. 
Let $\chi$ be the weight $1$ character of ${\rm SO}(2, {\C}) \simeq {\C}^{\ast}$. 
The $2$-dimensional standard representation of ${\rm SO}(2, {\C})$ is $\chi \oplus \chi^{-1}$. 
The pullback of $\chi$ to $T_{1}\times T_2$ by \eqref{eqn: Levi part n=2} is the character $\chi_1^{-1}\boxtimes \chi_2$. 
Hence the pullback of the standard representation of ${\rm SO}(2, {\C})$ to $T_1\times T_2$ is 
$(\chi_1^{-1}\boxtimes \chi_2)\oplus (\chi_1\boxtimes \chi_2^{-1})$. 
This explains \eqref{eqn: 2nd Hodge n=2} from representation theory. 

By Lemma \ref{lem: 2nd Hodge n=2}, a general automorphic vector bundle ${\Elk}$ on $Q$ 
decomposes into a direct sum of various line bundles $\mathcal{O}_{{\proj}^1\times {\proj}^1}(a, b)$. 
This means that vector-valued orthogonal modular forms in $n=2$ decompose into 
tuples of scalar-valued Hilbert modular forms of various weights, so we have nothing new here.

\subsection{Siegel modular $3$-folds}\label{sssec: n=3}

When $n=3$, the accidental isomorphism between the real Lie groups is 
${\rm PSp}(4, {\R})\simeq {\rm SO}^{+}(2, 3)$. 
Its complexification is 
${\rm PSp}(4, {\C})\simeq {\rm SO}(5, {\C})$, 
which lifts to 
${\rm Sp}(4, {\C})\simeq {\rm Spin}(5, {\C})$. 
The isomorphism between the compact duals is provided by the Pl\"ucker embedding 
${\rm LG}(2, 4)\hookrightarrow {\proj}V={\proj}^4$ of the Lagrangian Grassmannian ${\rm LG}(2, 4)$. 
Here $V$ is the $5$-dimensional irreducible representation of ${\rm Sp}(4, {\C})$ appearing in $\wedge^2{\C}^4$. 
The Pl\"ucker embedding maps ${\rm LG}(2, 4)$ to a $3$-dimensional quadric $Q\subset {\proj}^4$, 
and hence gives an isomorphism between the Siegel upper half space of genus $2$ and the type IV domain in $n=3$. 

Let $\mathcal{F}$ be the rank $2$ universal sub vector bundle over ${\rm LG}(2, 4)$. 
(This is the weight $1$ Hodge bundle for Siegel modular $3$-folds.) 
Since the Pl\"ucker embedding is defined by $\mathcal{O}_{{\rm LG}}(1)= \det \mathcal{F}^{\vee}$, 
the Hodge line bundle ${\LL}=\mathcal{O}_{Q}(-1)$ on $Q$ is identified with $\det \mathcal{F}$ on ${\rm LG}(2, 4)$. 
This means that orthogonal modular forms of weight $k$ correspond to Siegel modular forms of weight $k$. 

We explain the representation-theoretic aspect. 
The reductive part of a standard parabolic subgroup of ${\rm Sp}(4, {\C})$ is isomorphic to ${\rm GL}(2, {\C})$. 
The corresponding group in ${\rm PSp}(4, {\C})$ is ${\rm GL}(2, {\C})/-1$. 
We have a natural isomorphism 
\begin{equation}\label{eqn: Levi part n=3}
{\rm GL}(2, {\C})/-1 \simeq {\C}^{\ast} \times {\rm PGL}(2, {\C}) \simeq {\C}^{\ast} \times {\rm SO}(3, {\C}), 
\end{equation}
where ${\rm GL}(2, {\C}) \to {\C}^{\ast}$ in the first isomorphism is the determinant character, and 
${\rm PGL}(2, {\C})\simeq {\rm SO}(3, {\C})$ in the second isomorphism is the accidental isomorphism in $n=1$. 
This gives the isomorphism between the reductive parts of standard parabolic subgroups of 
${\rm PSp}(4, {\C})$ and ${\rm SO}(5, {\C})$. 
By construction, the pullback of the weight $1$ character of ${\C}^{\ast}$ to ${\rm GL}(2, {\C})$ by \eqref{eqn: Levi part n=3} 
is the determinant character of ${\rm GL}(2, {\C})$. 
This explains ${\LL} \simeq \det \mathcal{F}$ from representation theory. 

The second Hodge bundle ${\E}$ is described as follows. 

\begin{lemma}\label{lem: 2nd Hodge n=3}
We have an ${\rm SO}(5, {\C})$-equivariant isomorphism  
\begin{equation}\label{eqn: 2nd Hodge n=3}
{\E} \simeq {\rm Sym}^2\mathcal{F} \otimes {\LL}^{-1}. 
\end{equation}
\end{lemma}

\begin{proof} 
As it is well-known, we have an ${\rm Sp}(4, {\C})$-equivariant isomorphism 
$\Omega^{1}_{{\rm LG}}\simeq {\rm Sym}^2\mathcal{F}$ (see, e.g., \cite{vdG} \S 14). 
Then \eqref{eqn: 2nd Hodge n=3} follows from the isomorphism ${\E}\simeq \Omega^{1}_{{\rm LG}}\otimes {\LL}^{-1}$ in \eqref{eqn: TD}. 
\end{proof}

Note that $\mathcal{F}$ is not ${\rm SO}(5, {\C})$-linearized but ${\rm Sym}^2\mathcal{F}$ is. 
At the level of representations, the isomorphism \eqref{eqn: 2nd Hodge n=3} comes from the following fact: 
the symmetric square of the standard representation of ${\rm GL}(2, {\C})$, 
when viewed as a representation of ${\C}^{\ast} \times {\rm SO}(3, {\C})$ via \eqref{eqn: Levi part n=3}, 
is isomorphic to the tensor product of 
the weight $1$ character of ${\C}^{\ast}$ and the standard representation of ${\rm SO}(3, {\C})$.

The full orthogonal group ${\rm O}(5, {\C})$ is  ${\rm SO}(5, {\C}) \times \{ \pm {\rm id} \}$. 
As an ${\rm O}(5, {\C})$-vector bundle, we have 
\begin{equation*}
{\E} \simeq {\rm Sym}^2\mathcal{F} \otimes {\LL}^{-1} \otimes \det. 
\end{equation*}
The twist by $\det$ cannot be detected at the side of ${\rm Sp}(4, {\C})$.

\subsection{Hermitian modular $4$-folds}\label{sssec: n=4}

When $n=4$, the accidental isomorphism between the real Lie groups is 
${\rm SU}(2, 2)/-1 \simeq {\rm SO}^{+}(2, 4)$. 
The complexification is ${\rm SL}(4, {\C})/-1 \simeq {\rm SO}(6, {\C})$. 
This lifts to ${\rm SL}(4, {\C}) \simeq {\rm Spin}(6, {\C})$. 
The isomorphism between the compact duals is provided by the Pl\"ucker embedding 
${\rm G}(2, 4)\hookrightarrow {\proj}(\wedge^2{\C}^4)={\proj}^5$ 
of the Grassmannian ${\rm G}(2, 4)$. 
This maps ${\rm G}(2, 4)$ to a $4$-dimensional quadric $Q\subset {\proj}^5$, 
and gives an isomorphism between the Hermitian upper half space of degree $2$ and the type IV domain in $n=4$. 

The reductive part of a standard parabolic subgroup of ${\rm SL}(4, {\C})$ is the group 
\begin{equation*}
G= \Biggl\{ \: \begin{pmatrix} g_1 & 0 \\ 0 & g_2 \end{pmatrix} \: \Biggm\vert \: g_1,g_2\in {\rm GL}(2, {\C}), \: \det g_2 = \det g_1^{-1} \: \Biggr\}. 
\end{equation*}
The corresponding group in ${\rm SL}(4, {\C})/-1$ is $G/-1$. 
We have a natural isomorphism 
\begin{equation}\label{eqn: Levi part n=4}
G/-1 \simeq {\C}^{\ast}\times ({\rm SL}(2, {\C})\times {\rm SL}(2, {\C})/(-1, -1)) 
\simeq {\C}^{\ast}\times {\rm SO}(4, {\C}). 
\end{equation}
Here the first isomorphism sends $(g_1, g_2)\in G$ to 
$(\det g_1, \pm \alpha^{-1} g_1, \pm \alpha g_2)$ 
where $\alpha$ is one of the square roots of $\det g_1$, and  
the second isomorphism is given by the accidental isomorphism in $n=2$. 
This is the isomorphism between the reductive parts of standard parabolic subgroups of 
${\rm SL}(4, {\C})/-1$ and ${\rm SO}(6, {\C})$. 

Let $\mathcal{F}$, $\mathcal{G}$ be the universal sub and quotient vector bundles on ${\rm G}(2, 4)$ respectively. 
Since the Pl\"ucker embedding is defined by 
$\mathcal{O}_{{\rm G}(2,4)}(1)=\det \mathcal{G}=(\det \mathcal{F})^{-1}$, 
the Hodge line bundle ${\LL}=\mathcal{O}_{Q}(-1)$ is isomorphic to $\det \mathcal{F}$. 
Thus orthogonal modular forms of weight $k$ correspond to Hermitian modular forms of weight $k$. 
At the level of representations, this comes from the fact that 
the pullback of the weight $1$ character of ${\C}^{\ast}$ to $G$ by \eqref{eqn: Levi part n=4} 
is the character of $G$ given by $(g_1, g_2)\mapsto \det g_1$. 

The second Hodge bundle ${\E}$ is described as follows. 

\begin{lemma}\label{lem: 2nd Hodge n=4}
We have an ${\rm SO}(6, {\C})$-equivariant isomorphism  
\begin{equation}\label{eqn: 2nd Hodge n=4}
{\E} \simeq \mathcal{F}\otimes \mathcal{G}. 
\end{equation}
\end{lemma}

\begin{proof}
We have a canonical isomorphism 
$T_{{\rm G}(2,4)}\simeq \mathcal{F}^{\vee}\otimes \mathcal{G}$. 
The natural symplectic form $\mathcal{F}\otimes \mathcal{F}\to \det \mathcal{F}$ 
induces an isomorphism $\mathcal{F}^{\vee}\simeq \mathcal{F}\otimes {\LL}^{-1}$. 
Therefore, by \eqref{eqn: TD}, we have 
\begin{equation*}
{\E} \simeq T_{{\rm G}(2,4)}\otimes {\LL} \simeq 
\mathcal{F}^{\vee}\otimes\mathcal{G} \otimes {\LL} \simeq \mathcal{F}\otimes \mathcal{G}. 
\end{equation*}
This proves \eqref{eqn: 2nd Hodge n=4}. 
\end{proof}

Note that each $\mathcal{F}$, $\mathcal{G}$ is not ${\rm SO}(6, {\C})$-linearized, 
but $\mathcal{F}\otimes \mathcal{G}$ is. 
At the level of representations, the isomorphism \eqref{eqn: 2nd Hodge n=4} comes from the following correspondence. 
Let $V_i$, $i=1, 2$, be the representation of $G$ obtained as the pullback of the standard representation of ${\rm GL}(2, {\C})$ by 
the $i$-th projection $G\to {\rm GL}(2, {\C})$, $(g_1, g_2)\mapsto g_i$. 
Then $V_1, V_2$ correspond to the homogeneous vector bundles $\mathcal{F}, \mathcal{G}$ respectively. 
Each $V_1, V_2$ is not a representation of $G/-1$, but $V_1\otimes V_2$ is. 
Then, as a representation of 
${\C}^{\ast}\times ({\rm SL}(2, {\C})^2/(-1, -1))$ 
via the first isomorphism in \eqref{eqn: Levi part n=4}, 
$V_1\otimes V_2$ is isomorphic to the external tensor product of 
the standard representations of the two ${\rm SL}(2, {\C})$ 
(with weight $0$ on ${\C}^{\ast}$). 
This in turn is the standard representation of ${\rm SO}(4, {\C})$ via the second isomorphism in \eqref{eqn: Levi part n=4}. 
This explains the isomorphism \eqref{eqn: 2nd Hodge n=4} from representation theory.  

Finally, ${\rm O}(6, {\C})$ is the semi-product $\frak{S}_{2}\ltimes{\rm SO}(6, {\C})$ 
where $\frak{S}_{2}=\langle \iota \rangle$ acts on ${\rm G}(2, 4)$ by the following involution: 
choose a symplectic form on ${\C}^{4}$ (say the standard one), 
and sends $2$-dimensional subspaces $W\subset {\C}^{4}$ to $W^{\perp}\subset {\C}^{4}$. 
This involution exchanges the two ${\proj}^3$-families of planes on ${\rm G}(2, 4)=Q$. 
(This is essentially the involution $Z\mapsto Z'$ in \cite{FS} \S 1 on the Hermitian upper half space.) 
The involution $\iota$ acts on the vector bundle $\mathcal{F}\otimes\mathcal{G} \simeq \mathcal{F}^{\vee}\otimes \mathcal{G}^{\vee}$ 
by $\iota^{\ast}\mathcal{F}\simeq \mathcal{G}^{\vee}$ and $\iota^{\ast}\mathcal{G}\simeq \mathcal{F}^{\vee}$. 
Then \eqref{eqn: 2nd Hodge n=4} is an ${\rm O}(6, {\C})$-equivariant isomorphism.


\chapter{Vector-valued modular forms}\label{sec: vector-valued}

In this chapter we define vector-valued orthogonal modular forms (\S \ref{ssec: automorphic VB}) 
and explain their Fourier expansions at $0$-dimensional cusps (\S \ref{ssec: tube domain} -- \S \ref{ssec: geometry Fourier}). 
These are the most fundamental parts of this monograph. 
The rest of this chapter (\S \ref{ssec: SO} -- \S \ref{ssec: higher Chow}) is devoted to supplementary materials: 
the passage from ${\rm O}$ to ${\rm SO}$, an example of explicit construction, and an interaction with algebraic cycles.

\section{Representations of ${\On}$}\label{ssec: rep O(n)}

We begin by recollection of some basic facts from the representation theory for ${\On}$. 
Our main reference for representation theory is \cite{Ok} \S 8 
(whose main contents are more or less covered by \cite{FH} \S 19 and \cite{GW} \S 5.5.5, \S 10.2). 
In what follows and in \S \ref{ssec: SO}, all representations are assumed to be finite-dimensional over ${\C}$. 

Irreducible representations of ${\On}$ are labelled by partitions 
$\lambda=(\lambda_1\geq \cdots \geq \lambda_n\geq 0)$ 
such that ${}^t\lambda_1 + {}^t\lambda_2 \leq n$, 
where ${}^t \lambda$ is the transpose of $\lambda$. 
The irreducible representation corresponding to such a partition $\lambda$ is constructed as follows. 
Let $V={\C}^n$ be the standard representation of ${\On}$. 
Let $d=|\lambda|=\sum_{i}\lambda_i$ be the size of $\lambda$. 
We denote by $V^{[d]}$ the intersection of the kernels of the contraction maps $V^{\otimes d}\to V^{\otimes d-2}$ for all pairs of indices. 
Vectors in $V^{[d]}$ are called \textit{traceless tensors} or \textit{harmonic tensors} in the literature. 
The symmetric group $\frak{S}_{d}$ acts on $V^{\otimes d}$ naturally and preserves $V^{[d]}$. 
Let $T=T_{\lambda}^{\downarrow}$ be the column canonical tableau on $\lambda$ 
(namely $1, 2, \cdots, {}^{t}\lambda_{1}$ on the first column, 
${}^{t}\lambda_{1}+1, \cdots, {}^{t}\lambda_{1}+{}^{t}\lambda_{2}$ on the second column, $\cdots$). 
Let 
$c_{\lambda}=b_{\lambda}a_{\lambda} \in {\C}\frak{S}_{d}$ 
be the Young symmetrizer associated to $T$, where 
\begin{equation*}
a_{\lambda}=\sum_{\sigma\in H_{T}}\sigma, \qquad 
b_{\lambda}=\sum_{\tau\in V_{T}}{\rm sgn}(\tau) \tau 
\end{equation*}
as usual. 
($H_T$ and $V_T$ are the row and the column Young subgroups of $\frak{S}_{d}$ for the tableau $T$ respectively.) 
We apply the orthogonal Schur functor for $\lambda$ to $V$: 
\begin{equation*}
V_{\lambda} = c_{\lambda}\cdot V^{[d]} = V^{[d]}\cap (c_{\lambda}\cdot V^{\otimes d}). 
\end{equation*}
This space $V_{\lambda}$ is the irreducible representation of ${\On}$ labelled by the partition $\lambda$. 
Since $b_{\lambda}$ maps $V^{\otimes d}$ to 
$\wedge^{{}^t \lambda_{1}}V \otimes \cdots \otimes \wedge^{{}^t \lambda_{\lambda_{1}}}V$, 
we have 
\begin{equation}\label{eqn: Vlambda in wedge tensor}
V_{\lambda} \: \subset \: 
\wedge^{{}^t \lambda_{1}}V \otimes \cdots \otimes \wedge^{{}^t \lambda_{\lambda_{1}}}V 
\: \subset \: V^{\otimes d}. 
\end{equation}
If we take a basis $e_{1}, \cdots, e_{n}$ of $V$ such that 
$(e_{i}, e_{j})=1$ when $i+j=n+1$ and $(e_{i}, e_{j})=0$ otherwise, 
$V_{\lambda}$ especially contains the vector 
\begin{equation}\label{eqn: Vlambda highest weight vector}
(e_{1}\wedge \cdots \wedge e_{{}^t \lambda_{1}}) \otimes 
(e_{1}\wedge \cdots \wedge e_{{}^t \lambda_{2}}) \otimes \cdots \otimes (e_{1}\wedge \cdots \wedge e_{{}^t \lambda_{\lambda_{1}}}) 
\end{equation}
(see \cite{Ok} \S 8.3.1).

\begin{example}\label{ex: wedge sym}
(1) The exterior tensor $\wedge^{d}V$ for $0\leq d \leq n$ corresponds to the partition $\lambda=(1^d)=(1,\cdots, 1)$. 
By abuse of notation, we sometimes write $\lambda = 1, {\rm St}, \wedge^{d}, \det$ instead of 
$\lambda = (0), (1), (1^d), (1^n)$ respectively. 

(2) The symmetric tensor ${\rm Sym}^{d}V$ is reducible and decomposes as 
\begin{eqnarray*}
{\rm Sym}^{d}V & = & V_{(d)}\oplus {\rm Sym}^{d-2}V \: \: = \: \: \cdots \\ 
& = & V_{(d)}\oplus V_{(d-2)}\oplus \cdots \oplus V_{(1) \textrm{or} (0)}. 
\end{eqnarray*}
Geometrically, $V_{(d)}$ is the cohomology $H^{0}(\mathcal{O}_{Q_{n-2}}(d))$ 
for the isotropic quadric $Q_{n-2}\subset {\proj}V$ of dimension $n-2$. 
\end{example}


\section{Automorphic vector bundles}\label{ssec: automorphic VB}

In this section we define automorphic vector bundles and vector-valued modular forms. 
Let $L$ be a lattice of signature $(2, n)$. 
For simplicity of exposition we assume $n\geq 3$ so that the Koecher principle holds. 
(This assumption can be somewhat justified by our calculation of ${\E}$ in the case $n \leq 2$ in \S \ref{ssec: accidental}.) 
Let $\lambda=(\lambda_1\geq \cdots \geq \lambda_n)$ be a partition as in \S \ref{ssec: rep O(n)} 
and let $d=|\lambda|$. 
Recall that the second Hodge bundle ${\E}$ is endowed with a canonical quadratic form. 
Let ${\E}^{[d]}\subset {\E}^{\otimes d}$ be the intersection of 
the kernels of the contractions ${\E}^{\otimes d} \to {\E}^{\otimes d-2}$ for all pairs of indices. 
The fibers of ${\E}^{[d]}$ consist of traceless tensors in the fibers of ${\E}^{\otimes d}$. 
The symmetric group $\frak{S}_{d}$ acts on ${\E}^{\otimes d}$ fiberwisely and preserves ${\E}^{[d]}$.  
We define the vector bundle ${\El}$ by applying the orthogonal Schur functor for $\lambda$ relatively to ${\E}$: 
\begin{equation*}
{\El} = c_{\lambda}\cdot {\E}^{[d]} = {\E}^{[d]} \cap (c_{\lambda}\cdot {\E}^{\otimes d}). 
\end{equation*}
By construction ${\El}$ is a sub vector bundle of ${\E}^{\otimes d}$, 
naturally defined over $Q$ and is ${\OLC}$-invariant. 

Let $I$ be a rank $1$ primitive isotropic sublattice of $L$. 
Recall from \S \ref{ssec: trivialize} that we have the $I$-trivialization ${\E}\simeq V(I)\otimes \mathcal{O}_{Q(I)}$ over $Q(I)=Q-Q\cap {\proj}{\ICp}$.  
Let $V(I)_{\lambda}$ be the irreducible representation of ${\rm O}(V(I))\simeq {\On}$ 
obtained by applying the orthogonal Schur functor for $\lambda$ to $V(I)$. 
Since the $I$-trivialization of ${\E}$ preserves the quadratic forms, 
it induces an isomorphism 
\begin{equation*}
{\El} \simeq V(I)_{\lambda}\otimes \mathcal{O}_{Q(I)} 
\end{equation*}
over $Q(I)$. 
We call this isomorphism the $I$-\textit{trivialization} of ${\El}$. 

Next for $k\in {\Z}$ we consider the tensor product 
\begin{equation*}
{\Elk} = {\El}\otimes {\LL}^{\otimes k}. 
\end{equation*}
This is an ${\OLC}$-equivariant vector bundle on $Q$. 
If we write 
\begin{equation*}
{\VIlk}=V(I)_{\lambda}\otimes (I_{{\C}}^{\vee})^{\otimes k}, 
\end{equation*}
the $I$-trivializations of ${\El}$ and ${\LL}^{\otimes k}$ induce an isomorphism   
\begin{equation*}
{\Elk} \simeq {\VIlk}\otimes \mathcal{O}_{Q(I)} 
\end{equation*}
over $Q(I)$. 
This is equivariant with respect to the stabilizer of $I_{{\C}}$ in ${\OLC}$. 
We call this isomorphism the $I$-\textit{trivialization} of ${\Elk}$. 

In what follows, we work over ${\D}$. 
We use the same notations ${\El}, {\Elk}$ for the restriction of ${\El}, {\Elk}$ to ${\D}$. 
These are ${\OLR}$-equivariant vector bundles on ${\D}$. 
Like \eqref{eqn: E standard rep}, we have an ${\OLR}$-equivariant isomorphism  
\begin{equation}\label{eqn: El=GKVl}
{\El}\simeq {\OLR}\times_{K}({\El})_{[\omega]} \simeq {\OLR}\times_{K} V_{\lambda}, 
\end{equation}
where $K$ is the stabilizer of $[\omega]$ in ${\OLR}$. 
The $I$-trivialization of ${\Elk}$ is defined over ${\D}$. 
Let $j(g, [\omega])$ be the factor of automorphy for the ${\OLR}$-action on ${\Elk}$ with respect to the $I$-trivialization.  
This is a ${\rm GL}({\VIlk})$-valued function on ${\OLR}\times {\D}$. 
Since the $I$-trivialization is equivariant with respect to the stabilizer of $I_{{\R}}$ in ${\OLR}$, 
we especially have the following. 

\begin{lemma}\label{lem: f.a. stab}
When $g\in {\OLR}$ stabilizes $I_{{\R}}$, the value of $j(g, [\omega])$ is constant over ${\D}$, 
given by the natural action of $g$ on ${\VIlk}$. 
\end{lemma}


Now let ${\G}$ be a finite-index subgroup of ${\OL}$. 
We call a ${\G}$-invariant holomorphic section of ${\Elk}$ over ${\D}$ 
a \textit{modular form of weight} $(\lambda, k)$ with respect to ${\G}$. 
By the $I$-trivialization, a modular form of weight $(\lambda, k)$ is identified with 
a ${\VIlk}$-valued holomorphic function $f$ on ${\D}$ such that 
\begin{equation*}
f(\gamma[\omega]) = j(\gamma, [\omega])f([\omega]) 
\end{equation*}
for every $\gamma\in {\G}$ and $[\omega]\in {\D}$. 
We denote by ${\MG}$ the space of modular forms of weight $(\lambda, k)$ with respect to ${\G}$. 
When $\lambda=(0)$, we especially write $M_{(0),k}({\G})=M_{k}({\G})$ as usual. 

When $-{\rm id}\in {\G}$, the weight $(\lambda, k)$ satisfies a parity condition. 
We state it in a slightly generalized form. 

\begin{lemma}
Let $[\omega]\in {\D}$ and ${\G}_{[\omega]}$ be the stabilizer of $[\omega]$ in ${\G}$. 
The value of a ${\G}$-modular form of weight $(\lambda, k)$ at $[\omega]$ is contained in 
the ${\G}_{[\omega]}$-invariant part of $({\Elk})_{[\omega]}$.   
In particular, when $-{\rm id}\in {\G}$ and $k+|\lambda|$ is odd, we have ${\MG}=0$. 
\end{lemma}

\begin{proof}
The first assertion follows from the ${\G}_{[\omega]}$-invariance of the section. 
As for the second assertion, we note that $-{\rm id}$ acts on both ${\LL}$ and ${\E}$ as the scalar multiplication by $-1$. 
Since ${\El}$ is a sub vector bundle of ${\E}^{\otimes |\lambda|}$, 
$-{\rm id}$ acts on ${\Elk}$ as the scalar multiplication by $(-1)^{k+|\lambda|}$. 
Therefore, when $k+|\lambda|$ is odd, $-{\rm id}$ has no nonzero invariant part in every fiber of ${\Elk}$. 
\end{proof}

Product of vector-valued modular forms can be given as follows. 
Suppose that we have a nonzero ${\On}$-homomorphism 
\begin{equation}\label{eqn: tensor decomposition}
\varphi : V_{\lambda_{1}} \otimes V_{\lambda_{2}} \to V_{\lambda_{3}} 
\end{equation}
for partitions $\lambda_{1}, \lambda_{2}, \lambda_{3}$ for ${\On}$. 
This uniquely induces an ${\OLR}$-equivariant homomorphism 
\begin{equation*}
\varphi : \mathcal{E}_{\lambda_{1}, k_{1}}\otimes \mathcal{E}_{\lambda_{2}, k_{2}} \to \mathcal{E}_{\lambda_{3}, k_{1}+k_{2}}. 
\end{equation*}
If $f_{1}, f_{2}$ are ${\G}$-modular forms of weight $(\lambda_{1}, k_{1})$, $(\lambda_{2}, k_{2})$ respectively, then 
\begin{equation*}
f_1 \times_{\varphi} f_2 := \varphi(f_{1}\otimes f_{2}) 
\end{equation*} 
is a ${\G}$-modular form of weight $(\lambda_{3}, k_{1}+k_{2})$. 
This is the ``$\varphi$-product" of $f_{1}$ and $f_{2}$. 
Note that a homomorphism \eqref{eqn: tensor decomposition} exists exactly when  
$V_{\lambda_{3}}$ appears in the irreducible decomposition of $V_{\lambda_{1}} \otimes V_{\lambda_{2}}$, 
and it is unique up to constant when the multiplicity is $1$. 
This information can be read off from the Littlewood-Richardson numbers 
(\cite{Ki0}, \cite{KT0}, see also \cite{Ok} \S 12). 

The map \eqref{eqn: tensor decomposition} also uniquely induces an ${\rm O}(V(I))$-homomorphism 
\begin{equation}\label{eqn: tensor decomposition V(I)}
\varphi_{I} : V(I)_{\lambda_{1},k_{1}} \otimes V(I)_{\lambda_{2},k_{2}} \to V(I)_{\lambda_{3},k_{1}+k_{2}}. 
\end{equation}
If we denote by $\iota$ the relevant $I$-trivialization maps, then we have 
\begin{equation}\label{eqn: product I-trivialization}
\iota (f_{1}) \times_{\varphi_{I}} \iota (f_{2}) = \iota (f_{1}\times_{\varphi} f_{2}). 
\end{equation}
In this sense, $\varphi$-product and $I$-trivialization are compatible. 
 
It will be useful to know how orthogonal weights $(\lambda, k)$ in $n=3, 4$ are translated by the accidental isomorphisms. 
For simplicity we assume ${}^t \lambda_{1}<n/2$, namely ${}^t \lambda_{1}=1$. 
See \S \ref{ssec: SO} for some justification of this assumption. 
(There is no essential loss of generality when $n=3$.) 
Henceforth we write $\lambda=(d)$ with $d$ a natural number.  

\begin{example}\label{ex: weight n=3}
Let $n=3$. 
Let $\mathcal{F}$ be the rank $2$ Hodge bundle considered in \S \ref{sssec: n=3}. 
Automorphic vector bundles on Siegel modular $3$-folds can be expressed as 
${\rm Sym}^{j}\mathcal{F}\otimes {\LL}^{\otimes l}$ with $j\in {\Z}_{\geq 0}$ and $l\in {\Z}$. 
In the literature this is often referred to as weight $({\rm Sym}^j, \det^{l})$. 
This corresponds to the highest weight 
$(\rho_{1}, \rho_{2})= (j+l, l)$ of ${\rm GL}(2, {\C})$. 
When $j=2d$ is even, we have 
\begin{equation*}
{\rm Sym}^{2d}\mathcal{F} \simeq ({\rm Sym}^{2}\mathcal{F})_{(d)} \simeq {\E}_{(d)}\otimes {\LL}^{\otimes d} 
\end{equation*}
by Lemma \ref{lem: 2nd Hodge n=3}. 
Therefore  
\begin{equation*}
{\rm Sym}^{2d}\mathcal{F}\otimes {\LL}^{\otimes l} \simeq {\E}_{(d)}\otimes {\LL}^{\otimes l+d}. 
\end{equation*}
Thus we have the following correspondence of weights: 
\begin{eqnarray*}
& & \textrm{orthogonal weight} \: \: ((d), k) \\ 
& \leftrightarrow & \textrm{Siegel weight} \: \: ({\rm Sym}^j, {\rm det}^{l}) \: \: \textrm{with} \: \: (j, l)=(2d, k-d) \\ 
& \leftrightarrow & {\rm GL}(2, {\C})\textrm{-weight} \: \: (\rho_{1}, \rho_{2}) = (k+d, k-d) 
\end{eqnarray*} 
\end{example}

\begin{example}\label{ex: weight n=4}
Let $n=4$. 
Let $\mathcal{F}$ and $\mathcal{G}$ be the rank $2$ Hodge bundles considered in \S \ref{sssec: n=4}. 
Automorphic vector bundles on Hermitian modular $4$-folds can be expressed as 
\begin{equation}\label{eqn: auto VB Hermitian 4-fold}
{\LL}^{\otimes k} \otimes {\rm Sym}^{j_{1}}\mathcal{F}\otimes {\rm Sym}^{j_{2}}\mathcal{G}, \qquad 
k\in {\Z}, \: \: j_{1}, j_{2} \in {\Z}_{\geq 0}. 
\end{equation} 
On the other hand, in \cite{FS} \S 2, weights of vector-valued Hermitian modular forms of degree $2$ are expressed as 
$(r, \, \rho_{1}\boxtimes \rho_{2})$ where $r\in {\Z}$ and 
$\rho_{1}, \rho_{2}$ are symmetric tensors of the standard representation of ${\rm GL}(2, {\C})$. 
(We are working with ${\rm SU}(2, 2)$ rather than ${\rm U}(2, 2)$, 
and we do not consider twist by a character as in \cite{FS}.) 
Then ${\LL}$ corresponds to the weight $r=1$, 
$\mathcal{F}$ corresponds to the weight $\rho_{1}={\rm St}$, 
and ${\LL}\otimes \mathcal{G}\simeq \mathcal{G}^{\vee}\simeq \iota^{\ast}\mathcal{F}$ corresponds to the weight $\rho_{2}={\rm St}$. 
Thus the vector bundle \eqref{eqn: auto VB Hermitian 4-fold} corresponds to the Hermitian weight 
$(r, \, \rho_{1}\boxtimes \rho_{2})$ with  
$r=k-j_{2}$, $\rho_{1}={\rm Sym}^{j_1}$ and $\rho_{2}={\rm Sym}^{j_2}$. 

Now, by Lemma \ref{lem: 2nd Hodge n=4}, we have 
\begin{equation*}
{\E}_{(d)}\simeq {\rm Sym}^{d}\mathcal{F}\otimes {\rm Sym}^{d}\mathcal{G}. 
\end{equation*} 
Therefore the orthogonal weight $((d), k)$ corresponds to the Hermitian weight 
$(r, \, \rho_{1}\boxtimes \rho_{2})$ with $r=k-d$ and $\rho_{1}=\rho_{2}={\rm Sym}^{d}$. 
In \cite{FS} \S 3 and \S 4, some examples in the case $d=1$ are studied in detail. 
\end{example}

\section{Tube domain realization}\label{ssec: tube domain}

In this section we recall the tube domain realization of ${\D}$ associated to a $0$-dimensional cusp. 
We refer the reader to \cite{GHS07}, \cite{Lo}, \cite{Ma1} for some more details. 
This section is preliminaries for the Fourier expansion (\S \ref{ssec: Fourier expansion}). 

We choose a rank $1$ primitive isotropic sublattice $I$ of $L$, 
which is fixed throughout \S \ref{ssec: tube domain} -- \S \ref{ssec: geometry Fourier}. 
Recall that this corresponds to the $0$-dimensional cusp $[I_{{\C}}]$ of ${\D}$. 
The ${\Z}$-module $(I^{\perp}/I)\otimes_{{\Z}} I$ is canonically endowed with the structure of a hyperbolic lattice, 
from the quadratic form on $I^{\perp}/I$ and the standard quadratic form $I\times I \to I^{\otimes 2}\simeq {\Z}$ on $I$ 
in which the generators of $I$ have norm $1$. 
For $F={\Q}, {\R}, {\C}$ we write 
\begin{equation*}
{\UIF} =(I^{\perp}/I)_F\otimes_F I_F = V(I)_{F}\otimes_{F} I_{F}.  
\end{equation*}
This is a quadratic space over $F$, hyperbolic when $F={\Q}, {\R}$. 

\subsection{Tube domain realization}

The linear projection ${\proj}L_{{\C}}\dashrightarrow {\proj}(L/I)_{{\C}}$ from the point $[I_{{\C}}]\in Q$ defines an isomorphism 
\begin{equation}\label{eqn: tube domain 1}
Q(I) \: \to \: {\proj}(L/I)_{{\C}}-{\proj}V(I). 
\end{equation}
We choose, as an auxiliary data, a rank $1$ sublattice $I'\subset L$ such that $(I, I')\ne 0$. 
This determines a base point of the affine space ${\proj}(L/I)_{{\C}}-{\proj}V(I)$ 
and hence an isomorphism 
\begin{equation}\label{eqn: tube domain 2}
{\proj}(L/I)_{{\C}}-{\proj}V(I) \to V(I)\otimes_{{\C}} I_{{\C}} = {\UIC}. 
\end{equation}
Since the quadratic form on ${\UIR}$ is hyperbolic, the set of vectors $v\in {\UIR}$ with $(v, v)>0$ consists of two connected components. 
The choice of the component ${\D}$ determines one of them, which we denote by $\mathcal{C}_{I}$ (the positive cone). 
Let 
\begin{equation*}
{\DI} = \{ Z\in {\UIC} \; | \;{\rm Im}(Z)\in \mathcal{C}_{I} \} 
\end{equation*}
be the tube domain associated to $\mathcal{C}_{I}$. 
Then the composition of \eqref{eqn: tube domain 1} and \eqref{eqn: tube domain 2} gives an isomorphism 
\begin{equation}\label{eqn: tube domain 3}
{\D} \: \stackrel{\simeq}{\to} \: {\DI} \subset U(I)_{{\C}}. 
\end{equation}
This is the tube domain realization of ${\D}$ associated to $I$. 
If we change $I'$, this isomorphism is shifted by the translation by a vector of ${\UIQ}$.

\subsection{Stabilizer}\label{sssec: stabilizer Q}

Next we recall the structure of the stabilizer of the $I$-cusp. 
Let $F={\Q}, {\R}$. 
We denote by ${\GIF}$ the stabilizer of $I$ in ${\rm O}^{+}(L_{F})$ (not the stabilizer of $I_{F}$). 
Elements of ${\GIF}$ act on ${\UIF}$ as isometries. 
Let ${\rm O}^{+}({\UIF})$ be the subgroup of ${\rm O}({\UIF})$ preserving the positive cone $\mathcal{C}_{I}$. 
By \eqref{eqn: stabilizer isotropic line}, ${\GIF}$ sits in the canonical exact sequence 
\begin{equation}\label{eqn: G(I)K}
0 \to {\UIF} \to {\GIF} \to {\rm O}^{+}({\UIF})\times {\rm GL}(I) \to 1. 
\end{equation}
Here the subgroup ${\UIF}$ consists of the Eichler transvections of $L_{F}$ with respect to the isotropic line $I_{F}$. 
The adjoint action of ${\GIF}$ on ${\UIF}$ via \eqref{eqn: G(I)K} coincides with the natural action of 
${\GIF}$ on $(I^{\perp}/I)_{F}\otimes I_{F}$. 

The choice of $I'$ determines the lift $V(I)_{F} \simeq (I_{F}\oplus I'_{F})^{\perp}$ of $V(I)_{F}$ in $I_{F}^{\perp}$, 
and thus a splitting $L_{F}\simeq U_{F}\oplus V(I)_{F}$. 
This determines a section of \eqref{eqn: G(I)K}  
\begin{equation*}
{\rm O}^{+}({\UIF})\times {\rm GL}(I) \hookrightarrow {\GIF}, 
\end{equation*}
by letting ${\rm O}^{+}({\UIF})\simeq {\rm O}^{+}(V(I)_{F})$ act on the lifted component $V(I)_{F}\subset L_{F}$ 
and mapping ${\rm GL}(I)=\{ \pm 1 \}$ to $\{ \pm {\rm id} \}$.  
In this way, from the choice of $I'$, we obtain a splitting of \eqref{eqn: G(I)K}: 
\begin{equation}\label{eqn: split G(I)K}
{\GIF} \simeq ({\rm O}^{+}({\UIF})\times {\rm GL}(I)) \ltimes {\UIF}, 
\end{equation}
where ${\rm O}^{+}({\UIF})$ acts on ${\UIF}$ in the natural way and ${\rm GL}(I)$ acts on ${\UIF}$ trivially. 
This splitting is compatible with the tube domain realization in the following sense. 
We translate the ${\GIF}$-action on ${\D}$ to action of ${\GIF}$ on ${\DI}$ 
via the tube domain realization \eqref{eqn: tube domain 3} defined by (the same) $I'$. 
Then, 
\begin{itemize}
\item the unipotent radical ${\UIF}\subset {\GIF}$ acts on ${\DI}$ as the translation by ${\UIF}$ on ${\UIC}$, 
\item the lifted group ${\rm O}^{+}({\UIF})$ in \eqref{eqn: split G(I)K} acts on ${\DI}$ by its linear action on ${\UIC}$, 
\item the lifted group ${\rm GL}(I)=\{ \pm {\rm id} \}$ acts trivially. 
\end{itemize}

Now let ${\G}$ be a finite-index subgroup of ${\OL}$. 
We write 
\begin{equation*}
{\GIZ}={\GIQ}\cap {\G}, \quad {\UIZ}={\UIQ}\cap {\G}, \quad {\GIZbar}={\GIZ}/{\UIZ}. 
\end{equation*}
The group ${\GIZ}$ is the stabilizer of $I$ in ${\G}$. 
The exact sequence 
\begin{equation}\label{eqn: exact seq GIZ}
0 \to {\UIZ} \to {\GIZ} \to {\GIZbar} \to 1 
\end{equation}
is naturally embedded in \eqref{eqn: G(I)K}. 
The group ${\UIZ}$ is a full lattice in ${\UIQ}$. 
It defines the algebraic torus 
\begin{equation*}
T(I)={\UIC}/{\UIZ}. 
\end{equation*}
Then the tube domain realization \eqref{eqn: tube domain 3} induces an isomorphism 
\begin{equation*}
{\D}/{\UIZ} \: \stackrel{\simeq}{\to} \: {\DI}/{\UIZ} \subset T(I). 
\end{equation*}
The group ${\GIZbar}$ acts on ${\D}/{\UIZ}\simeq {\DI}/{\UIZ}$. 
Let $\bar{\gamma}\in {\GIZbar}$ and let $\gamma\in {\GIZ}$ be its lift. 
According to the splitting \eqref{eqn: split G(I)K}, we express $\gamma$ as 
\begin{equation}\label{eqn: GIZZ element}
\gamma = (\gamma_{1}, \varepsilon, \alpha), \qquad 
\gamma_{1} \in {\rm O}^{+}({\UIZ}), \; \; \varepsilon=\pm {\rm id}, \; \; \alpha\in {\UIQ}. 
\end{equation}
Here $\gamma_{1}$, a priori an element of ${\rm O}^{+}({\UIQ})$, 
is contained in ${\rm O}^{+}({\UIZ})$ because the adjoint action of ${\GIZ}$ on ${\UIQ}$ preserves the lattice ${\UIZ}$. 
Then the action of $\bar{\gamma}$ on ${\DI}/{\UIZ}$ is given by the linear action by $\gamma_{1}$ plus 
the translation by $[\alpha] \in {\UIQ}/{\UIZ}$. 
Note that $\bar{\gamma}$ is determined by $(\gamma_{1}, \varepsilon)$ 
because the projection ${\GIZbar}\to {\rm O}^{+}({\UIQ})\times {\rm GL}(I)$ is injective. 
Nevertheless the translation component $[\alpha]$ could be nontrivial because \eqref{eqn: exact seq GIZ} may not necessarily split.


\section{Fourier expansion}\label{ssec: Fourier expansion}

Let $I$ and $I'$ be as in \S \ref{ssec: tube domain}. 
Let $f$ be a modular form of weight $(\lambda, k)$ on ${\D}$ with respect to a finite-index subgroup ${\G}$ of ${\OL}$. 
By the $I$-trivialization ${\Elk}\simeq {\VIlk}\otimes{\OD}$ and the tube domain realization ${\D}\simeq {\DI}$, 
we regard $f$ as a ${\VIlk}$-valued holomorphic function on the tube domain ${\DI}$ (again denoted by $f$). 
The subgroup ${\UIZ}$ of ${\GIZ}$ acts on ${\DI}$ by translation, and acts on ${\VIlk}$ trivially . 
By Lemma \ref{lem: f.a. stab}, this shows that the function $f$ is invariant under the translation by the lattice ${\UIZ}$. 
Therefore it admits a Fourier expansion of the form 
\begin{equation}\label{eqn: Fourier}
f(Z) = \sum_{l\in {\UIZZ}} a(l) q^{l}, \qquad q^{l}=e((l, Z)),  
\end{equation}
for $Z\in {\DI}$, where $a(l)\in {\VIlk}$ and ${\UIZZ}\subset {\UIQ}$ is the dual lattice of ${\UIZ}$. 
This series is convergent when ${\rm Im}(Z)$ is sufficiently large. 
The Fourier coefficients $a(l)$ can be expressed as 
\begin{equation}\label{eqn: Fourier coefficient integral}
a(l) = \int_{{\UIR}/{\UIZ}} f(Z_{0}+v) \: e(-(Z_{0}+v, l)) \: dv, 
\end{equation}
where $dv$ is the flat volume form on ${\UIR}$ normalized so that ${\UIR}/{\UIZ}$ has volume $1$. 

The Koecher principle says that we have $a(l)\ne 0$ only when 
$l$ is in the closure of the positive cone $\mathcal{C}_{I}$, which is the dual cone of $\mathcal{C}_{I}$. 
See, e.g., \cite{vdG} p.191 for a proof of the Koecher principle in the vector-valued Siegel modular case. 
The present case can be proved similarly by using \eqref{eqn: Fourier coefficient integral} and Proposition \ref{prop: Fourier coeff} below. 
See also \cite{Br} Proposition 4.15 for the scalar-valued case. 
In general, when $n\leq 2$, the condition $a(l)\ne 0 \Rightarrow l\in \overline{\mathcal{C}_{I}}$ 
is the holomorphicity condition required around the $I$-cusp. 
The modular form $f$ is called a \textit{cusp form} if $a(l)\ne0$ only when $l\in \mathcal{C}_{I}$ 
at every $0$-dimensional cusp $I$. 
We denote by ${\SG} \subset {\MG}$ the subspace of cusp forms. 

It should be noted that the Fourier expansion depends on the choice of $I'$. 
If we change $I'$, the tube domain realization is shifted by the translation by a vector of ${\UIQ}$, say $v_{0}$. 
Then we need to replace $f(Z)$ by $f(Z+v_{0})$, 
and the Fourier coefficient $a(l)$ is replaced by $e((l, v_{0})) \cdot a(l)$. 
In what follows, when we speak of Fourier expansion at the $I$-cusp, 
the choice of $I'$ (and hence of the tube domain realization ${\D}\to {\DI}$) is subsumed. 

The Fourier coefficients satisfy the following symmetry under ${\GIZbar}$. 

\begin{proposition}\label{prop: Fourier coeff}
Let $\bar{\gamma}\in {\GIZbar}$. 
Let $\gamma=(\gamma_{1}, \varepsilon, \alpha)$ be its lift in ${\GIZ}$ expressed as in \eqref{eqn: GIZZ element}. 
Then we have 
\begin{equation}\label{eqn: Fourier coeff GIZ}
a(\gamma_{1}l) = e(-(\gamma_{1}l, \alpha))\cdot \gamma(a(l)) 
\end{equation}
for every $l\in {\UIZZ}$. 
\end{proposition}

\begin{proof}
By Lemma \ref{lem: f.a. stab}, the factor of automorphy for $\gamma$ is given by its natural action on ${\VIlk}$. 
Therefore we have 
\begin{equation*}
f(\gamma(Z))= \gamma(f(Z)), \quad Z\in {\DI}, 
\end{equation*}
where $\gamma$ acts on ${\DI}$ via the tube domain realization ${\D}\simeq {\DI}$. 
We compute the Fourier expansion of both sides. 
Since $\gamma(Z)=\gamma_{1}Z+\alpha$, we have 
\begin{eqnarray*}
f(\gamma(Z)) & = & \sum_{l} a(l) e((l, \gamma_{1}Z+\alpha)) \\ 
& = & \sum_{l} a(l) e((l, \alpha)) e((\gamma_{1}^{-1}l, Z)) \\ 
& = & \sum_{l} a(\gamma_{1}l) e((\gamma_{1}l, \alpha)) e((l, Z)). 
\end{eqnarray*}
In the last equality we rewrote $l$ as $\gamma_{1}l$. 
Comparing this with 
\begin{equation*}
\gamma(f(Z)) = \sum_{l} \gamma(a(l)) e((l, Z)), 
\end{equation*}
we obtain 
$\gamma(a(l)) = e((\gamma_{1}l, \alpha)) a(\gamma_{1}l)$. 
\end{proof}

In the right hand side of \eqref{eqn: Fourier coeff GIZ}, the action of $\gamma$ on $a(l)\in {\VIlk}$ is determined by $(\gamma_{1}, \varepsilon)$. 
More precisely, $\gamma$ acts on $I_{{\C}}$ by $\varepsilon\in \{ \pm 1\}$, 
and on $V(I)={\UIC}\otimes {\ICv}$ by $\gamma_{1}\otimes \varepsilon$. 

Proposition \ref{prop: Fourier coeff} implies the vanishing of the constant term $a(0)$ in most cases. 

\begin{proposition}\label{cor: a(0)=0}
Assume that $\lambda \ne 1, \det$. 
Then $a(0)=0$. 
\end{proposition}

\begin{proof}
We apply Proposition \ref{prop: Fourier coeff} to $l=0$ and elements $\bar{\gamma}$ in the subgroup 
\begin{equation}\label{eqn: subgroup GIZbar a(0)=0}
\{ \: \bar{\gamma}\in {\GIZbar} \: | \: \varepsilon=1, \det \gamma_{1} =1 \: \}  
\end{equation}
of ${\GIZbar}$. 
By trivializing $I\simeq {\Z}$, we identify ${\VIlk}={\VIl}$. 
We also identify ${\rm SO}({\UIQ}) = {\rm SO}(V(I)_{{\Q}})$ naturally. 
Then elements $\bar{\gamma}$ of the group \eqref{eqn: subgroup GIZbar a(0)=0}  
act on ${\VIlk}$ by the action of $\gamma_{1}\in {\rm SO}(V(I)_{{\Q}})$ on ${\VIl}$. 
Therefore, by Proposition \ref{prop: Fourier coeff}, we find that 
$a(0)=\gamma_{1}(a(0)) \in {\VIl}$ 
for every such $\bar{\gamma}$. 
The mapping $\bar{\gamma}\mapsto \gamma_{1}$ 
embeds the group \eqref{eqn: subgroup GIZbar a(0)=0} into ${\rm SO}(V(I)_{{\Q}})$, 
and the image is an arithmetic subgroup of ${\rm SO}(V(I)_{{\Q}})$. 
By the density theorem of Borel \cite{Borel} (see also \cite{Ra} Corollary 5.15), 
it is Zariski dense in ${\rm SO}(V(I))$. 
Therefore the vector $a(0)\in {\VIl}$ is invariant under the action of ${\rm SO}(V(I))$ on ${\VIl}$. 
However, by our assumption $\lambda \ne 1, \det$, 
the ${\SOn}$-representation $V_{\lambda}$ contains no nonzero invariant vector (cf.~\S \ref{ssec: SO}). 
Therefore $a(0)=0$. 
\end{proof}

\begin{remark}
Since $V(I)$ and $I_{{\C}}$ have the natural ${\Q}$-structures $V(I)_{{\Q}}$ and $I_{{\Q}}$ respectively, 
${\VIlk}$ has the natural ${\Q}$-structure $V(I)_{{\Q},\lambda}\otimes (I_{{\Q}}^{\vee})^{\otimes k}$ where 
$V(I)_{{\Q},\lambda}=c_{\lambda}\cdot V(I)_{{\Q}}^{[d]}$ is the ${\Q}$-representation of ${\rm O}(V(I)_{{\Q}})$ 
obtained by applying the orthogonal Schur functor to $V(I)_{{\Q}}$. 
Thus we can speak of rationality and algebraicity of the Fourier coefficients $a(l)$. 
(Rationality depends on the choice of $I'$, but algebraicity does not 
because the transition constant $e((l, v_{0}))$ is a root of unity.) 
When the homomorphism $\varphi_{I}$ in \eqref{eqn: tensor decomposition V(I)} is defined over ${\Q}$, 
the $\varphi$-product of two modular forms with rational Fourier coefficients at the $I$-cusp 
again has rational Fourier coefficients by \eqref{eqn: product I-trivialization}. 
\end{remark}

\section{Geometry of Fourier expansion}\label{ssec: geometry Fourier}

Let $I$ and $I'$ be as in \S \ref{ssec: tube domain} and \S \ref{ssec: Fourier expansion}. 
In this section we recall the partial toroidal compactifications of ${\D}/{\UIZ}$ following \cite{AMRT} 
and explain the Fourier expansion from that point of view. 

\subsection{Partial toroidal compactification}\label{sssec: partial compact}

We write ${\XI}={\D}/{\UIZ}$. 
The tube domain realization identifies ${\XI}$ with the open set ${\DI}/{\UIZ}$ of the torus $T(I)$. 
Let $\mathcal{C}_{I}^{+}=\mathcal{C}_{I}\cup \bigcup_{v}{\R}_{\geq0}v$ 
be the union of the positive cone $\mathcal{C}_{I}$ and the rays ${\R}_{\geq0}v$ generated by 
rational isotropic vectors $v$ in $\overline{\mathcal{C}_{I}}$. 
Let $\Sigma_{I}=(\sigma_{\alpha})$ be a rational polyhedral cone decomposition of $\mathcal{C}_{I}^{+}$, 
namely a fan in ${\UIR}$ whose support is $\mathcal{C}_{I}^{+}$.  
Note that every rational isotropic ray in $\mathcal{C}_{I}^{+}$ must be included in $\Sigma_{I}$. 
We will often abbreviate $\Sigma_{I}=\Sigma$ when $I$ is clear from the context. 
The fan $\Sigma$ is said to be \textit{${\GIZ}$-admissible} if it is preserved by the ${\GIZ}$-action on ${\UIR}$ 
and there are only finitely many cones up to the ${\GIZ}$-action. 
The fan $\Sigma$ is called \textit{regular} if each cone $\sigma_{\alpha}$ is generated by a part of a ${\Z}$-basis of ${\UIZ}$. 
It is possible to choose $\Sigma$ to be ${\GIZ}$-admissible and regular (\cite{AMRT}, \cite{FC}). 

Let $\Sigma$ be ${\GIZ}$-admissible. 
It determines a ${\GIZbar}$-equivariant torus embedding $T(I)\hookrightarrow T(I)^{\Sigma}$. 
The toric variety $T(I)^{\Sigma}$ is normal; it is nonsingular if $\Sigma$ is regular. 
The cones $\sigma$ in $\Sigma$ correspond to the boundary strata of $T(I)^{\Sigma}$, say $\Delta_{\sigma}$. 
A stratum $\Delta_{\sigma}$ is in the closure of another stratum $\Delta_{\tau}$ if and only if $\tau$ is a face of $\sigma$. 
Each stratum $\Delta_{\sigma}$ is isomorphic to the quotient torus of $T(I)$ defined by the quotient lattice 
${\UIZ}/{\UIZ}\cap \langle {\sigma} \rangle$, where $\langle {\sigma} \rangle$ is the ${\R}$-span of $\sigma$. 
In particular, the rays ${\R}_{\geq0}v$ in $\Sigma$ correspond to the boundary strata of codimension $1$, say $\Delta_{v}$. 
If we take $v$ to be a primitive vector of ${\UIZ}$, 
the stratum $\Delta_{v}$ is isomorphic to the quotient torus of $T(I)$ defined by ${\UIZ}/{\Z}v$.  
The variety $T(I)^{\Sigma}$ is nonsingular along $\Delta_{v}$. 
If we take a vector $l\in {\UIZZ}$ with $(v, l)=1$, 
then $q^{l}=e((l, Z))$ is a character of $T(I)$ and extends holomorphically over $\Delta_{v}$. 
The divisor $\Delta_v$ is defined by $q^{l}=0$. 
More generally, a character $q^{l}$ of $T(I)$ where $l\in {\UIZZ}$ extends holomorphically 
around a boundary stratum $\Delta_{\sigma}$ 
(i.e., extends over $\Delta_{\sigma}$ and the strata $\Delta_{\tau}$ which contains $\Delta_{\sigma}$ in its closure)  
if and only if $(l, \sigma)\geq 0$, or in other words, $l$ is in the dual cone of $\sigma$. 
If moreover $l$ has positive pairing with the relative interior of $\sigma$, 
the extended function vanishes identically at $\Delta_{\sigma}$. 

Now let ${\XIcpt}$ be the interior of the closure of ${\XI}$ in $T(I)^{\Sigma}$. 
We call ${\XIcpt}$ the \textit{partial toroidal compactification} of ${\XI}$ defined by the fan $\Sigma$. 
As a partial compactification of ${\XI}={\D}/{\UIZ}$, this does not depend on the choice of $I'$. 
When a cone $\sigma\in \Sigma$ is not an isotropic ray, its relative interior is contained in $\mathcal{C}_{I}$, 
and the corresponding boundary stratum $\Delta_{\sigma}$ of $T(I)^{\Sigma}$ is totally contained in ${\XIcpt}$. 
On the other hand, when $\sigma={\R}_{\geq 0}v$ is an isotropic ray, only an open subset of $\Delta_{v}$ is contained in ${\XIcpt}$. 
(This will be glued with the boundary of the partial toroidal compactification over the corresponding $1$-dimensional cusp: 
see \S \ref{ssec: partial compact}.) 
By abuse of notation, we still write $\Delta_{v}$ for the boundary stratum in ${\XIcpt}$ in this case.

\subsection{Fourier expansion and Taylor expansion}\label{sssec: Fourier revisit}
 
Let $f(Z)=\sum_{l}a(l)q^{l}$ be the Fourier expansion of a ${\G}$-modular form of weight $(\lambda, k)$ at the $I$-cusp. 
This can be viewed as the expansion of the ${\VIlk}$-valued function $f$ on ${\XI}$ by the characters of $T(I)$. 

\begin{lemma}\label{lem: modular form extend}
The function $f$ on ${\XI}$ extends holomorphically over ${\XIcpt}$. 
When $\lambda\ne 1, \det$ and $\sigma$ is not an isotropic ray, 
$f$ vanishes at the corresponding boundary stratum $\Delta_{\sigma}$. 
When $f$ is a cusp form, it vanishes at every boundary stratum $\Delta_{\sigma}$. 
\end{lemma}

\begin{proof}
Since the dual cone of $\overline{\mathcal{C}_{I}}$ is $\overline{\mathcal{C}_{I}}$ itself,  
$\overline{\mathcal{C}_{I}}$ is contained in the dual cone of every cone $\sigma$ in $\Sigma$. 
Therefore, if $l\in {\UIZZ}\cap \overline{\mathcal{C}_{I}}$, then $l$ is contained in the dual cone of every $\sigma$, 
which implies that the function $q^{l}$ extends holomorphically over ${\XIcpt}$. 
By the cusp condition in the Fourier expansion, 
this shows that the function $f$ extends holomorphically over ${\XIcpt}$. 

When $\sigma$ is not an isotropic ray, its relative interior is contained in $\mathcal{C}_{I}$. 
Hence any nonzero vector $l \in {\UIZZ}\cap \overline{\mathcal{C}_{I}}$ 
has positive pairing with the relative interior of $\sigma$. 
This shows that the corresponding character $q^{l}$ vanishes at the boundary stratum $\Delta_{\sigma}$. 
It follows that $f|_{\Delta_{\sigma}}$ is the constant $a(0)$. 
By Proposition \ref{cor: a(0)=0}, this vanishes when $\lambda \ne 1, \det$. 

Finally, if $f$ is a cusp form, we have $a(l)\ne 0$ only when $l\in \mathcal{C}_{I}$. 
Such a vector $l$ has positive pairing with the relative interior of every cone $\sigma\in \Sigma$, 
and so $q^{l}$ vanishes at $\Delta_{\sigma}$. 
Therefore $f$ vanishes at the boundary of ${\XIcpt}$. 
\end{proof}

Let us explain that the Fourier expansion gives Taylor expansion along each boundary divisor. 
Let $\sigma={\R}_{\geq0}v$ be a ray in $\Sigma$ with $v\in {\UIZ}$ primitive. 
We can rewrite the Fourier expansion of $f$ as 
\begin{equation}\label{eqn: Fourier expansion regroup}
f(Z) = \sum_{m\geq0}  \sum_{\substack{l\in {\UIZZ} \\ (l, v)=m }} a(l)q^l. 
\end{equation}
We choose a vector $l_{0}\in {\UIZZ}$ with $(l_0, v)=1$ and put $q_0=q^{l_{0}}$. 
The boundary divisor $\Delta_{v}$ is defined by $q_{0}=0$. 
We put 
\begin{equation*}
\phi_{m} = \sum_{\substack{l\in {\UIZZ} \\ (l, v)=m }} a(l)q^{l-ml_0} 
=\sum_{l\in v^{\perp}\cap {\UIZZ}} a(l+ml_0)q^l. 
\end{equation*}
Note that $v^{\perp}\cap {\UIZZ}$ is the dual lattice of ${\UIZ}/{\Z}v$ 
and hence is the character lattice of the quotient torus $\Delta_v$. 
Therefore $\phi_{m}$ is (the pullback of) a ${\VIlk}$-valued function on $\Delta_{v}$. 
Then \eqref{eqn: Fourier expansion regroup} can be rewritten as 
\begin{equation*}
f(Z) = \sum_{m\geq 0} \phi_{m} q_{0}^{m}. 
\end{equation*}
This is the Taylor expansion of $f$ along the divisor $\Delta_{v}$ with normal parameter $q_{0}$, 
and $\phi_{m}$ (as a function on $\Delta_{v}$) is the $m$-th Taylor coefficient. 
In particular, the restriction of $f$ to $\Delta_{v}$ is given by $\phi_{0}$: 
\begin{equation*}
f|_{\Delta_{v}} = \phi_{0} = \sum_{l\in v^{\perp}\cap {\UIZZ}}a(l)q^{l}. 
\end{equation*}
When $(v, v)\ne 0$, this reduces to $a(0)$ because $v^{\perp}\cap \overline{\mathcal{C}_{I}}= \{ 0 \}$ holds 
(cf.~the proof of Lemma \ref{lem: modular form extend}). 
On the other hand, when $(v, v)=0$, this reduces to 
\begin{equation}\label{eqn: restrict torus boundary divisor}
f|_{\Delta_{v}}  = \sum_{l\in {\Q}v\cap {\UIZZ}}a(l)q^{l} 
\end{equation}
because $v^{\perp}\cap \overline{\mathcal{C}_{I}}={\R}_{\geq 0}v$. 

\begin{remark}\label{rmk: when l_0 not in UIZZ}
Sometimes it is useful to allow $l_0$ from an overlattice of ${\UIZZ}$, 
e.g., when considering the Fourier-Jacobi expansion (\S \ref{sec: FJ}). 
Then $q_{0}$ and $\phi_{m}$ are still defined, as functions on a finite cover of $T(I)$. 
\end{remark}

\subsection{Canonical extension}\label{sssec: cano exte 0dim cusp}

In \S \ref{ssec: Fourier expansion} and \S \ref{ssec: geometry Fourier}, 
we regarded modular forms as ${\VIlk}$-valued functions via the $I$-trivialization. 
Let us go back to the viewpoint of sections of ${\Elk}$. 
The vector bundle ${\Elk}$ on ${\D}$ descends to a vector bundle on ${\XI}={\D}/{\UIZ}$, 
which we again denote by ${\Elk}$. 
We extend it over ${\XIcpt}$ as follows. 

Since the $I$-trivialization ${\Elk}\simeq {\VIlk}\otimes {\OD}$ is equivariant with respect to ${\UIZ}$ 
which acts on ${\VIlk}$ trivially,  
it descends to an isomorphism ${\Elk}\simeq {\VIlk}\otimes \mathcal{O}_{{\XI}}$ over ${\XI}$. 
Then we can extend ${\Elk}$ to a vector bundle over ${\XIcpt}$ (still denoted by ${\Elk}$) 
so that this isomorphism extends to ${\Elk}\simeq {\VIlk}\otimes \mathcal{O}_{{\XIcpt}}$ over ${\XIcpt}$. 
In other words, the extension is defined so that the frame of ${\Elk}$ over ${\XI}$ 
corresponding to a basis of ${\VIlk}$ by the $I$-trivialization extends to a frame of the extended bundle ${\Elk}$. 
This is an explicit form of Mumford's canonical extension \cite{Mu}. 
By construction, a section $f$ of ${\Elk}$ over ${\XI}$ extends  
to a holomorphic section of the extended bundle ${\Elk}$ over ${\XIcpt}$ 
if and only if $f$ viewed as a ${\VIlk}$-valued function via the $I$-trivialization extends holomorphically over ${\XIcpt}$. 
Then Lemma \ref{lem: modular form extend} can be restated as follows. 

\begin{lemma}\label{lem: modular form extend II}
A modular form $f\in {\MG}$ as a section of ${\Elk}$ over ${\XI}$ extends to 
a holomorphic section of the extended bundle ${\Elk}$ over ${\XIcpt}$. 
When $\lambda \ne 1, \det$ and $\sigma$ is not an isotropic ray, 
this extended section vanishes at $\Delta_{\sigma}$. 
When $f$ is a cusp form, this section vanishes at every $\Delta_{\sigma}$. 
\end{lemma}

\section{Special orthogonal groups}\label{ssec: SO}

In the theory of orthogonal modular forms, there is an option at the outset: which Lie group to mainly work with. 
The full orthogonal group ${\rm O}$, or the special orthogonal group ${\rm SO}$, or the spin group ${\rm Spin}$, 
or even the pin group ${\rm Pin}$. 
We decided to start with ${\rm O}$ for two reasons: 
(1) in some applications we need to consider subgroups ${\G}$ of ${\OL}$ not contained in ${\rm SO}^{+}(L)$, and 
(2) the explicit construction by the orthogonal Schur functor for ${\E}$ will be useful at some points. 

On the other hand, it is sometimes more convenient to work with ${\rm SO}$. 
In this section we explain the switch from ${\rm O}$ to ${\rm SO}$. 
The contents of this section will be used only in 
\S \ref{ssec: reduce Fourier coefficient}, \S \ref{sec: L2} and \S \ref{sec: VT II}, 
so the reader may skip it for the moment. 

\subsection{Representations of ${\SOn}$}\label{sssec: SO rep}

We first recall some basic facts from the representation theory of ${\SOn}$ following 
\cite{Ok} \S 4, \S 8 and \cite{FH} \S 19. 
Irreducible representations of ${\SOn}$ are labelled by their highest weights. 
When $n=2m$ is even, the highest weights are expressed by 
$m$-tuples $\rho=(\rho_1, \cdots, \rho_m)$ of integers, nonnegative for $i<m$, such that 
$\rho_1\geq \cdots \geq \rho_{m-1} \geq |\rho_{m}|$. 
We write $\rho^{\dag}=(\rho_1, \cdots, \rho_{m-1}, -\rho_m)$ for such $\rho$. 
When $n=2m+1$ is odd, the highest weights are expressed by 
$m$-tuples $\rho=(\rho_1, \cdots, \rho_m)$ of nonnegative integers such that $\rho_1\geq \cdots \geq \rho_{m}\geq 0$. 
We denote by $W_{\rho}$ the irreducible representation of ${\SOn}$ with highest weight $\rho$. 
The dual representation $W_{\rho}^{\vee}$ is isomorphic to $W_{\rho}$ itself when $n$ is odd or $4|n$, 
while it is isomorphic to $W_{\rho^{\dag}}$ in the case $n\equiv 2$ mod $4$. 

By the Weyl unitary trick, $W_{\rho}$ remains irreducible as a representation of ${\rm SO}(n, {\R})\subset {\SOn}$, 
and the above classification is the same as the classification of irreducible ${\C}$-representations of ${\rm SO}(n, {\R})$. 

The restriction rule from ${\On}$ to ${\SOn}$ is as follows (\cite{Ok} Proposition 8.24). 
Let $\lambda=(\lambda_1\geq \cdots \geq \lambda_n\geq 0)$ be a partition expressing 
an irreducible representation $V_{\lambda}$ of ${\On}$. 
We define a highest weight $\bar{\lambda}$ for ${\SOn}$ by 
\begin{equation*}
\bar{\lambda} = (\lambda_{1}-\lambda_{n}, \lambda_{2}-\lambda_{n-1}, \cdots, \lambda_{[n/2]}-\lambda_{n+1-[n/2]}). 
\end{equation*}
Note that $\bar{\lambda}$ itself can be viewed as a partition for ${\On}$. 
When $n$ is odd or $n=2m$ is even with ${}^t\lambda_{1}\ne m$, 
the ${\On}$-representation $V_{\lambda}$ remains irreducible as a representation of ${\SOn}$, 
with highest weight $\bar{\lambda}$. 
The vector defined in \eqref{eqn: Vlambda highest weight vector} is a highest weight vector. 
Thus $V_{\lambda}\simeq W_{\bar{\lambda}}$ as a representation of ${\SOn}$ in this case. 
In particular, since the highest weight for the partition $\bar{\lambda}$ is $\bar{\lambda}$ itself, 
we have $V_{\lambda}\simeq V_{\bar{\lambda}}$ as ${\SOn}$-representations. 
More specifically, 
when ${}^{t}\lambda_{1}<n/2$ we have $\bar{\lambda}=\lambda$, 
while when ${}^{t}\lambda_{1}>n/2$ 
we have $V_{\lambda}\simeq V_{\bar{\lambda}}\otimes \det$ as ${\On}$-representations. 
(In the latter case, the partitions $\lambda$ and $\bar{\lambda}$ are called \textit{associated} in \cite{Ok} and \cite{FH}.) 

In the remaining case, namely when $n=2m$ is even and ${}^t\lambda_{1}=m$, 
$V_{\lambda}$ gets reducible when restricted to ${\SOn}$. 
More precisely,  
\begin{equation}\label{eqn: restrict SO vs O exceptional}
V_{\lambda}\simeq  W_{\bar{\lambda}} \oplus  W_{\bar{\lambda}^{\dag}} 
\end{equation} 
as a representation of ${\SOn}$. 
Note that $\bar{\lambda}=\lambda$ and $\lambda_{m}\ne 0$ in this case. 
Since $\bar{\lambda}\ne \bar{\lambda}^{\dag}$, this decomposition is unique. 
In this case, $V_{\lambda}$ is the induced representation from 
the representation $W_{\bar{\lambda}}$ of ${\SOn}\subset {\On}$.

\subsection{Automorphic vector bundles}\label{ssec: automorphic SO}

We go back to the automorphic vector bundles on ${\D}$. 
We choose a base point $[\omega_{0}]\in {\D}$. 
Let $K\simeq {\rm SO}(2, {\R})\times {\rm O}(n, {\R})$ and $S \! K\simeq {\rm SO}(2, {\R})\times {\rm SO}(n, {\R})$ 
be the stabilizers of $[\omega_{0}]$ in ${\OLR}$ and in ${\SOLR}$ respectively (cf.~\S \ref{ssec: domain}). 

\begin{proposition}\label{prop: El SO}
The following holds. 

(1) If either $n$ is odd or $n=2m$ is even with ${}^t\lambda_{1}\ne m$, 
then ${\El}$ remains irreducible as an ${\SOLR}$-equivariant vector bundle, 
and we have ${\El}\simeq {\SOLR}\times_{S\! K} W_{\bar{\lambda}}$. 
In particular, we have ${\El}\simeq {\E}_{\bar{\lambda}}$ as ${\SOLR}$-equivariant vector bundles. 

(2) If $n$ is even and ${}^t\lambda_{1}=n/2$, 
then ${\El}$ as an ${\SOLR}$-vector bundle decomposes into the direct sum of two non-isomorphic vector bundles: 
\begin{equation}\label{eqn: decompose O vs SO}
{\El}\simeq \mathcal{E}_{\lambda}^{+}\oplus \mathcal{E}_{\lambda}^{-} 
\end{equation}
with each component isomorphic to 
${\SOLR}\times_{S\! K} W_{\bar{\lambda}}$ and ${\SOLR}\times_{S\! K} W_{\bar{\lambda}^{\dag}}$ 
respectively. 
\end{proposition}

\begin{proof}
By \eqref{eqn: El=GKVl}, we have 
${\El} \simeq {\OLR}\times_{K} V_{\lambda}$ 
as an ${\OLR}$-equivariant vector bundle. 
Therefore, as an ${\SOLR}$-equivariant vector bundle, we have 
${\El} \simeq {\SOLR}\times_{S\! K} V_{\lambda}$. 
Note that the representation of 
${\rm O}(n, {\R}) \simeq {\rm O}(H_{\omega_{0}}^{\perp})\subset K$ on 
$V_{\lambda}=(\omega_{0}^{\perp}/{\C}\omega_{0})_{\lambda} \simeq (H_{\omega_{0}}^{\perp}\otimes_{{\R}}{\C})_{\lambda}$ 
extends to a representation of 
${\On}\simeq {\rm O}(H_{\omega_{0}}^{\perp}\otimes_{{\R}} {\C})$ naturally. 
Then our assertions follow from the restriction rule for ${\SOn}\subset {\On}$. 
\end{proof} 

At each fiber, the decomposition \eqref{eqn: decompose O vs SO} is 
the irreducible decomposition of $(\omega^{\perp}/{\C}\omega)_{\lambda}$ 
as a representation of ${\rm SO}(\omega^{\perp}/{\C}\omega)$. 
The $I$-trivialization respects the decomposition \eqref{eqn: decompose O vs SO} in the following sense. 
As a representation of ${\rm SO}(V(I))$, ${\VIl}$ decomposes according to \eqref{eqn: restrict SO vs O exceptional}, 
which we denote by ${\VIl}=W(I)_{\bar{\lambda}}\oplus W(I)_{\bar{\lambda}^{\dag}}$. 
By the uniqueness of the decomposition \eqref{eqn: restrict SO vs O exceptional}, 
the $I$-trivialization ${\El}\simeq {\VIl}\otimes{\OD}$ sends the decomposition \eqref{eqn: decompose O vs SO} of ${\El}$ to 
the decomposition 
\begin{equation*}
{\VIl}\otimes{\OD} = (W(I)_{\bar{\lambda}}\otimes{\OD}) \oplus (W(I)_{\bar{\lambda}^{\dag}}\otimes{\OD})  
\end{equation*}
of ${\VIl}\otimes{\OD}$. 
Thus we have the $I$-trivializations 
\begin{equation}\label{eqn: I-trivialization SO reducible}
\mathcal{E}_{\lambda}^{+} \simeq W(I)_{\bar{\lambda}}\otimes{\OD}, \qquad  
\mathcal{E}_{\lambda}^{-} \simeq W(I)_{\bar{\lambda}^{\dag}}\otimes{\OD} 
\end{equation}
of each component $\mathcal{E}_{\lambda}^{+}$, $\mathcal{E}_{\lambda}^{-}$.

\section{Rankin-Cohen brackets}\label{ssec: Rankin-Cohen}

In this section, as an example of explicit construction of vector-valued modular forms, 
we define the Rankin-Cohen bracket of two scalar-valued modular forms. 
This is a general method: see, e.g., \cite{Sa}, \cite{Ib}, \cite{CG13}, \cite{FS}, \cite{FS2} 
for the case of some other types of modular forms, 
where Rankin-Cohen bracket is a successful technique for explicitly describing some modules of vector-valued modular forms. 

Let $f$, $g$ be nonzero scalar-valued modular forms of weight $k$, $l$ respectively for ${\G}<{\OL}$. 
We define the Rankin-Cohen bracket of $f$ and $g$ by 
\begin{equation*}
\{ f, g \} = (g^{k+1}/f^{l-1}) \otimes d(f^{l}/g^{k}). 
\end{equation*}
Here $g^{k+1}/f^{l-1}$ is a meromorphic section of ${\LL}^{\otimes l(k+1)-k(l-1)}={\LL}^{\otimes k+l}$, and 
$d(f^{l}/g^{k})$ is the exterior differential of the meromorphic function $f^{l}/g^{k}$ on ${\D}$. 
Thus $d(f^{l}/g^{k})$ is a meromorphic $1$-form on ${\D}$. 
It is immediate to see that $\{ g, f \} = - \{ f, g \}$. 
When $k=l$, the Rankin-Cohen bracket reduces to the more simple expression 
\begin{eqnarray*}
\{ f, g \} & = & 
(g^{k+1}/f^{k-1})\otimes k \, (f/g)^{k-1} \cdot d(f/g) \\ 
& = & k \, g^{2} \otimes d(f/g). 
\end{eqnarray*}

\begin{proposition}\label{prop: Rankin-Cohen}
The Rankin-Cohen bracket $\{ f, g \}$ is a modular form of weight $({\rm St}, k+l+1)$ for ${\G}$. 
We have $\{ f, g \} \ne 0$ unless when $f^{l}$ is a constant multiple of $g^{k}$. 
\end{proposition}

\begin{proof}
Since $g^{k+1}/f^{l-1}$ and $d(f^{l}/g^{k})$ are meromorphic sections of 
${\LL}^{\otimes k+l}$ and $\Omega_{{\D}}^{1} \simeq {\E}\otimes {\LL}$ respectively, 
$\{ f, g \}$ is a meromorphic section of ${\E}\otimes {\LL}^{\otimes k+l+1}$, i.e., has weight $({\rm St}, k+l+1)$. 
The ${\G}$-invariance is obvious from the definition. 
It remains to check the holomorphicity over ${\D}$. 
We take a frame $s$ of ${\LL}$ and write $f=\tilde{f}s^{\otimes k}$, $g=\tilde{g}s^{\otimes l}$ 
with $\tilde{f}$, $\tilde{g}$ holomorphic functions on ${\D}$. 
Then 
\begin{eqnarray*}
\{ f, g \} & = & (\tilde{g}^{k+1}/\tilde{f}^{l-1}) s^{\otimes k+l} \otimes d(\tilde{f}^{l}/\tilde{g}^{k}) \\ 
& = & s^{\otimes k+l} \otimes (l \, (d\tilde{f}) \tilde{g} - k\, (d\tilde{g})\tilde{f}). 
\end{eqnarray*}
From this expression, we find that $\{ f, g\}$ is holomorphic. 
The nonvanishing assertion is apparent. 
\end{proof}

When $f=0$ or $g=0$, we simply set $\{ f, g\}=0$. 
Then the Rankin-Cohen bracket defines a bilinear map 
\begin{equation*}
M_{k}({\G}) \times M_{l}({\G}) \to M_{{\rm St}, k+l+1}({\G}). 
\end{equation*}
When $k=l$, this induces 
$\wedge^{2}M_{k}({\G}) \to M_{{\rm St}, 2k+1}({\G})$ 
by the anti-commutativity. 

\section{Higher Chow cycles on $K3$ surfaces}\label{ssec: higher Chow}

One of the geometric significance of vector-valued modular forms on ${\D}$ is 
the appearance of the middle graded piece of the Hodge filtration, 
while scalar-valued modular forms are concerned only with the first piece. 
Thus the connection between modular forms and geometry related to the variation of Hodge structures on ${\D}$ 
shows up fully. 
In this section we present such an example of geometric construction of 
vector-valued modular forms with singularities. 
This section is independent of the rest of the monograph. 

Let $\pi\colon X\to B$ be a smooth family of $K3$ surfaces.
We say that $\pi\colon X\to B$ is \textit{lattice-polarized} with period lattice $L$ 
if we have a sub local system $\Lambda_{NS}$ of $R^{2}\pi_{\ast}{\Z}$ 
whose fibers are primitive hyperbolic sublattices of the N\'eron-Severi lattices of the $\pi$-fibers $X_{b}$ 
and the fibers of $\Lambda_{T}=\Lambda_{NS}^{\perp}$ are isometric to $L$. 
Let $\tilde{B}$ be an unramified cover of $B$ where the local system $\Lambda_{T}$ can be trivialized (e.g., the universal cover) 
and let $\tilde{X}=X\times_{B}\tilde{B}$. 
After choosing a base point $o\in \tilde{B}$ and an isometry $(\Lambda_{T})_{o}\simeq L$, 
we have the period map 
\begin{equation*}
\tilde{\mathcal{P}}: \tilde{B} \to {\D}, \quad b\mapsto [H^{2,0}(\tilde{X}_{b})\subset L_{{\C}}]. 
\end{equation*}
If ${\G}$ is a finite-index subgroup of ${\OL}$ which contains the monodromy group of $\Lambda_{T}$, 
$\tilde{\mathcal{P}}$ descends to a holomorphic map 
\begin{equation*}
\mathcal{P}: B \to \mathcal{F}({\G}).  
\end{equation*}
When $B$ is algebraic, $\mathcal{P}$ is a morphism of algebraic varieties by Borel's extension theorem. 

Let $Z=(Z_{b})$ be a family of higher Chow cycles in $CH^{2}(X_{b}, 1)$. 
By this, we mean that 
\begin{itemize}
\item $Z$ is a higher Chow cycle of type $(2, 1)$ on the total space $X$, i.e., 
a codimension $2$ cycle on $X\times \mathbb{A}^{1}$ which meets $X\times \{ 0\}$ and $X\times \{ 1\}$ properly and 
satisfies $Z|_{X\times \{ 0\}}=Z|_{X\times \{ 1\}}$, and 
\item the restriction $Z_{b}=Z|_{X_{b}}$ to each fiber $X_{b}$ is well-defined, i.e., 
without using the moving lemma, $Z$ already intersects with $X_{b}\times \mathbb{A}^{1}$ properly and 
gives a higher Chow cycle on $X_{b}$. 
\end{itemize} 
The normal function $\nu_{Z}$ of $Z$ is defined 
as a holomorphic section of the fibration of the generalized intemediate Jacobians 
$\mathcal{H}/(F^{2}\mathcal{H}+R^{2}\pi_{\ast}{\Z})$. 
Here $\mathcal{H}=R^{2}\pi_{\ast}{\C}\otimes \mathcal{O}_{B}$ and 
$(F^{p}\mathcal{H})_{p}$ is the Hodge filtration on $\mathcal{H}$. 
The infinitesimal invariant $\delta\nu_{Z}$ of $\nu_{Z}$ is defined 
as a section of the middle cohomology sheaf of the Koszul complex 
\begin{equation}\label{eqn: Koszul geometric}
F^{2}\mathcal{H} \to (F^{1}\mathcal{H}/F^{2}\mathcal{H}) \otimes \Omega_{B}^{1} 
\to (\mathcal{H}/F^{1}\mathcal{H})\otimes \Omega_{B}^{2} 
\end{equation}
over $B$. 
See \cite{Vo}, \cite{Co2} for more details and examples. 

We explain the connection with vector-valued modular forms. 
We first consider the case where $\tilde{B}=B$ is an analytic open set of ${\D}$ 
and the period map $B\to {\D}$ coincides with the inclusion map. 
Then we can identify 
\begin{equation*}
F^{2}\mathcal{H}={\LL}|_{B}, \quad 
F^{1}\mathcal{H}/F^{2}\mathcal{H} = {\E}|_{B} \oplus (\Lambda_{NS}\otimes_{{\Z}}\mathcal{O}_{B}), \quad 
\mathcal{H}/F^{1}\mathcal{H}={\LL}^{-1}|_{B}. 
\end{equation*}
The Koszul complex \eqref{eqn: Koszul geometric} is the direct sum of 
the complex 
\begin{equation*}
0\to \Lambda_{NS}\otimes \Omega_{B}^{1}\to 0 
\end{equation*}
and the modular Koszul complex \eqref{eqn: Koszul} restricted to $B$: 
\begin{equation*}  
{\LL} \to {\E}\otimes \Omega_{B}^{1} \to {\LL}^{-1}\otimes \Omega_{B}^{2}. 
\end{equation*}
According to this decomposition, we can write $\delta\nu_{Z}$ as 
$((\delta\nu_{Z})_{pol}, (\delta\nu_{Z})_{prim})$ 
where $(\delta\nu_{Z})_{pol}$ is a section of $\Lambda_{NS}\otimes \Omega_{B}^{1}$ and 
$(\delta\nu_{Z})_{prim}$ is a section of the middle cohomology sheaf of the modular Koszul complex over $B$. 
By the calculation in Example \ref{ex: Koszul}, 
we see that 
\begin{equation*}
(\delta\nu_{Z})_{prim} \in H^{0}(B, {\E}_{(2)}\otimes {\LL}), 
\end{equation*}
namely $(\delta\nu_{Z})_{prim}$ is a local modular form of weight $(\lambda, k)=((2), 1)$ over $B$. 

Now we consider the case where the family $\pi\colon X\to B$ is algebraic, $-{\rm id}\not\in {\G}$,  
and the algebraic period map $\mathcal{P}\colon B\to{\FG}$ is birational. 
By removing some divisors from $B$ if necessary, we may assume that $\mathcal{P}$ is an open immersion and 
${\D}\to {\FG}$ is unramified over $B\subset {\FG}$. 
Then we may take $\tilde{B}$ to be a ${\G}$-invariant Zariski open set of ${\D}$. 
In this case, the Koszul complex \eqref{eqn: Koszul geometric} over $B$ is the direct sum of 
$0\to \Lambda_{NS}\otimes \Omega_{B}^{1}\to 0$ 
and the descent of the modular Koszul complex \eqref{eqn: Koszul} from $\tilde{B}\subset {\D}$ to $B\subset {\FG}$.  
Let $Z$ be a family of higher Chow cycles on $X\to B$ as above. 
According to the decomposition of the Koszul complex over $B$, we can write 
\begin{equation*}
\delta\nu_{Z} = ((\delta\nu_{Z})_{pol}, \; (\delta\nu_{Z})_{prim}) 
\end{equation*}
as in the local case. 
Then the pullback of the primitive part $(\delta\nu_{Z})_{prim}$ to $\tilde{B}$ 
is a ${\G}$-invariant holomorphic section of ${\E}_{(2)}\otimes {\LL}$ over $\tilde{B}$. 
By a vanishing theorem proved later (Theorem \ref{thm: VT I}), 
there is no nonzero holomorphic modular form of weight $((2), 1)$ on ${\D}$. 
Hence, if $(\delta\nu_{Z})_{prim}$ does not vanish identically, 
it must have a singularity at some component of the complement of $\tilde{B}$ in ${\D}$. 
In other words, the primitive part $(\delta\nu_{Z})_{prim}$ of the infinitesimal invariant $\delta\nu_{Z}$ of $Z$ 
is a modular form of weight $((2), 1)$ with singularities. 


\chapter{Witt operators}\label{sec: pullback}

In this chapter, as a functorial aspect of the theory, 
we study pullback of vector-valued modular forms to sub orthogonal modular varieties, 
an operation sometimes called the \textit{Witt operator}. 
Let $L$ be a lattice of signature $(2, n)$ and $L'$ be a primitive sublattice of $L$ of signature $(2, n')$. 
We put $K=(L')^{\perp}\cap L$ and $r={\rm rank}(K)=n-n'$. 
If we write ${\D}'=\mathcal{D}_{L'}$, then ${\D}'={\proj}L'_{{\C}}\cap {\D}$. 
Let $f$ be a vector-valued modular form on ${\D}$. 
In \S \ref{ssec: pullback} we study the restriction of $f$ to $\mathcal{D}'$. 
This produces a vector-valued modular form on $\mathcal{D}'$, 
whose weight (in general reducible) can be known from the branching rule for ${\rm O}(n', {\C})\subset {\On}$. 
An immediate consequence is the vanishing of $M_{\lambda,k}({\G})$ in $k\leq 0$ (Proposition \ref{prop: vanish k<0}). 
A more interesting situation is the case when $f$ vanishes identically at ${\D}'$, which we study in \S \ref{ssec: quasi-pullback}. 
In that case, we can define the so-called \textit{quasi-pullback} of $f$, which produces a \textit{cusp} form on ${\D}'$ 
(Proposition \ref{prop: quasi-pullback}). 
These operations will be useful when studying concrete examples. 

Restriction of modular forms to sub modular varieties has been considered classically 
for scalar-valued Siegel modular forms, going back to Witt \cite{Wi}.  
Quasi-pullback has been also considered in this case: see \cite{CG} \S 2 for a general treatment. 

Quasi-pullback of orthogonal modular forms was first considered for Borcherds products by Borcherds (\cite{Bo}, \cite{BKPSB}), 
and later for general scalar-valued modular forms by Gritsenko-Hulek-Sankaran (\cite{GHS} \S 8.4) in the case $r=1$. 
Our terminology "quasi-pullback" comes from this series of work. 
The cuspidality of quasi-pullback was first proved in \cite{GHS07}, \cite{GHS} in the scalar-valued case. 
Our Proposition \ref{prop: quasi-pullback} is the vector-valued generalization.

\section{Ordinary pullback}\label{ssec: pullback}

We embed ${\rm O}^{+}(L_{{\R}}')\times {\rm O}(K_{{\R}})$ in ${\OLR}$ naturally. 
This is the stabilizer of $L'_{{\R}}$ in ${\OLR}$.  
Let ${\G}$ be a finite-index subgroup of ${\OL}$. 
Then $\Gamma'={\G}\cap {\rm O}^{+}(L')$ is a finite-index subgroup of ${\rm O}^{+}(L')$, 
and $G={\G}\cap {\rm O}(K)$ is a finite group. 
The product ${\G}'\times G$ is a finite-index subgroup of the stabilizer of $L'$ in ${\G}$. 

Let ${\LL}'$, ${\E}'$ be the Hodge bundles on ${\D}'$. 
Since $\mathcal{O}_{{\proj}L_{{\C}}}(-1)|_{{\proj}L'_{{\C}}}=\mathcal{O}_{{\proj}L'_{{\C}}}(-1)$, 
we have ${\LL}|_{{\D}'}={\LL}'$. 
We also have a natural isomorphism 
\begin{equation}\label{eqn: restrict E}
{\E}|_{{\D}'}\simeq {\E}'\oplus (K_{{\C}}\otimes \mathcal{O}_{{\D}'}), 
\end{equation}
which at each fiber is the decomposition 
\begin{equation*}
(\omega^{\perp}\cap L_{{\C}})/{\C}\omega = ((\omega^{\perp}\cap L'_{{\C}})/{\C}\omega) \oplus K_{{\C}}. 
\end{equation*}
This corresponds to the decomposition 
${\rm St}={\rm St}' \oplus {\rm St}''$  
of the standard representation of ${\On}$ when restricted to the subgroup ${\rm O}(n', {\C})\times {\rm O}(r, {\C})$, 
where ${\rm St}'$ and  ${\rm St}''$ are the standard representations of ${\rm O}(n', {\C})$ and ${\rm O}(r, {\C})$ respectively.

Let $\lambda$ be a partition expressing an irreducible representation $V_{\lambda}$ of ${\On}$. 
We denote by 
\begin{equation}\label{eqn: restrict O(n')O(r)}
V_{\lambda} \simeq \bigoplus_{\alpha} V'_{\lambda'(\alpha)}\boxtimes V''_{\lambda''(\alpha)} 
\end{equation}
the irreducible decomposition as a representation of ${\rm O}(n', {\C})\times {\rm O}(r, {\C})$, 
where $V'_{\lambda'(\alpha)}$ (resp.~$V''_{\lambda''(\alpha)}$) is the irreducible representation of 
${\rm O}(n', {\C})$ (resp.~${\rm O}(r, {\C})$) with partition $\lambda'(\alpha)$ (resp.~$\lambda''(\alpha)$). 
See \cite{Ki}, \cite{KT} for an explicit description of this restriction rule in terms of the Littlewood-Richardson numbers. 
Let $k$ be an integer. 

\begin{proposition}\label{prop: restriction}
Restriction of modular forms to ${\D}'\subset {\D}$ defines a linear map 
\begin{equation*}\label{eqn: restrict}
M_{\lambda,k}({\G}) \to \bigoplus_{\alpha} M_{\lambda'(\alpha),k}({\G}') \otimes (K_{{\C}})_{\lambda''(\alpha)}^{G}, 
\quad f\mapsto f|_{{\D}'}. 
\end{equation*}
This maps cusp forms to cusp forms. 
\end{proposition}

For the proof of Proposition \ref{prop: restriction}, we need to calculate the Fourier expansion of $f|_{{\D}'}$. 
We take a rank $1$ primitive isotropic sublattice $I$ of $L'$. 
Let ${\UIZ}\subset {\UIQ}$ be as in \S \ref{ssec: tube domain} and 
we define $U(I)_{{\Z}}'\subset U(I)_{{\Q}}'$ similarly for $(L', {\G}')$.  
Then $U(I)_{{\Q}}'\subset {\UIQ}$ and $U(I)_{{\Z}}'\subset {\UIZ}$. 
If we write $K_{{\Q}}'=K_{{\Q}}\otimes I_{{\Q}}$, we have ${\UIQ}=U(I)_{{\Q}}'\oplus K_{{\Q}}'$. 
The tube domain realization with respect to $I$ (with $I'$ also taken from $L'$) identifies ${\D}'\subset {\D}$ with 
$\mathcal{D}_{I}'={\DI}\cap U(I)_{{\C}}' \subset {\DI}$. 

\begin{lemma}\label{lem: restrict Fourier expansion}
Let $f(Z)=\sum_{l\in U(I)_{{\Z}}^{\vee}}a(l)q^l$ be the Fourier expansion of $f\in M_{{\lambda},k}({\G})$ at the $I$-cusp of ${\D}$. 
Then we have 
\begin{equation}\label{eqn: Fourier expansion restrict}
f|_{\mathcal{D}_{I}'}(Z') = \sum_{l'\in (U(I)'_{{\Z}})^{\vee}} b(l') (q')^{l'}, \quad (q')^{l'}=e((l', Z')),  
\end{equation}
for $Z'\in {\D}_{I}'$, 
where  
\begin{equation*}
b(l')=\sum_{\substack{l''\in K_{{\Q}}' \\ l'+l''\in {\UIZZ}}} a(l'+l''). 
\end{equation*}
\end{lemma}

\begin{proof}
Let $\pi \colon {\UIQ}\to U(I)_{{\Q}}'$ be the orthogonal projection. 
This maps ${\UIZZ}$ to a sublattice of $(U(I)_{{\Z}}')^{\vee}$. 
For $l\in {\UIZZ}$, the restriction of the function $q^{l}=e((l, Z))$ to ${\D}_{I}'\subset {\DI}$ is $(q')^{\pi(l)}=e((\pi(l), Z'))$. 
Then our assertion follows by substituting $q^l=(q')^{\pi(l)}$ in $f=\sum_{l}a(l)q^l$. 
Note that the sum defining $b(l')$ is actually a finite sum 
by the condition $l'+l''\in \overline{\mathcal{C}_{I}}$ (the cusp condition for $f$) and 
the fact that $K'_{{\Q}}$ is negative-definite. 
\end{proof}

Now we prove Proposition \ref{prop: restriction}. 

\begin{proof}[(Proof of Proposition \ref{prop: restriction})]
From the expression \eqref{eqn: El=GKVl} and the decomposition \eqref{eqn: restrict O(n')O(r)}, we see that  
\begin{equation}\label{eqn: Elk D'}
{\El}|_{{\D}'}\simeq \bigoplus_{\alpha} \mathcal{E}'_{\lambda'(\alpha)} \otimes (K_{{\C}})_{\lambda''(\alpha)} 
\end{equation} 
as an ${\rm O}^{+}(L'_{{\R}})\times {\rm O}(K_{{\R}})$-equivariant vector bundle on ${\D}'$. 
With the isomorphism ${\LL}|_{{\D}'}={\LL}'$, we obtain  
\begin{equation*}
{\Elk}|_{{\D}'}\simeq \bigoplus_{\alpha} \mathcal{E}'_{\lambda'(\alpha),k} \otimes (K_{{\C}})_{\lambda''(\alpha)}. 
\end{equation*} 
If $f$ is a ${\G}$-invariant section of ${\Elk}$ over ${\D}$, 
this shows that $f|_{{\D}'}$ is a ${\G}'\times G$-invariant section of 
$\bigoplus_{\alpha} \mathcal{E}'_{\lambda'(\alpha),k} \otimes (K_{{\C}})_{\lambda''(\alpha)}$ 
over ${\D}'$. 
Hence it is a ${\G}'$-invariant section of 
$\bigoplus_{\alpha} \mathcal{E}'_{\lambda'(\alpha),k} \otimes (K_{{\C}})_{\lambda''(\alpha)}^{G}$ 
over ${\D}'$. 

Holomorphicity of $f|_{{\D}'}$ at the cusps of ${\D}'$ holds automatically when $n'\geq 3$ by the Koecher principle. 
In general, this can be seen from Lemma \ref{lem: restrict Fourier expansion} as follows. 
Let $I$ and $K'_{{\Q}}$ be as in Lemma \ref{lem: restrict Fourier expansion}. 
Since $K'_{{\Q}}$ is negative-definite, 
the orthogonal projection ${\UIR}\to U(I)_{{\R}}'$ maps the positive cone $\mathcal{C}_{I}$ of ${\UIR}$ 
to the positive cone $\mathcal{C}_{I}'$ of $U(I)'_{{\R}}$, 
and maps $\overline{\mathcal{C}_{I}}$ to $\overline{\mathcal{C}_{I}'}$. 
Hence the vectors $l'$ in \eqref{eqn: Fourier expansion restrict} actually range over 
$(U(I)_{{\Z}}')^{\vee} \cap \overline{\mathcal{C}_{I}'}$. 
This proves the holomorphicity of $f|_{{\D}'}$ around the $I$-cusp of ${\D}'$. 
Since $I$ is arbitrary, $f$ is holomorphic at all cusps of ${\D}'$. 
When $f$ is a cusp form, 
the vectors $l'$ range over $(U(I)_{{\Z}}')^{\vee} \cap \mathcal{C}_{I}'$ for the same reason. 
This means that $f|_{{\D}'}$ is a cusp form. 
This proves Proposition \ref{prop: restriction}.  
\end{proof}


\begin{example}
Let us look at a typical example. 
Let $\lambda={\rm St}$. 
As noticed before, this decomposes as ${\rm St}={\rm St}' \oplus {\rm St}''$ 
when restricted to ${\rm O}(n', {\C})\times {\rm O}(r, {\C})$, 
which corresponds to the decomposition \eqref{eqn: restrict E}.  
Therefore restriction to ${\D}'$ gives a linear map 
\begin{equation*}\label{eqn: restrict standard}
M_{{\rm St},k}({\G}) \to M_{{\rm St'},k}({\G}')\: \oplus \: (M_{k}({\G}')\otimes K_{{\C}}^{G}). 
\end{equation*}
The first component $M_{{\rm St},k}({\G}) \to M_{{\rm St'},k}({\G}')$ can be considered as the main component of the restriction, 
but we also obtain some scalar-valued modular forms in $M_{k}({\G}')\otimes K_{{\C}}^{G}$ as "extra" components. 
When $G$ fixes no nonzero vector of $K$, these extra components vanish. 
For example, this happens when ${\G}$ contains a reflection and $L'$ is the fixed lattice of this reflection. 
%
\end{example}

As an application of Proposition \ref{prop: restriction}, we obtain the following elementary vanishing theorem. 
Although this will be superseded later (\S \ref{sec: VT I}), 
we present it here because it can be proved easily and is already informative. 

\begin{proposition}\label{prop: vanish k<0}
When $k<0$, we have ${\MG}=0$. 
Moreover, we have $M_{\lambda,0}({\G})=0$ when $\lambda\ne 1, \det$. 
\end{proposition}

\begin{proof}
Let $f\in {\MG}$ with $k<0$. 
We consider restriction of $f$ to $1$-dimensional domains $\mathcal{D}_{L'}\subset {\D}$ 
for sublattices $L'\subset L$ of signature $(2, 1)$. 
As a representation of ${\rm O}(1, {\C})=\{ \pm {\rm id} \}$, 
$V_{\lambda}$ is a direct sum of copies of the trivial character and the determinant character. 
By Proposition \ref{prop: restriction} and the calculation in \S \ref{sssec: n=1}, we see that 
$f|_{{\D}_{L'}}$ is a tuple of scalar-valued modular forms of weight $2k<0$ on the upper half plane ${\D}_{L'}$. 
Since there is no nonzero elliptic modular form of negative weight, 
we find that $f$ vanishes identically at ${\D}_{L'}$. 
Now, if we vary $L'$, then ${\D}_{L'}$ run over a dense subset of ${\D}$. 
Therefore $f\equiv 0$. 

When $f\in M_{\lambda,0}({\G})$ with $\lambda\ne 1, \det$, 
by combining Proposition \ref{cor: a(0)=0} and Lemma \ref{lem: restrict Fourier expansion}, 
we see that $f|_{{\D}_{L'}}$ is a tuple of scalar-valued cusp forms of weight $0$ on ${\D}_{L'}$, 
which vanish identically. 
Therefore $f\equiv 0$ similarly.  
\end{proof}

The idea to deduce a vanishing theorem by considering restriction to sub modular varieties is classical. 
In the case of Siegel modular forms, this goes back to Freitag \cite{Fr79}. 

Proposition \ref{prop: vanish k<0} in particular implies the following. 

\begin{proposition}\label{prop: dual hol tensor}
Let $n\geq 3$. 
Assume that $\langle {\G}, -{\rm id} \rangle$ does not contain a reflection. 
Let $X$ be the regular locus of ${\FG}={\G}\backslash {\D}$. 
Then $H^{0}(X, T_{X}^{\otimes k})=0$ for every $k>0$. 
\end{proposition}

\begin{proof}
Let $\pi\colon {\D}\to {\FG}$ be the projection and $X'\subset X$ be the locus where $\pi$ is unramified. 
By \cite{GHS07}, the absence of reflection in $\langle {\G}, -{\rm id} \rangle$ implies that $\pi$ is unramified in codimension $1$, 
so the complement of $\pi^{-1}(X')$ in ${\D}$ has codimension $\geq 2$. 
Since we can pullback sections of $T_{X'}^{\otimes k}$ by the \'etale map $\pi^{-1}(X')\to X'$, 
we see that 
\begin{equation*}
H^{0}(X, T_{X}^{\otimes k})  =  H^{0}(X', T_{X'}^{\otimes k})  
 =  H^{0}(\pi^{-1}(X'), T_{\pi^{-1}(X')}^{\otimes k})^{{\G}}  
 =  H^{0}({\D}, T_{{\D}}^{\otimes k})^{{\G}}. 
\end{equation*}
Since $T_{{\D}}\simeq {\E}\otimes {\LL}^{-1}$ by \eqref{eqn: TD}, 
we find that 
\begin{equation*}
H^{0}(X, T_{X}^{\otimes k}) = H^{0}({\D}, {\E}^{\otimes k}\otimes {\LL}^{\otimes -k})^{{\G}} 
= \bigoplus_{i} M_{\lambda(i), -k}({\G}), 
\end{equation*}
where $\lambda(i)$ run over the irreducible summands of ${\rm St}^{\otimes k}$. 
By Proposition \ref{prop: vanish k<0}, the last space vanishes when $-k<0$. 
\end{proof}

\section{Quasi-pullback}\label{ssec: quasi-pullback}

In this section we show that when $f|_{{\D}'}\equiv 0$, we can still obtain a nonzero \textit{cusp} form on ${\D}'$ 
by considering the Taylor expansion of $f$ along ${\D}'$. 
We assume $n'\geq 3$ for simplicity of exposition, but the results below hold also when $n'\leq 2$ 
(see the proof of Proposition \ref{prop: quasi-pullback}). 

We first describe the normal bundle $\mathcal{N}=\mathcal{N}_{{\D}'/{\D}}$ of ${\D}'$ in ${\D}$. 

\begin{lemma}\label{lem: normal bundle}
We have $\mathcal{N}\simeq ({\LL}')^{-1}\otimes K_{{\C}}$ 
as an ${\rm O}^{+}(L'_{{\R}})\times {\rm O}(K_{{\R}})$-equivariant vector bundle on ${\D}'$. 
\end{lemma}

\begin{proof}
By \eqref{eqn: TD} and \eqref{eqn: restrict E}, we have natural isomorphisms 
\begin{eqnarray*}
T_{{\D}}|_{{\D}'} & \simeq & ({\E}\otimes {\LL}^{-1})|_{{\D}'} 
\simeq ({\E}'\oplus (K_{{\C}}\otimes \mathcal{O}_{{\D}'})) \otimes ({\LL}')^{-1} \\ 
& \simeq & T_{{\D}'} \oplus (({\LL}')^{-1}\otimes K_{{\C}}). 
\end{eqnarray*}
This implies $\mathcal{N}\simeq ({\LL}')^{-1}\otimes K_{{\C}}$. 
\end{proof}

Let $\mathcal{I}$ be the ideal sheaf of ${\D}'\subset {\D}$ and $\nu \geq 0$. 
By Lemma \ref{lem: normal bundle} we have  
\begin{equation}\label{eqn: Ir/Ir+1}
\mathcal{I}^{\nu}/\mathcal{I}^{\nu+1}|_{{\D}'} \simeq 
{\rm Sym}^{\nu}\mathcal{N}^{\vee}\simeq ({\LL}')^{\otimes \nu} \otimes {\rm Sym}^{\nu}K_{{\C}}^{\vee}   
\end{equation}
as an ${\rm O}^{+}(L'_{{\R}})\times {\rm O}(K_{{\R}})$-equivariant vector bundle on ${\D}'$. 
Therefore we have the exact sequence  
\begin{equation}\label{eqn: exact sequence Im/Im+1}
0 \to \mathcal{I}^{\nu+1} {\Elk} \to \mathcal{I}^{\nu} {\Elk} \to 
{\El}|_{{\D}'}\otimes ({\LL}')^{\otimes k+\nu} \otimes {\rm Sym}^{\nu}K_{{\C}}^{\vee} \to 0 
\end{equation}
of sheaves on ${\D}$. 
By \eqref{eqn: Elk D'} we have an ${\rm O}^{+}(L'_{{\R}})\times {\rm O}(K_{{\R}})$-equivariant isomorphism 
\begin{equation*}
{\El}|_{{\D}'}\otimes ({\LL}')^{\otimes k+\nu} \otimes {\rm Sym}^{\nu}K_{{\C}}^{\vee} \simeq 
\bigoplus_{\alpha} \mathcal{E}'_{\lambda'(\alpha),k+\nu} \otimes (K_{{\C}})_{\lambda''(\alpha)}\otimes {\rm Sym}^{\nu}K_{{\C}}^{\vee}.  
\end{equation*}
Note that $K_{{\C}}^{\vee}\simeq K_{{\C}}$ canonically by the pairing on $K$. 
Taking global sections in \eqref{eqn: exact sequence Im/Im+1}, and then the ${\G}'\times G$-invariant part, 
we obtain the exact sequence 
\begin{eqnarray}\label{eqn: quasi-pullback sheaf}
& & 0\to H^{0}({\D}, \, \mathcal{I}^{\nu+1} {\Elk})^{{\G}'\times G} \to  H^{0}({\D}, \, \mathcal{I}^{\nu} {\Elk})^{{\G}'\times G} \nonumber \\ 
& & \qquad \to \bigoplus_{\alpha} M_{\lambda'(\alpha), k+\nu}({\G}')\otimes ((K_{{\C}})_{\lambda''(\alpha)}\otimes {\rm Sym}^{\nu}K_{{\C}})^{G}.  
\end{eqnarray}
By definition, a modular form $f\in M_{\lambda,k}({\G})$ vanishes to order $\geq \nu$ along ${\D}'$ 
if it is a section of the subsheaf $\mathcal{I}^{\nu} {\Elk}$ of ${\Elk}$. 
The \textit{vanishing order} of $f$ along ${\D}'$ is the largest $\nu$ 
for which $f$ is a section of $\mathcal{I}^{\nu} {\Elk}$. 

\begin{definition}
Let $f\in M_{\lambda,k}({\G})$ and $\nu$ be the vanishing order of $f$ at ${\D}'$. 
We define the \textit{quasi-pullback} of $f$
\begin{equation*}
f||_{{\D}'} \; \;  \in \; 
\bigoplus_{\alpha} M_{\lambda'(\alpha), k+\nu}({\G}')\otimes ((K_{{\C}})_{\lambda''(\alpha)}\otimes {\rm Sym}^{\nu}K_{{\C}})^{G} 
\end{equation*}
as the image of $f$ by the last map in \eqref{eqn: quasi-pullback sheaf}. 
\end{definition}

By the exactness of \eqref{eqn: quasi-pullback sheaf} and the definition of the vanishing order, we have $f||_{{\D}'}\not\equiv 0$. 
Note that the vanishing order $\nu$ contributes to the increase $k \leadsto k+\nu$ of the scalar weight. 
When $\nu=0$, the quasi-pullback is just the ordinary pullback considered in \S \ref{ssec: pullback}. 

\begin{example}
When $r=1$, ignoring the symmetry by $G\subset \{ \pm {\rm id} \}$, 
the quasi-pullback $f||_{{\D}'}$ belongs to $\bigoplus_{\alpha}M_{\lambda'(\alpha),k+\nu}({\G}')$. 
Explicitly, $f||_{{\D}'}$ is given by the restriction of $f/(\cdot, \delta)^{\nu}$ to ${\D}'$, 
where $\delta$ is a nonzero vector of $K$ and $(\cdot, \delta)$ is the section of $\mathcal{O}(1)$ 
defined by the pairing with $\delta$. 
\end{example}

\begin{example}
The quasi-pullback of a Borcherds product $f$ considered by Borcherds (\cite{Bo}, \cite{BKPSB}) is defined as 
$f/\prod_{\delta}(\delta, \cdot)|_{{\D}'}$, 
where $\delta$ run over primitive vectors in $K$ (with multiplicity)  
such that $f$ vanishes at $\delta^{\perp}\cap {\D}$. 
This is a single scalar-valued modular form (again a Borcherds product), 
while our quasi-pullback produces a tuple of scalar-valued modular forms, or more canonically, 
a ${\rm Sym}^{\nu}K_{{\C}}$-valued modular form. 
The relationship is as follows. 

The denominator $\prod_{\delta}(\delta, \cdot)$ is a section of $\mathcal{I}^{\nu}\cdot \mathcal{O}(\nu)$ over ${\D}$. 
This corresponds to a sheaf homomorphism $\iota\colon {\LL}^{\otimes \nu}\to \mathcal{I}^{\nu}$. 
By a property of Borcherds products, 
$f$ is a section of the subsheaf $\iota({\LL}^{\otimes \nu})\cdot {\LL}^{\otimes k}$ of $\mathcal{I}^{\nu}\cdot {\LL}^{\otimes k}$. 
Let 
$\bar{\iota}\colon ({\LL}')^{\otimes \nu}\to {\rm Sym}^{\nu}\mathcal{N}^{\vee}$ 
be the embedding induced by $\iota|_{{\D}'}$ and \eqref{eqn: Ir/Ir+1}. 
Under the isomorphism 
${\rm Sym}^{\nu}\mathcal{N}^{\vee}\simeq ({\LL}')^{\otimes \nu}\otimes {\rm Sym}^{\nu}K_{{\C}}^{\vee}$, 
this corresponds to the vector $\prod_{\delta}(\cdot , \delta)$ of ${\rm Sym}^{\nu}K_{{\C}}^{\vee}$, 
which in turn corresponds to the vector $\prod_{\delta}\delta$ of ${\rm Sym}^{\nu}K_{{\C}}$. 
Then $f||_{{\D}'}$ as a section of ${\rm Sym}^{\nu}\mathcal{N}^{\vee}\otimes ({\LL}')^{\otimes k}$ 
takes values in the sub line bundle 
$\bar{\iota}(({\LL}')^{\otimes \nu})\otimes ({\LL}')^{\otimes k} \simeq ({\LL}')^{\otimes k+\nu}$. 
This section of $({\LL}')^{\otimes k+\nu}$ is the quasi-pullback in \cite{Bo} and \cite{BKPSB}. 
\end{example}

Next we prove the cuspidality of quasi-pullback. 
In the case $\lambda=0$ and $r=1$, this is due to Gritsenko-Hulek-Sankaran (\cite{GHS} Theorem 8.18). 

\begin{proposition}\label{prop: quasi-pullback}
Let $f\in M_{\lambda,k}({\G})$ and $\nu$ be the vanishing order of $f$ at ${\D}'$. 
Suppose that $\nu>0$. 
Then $f||_{{\D}'}$ is a cusp form. 
Thus 
\begin{equation*}
f||_{{\D}'} \; \;  \in \; 
\bigoplus_{\alpha} S\!_{\lambda'(\alpha), k+\nu}({\G}')\otimes ((K_{{\C}})_{\lambda''(\alpha)}\otimes {\rm Sym}^{\nu}K_{{\C}})^{G}. 
\end{equation*}
\end{proposition}

For the proof of Proposition \ref{prop: quasi-pullback}, we calculate the Fourier expansion of $f||_{{\D}'}$. 
We work under the same setting and notation as in the proof of Lemma \ref{lem: restrict Fourier expansion}. 
We choose a basis of $K_{{\Q}}'$. 
According to the decomposition ${\UIQ}=U(I)_{{\Q}}'\oplus K_{{\Q}}'$, 
we express a point of ${\UIC}$ as $Z=(Z', z_1, \cdots, z_r)$ 
with $Z'\in U(I)_{{\C}}'$ and $z_i\in {\C}$. 
Then ${\D}_{I}'\subset {\DI}$ is defined by $z_{1}=\cdots=z_{r}=0$. 
The coordinates $z_1, \cdots, z_r$ give a trivialization of the conormal bundle $\mathcal{N}^{\vee}$ of ${\D}_{I}'$. 
The quasi-pullback $f||_{{\D}'}$ as a $V(I)_{\lambda,k}\otimes {\rm Sym}^{\nu}{\C}^{r}$-valued function on ${\D}_{I}'$ 
is given, up to constants, by the Taylor coefficients of $f$ along ${\D}_{I}'$ in degree $\nu$: 
\begin{equation*}
f||_{{\D}'}(Z') = 
\left( \frac{\partial^{\nu}f}{\partial z_1^{\nu_1} \cdots \partial z_r^{\nu_r}} (Z', 0) \right)_{\nu_1+\cdots + \nu_{r}=\nu} 
\end{equation*}
We calculate the Fourier expansion of the Taylor coefficients. 
In what follows, we identify $(K'_{{\Q}})^{\vee}\simeq {\Q}^{r}$ by the dual basis of the chosen basis of $K_{{\Q}}'$ 
and express vectors of $(K'_{{\Q}})^{\vee}$ as $(n_{1}, \cdots, n_{r})$, $n_{i}\in {\Q}$. 

\begin{lemma}\label{lem: Fourier Taylor}
Let $f(Z)=\sum_{l}a(l)q^{l}$ be the Fourier expansion of $f$. 
Let $(\nu_1, \cdots, \nu_{r})$ be an index with $\nu_1+\cdots + \nu_{r}=\nu$. 
Then we have 
\begin{equation*}
\frac{\partial^{\nu}f}{\partial z_1^{\nu_1} \cdots \partial z_r^{\nu_r}} (Z', 0) = 
(2\pi \sqrt{-1})^{\nu} \sum_{l'\in (U(I)_{{\Z}}')^{\vee}}b(l')(q')^{l'}, 
\end{equation*}
where $(q')^{l'}= e((l', Z'))$ and 
\begin{equation*}
b(l') =  \sum_{\substack{(n_1, \cdots, n_r)\in {\Q}^{r} \\ l'+(n_1, \cdots, n_{r})\in {\UIZZ}}} 
n_1^{\nu_1} \cdots n_r^{\nu_r} \cdot a(l'+(n_1, \cdots, n_r)). 
\end{equation*}
Here, by convention, $0^0=1$ but $0^{m}=0$ when $m>0$. 
\end{lemma}

Note that the sum defining $b(l')$ is actually a finite sum 
for the same reason as in Lemma \ref{lem: restrict Fourier expansion}. 

\begin{proof}
We can rewrite the Fourier expansion of $f$ as 
\begin{eqnarray*}
& & f(Z', z_1, \cdots, z_r) \\ 
& = & \sum_{l'} \sum_{(n_1, \cdots, n_r)}  
a(l'+(n_1, \cdots, n_r)) \cdot e((l'+(n_1, \cdots, n_r), (Z', z_1, \cdots, z_r))) \\ 
& = & \sum_{l'} \sum_{(n_1, \cdots, n_r)}  
a(l'+(n_1, \cdots, n_r)) \cdot e((l', Z')) \cdot \prod_{i=1}^{r} e(n_{i}z_{i}). 
\end{eqnarray*}
Here $l'$ ranges over $(U(I)_{{\Z}}')^{\vee}$ and 
$(n_1, \cdots, n_r)$ ranges over vectors in ${\Q}^{r}=(K_{{\Q}}')^{\vee}$ such that $l'+(n_1, \cdots, n_{r})\in {\UIZZ}$. 
Since we have 
\begin{equation*}
\frac{\partial^{\nu}\prod_{i}e(n_iz_i)}{\partial z_1^{\nu_1} \cdots \partial z_r^{\nu_r}} = 
(2\pi \sqrt{-1})^{\nu} \prod_{i}n_{i}^{\nu_{i}} \cdot e(n_{i}z_{i}),  
\end{equation*}
we see that 
\begin{eqnarray*}
& & \frac{\partial^{\nu}f}{\partial z_1^{\nu_1} \cdots \partial z_r^{\nu_r}} (Z', z_1, \cdots, z_r) \\ 
& = & 
(2\pi \sqrt{-1})^{\nu} \sum_{l'} \sum_{(n_1, \cdots, n_r)} a(l'+(n_1, \cdots, n_r)) \cdot (q')^{l'} \cdot \prod_{i} n_{i}^{\nu_{i}} \cdot e(n_{i}z_{i}). 
\end{eqnarray*}
Substituting $z_1= \cdots = z_r=0$, this proves Lemma \ref{lem: Fourier Taylor}. 
\end{proof}

Now we complete the proof of Proposition \ref{prop: quasi-pullback}. 

\begin{proof}[(Proof of Proposition \ref{prop: quasi-pullback})]
Let $l'$ be a vector in $\overline{\mathcal{C}_{I}'}\cap (U(I)_{{\Z}}')^{\vee}$ with $(l', l')=0$. 
For $(n_1, \cdots, n_r)\in {\Q}^{r}$, we have $l'+(n_1, \cdots, n_r)\in \overline{\mathcal{C}_{I}}$ only when $(n_1, \cdots, n_r)=(0, \cdots, 0)$ 
because $K_{{\Q}}'$ is negative-definite and perpendicular to $U(I)_{{\Q}}'$.  
By Lemma \ref{lem: Fourier Taylor}, this shows that 
\begin{equation*}
b(l') = 0^{\nu_{1}} \cdots 0^{\nu_{r}} \cdot a(l') = 0 
\end{equation*}
because $(\nu_{1}, \cdots, \nu_{r})\ne(0, \cdots, 0)$ by the assumption $\nu>0$. 
This proves Proposition \ref{prop: quasi-pullback}.  
\end{proof}


\chapter{Canonical extension over $1$-dimensional cusps}\label{sec: cano exte}

In this chapter we recall the partial toroidal compactification over a $1$-dimensional cusp 
and the canonical extension of the automorphic vector bundles over it. 
This provides a geometric basis for the Siegel operator (\S \ref{sec: Siegel}) and the Fourier-Jacobi expansion (\S \ref{sec: FJ}). 
Except for a few calculations in \S \ref{ssec: cano exte 1dim cusp} and \S \ref{ssec: L at boundary}, 
most contents of this chapter are essentially expository. 
We refer the reader to \cite{AMRT} for the general theory of toroidal compactification, 
to \cite{GHS07}, \cite{Lo}, \cite{Ma1} for its specialization to the case of orthogonal modular varieties 
(especially for more details on the contents of \S \ref{ssec: Siegel domain} -- \S \ref{ssec: partial compact}), 
and to \cite{Mu} for the general theory of canonical extension. 
Nevertheless, since this chapter is the basis of many later chapters, 
we tried to keep the presentation as self-contained, explicit, and coherent as possible. 

Throughout this chapter, $L$ is a lattice of signature $(2, n)$ with $n\geq 3$. 
We fix a rank $2$ primitive isotropic sublattice $J$ of $L$, which corresponds to 
a $1$-dimensional cusp of ${\D}={\D}_{L}$. 
We write $V(J)_{F}=(J^{\perp}/J)_{F}$ for $F={\Q}, {\R}, {\C}$.  
This is a quadratic space over $F$, negative-definite when $F={\Q}, {\R}$. 
We especially abbreviate $V(J)=V(J)_{{\C}}$. 
We also write $U(J)_F=\wedge^{2}J_{F}$. 
The choice of the component ${\D}$ determines an orientation of $J$ so that 
the ${\R}$-isomorphism $(\omega, \cdot )\colon J_{{\R}}\to {\C}$ preserves the orientation for any $[\omega]\in {\D}$. 
This determines the positive part of ${\UJR}$.  

For $2U=U\oplus U$, where $U$ is the integral hyperbolic plane, 
we will denote by $e_1, f_1$ and $e_2, f_2$ the standard hyperbolic basis of the first and the second components respectively. 
We say that an embedding $\iota\colon 2U_{F}\hookrightarrow L_{F}$ is compatible with $J$ if $\iota({\Z}e_1\oplus {\Z}e_2)=J$. 
This defines a lift $V(J)_{F}\simeq \iota(2U_{F})^{\perp}\cap L_F$ of $V(J)_{F}$ in $J^{\perp}_{F}$ and hence a splitting 
\begin{equation}\label{eqn: L=2U+V(J)}
L_{F} \simeq 2U_{F}\oplus V(J)_{F} = (J_{F}\oplus J_{F}^{\vee}) \oplus V(J)_{F}, 
\end{equation}
where we identify $\iota(\langle f_1, f_2\rangle)$ with $J_{F}^{\vee}$. 
We often choose a rank $1$ primitive sublattice $I$ of $J$. 
We say that $\iota\colon 2U_{F}\hookrightarrow L_{F}$ is compatible with $I\subset J$ 
if $\iota({\Z}e_1\oplus {\Z}e_2)=J$ and $\iota({\Z}e_1)=I$.

\section{Siegel domain realization}\label{ssec: Siegel domain}

In this section we recall the Siegel domain realization of ${\D}$ with respect to the $J$-cusp  
and explain its relation with the tube domain realization. 

\subsection{Siegel domain realization}\label{sssec: Siegel domain}

The filtration $J\subset J^{\perp} \subset L$ on $L$ determines the two-step linear projection 
\begin{equation}\label{eqn: linear projection}
{\proj}L_{{\C}} \stackrel{\pi_{1}}{\dashrightarrow} {\proj}(L/J)_{{\C}} \stackrel{\pi_{2}}{\dashrightarrow} {\proj}(L/J^{\perp})_{{\C}}. 
\end{equation}
Via the pairing on $L_{{\C}}$, this is identified with the dual projection 
\begin{equation*}
{\proj}L_{{\C}}^{\vee}\dashrightarrow {\proj}(J^{\perp}_{{\C}})^{\vee} \dashrightarrow {\proj}J_{{\C}}^{\vee}. 
\end{equation*}
The center of $\pi_{1}$ is ${\proj}J_{{\C}}$, and the center of $\pi_{2}$ is ${\proj}V(J)$. 
The projection $\pi_{2}$ identifies ${\proj}(L/J)_{{\C}}-{\proj}V(J)$ with an affine space bundle over ${\proj}(L/J^{\perp})_{{\C}}$. 
If we choose a lift $V(J)\hookrightarrow J^{\perp}_{{\C}}$ of $V(J)$, 
it defines a splitting $(L/J)_{{\C}}=V(J)\oplus (L/J^{\perp})_{{\C}}$, 
and so defines an isomorphism between the affine space bundle ${\proj}(L/J)_{{\C}}-{\proj}V(J)$ 
with the vector bundle $V(J)\otimes \mathcal{O}(1)$ over ${\proj}(L/J^{\perp})_{{\C}}$. 

We restrict \eqref{eqn: linear projection} to the isotropic quadric $Q\subset {\proj}L_{{\C}}$. 
The closure of a $\pi_1$-fiber is a plane containing ${\proj}J_{{\C}}$. 
When this plane is not contained in ${\proj}J_{{\C}}^{\perp}$, it intersects properly with $Q$ at two distinct lines, one being ${\proj}J_{{\C}}$. 
This shows that 
\begin{equation*}
\pi_{1}|_{Q} : Q - Q\cap {\proj}J^{\perp}_{{\C}} \; \to \; {\proj}(L/J)_{{\C}} - {\proj}V(J) 
\end{equation*}
is an affine line bundle. 

Next we restrict \eqref{eqn: linear projection} further to an enlargement of the domain ${\D}\subset Q$. 
Let ${\HJ}$ be the connected component of ${\proj}J_{{\C}}^{\vee}-{\proj}J_{{\R}}^{\vee}$ 
consisting of ${\C}$-linear maps $\phi \colon J_{{\C}}\to {\C}$ such that 
$\phi|_{J_{{\R}}}\colon J_{{\R}}\to {\C}$ is an orientation-preserving ${\R}$-isomorphism. 
By the canonical isomorphism ${\proj}J_{{\C}}^{\vee}\simeq {\proj}J_{{\C}}$, ${\HJ}$ corresponds to the $J$-cusp. 
We put $\mathcal{V}_{J}=\pi_{2}^{-1}({\HJ})$ and $\mathcal{D}(J)=(\pi_{1}|_{Q})^{-1}(\mathcal{V}_{J})$. 
Then ${\D}\subset \mathcal{D}(J)$. 
We thus have the extended two-step fibration 
\begin{equation}\label{eqn: Siegel domain realization}
{\D} \subset \mathcal{D}(J) \stackrel{\pi_{1}}{\to} \mathcal{V}_{J} \stackrel{\pi_{2}}{\to} {\HJ}, 
\end{equation}
where $\mathcal{V}_{J} \to {\HJ}$ is an affine space bundle isomorphic to $V(J)\otimes \mathcal{O}_{{\HJ}}(1)$, 
$\mathcal{D}(J) \to \mathcal{V}_{J}$ is an affine line bundle, 
and ${\D}\to \mathcal{V}_{J}$ is an upper half plane bundle inside $\mathcal{D}(J) \to \mathcal{V}_{J}$. 
This is the Siegel domain realization of ${\D}$ with respect to $J$. 
(Up to this point, canonically determined by $J$.) 

\subsection{Relation with tube domain realization}\label{sssec: Siegel vs tube}

We choose a rank $1$ primitive sublattice $I$ of $J$. 
Recall from \S \ref{ssec: tube domain} that the tube domain realization at the $I$-cusp (before choosing a base point) 
is the canonical embedding 
\begin{equation*}
{\D}\subset Q(I) \stackrel{\simeq}{\to} {\proj}(L/I)_{{\C}}-{\proj}V(I) 
\end{equation*}
induced by the projection ${\proj}L_{{\C}}\dashrightarrow {\proj}(L/I)_{{\C}}$. 
Note that ${\D}(J)\subset Q(I)$. 
We can factor the projection $\pi_{1}$ in \eqref{eqn: linear projection} as: 
\begin{equation*}\label{eqn: I J affine projection}
{\proj}L_{{\C}} \dashrightarrow {\proj}(L/I)_{{\C}} \dashrightarrow {\proj}(L/J)_{{\C}} \dashrightarrow {\proj}(L/J^{\perp})_{{\C}}. 
\end{equation*}
Hence we have the following commutative diagram: 
\begin{equation*}\label{cd: Siegel tube affine}
\xymatrix{
{\proj}(L/I)_{{\C}}-{\proj}V(I) \ar[r] & {\proj}(L/J)_{{\C}} - {\proj}(I^{\perp}/J)_{{\C}} \ar[r] & {\proj}(L/J^{\perp})_{{\C}} - {\proj}(I^{\perp}/J^{\perp})_{{\C}} \\ 
{\D} \subset {\D}(J) \quad \;  \ar@{^{(}-_>}[u] \ar[r]^{\pi_{1}} & {\VJ} \ar@{^{(}-_>}[u] \ar[r]^{\pi_{2}} & {\HJ}. \ar@{^{(}-_>}[u]  
}
\end{equation*}
Here the upper row is projections of affine spaces, 
the left vertical map is the tube domain realization at $I$, 
and other vertical maps are natural inclusions. 
The two squares are cartesian, i.e., ${\D}(J) \to {\VJ} \to {\HJ}$ is the restriction of the upper row over ${\HJ}$. 
Thus the Siegel domain realization at $J$ can be given by a decomposition of the tube domain realization at $I\subset J$. 

Next we choose a rank $1$ isotropic sublattice $I'\subset L$ with $(I, I')\ne 0$ 
and accordingly a base point of the affine space ${\proj}(L/I)_{{\C}} - {\proj}V(I)$. 
This identifies the upper row of the above diagram with the linear maps 
\begin{equation*}\label{eqn: Siegel tube linear I}
{\UIC}=(I^{\perp}/I)_{{\C}}\otimes I_{{\C}} \to (I^{\perp}/J)_{{\C}}\otimes I_{{\C}} \to (I^{\perp}/J^{\perp})_{{\C}}\otimes I_{{\C}}. 
\end{equation*}
We identify ${\UJC}=\wedge^2 J_{{\C}}$ with the isotropic line $(J/I)_{{\C}}\otimes I_{{\C}}$ in ${\UIC}$. 
Then this is written as the quotient maps 
\begin{equation}\label{eqn: Siegel tube linear II}
{\UIC} \to {\UIC}/U(J)_{{\C}} \to {\UIC}/U(J)_{{\C}}^{\perp}. 
\end{equation}
Therefore, after choosing the base point $I'$, the above commutative diagram can be rewritten as 
\begin{equation*}
\xymatrix{
{\UIC} \ar[r]^{\pi_{1}} & {\UIC}/U(J)_{{\C}} \ar[r]^{\pi_{2}} & {\UIC}/U(J)_{{\C}}^{\perp} \\ 
{\D} \subset {\D}(J) \quad \; \ar@{^{(}-_>}[u] \ar[r]^{\pi_{1}} & {\VJ} \ar@{^{(}-_>}[u] \ar[r]^{\pi_{2}} & {\HJ} \ar@{^{(}-_>}[u]  
}
\end{equation*}
where the vertical embeddings are defined by $I'$ 
and the two squares are cartesian. 
This gives a simpler (but depending on $I, I'$) expression of the Siegel domain realization. 

Finally, we introduce coordinates. 
Let $v_{J}$ be the positive generator of $\wedge^{2}J\simeq {\Z}$. 
We choose an isotropic vector $l_{J}\in {\UIQ}$ with $(v_J, l_J)=1$. 
This defines a splitting ${\UIQ}\simeq U_{{\Q}}\oplus K_{{\Q}}$ where $K_{{\Q}}=V(J)_{{\Q}}\otimes I_{{\Q}}$, 
which determines a splitting of \eqref{eqn: Siegel tube linear II}. 
Accordingly, we express a point of ${\UIC} \simeq {\C}l_J \times K_{{\C}} \times {\C}v_{J}$ as 
\begin{equation}\label{eqn: coordinate tube Siegel}
Z=(\tau, z, w)=\tau l_J + z + w v_{J}, \qquad z\in K_{{\C}}, \; \tau, w\in {\C}. 
\end{equation}
In this coordinates, the $I$-directed Siegel domain realization \eqref{eqn: Siegel tube linear II} is expressed by 
\begin{equation*}
(\tau, z, w) \mapsto (\tau, z) \mapsto \tau. 
\end{equation*}
The $w$-component gives coordinates on the $\pi_{1}$-fibers ($\simeq {\UJC}$), 
and $\tau$ gives coordinates on the base ${\UIC}/U(J)_{{\C}}^{\perp}\simeq U(J)_{{\C}}^{\vee}$. 
The images of the embeddings 
\begin{equation*}
{\D}(J)\hookrightarrow {\UIC}, \quad 
\mathcal{V}_{J}\hookrightarrow {\UIC}/{\UJC}, \quad 
{\HJ}\hookrightarrow {\UIC}/U(J)_{{\C}}^{\perp} 
\end{equation*}
are all defined by the inequality ${\rm Im}(\tau)>0$, 
and the tube domain ${\DI}\subset {\UIC}$ is defined by the inequalities 
\begin{equation*}
-({\rm Im}(z), {\rm Im}(z)) \: < \: 2 \, {\rm Im}(\tau) \cdot {\rm Im}(w), \qquad {\rm Im}(\tau)>0. 
\end{equation*}  
Thus the choice of $I$, $I'$, $l_{J}$ defines a passage from the canonical presentation \eqref{eqn: Siegel domain realization} 
to a more classical presentation of the Siegel domain realization.  

\begin{remark}
The choice of $I'$ and $l_{J}$ is almost equivalent to the choice of an embedding $2U_{{\Q}}\hookrightarrow L_{{\Q}}$ 
compatible with $I_{{\Q}}\subset J_{{\Q}}$. 
More precisely, we choose one of the two generators of $I\simeq{\Z}$, say $v_{I}$. 
Let $v_{I}'\in I_{{\Q}}'$ be the dual vector of $v_{I}$ in $I'_{{\Q}}$. 
We can write $v_{J}=\tilde{v}_{J}\otimes v_{I}$ and $l_{J}=\tilde{l}_{J}\otimes v_{I}$ 
for some vectors $\tilde{v}_{J}\in (I'_{{\Q}})^{\perp}\cap J_{{\Q}}$ and $\tilde{l}_{J}\in (I'_{{\Q}})^{\perp}\cap I_{{\Q}}^{\perp}$. 
This defines an embedding $2U_{{\Q}}\hookrightarrow L_{{\Q}}$ compatible with $I_{{\Q}}\subset J_{{\Q}}$ 
by sending 
\begin{equation*}
e_{1}\mapsto v_{I}, \quad f_{1}\mapsto v_{I}', \quad e_{2}\mapsto \tilde{v}_{J}, \quad f_{2}\mapsto \tilde{l}_{J}. 
\end{equation*}
\end{remark}

\section{Jacobi group}\label{ssec: Jacobi group}

In this section we describe the rational/real Jacobi group of the $J$-cusp and its action on the Siegel domain realization. 

Let $F={\Q}$, ${\R}$. 
Let ${\GJF}$ be the subgroup of the stabilizer of $J_{F}$ in ${\rm O}(L_{F})$ acting trivially on $\wedge^{2}J_F$ and $V(J)_F$. 
We call ${\GJF}$ the \textit{Jacobi group} for $J$ over $F$. 
(It is certainly useful to take into account the action on $V(J)_{F}$, 
but here we refrain from doing so for simplicity of exposition.) 
The Jacobi group has the canonical filtration 
\begin{equation*}
{\UJF} \subset W(J)_{F} \subset {\GJF} 
\end{equation*}
defined by 
\begin{equation*}
W(J)_{F} = \ker ({\GJF} \to {\rm SL}(J_F)), 
\end{equation*} 
\begin{equation*}
U(J)_{F} = \ker ({\GJF} \to {\rm GL}(J_{F}^{\perp})).  
\end{equation*}
The group $U(J)_{F}$ consists of the Eichler transvections $E_{l\otimes l'}$ for $l, l'\in J_{F}$. 
Since $E_{l'\otimes l}=E_{-l\otimes l'}$, $U(J)_F$ is canonically isomorphic to $\wedge^{2}J_{F}$. 
This justifies our use of the notation $U(J)_F$. 
We also have the canonical isomorphism 
\begin{equation*}
V(J)_{F}\otimes J_F \to W(J)_{F}/{\UJF}, \qquad 
m\otimes l \mapsto E_{\tilde{m}\otimes l}  \: \: {\rm mod} \; {\UJF}, 
\end{equation*}
where $\tilde{m}\in J_{F}^{\perp}$ is a lift of $m\in V(J)_{F}$. 
The linear space $V(J)_{F}\otimes J_F$ has a canonical $U(J)_F$-valued symplectic form 
as the tensor product of the quadratic form on $V(J)_{F}$ and the canonical $\wedge^{2}J_{F}$-valued symplectic form on $J_{F}$. 
We thus have the canonical exact sequences  
\begin{equation}\label{eqn: rational Jacobi sequence}
0 \to W(J)_{F} \to {\GJF} \to {\SL}(J_F) \to 1, 
\end{equation}
\begin{equation*}\label{eqn: Heisenberg}
0 \to {\UJF} \to W(J)_{F} \to V(J)_{F}\otimes J_F \to 0. 
\end{equation*}
The group $U(J)_{F}$ is the center of ${\GJF}$, and $W(J)_{F}$ is the unipotent radical of ${\GJF}$.  
The first sequence \eqref{eqn: rational Jacobi sequence} splits 
if we choose an embedding $2U_{F}\hookrightarrow L_F$ compatible with $J_F$ and 
hence a splitting $L_F\simeq (J_{F}\oplus J_{F}^{\vee})\oplus V(J)_{F}$ as in \eqref{eqn: L=2U+V(J)}: 
\begin{equation}\label{eqn: Jacobi group split}
{\GJF} \simeq {\SL}(J_F)\ltimes W(J)_{F}. 
\end{equation}
Here the lifted group ${\SL}(J_F)\subset {\GJF}$ acts on the component $J_F\oplus J_F^{\vee}$ in the natural way. 
The adjoint action of ${\SL}(J_F)$ on $W(J)_F/{\UJF}\simeq V(J)_F\otimes J_F$ is 
the tensor product of the natural action of ${\SL}(J_{F})$ on $J_F$ 
and the trivial action on $V(J)_F$. 
The group $W(J)_{F}$ is isomorphic to the Heisenberg group for the symplectic space $V(J)_F\otimes J_F$ with center ${\UJF}$. 
We call $W(J)_{F}$ the \textit{Heisenberg group} for $J$ over $F$. 

If $I$ is a rank $1$ primitive sublattice of $J$, we have 
\begin{equation}\label{eqn: U(J) U(I) G(J)}
{\UJF} \subset {\UIF} \subset {\GJF}, 
\end{equation}
as can be seen from the definitions. 
In ${\UIF}=(I^{\perp}/I)_{F}\otimes I_{F}$, ${\UJF}$ corresponds to the isotropic line $(J/I)_{F}\otimes I_{F}$. 
We also have $W(J)_{F}\subset {\GIF}$ and 
\begin{equation*}
{\UIF}\cap W(J)_{F} = U(J)_{F}^{\perp} = (J^{\perp}/I)_F\otimes I_{F}. 
\end{equation*}
The image of $W(J)_{F}$ in ${\rm O}(V(I)_{F})$ is the group of Eichler transvections of $V(I)_{F}$ 
with respect to the isotropic line $(J/I)_{F}$. 

The Jacobi group ${\GJF}$ preserves the Siegel domain realization \eqref{eqn: Siegel domain realization} by definition. 
The actions of the factors $U(J)_{F}$, $W(J)_{F}/U(J)_{F}$, ${\SL}(J_{F})$ of ${\GJF}$ 
on the spaces in \eqref{eqn: Siegel domain realization} are described as follows. 

(1) The group ${\UJF}$ acts on $\mathcal{V}_{J}$ trivially. 
The projection ${\D}(J)\to \mathcal{V}_J$ is a principal $U(J)_{{\C}}$-bundle, 
where $U(J)_{{\C}}=\wedge^{2}J_{{\C}}$ is the group of Eichler transvections $E_{l\otimes l'}$ with $l, l'\in J_{{\C}}$. 

(2) The Heisenberg group $W(J)_{F}$ acts on ${\HJ}$ trivially. 
The quotient $W(J)_F/{\UJF}\simeq V(J)_F\otimes J_{F}$ acts on the fibers of $\mathcal{V}_J\to {\HJ}$ by translation. 
More precisely, if $\tau$ is a point of ${\HJ}\subset{\proj}J_{{\C}}^{\vee}$ and 
$J_{{\C}}=J^{1,0}\oplus J^{0,1}$ is the corresponding Hodge decomposition of $J_{{\C}}$ 
(where $J^{1,0}$ is the kernel),  
the fiber of $\mathcal{O}_{{\HJ}}(1)$ over $\tau$ is $J_{{\C}}/J^{1,0}$. 
So the fiber $(\mathcal{V}_{J})_{\tau}$ of $\mathcal{V}_{J}$ over $\tau$ is an affine space for $V(J)\otimes_{{\C}} (J_{{\C}}/J^{1,0})$. 
On the other hand, we have a natural projection $V(J)_{{\R}}\otimes_{{\R}} J_{{\R}} \to V(J)\otimes_{{\C}} (J_{{\C}}/J^{1,0})$ 
which is an ${\R}$-isomorphism. 
Then the action of an element of $W(J)_{{\R}}/{\UJR}\simeq V(J)_{{\R}}\otimes_{{\R}} J_{{\R}}$ 
on the affine space $(\mathcal{V}_{J})_{\tau}$ is the translation by its projection image in $V(J)\otimes_{{\C}} (J_{{\C}}/J^{1,0})$. 

(3) To describe the action of ${\SL}(J_{F})$, we take an embedding $2U_{F}\hookrightarrow L_{F}$ compatible with $J_{F}$. 
As explained before, this induces an isomorphism $\mathcal{V}_{J}\simeq V(J)\otimes \mathcal{O}_{{\HJ}}(1)$ 
and a lift ${\SL}(J_F)\hookrightarrow {\GJF}$. 
Then the lifted group ${\SL}(J_F)$ acts on $\mathcal{V}_{J}$ by its equivariant action on $\mathcal{O}_{{\HJ}}(1)$.

\section{Partial toroidal compactification}\label{ssec: partial compact}

Let ${\G}$ be a finite-index subgroup of ${\OL}$. 
We take the intersection of ${\GJQ}$, $W(J)_{{\Q}}$, ${\UJQ}$ with ${\G}$ and denote them by  
\begin{equation*}
{\GJZ}={\GJQ}\cap {\G}, \quad W(J)_{{\Z}}=W(J)_{{\Q}}\cap {\G}, \quad {\UJZ}={\UJQ}\cap {\G}. 
\end{equation*}
By the orientation on $J$, we have a distinguished isomorphism ${\UJZ}\simeq {\Z}$. 
We also denote by $\Gamma(J)_{{\Z}}^{\ast}$ the stabilizer of $J$ in ${\G}$. 
The integral Jacobi group ${\GJZ}$ is of finite index in $\Gamma(J)_{{\Z}}^{\ast}$ because 
\begin{equation*}
\Gamma(J)_{{\Z}}^{\ast}/{\GJZ} \hookrightarrow {\rm O}(J^{\perp}/J) 
\end{equation*}
and ${\rm O}(J^{\perp}/J)$ is a finite group. 
If ${\G}$ is neat, we have $\Gamma(J)_{{\Z}}^{\ast}={\GJZ}$. 

We put 
\begin{equation*}
{\GJZbar}={\GJZ}/{\UJZ},  \quad  {\GJFbar}={\GJF}/{\UJZ}
\end{equation*}
for $F={\Q}, {\R}$. 
These quotients make sense because ${\UJF}$ is the center of ${\GJF}$. 
By definition we have the canonical exact sequence 
\begin{equation*}
0 \to W(J)_{{\Z}}/{\UJZ} \to {\GJZbar} \to {\GJZ}/W(J)_{{\Z}} \to 1, 
\end{equation*}
which is canonically embedded in the quotient of \eqref{eqn: rational Jacobi sequence} by $U(J)_{F}$: 
more specifically, ${\GJZ}/W(J)_{{\Z}}$ is embedded in ${\SL}(J)$ as a finite-index subgroup, 
and $W(J)_{{\Z}}/{\UJZ}$ is embedded in $V(J)_{{\Q}}\otimes J_{{\Q}}$ as a full lattice. 

Let $T(J)=U(J)_{{\C}}/{\UJZ}\simeq {\C}^{\ast}$ be the $1$-dimensional torus defined by ${\UJZ}$. 
We denote by $\overline{T(J)}\simeq {\C}$ the natural partial compactification of $T(J)$.  
We take the quotient of ${\D}\subset {\D}(J)$ by ${\UJZ}$: 
\begin{equation*}
{\XJ}={\D}/{\UJZ}, \quad \mathcal{T}(J)={\D}(J)/{\UJZ}. 
\end{equation*}
Then $\mathcal{T}(J)$ is a principal $T(J)$-bundle over $\mathcal{V}_{J}$, 
which contains ${\XJ}$ as a fibration of punctured discs. 
Let $\overline{\mathcal{T}(J)}=\mathcal{T}(J)\times_{T(J)}\overline{T(J)}$ be the relative torus embedding. 
This has the structure of a line bundle on $\mathcal{V}_{J}$: 
the scalar multiplication on each fiber is given by the action of $T(J)\simeq {\C}^{\ast}$, 
and the sum is determined by the scalar multiplication because the fiber is $1$-dimensional. 
The group ${\GJRbar}$ acts on $\mathcal{T}(J)$ naturally, 
and this extends to an action on $\overline{\mathcal{T}(J)}$. 
The fact that ${\GJR}$ commutes with ${\UJC}$ implies that 
the action of ${\GJRbar}$ on $\overline{\mathcal{T}(J)}$ is an equivariant action on the line bundle. 

Let ${\XJcpt}$ be the interior of the closure of ${\XJ}$ in $\overline{\mathcal{T}(J)}$. 
We call ${\XJcpt}$ the \textit{partial toroidal compactification} of ${\XJ}$. 
This is a disc bundle over $\mathcal{V}_{J}$ obtained by filling the origins in the punctured disc bundle ${\XJ}\to \mathcal{V}_{J}$. 
Let $\Delta_{J}$ be the boundary divisor of ${\XJcpt}$. 
This is naturally isomorphic to $\mathcal{V}_{J}$. 
We denote by $\Theta_{J}$ the conormal bundle of $\Delta_{J}$ in ${\XJcpt}$. 
This is a ${\GJRbar}$-equivariant line bundle on $\Delta_{J}$. 
(Although the subgroup ${\UJR}/{\UJZ}$ of ${\GJRbar}$ acts on $\Delta_{J}$ trivially, 
it acts on the fibers of $\Theta_{J}$ by rotations.) 

\begin{lemma}\label{lem: partial compact in Theta LB}
We have a natural ${\GJRbar}$-equivariant isomorphism 
$\Theta^{\vee}\simeq \overline{\mathcal{T}(J)}$ 
of line bundles on $\Delta_{J}$. 
\end{lemma}

\begin{proof}
Since $\Delta_{J}$ is the zero section of the line bundle $\overline{\mathcal{T}(J)}$, 
its normal bundle in ${\XJcpt}$ is the same as the normal bundle in $\overline{\mathcal{T}(J)}$, 
which is isomorphic to $\overline{\mathcal{T}(J)}$ itself. 
\end{proof}

The partial compactification ${\XJcpt}$ already appears in essence in the partial compactifications ${\XIcpt}$ 
for $I\subset J$ considered in \S \ref{sssec: partial compact}.  
Recall that the isotropic ray $\sigma_{J}=({\UJR})_{\geq 0}$ appears in every ${\GIZ}$-admissible fan $\Sigma$ as in \S \ref{sssec: partial compact}. 
Since ${\UJZ}\subset{\UIZ}$, we have a natural \'etale map ${\XJ}\to {\XI}$ which is a free quotient map by ${\UIZ}/{\UJZ}$. 

\begin{lemma}\label{lem: glue}
The map ${\XJ}\to {\XI}$ extends to an \'etale map ${\XJcpt}\to {\XIcpt}$. 
The image of $\Delta_{J}$ is a Zariski open set of the boundary divisor of ${\XIcpt}$ associated to the isotropic ray $\sigma_{J}$. 
\end{lemma}

\begin{proof}
Since ${\D}(J)\subset Q(I)$, we have the following commutative diagram (cf.~\S \ref{sssec: Siegel vs tube}): 
\begin{equation*}
\xymatrix{
\mathcal{T}(J) \ar@{^{(}-_>}[r] \ar[d] & Q(I)/{\UJZ} \ar[r] \ar[d] & T(I) \ar[d] \\ 
\mathcal{V}_{J} \ar@{^{(}-_>}[r] & Q(I)/U(J)_{{\C}} \ar[r]  & T(I)/T(J).
}
\end{equation*}
Here the vertical maps are principal $T(J)$-bundles, 
and the two right horizontal maps are free quotients by ${\UIZ}/{\UJZ}$. 
The two squares are cartesian: 
the right is the pullback of a principal $T(J)$-bundle to a ${\UIZ}/{\UJZ}$-cover, 
and the left is the restriction to an open set.  
Since the upper row is $T(J)$-equivariant, it extends to 
\begin{equation*}
\overline{\mathcal{T}(J)} \hookrightarrow (Q(I)/{\UJZ})\times_{T(J)}\overline{T(J)} \to T(I)\times_{T(J)}\overline{T(J)}.  
\end{equation*}
The second map is still a free quotient by ${\UIZ}/{\UJZ}$. 
The image of $\Delta_{J}\subset \overline{\mathcal{T}(J)}$ by this map is 
an open set of the (unique) boundary divisor of $T(I)\times_{T(J)}\overline{T(J)}$. 
Since $T(I)\times_{T(J)}\overline{T(J)}$ is the torus embedding of $T(I)$ associated to the ray $\sigma_J$, 
it is a Zariski open set of $T(I)^{\Sigma}$. 
Thus we obtain an \'etale map $\overline{\mathcal{T}(J)} \to T(I)^{\Sigma}$ 
which maps $\Delta_{J}$ to an open set of the boundary divisor of $T(I)^{\Sigma}$ corresponding to $\sigma_{J}$. 
\end{proof}

\section{Canonical extension}\label{ssec: cano exte 1dim cusp}

In this section, which is the central part of \S \ref{sec: cano exte}, 
we extend the automorphic vector bundles ${\Elk}$ over ${\XJcpt}$. 
This is an explicit form of Mumford's canonical extension \cite{Mu} which is suitable for dealing with the Fourier-Jacobi expansion. 
We use the same notations ${\LL}$, ${\E}$, ${\El}$, ${\Elk}$ for the descends of these vector bundles to ${\XJ}$. 
They are ${\GJRbar}$-equivariant vector bundles on ${\XJ}$. 

We choose an adjacent $0$-dimensional cusp $I\subset J$. 
Since ${\UJZ}\subset {\GIR}$, the $I$-trivialization 
${\Elk}\simeq {\VIlk}\otimes {\OD}$ over ${\D}$ descends to an isomorphism 
${\Elk}\simeq {\VIlk}\otimes \mathcal{O}_{{\XJ}}$ over ${\XJ}={\D}/{\UJZ}$. 
Thus we still have the $I$-trivialization over ${\XJ}$. 
This is equivariant with respect to $({\GIR}\cap {\GJR})/{\UJZ}$. 
We extend ${\Elk}$ to a vector bundle over ${\XJcpt}$ (still use the same notation) by requiring that this isomorphism extends to 
${\Elk}\simeq {\VIlk}\otimes \mathcal{O}_{{\XJcpt}}$. 
We call it the \textit{canonical extension} of ${\Elk}$ over ${\XJcpt}$. 
This is the pullback of the canonical extension over ${\XIcpt}$ defined in \S \ref{sssec: cano exte 0dim cusp} 
by the gluing map ${\XJcpt}\to {\XIcpt}$ in Lemma \ref{lem: glue}. 
By construction, the frame of ${\Elk}$ over ${\XJ}$ corresponding to a basis of ${\VIlk}$ via the $I$-trivialization 
extends to a frame of the extended bundle over ${\XJcpt}$. 

\begin{proposition}\label{prop: cano exte well-defined}
The canonical extension of ${\Elk}$ over ${\XJcpt}$ defined above does not depend on the choice of $I$. 
The action of ${\GJRbar}$ on ${\Elk}$ over ${\XJ}$ extends to action on the canonical extension of ${\Elk}$ over ${\XJcpt}$. 
\end{proposition}

The proof of this proposition amounts to the following assertion. 

\begin{lemma}\label{lem: I vs I'}
The factor of automorphy of the ${\GJR}$-action on ${\Elk}$ with respect to the $I$-trivialization 
is constant on each fiber of $\pi_{1}\colon {\D} \to {\VJ}$. 
In particular, if $I'$ is another ${\R}$-line in $J_{{\R}}$, 
the difference of the $I$-trivialization and the $I'$-trivialization at $[\omega]\in {\D}$ as the composition map 
\begin{equation}\label{eqn: I vs I'}
{\VIlk} \to ({\Elk})_{[\omega]} \to V(I')_{\lambda,k} 
\end{equation}
is constant on each $\pi_{1}$-fiber. 
\end{lemma}

\begin{proof}
Let $j(\gamma, [\omega])$ be the factor of automorphy in question. 
This is a ${\rm GL}({\VIlk})$-valued function on ${\GJR}\times {\D}$. 
What has to be shown is that 
$j(\gamma, [\omega])=j(\gamma, [\omega'])$ if $\pi_{1}([\omega])=\pi_{1}([\omega'])$. 
We consider the natural extension of ${\Elk}$ over ${\D}(J)$,  
on which the group ${\UJC}\cdot {\GJR}$ acts equivariantly. 
Note that ${\UJC}$ commutes with ${\GJR}$. 
We can write $[\omega']=g[\omega]$ for some $g\in {\UJC}$. 
Since ${\UJC}$ acts trivially on $I_{{\C}}$ and $V(I)$, we have $j(g, \cdot )\equiv {\rm id}$. 
Therefore 
\begin{equation*}
j(\gamma, g[\omega]) = j(\gamma g, [\omega]) = j(g\gamma, [\omega]) = j(\gamma, [\omega]). 
\end{equation*}

As for the second assertion, 
we choose $\gamma\in {\GJR}$ with $\gamma(I_{{\R}})=I'$. 
Then \eqref{eqn: I vs I'} coincides with the isomorphism 
\begin{equation*}
\gamma \circ j(\gamma^{-1}, [\omega]) : {\VIlk} \to {\VIlk} \to V(I')_{\lambda,k}. 
\end{equation*}
Hence the constancy of $j(\gamma^{-1}, [\omega])$ over $\pi_{1}$-fibers implies that of \eqref{eqn: I vs I'}. 
\end{proof}

Now we can prove Proposition \ref{prop: cano exte well-defined}. 

\begin{proof}[(Proof of Proposition \ref{prop: cano exte well-defined})]
Let $I, I'$ be two rank $1$ primitive sublattices of $J$. 
By the second assertion of Lemma \ref{lem: I vs I'}, the difference of the $I$-trivialization and the $I'$-trivialization 
\begin{equation}\label{eqn: I vs I' II}
{\VIlk}\otimes \mathcal{O}_{{\XJ}} \to {\Elk} \to V(I')_{\lambda,k}\otimes \mathcal{O}_{{\XJ}}, 
\end{equation} 
viewed as a ${\rm GL}(n, {\C})$-valued holomorphic function on ${\XJ}$ via basis of ${\VIlk}$ and $V(I')_{\lambda,k}$, 
is constant on each fiber of ${\XJ}\to {\VJ}$. 
Therefore it extends to a ${\rm GL}(n, {\C})$-valued holomorphic function over ${\XJcpt}$. 
This implies that \eqref{eqn: I vs I' II} extends to an isomorphism 
\begin{equation*}
{\VIlk}\otimes \mathcal{O}_{{\XJcpt}} \to V(I')_{\lambda,k}\otimes \mathcal{O}_{{\XJcpt}} 
\end{equation*} 
over ${\XJcpt}$. 
Thus the two extensions agree. 

Extendability of the ${\GJRbar}$-action on ${\Elk}$ can be verified as follows. 
Let $\gamma \in {\GJR}$. 
The $\gamma$-action on ${\Elk}$ sends a frame corresponding to a basis of ${\VIlk}$ via the $I$-trivialization 
to a frame corresponding to a basis of $V(\gamma I)_{\lambda,k}$ via the $\gamma I$-trivialization. 
By Lemma \ref{lem: I vs I'} again, the latter extends to a frame over ${\XJcpt}$ also in the $I$-trivialization. 
Thus $\gamma$ sends an extendable frame to an extendable frame. 
This means that the $\gamma$-action extends over ${\XJcpt}$. 
\end{proof}

The fact that the canonical extension comes with an $I$-trivialization (but independent of it) 
enables us to develop the theory of Fourier-Jacobi expansion (\S \ref{sec: FJ}) in an intrinsic but still explicit way. 
The following property will play a fundamental role in \S \ref{sec: FJ}. 

\begin{proposition}\label{prop: L and E as GJR bundle}
Let $\pi_{1}\colon {\XJcpt}\to \mathcal{V}_{J}\simeq \Delta_{J}$ be the projection. 
Then we have a ${\GJRbar}$-equivariant isomorphism ${\Elk}\simeq \pi_{1}^{\ast}({\Elk}|_{\Delta_{J}})$ over ${\XJcpt}$. 
\end{proposition}

\begin{proof}
We fix a rank $1$ primitive sublattice $I\subset J$ and let 
$j(\gamma, [\omega])$ be the factor of automorphy of the ${\GJR}$-action on ${\Elk}$ 
with respect to the $I$-trivialization. 
By Lemma \ref{lem: I vs I'}, the ${\rm GL}({\VIlk})$-valued function $j(\gamma, [\omega])$ on ${\GJR}\times {\XJ}$ descends to 
a ${\rm GL}({\VIlk})$-valued function on ${\GJR}\times \Delta_{J}$. 
This gives the factor of automorphy of the ${\GJR}$-action on ${\Elk}|_{\Delta_{J}}$ 
with respect to the $I$-trivialization ${\Elk}|_{\Delta_{J}}\simeq {\VIlk}\otimes \mathcal{O}_{\Delta_{J}}$. 
The fact that its pullback agrees with the factor of automorphy of ${\Elk}$ implies that the composition 
\begin{equation*}\label{eqn: descend of E}
\pi_{1}^{\ast}({\Elk}|_{\Delta_{J}}) \to \pi_{1}^{\ast}({\VIlk}\otimes \mathcal{O}_{\Delta_{J}}) \simeq {\VIlk}\otimes \mathcal{O}_{{\XJcpt}} \to {\Elk} 
\end{equation*}
gives a ${\GJR}$-equivariant isomorphism 
$\pi_{1}^{\ast}({\Elk}|_{\Delta_{J}}) \to {\Elk}$ over ${\XJcpt}$, 
where the first isomorphism is the pullback of the $I$-trivialization over $\Delta_{J}$, 
and the last isomorphism is the $I$-trivialization over ${\XJcpt}$. 
\end{proof}

\begin{remark}\label{rmk: pullback and I-trivialization}
By the proof, we have the following commutative diagram: 
\begin{equation*}
\xymatrix{
\pi_{1}^{\ast}({\Elk}|_{\Delta_{J}}) \ar[r] \ar[d] & {\Elk} \ar[d] \\ 
\pi_{1}^{\ast}({\VIlk}\otimes \mathcal{O}_{\Delta_{J}}) \ar[r] & {\VIlk}\otimes \mathcal{O}_{{\XJcpt}}.
}
\end{equation*}
Here the upper arrow is the isomorphism in Proposition \ref{prop: L and E as GJR bundle}, 
the vertical arrows are the $I$-trivializations, 
and the lower arrow is the natural isomorphism. 
\end{remark}

\begin{remark}
Although the canonical extension at the level of ${\XJcpt}$ still has a trivialization (by construction), 
this no longer holds when passing to the full toroidal compactifications (\S \ref{ssec: toroidal cpt}). 
Around $\Delta_{J}$ we need to further take the quotient by ${\GJZbar}$, 
which does not preserve the trivialization. 
\end{remark}

\section{The Hodge line bundle at the boundary}\label{ssec: L at boundary}

In this section we study the Hodge line bundle ${\LL}$ relative to the $J$-cusp 
and show that its canonical extension can be understood more directly. 
Let 
\begin{equation*}
{\LL}_{J}=\mathcal{O}_{{\HJ}}(-1) = \mathcal{O}_{{\proj}(L/J^{\perp})_{{\C}}}(-1)|_{{\HJ}} 
\end{equation*}
be the Hodge bundle over the upper half plane ${\HJ}$. 
The group ${\GJR}$ acts on ${\LL}_{J}$ equivariantly via the natural map ${\GJR}\to {\SL}(J_{{\R}})$. 
Let $\pi=\pi_{2}\circ \pi_{1} \colon {\D} \to {\HJ}$ be the projection from ${\D}$ to ${\HJ}$. 

\begin{lemma}\label{lem: L=LJ}
We have a ${\GJR}$-equivariant isomorphism ${\LL}\simeq \pi^{\ast}{\LL}_{J}$ over ${\D}$. 
\end{lemma}

\begin{proof}
Recall that $\pi$ is restriction of the projection
${\proj}L_{{\C}}\dashrightarrow {\proj}(L/J^{\perp})_{{\C}}$. 
Since this is induced by the linear map $L_{{\C}}\to (L/J^{\perp})_{{\C}}$, 
we have a natural isomorphism 
$\pi^{\ast}\mathcal{O}_{{\proj}(L/J^{\perp})_{{\C}}}(-1) \simeq \mathcal{O}_{{\proj}L_{{\C}}}(-1)$ 
over ${\proj}L_{{\C}}-{\proj}J_{{\C}}^{\perp}$. 
By restricting this isomorphism to ${\D}$, we obtain ${\LL}\simeq \pi^{\ast}{\LL}_{J}$. 
Since the projection $L_{{\C}}\to (L/J^{\perp})_{{\C}}$ is ${\GJR}$-equivariant, 
so is the isomorphism ${\LL}\simeq \pi^{\ast}{\LL}_{J}$. 
\end{proof}

The fiber of $\pi^{\ast}{\LJ}$ over $[\omega]\in {\D}$ is the image of the projection ${\C}\omega \to (L/J^{\perp})_{{\C}}$, 
and the isomorphism ${\LL}\to \pi^{\ast}{\LJ}$ over $[\omega]$ is identified with the natural map 
${\C}\omega \to {\rm Im}({\C}\omega \to (L/J^{\perp})_{{\C}})$. 

The projection ${\D}\to {\HJ}$ descends to ${\XJ}\to {\HJ}$ and extends to ${\XJcpt}\to {\HJ}$ naturally. 
We denote it again by $\pi\colon {\XJcpt}\to {\HJ}$. 
The isomorphism in Lemma \ref{lem: L=LJ} descends to a ${\GJRbar}$-equivariant isomorphism 
${\LL}\simeq \pi^{\ast}{\LJ}|_{{\XJ}}$ over ${\XJ}$. 
We have respective extension of both sides over ${\XJcpt}$: 
for ${\LL}$ the canonical extension constructed in \S \ref{ssec: cano exte 1dim cusp}, and 
for $\pi^{\ast}{\LJ}|_{{\XJ}}$ the natural extension $\pi^{\ast}{\LJ}$. 
It turns out that these two extensions agree: 

\begin{proposition}\label{prop: L at 1dim cusp}
The isomorphism ${\LL}\simeq \pi^{\ast}{\LJ}|_{{\XJ}}$ over ${\XJ}$ extends to a ${\GJRbar}$-equivariant isomorphism 
between the canonical extension of ${\LL}$ and $\pi^{\ast}{\LJ}$ over ${\XJcpt}$. 
In particular, we have ${\LL}|_{\Delta_{J}}\simeq \pi_{2}^{\ast}{\LJ}$ over $\Delta_{J}$. 
\end{proposition}

\begin{proof}
We choose a rank $1$ primitive sublattice $I\subset J$. 
Recall that the canonical extension of ${\LL}$ is defined via the $I$-trivialization of ${\LL}$, which we denote by  
$\iota_{I}\colon {\LL}\simeq {\ICv}\otimes \mathcal{O}_{{\XJ}}$. 
On the other hand, we also have a trivialization 
$\iota_{I}' \colon {\LL}_{J}\simeq {\ICv}\otimes \mathcal{O}_{{\HJ}}$ of ${\LJ}=\mathcal{O}_{{\HJ}}(-1)$ 
over ${\HJ}\subset {\proj}(L/J^{\perp})_{{\C}}$ 
induced by the pairing between $(L/J^{\perp})_{{\C}}$ and $I_{{\C}}$. 
The natural extension $\pi^{\ast}{\LJ}$ of $\pi^{\ast}{\LJ}|_{{\XJ}}$ over ${\XJcpt}$ 
coincides with the extension via the trivialization 
\begin{equation}\label{eqn: I-trivialization piLJ}
\pi^{\ast}{\LL}_{J}|_{{\XJ}} \stackrel{\pi^{\ast}\iota_{I}'}{\to} 
\pi^{\ast}({\ICv}\otimes \mathcal{O}_{{\HJ}})|_{{\XJ}} = {\ICv}\otimes \mathcal{O}_{{\XJ}}, 
\end{equation}
because $\pi^{\ast}\iota_{I}'$ is defined over ${\XJcpt}$. 

We observe that the composition of \eqref{eqn: I-trivialization piLJ} with the isomorphism ${\LL}\simeq \pi^{\ast}{\LL}_{J}|_{{\XJ}}$ 
in Lemma \ref{lem: L=LJ} coincides with the $I$-trivialization $\iota_{I}$ of ${\LL}$: 
this is just the remark that taking the pairing of a vector $\omega\in L_{{\C}}$ with $I_{{\C}}$ (this is $\iota_{I}$) 
is the same as projecting $\omega$ to $(L/J^{\perp})_{{\C}}$ (this is ${\LL}\to \pi^{\ast}{\LJ}$) 
and then taking pairing with $I_{{\C}}$ (this is $\pi^{\ast}\iota_{I}'$). 
From this coincidence, we see that the isomorphism in Lemma \ref{lem: L=LJ} 
extends to an isomorphism over ${\XJcpt}$ from the extension of ${\LL}$ via $\iota_{I}$ (this is the canonical extension of ${\LL}$) 
to the extension of  $\pi^{\ast}{\LL}_{J}|_{{\XJ}}$ via $\pi^{\ast}\iota_{I}'$ (this is $\pi^{\ast}{\LJ}$). 
The ${\GJRbar}$-equivariance holds by continuity. 
\end{proof}

Thus the canonical extension of ${\LL}$ defined in \S \ref{ssec: cano exte 1dim cusp} via the $I$-trivialization 
can be understood more directly 
as the \textit{canonical} (verbatim) extension $\pi^{\ast}{\LJ}$ of $\pi^{\ast}{\LJ}|_{{\XJ}}$.  

\begin{remark}\label{rmk: I-trivialization LJ}
By the proof of Proposition \ref{prop: L at 1dim cusp}, 
${\LJ}$ is endowed with the $I$-trivialization ${\ICv}\otimes \mathcal{O}_{{\HJ}}\to {\LJ}$ 
induced by the pairing between $(L/J^{\perp})_{{\C}}$ and $I_{{\C}}$, 
and its pullback by $\pi$ agrees with the $I$-trivialization of ${\LL}$ via the isomorphism ${\LL}\simeq \pi^{\ast}{\LJ}$. 
\end{remark}

\section{Toroidal compactification}\label{ssec: toroidal cpt}

In this section we recall the (full) toroidal compactifications of the modular variety ${\FG}={\G}\backslash {\D}$ following \cite{AMRT}. 
While this provides a background for our geometric approach, 
logically it will be used only in \S \ref{sec: L2} in a rather auxiliary way, so the reader may skip it for the moment. 

The data for constructing a toroidal compactification of ${\FG}$ is a collection $\Sigma=(\Sigma_{I})$ 
of ${\GIZ}$-admissible rational polyhedral cone decomposition of $\mathcal{C}_{I}^{+}\subset {\UIR}$ in the sense of \S \ref{sssec: partial compact}, 
one for each ${\G}$-equivalence class of rank $1$ primitive isotropic sublattices $I$ of $L$. 
Two fans $\Sigma_{I}$, $\Sigma_{I'}$ for different ${\G}$-equivalence classes $I$, $I'$ are independent, 
and no choice is required for rank $2$ isotropic sublattices $J$ (it is canonical). 
Then the toroidal compactification is defined by 
\begin{equation*}
{\FGcpt} = 
\left( {\D} \sqcup \bigsqcup_{I} \mathcal{X}(I)^{\Sigma_{I}} \sqcup \bigsqcup_{J} {\XJcpt} \right) / \sim, 
\end{equation*}
where $I$ (resp.~$J$) run over all primitive isotropic sublattices of $L$ of rank $1$ (resp.~rank $2$), 
and $\sim$ is the equivalence relation generated by the following \'etale maps: 
\begin{enumerate}
\item The $\gamma$-action ${\D}\to {\D}$, $\mathcal{X}(I)^{\Sigma_{I}}\to \mathcal{X}(\gamma I)^{\Sigma_{\gamma I}}$, 
${\XJcpt}\to \overline{\mathcal{X}(\gamma J)}$ for $\gamma\in {\G}$. 
\item The gluing maps ${\D}\to \mathcal{X}(I)^{\Sigma_{I}}$, ${\D}\to {\XJcpt}$ and ${\XJcpt}\to \mathcal{X}(I)^{\Sigma_{I}}$ for $I\subset J$ 
as in Lemma \ref{lem: glue}. 
\end{enumerate} 
By \cite{AMRT} \S III.5, ${\FGcpt}$ is a compact Moishezon space which contains $\mathcal{F}({\G})$ as a Zariski open set 
and has a morphism ${\FGcpt}\to {\FG}^{bb}$ to the Baily-Borel compactification. 
We have natural maps 
\begin{equation}\label{eqn: boundary local chart}
\mathcal{X}(I)^{\Sigma_{I}}/{\GIZbar} \to {\FGcpt}, \quad 
{\XJcpt}/(\Gamma(J)_{{\Z}}^{\ast}/{\UJZ}) \to {\FGcpt}. 
\end{equation}
These maps are isomorphims in a neighborhood of the locus of boundary points lying over the $I$-cusp and the $J$-cusp respectively 
(see \cite{AMRT} p.175). 
We may choose $\Sigma$ so that ${\FGcpt}$ is projective. 
When ${\G}$ is neat and each fan $\Sigma_{I}$ is regular, i.e., 
every cone is generated by a part of a ${\Z}$-basis of ${\UIZ}$, 
then ${\FGcpt}$ is nonsingular (\cite{AMRT} \S III.7). 

Next we explain the canonical extension of ${\Elk}$ over $\mathcal{F}({\G})^{\Sigma}$ (cf.~\cite{Mu}). 
We assume that ${\G}$ is neat and $\Sigma$ is regular. 
Then not only ${\G}$ itself but also ${\GIZbar}$ and $\Gamma(J)_{{\Z}}^{\ast}/{\UJZ}={\GJZbar}$ are torsion-free, 
so the quotient map 
\begin{equation*}
{\D} \sqcup \bigsqcup_{I} \mathcal{X}(I)^{\Sigma_{I}} \sqcup \bigsqcup_{J} {\XJcpt} \to {\FGcpt} 
\end{equation*} 
is \'etale. 
The vector bundle ${\Elk}$ is initially defined over ${\D}$ and hence over 
${\D} \sqcup \bigsqcup_{I} {\XI} \sqcup \bigsqcup_{J}{\XJ}$. 
In \S \ref{sssec: cano exte 0dim cusp} and \S \ref{ssec: cano exte 1dim cusp}, 
we constructed the canonical extension of ${\Elk}$ over $\mathcal{X}(I)^{\Sigma_{I}}$ and ${\XJcpt}$ respectively. 
By construction we have a natural isomorphism $p^{\ast}{\Elk}\simeq {\Elk}$ for a gluing map $p$ in (2) above. 
Moreover, we have a natural isomorphism $\gamma^{\ast}{\Elk}\simeq {\Elk}$ for the action of $\gamma \in {\G}$: 
this is evident for ${\D}$ and $\mathcal{X}(I)^{\Sigma_{I}}$, 
while it is assured by Proposition \ref{prop: cano exte well-defined} for ${\XJcpt}$. 
Since these isomorphisms are compatible with each other, 
the extended vector bundle ${\Elk}$ on ${\D} \sqcup \bigsqcup_{I} \mathcal{X}(I)^{\Sigma_{I}} \sqcup \bigsqcup_{J} {\XJcpt}$ 
descends to a vector bundle on ${\FGcpt}$. 
We denote it again by ${\Elk}$.  
This is the same as extending ${\Elk}$ on ${\FG}$ over 
the boundary of ${\FGcpt}$ by using the local charts \eqref{eqn: boundary local chart}. 

\begin{proposition}\label{prop: modular form toroidal cpt}
For ${\G}$ neat, we have ${\MG}=H^{0}({\FGcpt}, {\Elk})$. 
\end{proposition}

\begin{proof}
We have the natural inclusion 
\begin{equation*}
H^{0}({\FGcpt}, {\Elk}) \hookrightarrow H^{0}({\FG}, {\Elk}) = {\MG}. 
\end{equation*}
It is sufficient to see that this is surjective. 
Let $f\in {\MG}$. 
As a section of ${\Elk}$ over ${\XI}$, $f$ extends holomorphically over $\mathcal{X}(I)^{\Sigma_{I}}$ by Lemma \ref{lem: modular form extend II}.  
By the gluing, $f$ extends holomorphically over ${\XJcpt}$. 
Therefore, as a section of ${\Elk}$ over ${\FG}$, $f$ extends holomorphically over ${\FGcpt}$. 
\end{proof}

Let us remark an immediate consequence of this interpretation. 
We go back to a general finite-index subgroup ${\G}$ of ${\OL}$. 
For a fixed $\lambda$, the direct sum 
$\bigoplus_{k\geq 0} {\MG}$ 
is a module over the ring $\bigoplus_{k\geq 0}M_{k}({\G})$ of scalar-valued modular forms. 

\begin{proposition}
For each $\lambda$, the module $\bigoplus_{k} {\MG}$ is finitely generated over the ring $\bigoplus_{k}M_{k}({\G})$. 
\end{proposition}

\begin{proof}
We may assume that ${\G}$ is neat by replacing the given ${\G}$ by its neat subgroup of finite index. 
We take a smooth toroidal compactification ${\FGcpt}$ as above and 
let $\pi\colon {\FGcpt}\to {\FG}^{bb}$ be the projection to the Baily-Borel compactification. 
Then ${\LL}^{\otimes n}=\pi^{\ast}\mathcal{O}(1)$ for an ample line bundle $\mathcal{O}(1)$ on ${\FG}^{bb}$ by 
\cite{Mu} Proposition 3.4 (b). 
(In fact, ${\LL}$ itself descends, but we do not need that.) 
It suffices to show that for each $0\leq k_{0}<n$, the module 
$\bigoplus_{k}M_{\lambda,k_{0}+nk}({\G})$ is finitely generated over $\bigoplus_{k}M_{nk}({\G})$. 
By Proposition \ref{prop: modular form toroidal cpt}, we have 
\begin{eqnarray*}
\bigoplus_{k\geq 0} M_{\lambda,k_{0}+nk}({\G}) & = & 
\bigoplus_{k\geq 0} H^{0}({\FGcpt}, \; \mathcal{E}_{\lambda,k_{0}}\otimes \pi^{\ast}\mathcal{O}(k)) \\  
& \simeq & \bigoplus_{k\geq 0} H^{0}({\FG}^{bb}, \; \pi_{\ast}\mathcal{E}_{\lambda,k_{0}}\otimes \mathcal{O}(k)) 
\end{eqnarray*}
where the second isomorphism follows from the projection formula for $\pi$. 
Since ${\FG}^{bb}$ is projective, the last module is finitely generated over  
$\bigoplus_{k} H^{0}({\FG}^{bb}, \mathcal{O}(k)) = \bigoplus_{k}M_{nk}({\G})$ 
by a general theorem of Serre (see, e.g., \cite{MO} p.128). 
\end{proof}


\chapter{Geometry of Siegel operators}\label{sec: Siegel}

Let $L$ be a lattice of signature $(2, n)$ with $n\geq 3$ and ${\G}$ be a finite-index subgroup of ${\OL}$. 
Let $\lambda=(\lambda_{1}\geq \cdots \geq \lambda_{n})$ be a partition expressing an irreducible representation of ${\On}$. 
We assume $\lambda\ne 1, \det$. 
This in particular implies $\lambda_{n}=0$ and so ${}^t \lambda_{1}<n$. 
In Proposition \ref{cor: a(0)=0}, we proved that a modular form $f\in {\MG}$ always vanishes at all $0$-dimensional cusps. 
In this chapter we study the restriction of $f$ to a $1$-dimensional cusp, an operation usually called the \textit{Siegel operator}. 

Let $J$ be a rank $2$ primitive isotropic sublattice of $L$, which we fix throughout this chapter.  
A traditional way to define the Siegel operator $\Phi_{J}$ at the $J$-cusp 
is to choose a $0$-dimensional cusp $I\subset J$, 
take the $I$-trivialization and the coordinates $(\tau, z, w)$ as in \S \ref{sssec: Siegel vs tube}, 
and set 
\begin{equation}\label{eqn: naive Siegel}
(\Phi_{J}f)(\tau) = \lim_{t\to \infty} f(\tau, 0, it), \qquad \tau\in \mathbb{H}. 
\end{equation}
In this way it is easy to define the Siegel operator, but we have to check the modularity of $\Phi_{J}f$ 
and calculate its reduced weight \textit{after} defining it. 

In this chapter we take a more geometric approach working directly with the automorphic vector bundle ${\Elk}$. 
This improves the geometric understanding of the Siegel operator, 
and tells us a priori the modularity of $\Phi_{J}f$ and its weight. 
We work with the partial toroidal compactification ${\XJcpt}$, rather than with the Baily-Borel compactification,  
because the boundary structure of ${\XJcpt}$ is easier to handle and 
${\Elk}$ extends to a vector bundle over ${\XJcpt}$ as we have seen in \S \ref{sec: cano exte}.  
We also wanted to put the Siegel operator on the same ground as the Fourier-Jacobi expansion (\S \ref{sec: FJ}).  
Understanding the Siegel operator at the level of toroidal compactification 
will be useful in some geometric applications. 

Let $\Delta_{J}$ be the boundary divisor of ${\XJcpt}$ and $\pi_{2}\colon \Delta_{J}\to {\HJ}$ be the projection to the $J$-cusp. 
Let ${\LJ}$ be the Hodge bundle on ${\HJ}$. 
For $V(J)=(J^{\perp}/J)_{{\C}}$ we denote by $V(J)_{\lambda'}$ the irreducible representation of 
${\rm O}(V(J))\simeq {\rm O}(n-2, {\C})$ with partition $\lambda'=(\lambda_{2}\geq \cdots \geq \lambda_{n-1})$.
Our result is summarized as follows. 

\begin{theorem}\label{thm: Siegel operator}
Let $\lambda\ne 1, \det$. 
There exists a ${\GJR}$-invariant sub vector bundle ${\ElJ}$ of ${\El}$ with the following properties. 

(1) ${\ElJ}$ extends to a sub vector bundle of the canonical extension of ${\El}$ over ${\XJcpt}$. 

(2) We have a ${\GJRbar}$-equivariant isomorphism 
${\ElJ}|_{\Delta_{J}}\simeq \pi_{2}^{\ast}{\LL}_{J}^{\otimes \lambda_1}\otimes V(J)_{\lambda'}$.  

(3) If $f$ is a ${\G}$-modular form of weight $(\lambda, k)$, 
its restriction to $\Delta_{J}$ as a section of ${\Elk}$ takes values in the sub vector bundle 
${\ElJ}\otimes {\LL}^{\otimes k}|_{\Delta_{J}}\simeq \pi_{2}^{\ast}{\LL}_{J}^{\otimes k+\lambda_1}\otimes V(J)_{\lambda'}$ 
of ${\Elk}|_{\Delta_{J}}$. 

In particular, we have 
\begin{equation*}
f|_{\Delta_{J}}= \pi_{2}^{\ast}(\Phi_{J}f) 
\end{equation*}
for a $V(J)_{\lambda'}$-valued cusp form $\Phi_{J}f$ of weight $k+\lambda_{1}$ on ${\HJ}$ 
with respect to the image of ${\GJZ}\to {\SL}(J)$. 
If $f=\sum_{l}a(l)q^{l}$ is the Fourier expansion of $f$ at a $0$-dimensional cusp $I\subset J$, 
the Fourier expansion of $\Phi_{J}f$ at the $I$-cusp of ${\HJ}$ is given by 
\begin{equation}\label{eqn: Fourier expansion Siegel operator}
(\Phi_{J}f)(\tau)=\sum_{l\in \sigma_{J}\cap {\UIZZ}}a(l)e((l, \tau)), \qquad \tau\in {\HJ} \subset {\UIC}/U(J)_{{\C}}^{\perp},  
\end{equation}
where $\sigma_{J}=({\UJR})_{\geq 0}$ is the isotropic ray in ${\UIR}$ corresponding to $J$.  
\end{theorem}

In \eqref{eqn: Fourier expansion Siegel operator}, 
the pairing $(l, \tau)$ for $l\in \sigma_{J}$ and $\tau\in {\HJ}$ is the natural pairing between 
${\UJC}$ and ${\UIC}/U(J)_{{\C}}^{\perp}$. 
(This $\tau\in {\HJ}$ is different from the coordinate $\tau\in {\HH}$ in \S \ref{sssec: Siegel vs tube}, 
but rather is identified with the point $\tau l_{J}$ there.) 

A point here is that the vector bundle ${\Elk}$ ``reduces'' to the sub vector bundle ${\ElJ}\otimes {\LL}^{\otimes k}$ 
at the boundary divisor $\Delta_{J}$. 
This is the difference with the Siegel operator in the scalar-valued case. 
This reduction corresponds to the reduction 
$\lambda \leadsto \lambda_{1}\boxtimes \lambda'$ 
of the weight, 
and makes it possible to descend $f|_{\Delta_{J}}$ to ${\HJ}$. 
Roughly speaking, this reduction occurs as a result of taking the direct image of ${\Elk}$ to the Baily-Borel compactification. 
In this way, the naive Siegel operator \eqref{eqn: naive Siegel} can be more geometrically understood as 
\begin{equation*}
\textrm{restriction to $\Delta_{J}$} \: + \: \textrm{reduction to ${\ElJ}\otimes {\LL}^{\otimes k}$} \: + \: \textrm{descend to ${\HJ}$}. 
\end{equation*}
The sub vector bundle ${\ElJ}$ will be taken up again in \S \ref{ssec: J-filtration and representation} 
from the viewpoint of a filtration on ${\El}$. 

In \S \ref{ssec: reduce Fourier coefficient} we prepare some calculations related to ${\ElJ}$. 
In \S \ref{ssec: reduce Elk} we define ${\ElJ}$ and prove the properties (1), (2) in Theorem \ref{thm: Siegel operator}. 
The Siegel operator $\Phi_{J}$ is defined in \S \ref{ssec: Siegel operator}, 
and the remaining assertions of Theorem \ref{thm: Siegel operator} are proved there. 

\section{Invariant part for a unipotent group}\label{ssec: reduce Fourier coefficient}

This section is preliminaries for introducing the Siegel operator. 
We prove that the Fourier coefficients of a modular form in the $J$-ray are contained in the invariant subspace 
for a certain unipotent subgroup of ${\On}$, and study this space as a representation of ${\C}^{\ast}\times {\rm O}(n-2, {\C})$. 

Let $F={\Q}, {\R}$. 
Let $W(J)_{F} \subset{\GJF}$ 
be the Heisenberg group and the Jacobi group for $J$ over $F$ defined in \S \ref{ssec: Jacobi group}. 
We choose a rank $1$ primitive sublattice $I$ of $J$, 
and also a rank $1$ sublattice $I'$ of $L$ with $(I, I')\ne 0$. 
Let 
\begin{equation*}
{\G}(I, J)_{F} = {\GJF} \cap \: {\rm Ker}({\GIF}\to {\rm GL}(I)). 
\end{equation*}
By definition ${\G}(I, J)_{F}$ consists of isometries of $L_{F}$ which act trivially on 
$I_{F}$, $J_{F}/I_{F}$ and $V(J)_{F}=(J^{\perp}/J)_{F}$. 
As a subgroup of ${\GJF}$, ${\G}(I, J)_{F}$ contains $W(J)_{F}$, 
and the quotient ${\G}(I, J)_{F}/W(J)_{F}\simeq F$ is the subgroup of ${\GJF}/W(J)_{F}\simeq {\SL}(J_{F})$ 
which acts trivially on $I_{F}$. 

As a subgroup of ${\GIF}$, ${\G}(I, J)_{F}$ contains the unipotent radical $U(I)_{F}$ of ${\GIF}$ by \eqref{eqn: U(J) U(I) G(J)}. 
Let $U(J/I)_{F}$ be the subgroup of ${\rm O}(V(I)_{F})$ acting trivially on $J_F/I_F$ and $V(J)_{F}$. 
Then $U(J/I)_{F}$ is the image of ${\G}(I, J)_{F}$ in ${\rm O}(V(I)_{F})$. 
This is also the image of $W(J)_{F}$ in ${\rm O}(V(I)_{F})$. 
From \eqref{eqn: G(I)K}, we have the exact sequence 
\begin{equation}\label{eqn: G(I, J)}
0 \to U(I)_{F} \to {\G}(I, J)_{F} \to U(J/I)_{F} \to 0. 
\end{equation} 
By \eqref{eqn: stabilizer isotropic line}, the group $U(J/I)_{F}$ is the unipotent radical of the stabilizer of $J_{F}/I_{F}$ in ${\rm O}(V(I)_{F})$ 
and consists of the Eichler transvections of $V(I)_{F}$ with respect to $J_{F}/I_{F}$. 
We have a canonical isomorphism 
\begin{equation*}
U(J/I)_{F} \simeq V(J)_{F}\otimes_{F}(J_{F}/I_{F}). 
\end{equation*}
We define $U(J/I)_{{\C}}<{\rm O}(V(I))$ similarly. 

Now let $f$ be a modular form of weight $(\lambda, k)$ with respect to ${\G}$, and  
$f=\sum_{l}a(l)q^{l}$ be its Fourier expansion at $I$. 
We are interested in the Fourier coefficients $a(l)\in V(I)_{\lambda,k}$ for $l$ in the isotropic ray 
$\sigma_{J}=((J/I)_{{\R}}\otimes I_{{\R}})_{\geq 0}$ corresponding to $J$. 
We denote by  
\begin{equation*}
V(I)_{\lambda}^{U} = V(I)_{\lambda}^{U(J/I)_{{\C}}} 
\end{equation*}
the invariant subspace of ${\VIl}$ for the unipotent subgroup $U(J/I)_{{\C}}$ of ${\rm O}(V(I))$, and put  
\begin{equation*}
V(I)_{\lambda,k}^{U} = V(I)_{\lambda}^{U} \otimes (I_{{\C}}^{\vee})^{\otimes k} \: \: \subset \:  V(I)_{\lambda,k}. 
\end{equation*}

\begin{lemma}\label{lem: reduce a(l) 1-dim cusp}
If $l\in {\UIZZ}\cap \sigma_{J}$, then $a(l)\in V(I)_{\lambda,k}^{U}$. 
\end{lemma}

\begin{proof}
We take the splitting of \eqref{eqn: G(I, J)} for $F={\Q}$ following \eqref{eqn: split G(I)K}, 
and accordingly express elements of ${\G}(I, J)_{{\Q}}$ as $(\gamma_{1}, \alpha)$  
where $\gamma_{1} \in U(J/I)_{{\Q}} \subset {\rm O}(V(I)_{{\Q}})$ and $\alpha\in {\UIQ}$. 
(In the notation \eqref{eqn: GIZZ element}, this is $(\gamma_{1}\otimes {\rm id}_{I}, 1, \alpha)$.) 
There exists a finite-index subgroup $H$ of ${\G}(I, J)_{{\Q}} \cap {\G}$ such that 
$\alpha\in {\UIZ}$ for every element $(\gamma_{1}, \alpha)$ of $H$. 
The group ${\G}(I, J)_{{\Q}}$ acts trivially on the isotropic ray $\sigma_{J}$. 
Therefore, if $l\in {\UIZZ}\cap \sigma_{J}$, we see from Proposition \ref{prop: Fourier coeff} that 
\begin{equation*}
a(l) = a(\gamma_{1}l) = \gamma_{1}(a(l)) 
\end{equation*}
for every element $(\gamma_{1}, \alpha)$ of $H$. 
Here $\gamma_{1}\in U(J/I)_{{\Q}}$ acts on ${\VIlk}$ by its natural action on ${\VIl}$. 
This equality means that $a(l)$ is contained in the $H$-invariant subspace 
$V(I)_{\lambda,k}^{H}=V(I)_{\lambda}^{H}\otimes (I_{{\C}}^{\vee})^{\otimes k}$ of ${\VIlk}$. 
The image of $H$ by the projection 
${\G}(I, J)_{{\Q}} \to U(J/I)_{{\Q}}$, $(\gamma_{1}, \alpha)\mapsto \gamma_{1}$,  
is a full lattice in $U(J/I)_{{\Q}}$. 
In particular, it is Zariski dense in $U(J/I)_{{\C}}$. 
This shows that $V(I)_{\lambda}^{H}=V(I)_{\lambda}^{U}$, 
and so $a(l)\in V(I)_{\lambda,k}^{U}$. 
\end{proof}


Let $P(J/I)_{{\C}}$ be the stabilizer of the isotropic line $J_{{\C}}/I_{{\C}} \subset V(I)$ in ${\rm O}(V(I))$. 
Then $U(J/I)_{{\C}}$ is the unipotent radical of $P(J/I)_{{\C}}$ and sits in the exact sequence (cf.~\eqref{eqn: stabilizer isotropic line}) 
\begin{equation}\label{eqn: P(I,J)}
0 \to U(J/I)_{{\C}} \to P(J/I)_{{\C}} \to {\rm GL}(J_{{\C}}/I_{{\C}})\times {\rm O}(V(J)) \to 1. 
\end{equation}
Therefore $V(I)_{\lambda}^{U}$ is a representation of 
\begin{equation*}
{\rm GL}(J_{{\C}}/I_{{\C}})\times {\rm O}(V(J)) \: \simeq \: {\C}^{\ast}\times {\rm O}(n-2, {\C}) 
\: \simeq \: {\rm SO}(2, {\C}) \times {\rm O}(n-2, {\C}). 
\end{equation*}

\begin{proposition}\label{lem: V(I)lJ weight}
Let $\lambda\ne \det$. 
As a representation of ${\C}^{\ast} \times {\rm O}(V(J))$ we have 
\begin{equation*}
V(I)_{\lambda}^{U} \simeq \chi_{\lambda_{1}} \boxtimes V(J)_{\lambda'}, 
\end{equation*}
where $\chi_{\lambda_{1}}$ is the character of ${\C}^{\ast}$ of weight $\lambda_{1}$ 
and $V(J)_{\lambda'}$ is the irreducible representation of ${\rm O}(V(J))$ associated to 
the partition $\lambda'=(\lambda_{2}\geq \cdots \geq \lambda_{n-1})$. 
\end{proposition}

\begin{proof}
This is purely a representation-theoretic calculation. 
Let us first rewrite the setting. 
Let $V={\C}^n$ be an $n$-dimensional quadratic space over ${\C}$ 
with a basis $e_{1}, \cdots, e_{n}$ such that $(e_{i}, e_{j})=1$ if $i+j=n+1$ and $(e_{i}, e_{j})=0$ otherwise. 
Let $P$ be the stabilizer of the isotropic line ${\C}e_{1}$ in ${\rm O}(V)$ and let $V'=\langle e_{2}, \cdots, e_{n-1} \rangle$. 
Then 
\begin{equation*}
P = ({\C}^{\ast}\times {\rm O}(V')) \ltimes U, 
\end{equation*}
where 
${\C}^{\ast}={\rm SO}(\langle e_{1}, e_{n} \rangle)\simeq {\rm GL}({\C}e_{1})$ 
and $U$ is the group of unipotent matrices 
\begin{equation*}
\begin{pmatrix}
1 & -v^{\vee} & -(v, v)/2 \\ 0 & I_{n-2} & v \\ 0 & 0 & 1 
\end{pmatrix}
\qquad v\in V'. 
\end{equation*}
The problem is to calculate the $U$-invariant part $V_{\lambda}^{U}$ of $V_{\lambda}$ 
as a representation of ${\C}^{\ast}\times {\rm O}(V')$. 

\begin{step}
There exists a ${\C}^{\ast}\times {\rm O}(V')$-equivariant embedding 
$\chi_{\lambda_{1}} \boxtimes V'_{\lambda'} \hookrightarrow V_{\lambda}^{U}$. 
\end{step}

\begin{proof}
We write 
\begin{equation*}
W_{0}= \wedge^{{}^{t}\lambda_{1}}V \otimes \cdots \otimes  \wedge^{{}^{t}\lambda_{\lambda_{1}}}V, 
\end{equation*}
\begin{equation*}
W_{0}'= \wedge^{{}^{t}\lambda_{1}-1}V' \otimes \cdots \otimes  \wedge^{{}^{t}\lambda_{\lambda_{1}}-1}V', 
\end{equation*}
\begin{equation*}
W_{1}= ({\C}e_{1}\wedge \wedge^{{}^{t}\lambda_{1}-1}V') \otimes \cdots \otimes ({\C}e_{1}\wedge \wedge^{{}^{t}\lambda_{\lambda_{1}}-1}V'). 
\end{equation*}
We have a natural ${\C}^{\ast}\times {\rm O}(V')$-equivariant isomorphism 
\begin{equation*}
\iota : {\C}e_{1}^{\otimes \lambda_{1}}\otimes W_{0}' \stackrel{\simeq}{\to} W_{1} \: \: \subset W_{0}. 
\end{equation*}
Recall from \eqref{eqn: Vlambda in wedge tensor} that 
$V_{\lambda}\subset W_{0}$ and $V'_{\lambda'}\subset W_{0}'$. 
(Note that the transpose of $\lambda'$ is $({}^{t}\lambda_{1}-1, \cdots, {}^{t}\lambda_{\lambda_{1}}-1)$.) 
We shall show that the image of 
${\C}e_{1}^{\otimes \lambda_{1}}\otimes V'_{\lambda'}$ by $\iota$ is contained in $V_{\lambda}^{U}$. 
Since ${\C}e_{1}^{\otimes \lambda_{1}}\simeq \chi_{\lambda_{1}}$ as a representation of ${\C}^{\ast}$, 
this would imply our assertion. 

Since $U$ acts on $W_{1}$ trivially, 
it does so on $\iota({\C}e_{1}^{\otimes \lambda_{1}}\otimes V'_{\lambda'})$. 
Thus it suffices to see that 
$\iota({\C}e_{1}^{\otimes \lambda_{1}}\otimes V'_{\lambda'})$ is contained in $V_{\lambda}$. 
Recall from \eqref{eqn: Vlambda highest weight vector} that 
$V_{\lambda}$ and $V'_{\lambda'}$ respectively contain the vectors 
\begin{equation*}
v_{\lambda} = \bigotimes_{i=1}^{\lambda_{1}}(e_{1}\wedge \cdots \wedge e_{{}^{t}\lambda_{i}}), \quad 
v_{\lambda'}' = \bigotimes_{i=1}^{\lambda_{1}}(e_{2}\wedge \cdots \wedge e_{{}^{t}\lambda_{i}}). 
\end{equation*}
Since $\iota(e_{1}^{\otimes \lambda_{1}}\otimes v_{\lambda'}')=v_{\lambda}$, 
we have ${\rm O}(V')\cdot \iota(e_{1}^{\otimes \lambda_{1}}\otimes v_{\lambda'}')\subset V_{\lambda}$. 
Taking the linear hull and using the irreducibility of $V_{\lambda'}'$, 
we see that $\iota({\C}e_{1}^{\otimes \lambda_{1}}\otimes V'_{\lambda'})\subset V_{\lambda}$. 
\end{proof}

For the proof of Proposition \ref{lem: V(I)lJ weight}, it thus suffices to prove $\dim V_{\lambda'}'=\dim V_{\lambda}^{U}$. 
We use the restriction to ${\rm SO}(V)\subset {\rm O}(V)$. 
We first consider the case when $V_{\lambda}$ remains irreducible as a representation of ${\rm SO}(V)$. 
As recalled in \S \ref{sssec: SO rep}, this occurs exactly when $n$ is odd or $n$ is even with ${}^t \lambda_{1}\ne n/2$, 
and $V_{\lambda}$ has highest weight 
\begin{equation*}
\bar{\lambda} = (\bar{\lambda}_{1}, \cdots, \bar{\lambda}_{[n/2]}) = 
(\lambda_{1}, \lambda_{2}-\lambda_{n-1}, \cdots, \lambda_{[n/2]}-\lambda_{n+1-[n/2]}) 
\end{equation*}
in this case. 

\begin{step}
When $V_{\lambda}$ is irreducible as a representation of ${\rm SO}(V)$, 
$V_{\lambda}^{U}$ is irreducible as a representation of ${\rm SO}(V')$ with highest weight 
$\bar{\lambda}'=(\bar{\lambda}_{2}, \cdots, \bar{\lambda}_{[n/2]})$. 
In particular, we have $\dim V_{\lambda'}'=\dim V_{\lambda}^{U}$. 
\end{step}

\begin{proof}
Let $B$ and $B'$ be the groups of upper triangular matrices in ${\rm SO}(V)$ and ${\rm SO}(V')$ respectively (the standard Borel subgroups). 
Let $U_{0}$ and $U_{0}'$ be the groups of unipotent matrices in $B$ and $B'$ respectively. 
Then $U$ and $U_{0}'$ generate $U_{0}$. 
Therefore we have 
\begin{equation}\label{eqn: highest weight space}
V_{\lambda}^{U_{0}} = (V_{\lambda}^{U})^{U_{0}'}.  
\end{equation}
The space $V_{\lambda}^{U_{0}}$ is the highest weight space for the ${\rm SO}(V)$-representation $V_{\lambda}$, 
while $(V_{\lambda}^{U})^{U_{0}'}$ is the highest weight space for the ${\rm SO}(V')$-representation $V_{\lambda}^{U}$. 
The irreducibility of $V_{\lambda}$ as an ${\rm SO}(V)$-representation implies $\dim V_{\lambda}^{U_{0}} =1$, 
which in turn implies by \eqref{eqn: highest weight space} the irreducibility of $V_{\lambda}^{U}$ as a representation of ${\rm SO}(V')$. 

We shall calculate the highest weight of $V_{\lambda}^{U}$ for ${\rm SO}(V')$. 
Let $T$ and $T'$ be the groups of diagonal matrices in $B$ and $B'$ respectively. 
Then $T={\C}^{\ast}\times T'$. 
The highest weight $\bar{\lambda}=(\bar{\lambda}_{1}, \cdots, \bar{\lambda}_{[n/2]})$ of 
the ${\rm SO}(V)$-representation $V_{\lambda}$ is the weight of the $T$-action on the highest weight space $V_{\lambda}^{U_{0}}$. 
Therefore 
$T'$ acts by weight $\bar{\lambda}'=(\bar{\lambda}_{2}, \cdots, \bar{\lambda}_{[n/2]})$ on $V_{\lambda}^{U_{0}}$. 
By \eqref{eqn: highest weight space}, this means that   
the highest weight of $V_{\lambda}^{U}$ for ${\rm SO}(V')$ is $\bar{\lambda}'$. 
\end{proof}

It remains to cover the exceptional case where $V_{\lambda}$ gets reducible when restricted to ${\rm SO}(V)$, 
namely $n$ is even and ${}^{t}\lambda_{1}=n/2$.  

\begin{step}
We have $\dim V_{\lambda'}'=\dim V_{\lambda}^{U}$ even when 
$V_{\lambda}$ is reducible as a representation of ${\rm SO}(V)$. 
\end{step}

\begin{proof}
In this case, the irreducible summands of $V_{\lambda}$ have highest weight $\bar{\lambda}=\lambda$ and $\lambda^{\dag}$ respectively. 
We can argue similarly for each irreducible summand. 
This shows that $V_{\lambda}^{U}$ as a representation of ${\rm SO}(V')$ has two irreducible summands, 
of highest weight 
$\lambda'=(\lambda_{2}, \cdots, \lambda_{n/2})$ and $(\lambda^{\dag})'=(\lambda_{2}, \cdots, -\lambda_{n/2})$. 
On the other hand, $V_{\lambda'}'$ is also reducible as a representation of ${\rm SO}(V')$ 
with highest weight $\lambda'$ and $(\lambda')^{\dag}=(\lambda^{\dag})'$ by \S \ref{sssec: SO rep}. 
This implies that $V_{\lambda}^{U} \simeq V_{\lambda'}'$ as ${\rm SO}(V')$-representations. 
\end{proof}

The proof of Proposition \ref{lem: V(I)lJ weight} is now complete. 
\end{proof}

\section{The sub vector bundle ${\ElJ}$}\label{ssec: reduce Elk}

Let $\lambda \ne \det$. 
We define the sub vector bundle ${\ElJ}$ of ${\El}$ as the image of $V(I)_{\lambda}^{U}\otimes {\OD}$ 
by the $I$-trivialization $\iota_{I}\colon V(I)_{\lambda}\otimes {\OD}\to {\El}$.  

\begin{lemma}\label{lem: GJR invariance ElJ}
The sub vector bundle ${\ElJ}$ of ${\El}$ is ${\GJR}$-invariant. 
In particular, it does not depend on the choice of $I$. 
\end{lemma}

\begin{proof}
Let $\gamma\in {\GJR}$. 
What has to be shown is that the image of $V(I)_{\lambda}^{U}\otimes {\OD}$ by the composition homomorphism  
\begin{equation*}
V(I)_{\lambda}\otimes {\OD} \stackrel{\iota_{I}}{\to} {\El} \stackrel{\gamma}{\to} {\El} \stackrel{\iota_{I}^{-1}}{\to} V(I)_{\lambda}\otimes {\OD} 
\end{equation*}
is again $V(I)_{\lambda}^{U}\otimes {\OD}$. 
This homomorphism coincides with 
\begin{equation}\label{eqn: GJQ invariance ElJ}
V(I)_{\lambda}\otimes {\OD} \stackrel{\gamma}{\to} V(\gamma I)_{\lambda}\otimes {\OD} 
\stackrel{\iota_{\gamma I}}{\to} {\El} \stackrel{\iota_{I}^{-1}}{\to} V(I)_{\lambda}\otimes {\OD}, 
\end{equation}
where $\iota_{\gamma I}$ is the $\gamma I$-trivialization. 
The image of $V(I)_{\lambda}^{U}$ by $\gamma \colon V(I)_{\lambda}\to V(\gamma I)_{\lambda}$ 
is $V(\gamma I)_{\lambda}^{U}$, the invariant subspace of $V(\gamma I)_{\lambda}$ for the unipotent radical 
$U(\gamma J_{{\C}}/ \gamma I_{{\C}})=U(J_{{\C}}/ \gamma I_{{\C}})$ 
of the stabilizer of $J_{{\C}}/ \gamma I_{{\C}}$ in ${\rm O}(V(\gamma I))$. 
Therefore it suffices to show that the homomorphism 
\begin{equation*}
\iota_{\gamma I}^{-1} \circ \iota_{I} : V(I)_{\lambda}\otimes {\OD} \to V(\gamma I)_{\lambda}\otimes {\OD} 
\end{equation*}
sends $V(I)_{\lambda}^{U}\otimes {\OD}$ to $V(\gamma I)_{\lambda}^{U}\otimes {\OD}$. 

The problem is pointwise. 
Let $[\omega]\in {\D}$. 
At the fiber of ${\E}$ over $[\omega]$, the difference of the $I$-trivialization and the $\gamma I$-trivialization is the isometry 
\begin{equation*}
{\ICp}/I_{{\C}} \to {\ICp}\cap \omega^{\perp} \to \omega^{\perp}/{\C}\omega \to 
\gamma I_{{\C}}^{\perp}\cap \omega^{\perp} \to \gamma I_{{\C}}^{\perp}/\gamma I_{{\C}}. 
\end{equation*}
This sends the isotropic line $J_{{\C}}/I_{{\C}}$ of ${\ICp}/I_{{\C}}$ as 
\begin{equation*}
J_{{\C}}/I_{{\C}} \to J_{{\C}}\cap \omega^{\perp} = J_{{\C}}\cap \omega^{\perp} \to J_{{\C}}/\gamma I_{{\C}}. 
\end{equation*}
Therefore the induced isomorphism ${\rm O}(V(I))\to {\rm O}(V(\gamma I))$ sends 
the subgroup $U(J/I)_{{\C}}$ to $U(J_{{\C}}/\gamma I_{{\C}})$. 
It follows that the induced isomorphism 
\begin{equation*}
(\iota_{\gamma I}^{-1} \circ \iota_{I})_{[\omega]} : V(I)_{\lambda} \to V(\gamma I)_{\lambda}  
\end{equation*}
sends $V(I)_{\lambda}^{U}$ to $V(\gamma I)_{\lambda}^{U}$. 
\end{proof}

Recall that the canonical extension of ${\El}$ over the partial toroidal compactification ${\XJcpt}$ 
is defined via the $I$-trivialization $V(I)_{\lambda}\otimes \mathcal{O}_{{\XJ}} \to {\El}$. 
Therefore, by construction, ${\ElJ}$ extends to a sub vector bundle of 
the canonical extension of ${\El}$ (again denoted by ${\ElJ}$). 
The $I$-trivialization ${\El}\to V(I)_{\lambda}\otimes \mathcal{O}_{{\XJcpt}}$ over ${\XJcpt}$ 
sends ${\ElJ}$ to $V(I)_{\lambda}^{U} \otimes \mathcal{O}_{{\XJcpt}}$. 

\begin{proposition}\label{prop: descend ElJ}
There exists an ${\SL}(J_{{\R}})$-equivariant vector bundle $\Phi_{J}{\El}$ on ${\HJ}$ such that 
we have a ${\GJRbar}$-equivariant isomorphism 
\begin{equation*}
{\ElJ}|_{\Delta_{J}} \simeq \pi_{2}^{\ast}(\Phi_{J}{\El}) 
\end{equation*}
of vector bundles on $\Delta_{J}$. 
\end{proposition}

\begin{proof}
Let $j(\gamma, [\omega])$ be the factor of automorphy of the ${\GJR}$-action on ${\ElJ}$ 
with respect to the $I$-trivialization ${\ElJ}\simeq V(I)_{\lambda}^{U}\otimes {\OD}$. 
This is a ${\rm GL}(V(I)_{\lambda}^{U})$-valued function on ${\GJR}\times {\D}$. 
We shall prove the following. 
\begin{enumerate}
\item For fixed $\gamma$, the function $j(\gamma, [\omega])$ of $[\omega]$ is constant on each fiber of ${\D}\to {\HJ}$.  
\item $j(\gamma, [\omega])={\rm id}$ if $\gamma \in W(J)_{{\R}}$. 
\end{enumerate}
Since ${\GJR}/W(J)_{{\R}}\simeq {\SL}(J_{{\R}})$, 
these properties ensure that 
$j(\gamma, [\omega])$ descends to a ${\rm GL}(V(I)_{\lambda}^{U})$-valued function on ${\SL}(J_{{\R}})\times {\HJ}$. 
This function defines the factor of automorphy of an ${\SL}(J_{{\R}})$-equivariant vector bundle $\Phi_{J}{\El}$ on ${\HJ}$ such that 
${\ElJ}\simeq \pi^{\ast}(\Phi_{J}{\El})$ as ${\GJR}$-equivariant vector bundles on ${\D}$. 
This gives an isomorphism 
${\ElJ}|_{\Delta_{J}} \simeq \pi_{2}^{\ast}(\Phi_{J}{\El})$ 
over $\Delta_{J}$.  

We first check the property (2). 
Since $W(J)_{{\R}}$ acts on $I_{{\R}}$ trivially, 
we see from Lemma \ref{lem: f.a. stab} that the factor of automorphy of the $W(J)_{{\R}}$-action on ${\El}$ 
with respect to the $I$-trivialization is given by the natural action of $W(J)_{{\R}}$ on ${\VIl}$. 
Since the image of $W(J)_{{\R}}$ in ${\rm O}(V(I)_{{\R}})$ is equal to $U(J/I)_{{\R}}$, 
$W(J)_{{\R}}$ acts on $V(I)_{\lambda}^{U}$ trivially by definition. 
This implies (2). 

Next we verify the property (1). 
The fibers of ${\D}\to \mathcal{V}_{J}$ are contained in ${\UJC}$-orbits in ${\D}(J)\supset {\D}$, 
and the fibers of $\Delta_{J}\to {\HJ}$ are $W(J)_{{\R}}/{\UJR}$-orbits. 
In particular, the constancy on the fibers of ${\D}\to {\VJ}$ would follow from the constancy on ${\UJR}$-orbits and 
the identity theorem in complex analysis for ${\UJR}\subset {\UJC}$. 
Thus we are reduced to checking the constancy on $W(J)_{{\R}}$-orbits. 
Let $\gamma\in {\GJR}$ and $g\in W(J)_{{\R}}$. 
Then we can calculate 
\begin{eqnarray*}
j(\gamma, g([\omega])) & = & j(\gamma g, [\omega]) \circ j(g, [\omega])^{-1} = j(\gamma g, [\omega]) \\ 
& = & j(\gamma g \gamma^{-1}, \gamma([\omega])) \circ j(\gamma, [\omega]) = j(\gamma, [\omega]). 
\end{eqnarray*}
In the second and the last equalities we used the property (2) proved above, 
with the normality of $W(J)_{{\R}}$ in ${\GJR}$ in the last equality. 
The property (1) is thus proved. 
\end{proof}

\begin{remark}\label{rmk: I-trivialization PhiJEl}
By construction, $\Phi_{J}{\El}$ is endowed with a trivialization $V(I)_{\lambda}^{U}\otimes \mathcal{O}_{{\HJ}}\simeq \Phi_{J}{\El}$, 
whose pullback agrees with the $I$-trivialization $V(I)_{\lambda}^{U}\otimes \mathcal{O}_{\Delta_{J}}\simeq {\ElJ}|_{\Delta_{J}}$ 
of ${\ElJ}$ over $\Delta_{J}$. 
\end{remark}

We can calculate the weights of $\Phi_{J}{\El}$ by using Proposition \ref{lem: V(I)lJ weight}. 
Let ${\LJ}$ be the Hodge bundle on ${\HJ}$. 

\begin{proposition}\label{prop: PhiJEl}
There exists an ${\SL}(J_{{\R}})$-equivariant isomorphism 
\begin{equation}\label{eqn: PhiJEl weight}
\Phi_{J}{\El} \simeq {\LL}_{J}^{\otimes \lambda_{1}} \otimes V(J)_{\lambda'}, 
\end{equation}
of vector bundles on ${\HJ}$. 
\end{proposition}

The proof of Proposition \ref{prop: PhiJEl} is divided into several steps. 
Let us formulate the first half as preparatory lemmas as follows. 
Let $P(J)$ be the stabilizer of $J_{{\C}}$ in ${\OLC}$. 
We write $Q(J)=Q-Q\cap {\proj}J_{{\C}}^{\perp}$. 
Recall that ${\El}$ is naturally defined over $Q$ as an ${\OLC}$-equivariant vector bundle. 

\begin{lemma}\label{step+1}
The vector bundle ${\ElJ}$ extends to a $P(J)$-invariant sub vector bundle of ${\El}$ over $Q(J)$ 
(again denoted by ${\ElJ}$). 
\end{lemma}

\begin{proof}
For each ${\C}$-line $I'\subset J_{{\C}}$, the $I'$-trivialization 
$\iota_{I'}\colon V(I')_{\lambda}\otimes \mathcal{O}\to {\El}$ is defined over $Q(I')=Q-Q\cap {\proj}(I')^{\perp}$. 
The same argument as the second half of the proof of Lemma \ref{lem: GJR invariance ElJ} 
shows that for two ${\C}$-lines $I_{1}, I_{2} \subset J_{{\C}}$, 
we have 
\begin{equation*}
\iota_{I_{1}}(V(I_{1})_{\lambda}^{U}\otimes \mathcal{O}) = \iota_{I_{2}}(V(I_{2})_{\lambda}^{U}\otimes \mathcal{O}) 
\end{equation*}
over $Q(I_{1})\cap Q(I_{2})$. 
Therefore, by gluing the image of $\iota_{I'}$ for all ${\C}$-lines $I'\subset J_{{\C}}$, 
we obtain a sub vector bundle of ${\El}$ over $Q(J)=\cup_{I'}Q(I')$ which extends ${\ElJ}$. 
Since $\gamma\in P(J)$ sends $\iota_{I'}(V(I')_{\lambda}^{U}\otimes \mathcal{O})$ to 
$\iota_{\gamma I'}(V(\gamma I')_{\lambda}^{U}\otimes \mathcal{O})$ 
(cf.~the proof of Lemma \ref{lem: GJR invariance ElJ}), 
this sub vector bundle is $P(J)$-invariant. 
\end{proof}

\begin{lemma}\label{step+2}
Let $D_{J}=Q(J)\cap {\proj}{\ICp}$. 
The $I$-trivialization 
$V(I)_{\lambda}^{U}\otimes \mathcal{O}_{Q(I)}\to {\ElJ}$ over $Q(I)$ 
extends to an isomorphism 
\begin{equation}\label{eqn: step+2}
V(I)_{\lambda}^{U}\otimes \mathcal{O}_{Q(J)}\to {\ElJ}\otimes \mathcal{O}_{Q(J)}(\lambda_{1}D_{J})  
\end{equation}
over $Q(J)$, which is equivariant with respect to the stabilizer of $I_{{\C}}$ in $P(J)$. 
\end{lemma}

\begin{proof}
We choose an arbitrary embedding $2U_{{\C}}\hookrightarrow L_{{\C}}$ compatible with $I_{{\C}}\subset J_{{\C}}$ in the sense of \S \ref{sec: cano exte} 
and accordingly take a lift ${\rm GL}(J_{{\C}})\hookrightarrow P(J)$ of ${\rm GL}(J_{{\C}})$. 
Let $T\simeq {\C}^{\ast}$ be the subgroup of ${\rm GL}(J_{{\C}})$ consisting of matrices 
$\begin{pmatrix} 1 & 0 \\ 0 & \alpha \end{pmatrix}$, $\alpha\in {\C}^{\ast}$, 
with respect to the basis $e_{1}, e_{2}$ of $J_{{\C}}$. 
($e_{1}$ spans $I_{{\C}}$ and $e_{2}$ spans $J_{{\C}}/I_{{\C}}$.) 
The image of $T$ in $P(J/I)_{{\C}}$ is a lift of ${\rm GL}(J_{{\C}}/I_{{\C}})$ in \eqref{eqn: P(I,J)}. 
Then 
\begin{equation*}
V(I) = {\C}e_{2} \oplus V(J) \oplus {\C}f_{2} 
\end{equation*}
is the weight decomposition for $T$, where 
${\C}e_{2}$, $V(J)$,  ${\C}f_{2}$ have weight $1$, $0$, $-1$ respectively. 
A general $T$-orbit $C^{\circ}=T [\omega]$ in $Q(I)$ gives a flow 
converging to the point $p=[f_{2}]$ of $D_{J}$ as $\alpha\to 0$ from a normal direction. 
Let $C=C^{\circ}\cup p\simeq {\C}$ be the closure of such a $T$-orbit in $Q(J)$. 
The proof of Lemma \ref{step+2} is based on the following assertion. 

\begin{claim}\label{claim: I-trivialization boundary}
The $I$-trivialization ${\E}|_{C^{\circ}}\simeq V(I)\otimes \mathcal{O}_{C^{\circ}}$ over $C^{\circ}$ extends to an isomorphism 
\begin{equation*}
{\E}|_{C}  \: \: \simeq \: \: {\C}e_{2} \otimes \mathcal{O}_{C}(-p)  
\;  \oplus \;  V(J)\otimes \mathcal{O}_{C}  
\;  \oplus \;  {\C}f_{2} \otimes \mathcal{O}_{C}(p) 
\end{equation*}
over $C$. 
\end{claim}

We postpone the proof of this claim for a while and continue the proof of Lemma \ref{step+2}. 
From Claim \ref{claim: I-trivialization boundary}, we see that 
if ${\VIl}=\bigoplus_{r} V(r)$ is the weight decomposition for $T$ with $V(r)$ the weight $r$ subspace, 
the $I$-trivialization of ${\El}$ over $C^{\circ}$ extends to an isomorphism 
\begin{equation*}
{\El}|_{C} \simeq \bigoplus_{r} V(r)\otimes \mathcal{O}_{C}(-r p) 
\end{equation*}
over $C$. 
Since $V(I)_{\lambda}^{U}\subset V(\lambda_{1})$ by Proposition \ref{lem: V(I)lJ weight}, 
we obtain 
\begin{equation*}
{\E}_{\lambda}^{J}|_{C} \simeq V(I)_{\lambda}^{U} \otimes \mathcal{O}_{C}(-\lambda_{1} p). 
\end{equation*}
Finally, if we vary the embedding $2U_{\C}\hookrightarrow L_{{\C}}$, then
the point $p=[f_{2}]$ runs over $D_{J}$. 
This implies the assertion of Lemma \ref{step+2}. 
\end{proof}

We give the postponed proof of Claim \ref{claim: I-trivialization boundary}. 

\begin{proof}[(Proof of Claim \ref{claim: I-trivialization boundary})]
Let $v\in V(I)$ be a weight vector for $T$ with weight $r\in \{ -1, 0, 1\}$ 
and let $s_{v}$ be the corresponding section of ${\E}$. 
We calculate the limit behavior of $s_{v}$ on the $T$-orbit $C^{\circ}=T[\omega]$ as $\alpha\to 0$. 
We write $\gamma_{\alpha} = \begin{pmatrix} 1 & 0 \\ 0 & \alpha \end{pmatrix} \in T$. 
We lift $V(I)\hookrightarrow {\ICp}$ by the given embedding $U_{{\C}}\hookrightarrow L_{{\C}}$. 
By Lemma \ref{lem: basic section E} for $l=e_{1}\in I$, we have  
\begin{eqnarray}\label{eqn: calculate sv at boundary}
s_{v}(\gamma_{\alpha}[\omega]) 
& = & v - (v, s_{e_{1}}(\gamma_{\alpha}[\omega]))e_{1} \; \; \mod {\C}\gamma_{\alpha}(\omega) \nonumber \\ 
& = & v - (v, \gamma_{\alpha}(s_{e_{1}}([\omega])))e_{1} \; \; \mod {\C}\gamma_{\alpha}(\omega) \nonumber \\ 
& = & v - (\gamma_{\alpha}^{-1}(v), s_{e_{1}}([\omega]))e_{1} \; \; \mod {\C}\gamma_{\alpha}(\omega) \nonumber \\ 
& = & v - \alpha^{-r}(v, s_{e_{1}}([\omega]))e_{1}  \; \; \mod {\C}\gamma_{\alpha}(\omega). 
\end{eqnarray}
We take the ${\C}e_{2}$-trivialization of ${\E}$ and express $s_{v}(\gamma_{\alpha}[\omega])$ as a $V({\C}e_{2})$-valued function. 
We identify $V({\C}e_{2})={\C}e_{1} \oplus V(J) \oplus {\C}f_{1}$ naturally. 
Then, according to the weight $r$ of $v$, we have 
\begin{equation*}
s_{v}(\gamma_{\alpha}[\omega]) = 
\begin{cases}
v + C_{1}(v) e_{1} & v\in V(J) \\ 
\alpha^{-1} C_{2} e_{1} & v=e_{2} \\ 
\alpha C_{2}^{-1} f_{1} + \alpha C_{3} e_{1} + \alpha v_{0} & v=f_{2} 
\end{cases}
\end{equation*}
as a $V({\C}e_{2})$-valued function. 
Here $C_{1}(v)$ is a linear function on $V(J)$, $C_2\ne 0$ and $C_{3}$ are constants, 
and $v_{0}\in V(J)$ is some constant vector. 
These expressions in the cases $v\in V(J)$ and $v=e_{2}$ are apparent from \eqref{eqn: calculate sv at boundary}, 
because the vector in \eqref{eqn: calculate sv at boundary} is already perpendicular to $e_{2}$ in these cases. 
The case $v=f_{2}$ follows from the conditions 
\begin{equation*}
(s_{f_{2}}, s_{e_{2}})=1, \quad (s_{f_{2}}, s_{f_{2}})=0, \quad (s_{f_{2}}, s_{w})=0  \; \textrm{for} \; w\in V(J). 
\end{equation*}
(This can also be calculated by using the coordinates $(\tau, z, w)$ in \S \ref{sssec: Siegel vs tube}.) 
The assertion of Claim \ref{claim: I-trivialization boundary} now follows from these expressions. 
\end{proof}

Now we can complete the proof of Proposition \ref{prop: PhiJEl}. 

\begin{proof}[(Proof of Proposition \ref{prop: PhiJEl})] 
We pass from $Q(J)$ to ${\proj}J_{{\C}}^{\vee}$. 
By the same argument as the proof of Proposition \ref{prop: descend ElJ} 
with ${\GJR}$ replaced by $P(J)$ and $W(J)_{{\R}}$ replaced by the kernel of $P(J)\to {\rm GL}(J_{{\C}})\times {\rm O}(V(J))$, 
we find that the $P(J)$-equivariant vector bundle ${\ElJ}$ on $Q(J)$ 
descends to a ${\rm GL}(J_{{\C}})\times {\rm O}(V(J))$-equivariant vector bundle on ${\proj}J_{{\C}}^{\vee}$.  
This is an extension of $\Phi_{J}{\El}$, and we denote it again by $\Phi_{J}{\El}$. 
Let $p_{I}=I^{\perp}\cap {\proj}J_{{\C}}^{\vee}$ be the $I$-cusp of ${\HJ}$. 
Since $D_{J}$ is the fiber of $Q(J)\to {\proj}J_{{\C}}^{\vee}$ over $p_{I}$, 
we find that the isomorphism \eqref{eqn: step+2} descends to an isomorphism 
\begin{equation}\label{eqn: step+2 descend}
V(I)_{\lambda}^{U}\otimes \mathcal{O}_{{\proj}J_{{\C}}^{\vee}} \to 
\Phi_{J}{\El}\otimes \mathcal{O}_{{\proj}J_{{\C}}^{\vee}}(\lambda_{1}p_{I}).   
\end{equation}
This is equivariant with respect to the stabilizer of $I_{{\C}}$ in ${\rm GL}(J_{{\C}})$ and ${\rm O}(V(J))$. 
Note that these groups act on $V(I)_{\lambda}^{U}$ by the representation in Proposition \ref{lem: V(I)lJ weight}. 


\begin{claim}\label{step+3}
The element $g(\alpha)=\begin{pmatrix} \alpha & 0 \\ 0 & \alpha^{-1} \end{pmatrix}$ of ${\rm GL}(J_{{\C}})$, $\alpha\in {\C}^{\ast}$,  
acts on the fiber of $\Phi_{J}{\El}$ over $p_{I}$ as the scalar multiplication by $\alpha^{\lambda_{1}}$. 
\end{claim}

We prove Claim \ref{step+3}. 
By Proposition \ref{lem: V(I)lJ weight}, the matrices $\begin{pmatrix} 1 & 0 \\ 0 & \alpha \end{pmatrix}$ in ${\rm GL}(J_{{\C}})$ 
act on $V(I)_{\lambda}^{U}$ as the scalar multiplication by $\alpha^{\lambda_{1}}$. 
Moreover, the matrices $\begin{pmatrix} \beta & 0 \\ 0 & 1 \end{pmatrix}$ act on $V(I)$ trivially. 
It follows that $g(\alpha)$ acts on $V(I)_{\lambda}^{U}$ as the scalar multiplication by $\alpha^{-\lambda_{1}}$. 
On the other hand, since the tangent space of $p_{I}\in {\proj}J_{{\C}}^{\vee}$ is 
\begin{equation*}
\hom (I^{\perp}\cap J_{{\C}}^{\vee}, \: J_{{\C}}^{\vee}/(I^{\perp}\cap J_{{\C}}^{\vee})) \simeq 
\hom ((J/I)_{{\C}}^{\vee},  {\ICv}), 
\end{equation*}
the element $g(\alpha)$ acts on it by the multiplication by $\alpha^{-2}$. 
Hence $g(\alpha)$ acts on the fiber of $\mathcal{O}_{{\proj}J_{{\C}}^{\vee}}(-\lambda_{1}p_{I})$ over $p_{I}$ 
as the multiplication by $\alpha^{2\lambda_{1}}$. 
By the isomorphism \eqref{eqn: step+2 descend}, we find that 
$g(\alpha)$ acts on the fiber of $\Phi_{J}{\El}$ over $p_{I}$ as the scalar multiplication by 
$\alpha^{-\lambda_{1}}\cdot \alpha^{2\lambda_{1}} = \alpha^{\lambda_{1}}$. 
This proves Claim \ref{step+3}. 

\medskip 
 
We go back to the proof of Proposition \ref{prop: PhiJEl}. 
The torus consisting of the matrices $g(\alpha)$ is the reductive part of the stabilizer of $p_{I}$ in ${\rm SL}(J_{{\C}})$. 
Therefore Claim \ref{step+3} implies that $\Phi_{J}{\El}$ is isomorphic to 
a direct sum of copies of ${\LL}_{J}^{\otimes \lambda_{1}}$ as an ${\SL}(J_{{\C}})$-equivariant vector bundle on ${\proj}J_{{\C}}^{\vee}$. 
Moreover, by the isomorphism \eqref{eqn: step+2 descend} and Proposition \ref{lem: V(I)lJ weight}, 
the action of ${\rm O}(V(J))$ on the fibers of $\Phi_{J}{\El}$ is isomorphic to the representation $V(J)_{\lambda'}$. 
Therefore 
$\Phi_{J}{\El}\simeq {\LL}_{J}^{\otimes \lambda_{1}} \otimes V(J)_{\lambda'}$ 
as ${\rm SL}(J_{{\C}})\times {\rm O}(V(J))$-equivariant vector bundles on ${\proj}J_{{\C}}^{\vee}$. 
This finishes the proof of Proposition \ref{prop: PhiJEl}.   
\end{proof}

\begin{remark}\label{rmk: PhiJEl O(V(J))}
By the proof, the vector bundle $\Phi_{J}{\El}$ over ${\HJ}$ is in fact ${\rm SL}(J_{{\R}})\times {\rm O}(V(J)_{{\R}})$-linearized, 
and the isomorphism $\Phi_{J}{\El}\simeq \mathcal{L}_{J}^{\otimes \lambda_{1}} \otimes V(J)_{\lambda'}$ over ${\HJ}$ 
is ${\rm SL}(J_{{\R}})\times {\rm O}(V(J)_{{\R}})$-equivariant. 
\end{remark}

\section{The Siegel operator}\label{ssec: Siegel operator}

Combining the arguments so far, we can now define the Siegel operator at the $J$-cusp. 

\begin{proposition}
Let $f\in {\MG}$ with $\lambda\ne 1, \det$. 
There exists a cusp form $\Phi_{J}f$ on ${\HJ}$ with values in 
$\Phi_{J}{\El}\otimes {\LL}_{J}^{\otimes k}\simeq {\LL}_{J}^{\otimes \lambda_{1}+k}\otimes V(J)_{\lambda'}$ 
and invariant under the image of ${\GJZ}\to {\SL}(J)$ such that $f|_{\Delta_{J}}=\pi_{2}^{\ast}(\Phi_{J}f)$. 
If $f=\sum_{l}a(l)q^{l}$ is the Fourier expansion of $f$ at a $0$-dimensional cusp $I\subset J$, 
the Fourier expansion of $\Phi_{J}f$ at the $I$-cusp of ${\HJ}$ is given by 
\begin{equation*}
(\Phi_{J}f)(\tau) = \sum_{l\in \sigma_{J}\cap {\UIZZ}}a(l)e((l, \tau)), \qquad \tau \in {\HJ}\subset {\UIC}/U(J)_{{\C}}^{\perp}. 
\end{equation*}
\end{proposition}

Here we recall that ${\LJ}$ and $\Phi_{J}{\El}$ on ${\HJ}$ are endowed with $I$-trivializations 
whose pullback agree with the $I$-trivializations of ${\LL}$ and ${\E}_{\lambda}^{J}$ respectively 
(Remarks \ref{rmk: I-trivialization LJ} and \ref{rmk: I-trivialization PhiJEl}).  
These define the $I$-trivialization 
$\Phi_{J}{\El}\otimes {\LL}_{J}^{\otimes k} \simeq V(I)_{\lambda,k}^{U}\otimes \mathcal{O}_{{\HJ}}$ 
of $\Phi_{J}{\El}\otimes {\LL}_{J}^{\otimes k}$ whose pullback agrees with the $I$-trivialization of ${\ElJ}\otimes {\LL}^{\otimes k}$. 
The Fourier expansion of $\Phi_{J}f$ is done with respect to this trivialization. 

\begin{proof}
We choose a rank $1$ primitive sublattice $I\subset J$ and let $f=\sum_{l}a(l)q^{l}$ be the Fourier expansion of $f$ at $I$. 
By \eqref{eqn: restrict torus boundary divisor} and the gluing map ${\XJcpt}\to {\XIcpt}$ in Lemma \ref{lem: glue}, we see that 
\begin{equation}\label{eqn: f at DeltaJ}
f|_{\Delta_{J}}  = \sum_{l\in \sigma_{J}\cap {\UIZZ}} a(l) q^{l} 
\end{equation}
as a ${\VIlk}$-valued function on $\Delta_{J}\subset {\UIC}/{\UJC}$. 
By Lemma \ref{lem: reduce a(l) 1-dim cusp}, 
the function $f|_{\Delta_{J}}$ takes values in $V(I)_{\lambda,k}^{U}$.  
This in turn implies that $f|_{\Delta_{J}}$ as a section of ${\Elk}|_{\Delta_{J}}$ takes values in the sub vector bundle 
\begin{equation*}
{\ElJ}\otimes {\LL}^{\otimes k}|_{\Delta_{J}} \simeq 
\pi_{2}^{\ast}(\Phi_{J}{\El}\otimes {\LL}_{J}^{\otimes k}) \simeq 
\pi_{2}^{\ast}{\LL}_{J}^{\otimes \lambda_{1}+k}\otimes V(J)_{\lambda'}. 
\end{equation*}
Since the section $f|_{\Delta_{J}}$ is ${\GJZbar}$-invariant, 
it is in particular $W(J)_{{\Z}}/{\UJZ}$-invariant, 
and so it descends to a section of $\bar{\pi}_{2}^{\ast}{\LL}_{J}^{\otimes \lambda_{1}+k}\otimes V(J)_{\lambda'}$ over 
$\Delta_{J}/(W(J)_{{\Z}}/{\UJZ})$ where 
\begin{equation*}
\bar{\pi}_{2} :  \Delta_{J}/(W(J)_{{\Z}}/{\UJZ}) \to {\HJ} 
\end{equation*}
is the projection. 
Since $\bar{\pi}_{2}$ is a proper map (family of abelian varieties), we find that $f|_{\Delta_{J}}$ is constant on each $\pi_{2}$-fiber. 
Therefore $f|_{\Delta_{J}}=\pi_{2}^{\ast}(\Phi_{J}f)$ for a section $\Phi_{J}f$ of 
${\LL}_{J}^{\otimes \lambda_{1}+k}\otimes V(J)_{\lambda'}$ over ${\HJ}$. 
Since $f|_{\Delta_{J}}$ is ${\GJZbar}$-invariant, 
$\Phi_{J}f$ is invariant under the image of ${\GJZ}\to {\SL}(J)$. 

The fact that the pullback of the $I$-trivialization of $\Phi_{J}{\El}\otimes {\LL}_{J}^{\otimes k}$ agrees with 
the $I$-trivialization of ${\ElJ}\otimes {\LL}^{\otimes k}$ implies that 
the pullback of $\Phi_{J}f$ as a $V(I)_{\lambda,k}^{U}$-valued function by $\Delta_{J}\to {\HJ}$ 
equals to $f|_{\Delta_{J}}$ as a $V(I)_{\lambda,k}^{U}$-valued function. 
Therefore $\Phi_{J}f$ as a $V(I)_{\lambda,k}^{U}$-valued function on ${\HJ}$ is given by the right hand side of \eqref{eqn: f at DeltaJ}: 
\begin{equation*}
\Phi_Jf = \sum_{l\in \sigma_{J}\cap {\UIZZ}} a(l) q^{l}. 
\end{equation*}
Here $q^l$ for $l\in \sigma_{J}\cap {\UIZZ}$ is naturally viewed as a function on ${\HJ}\subset {\UIC}/U(J)_{{\C}}^{\perp}$ 
by the pairing between ${\UJC}$ and ${\UIC}/U(J)_{{\C}}^{\perp}$. 
This gives the Fourier expansion of $\Phi_Jf$ at the $I$-cusp of ${\HJ}$. 
By Proposition \ref{cor: a(0)=0}, $\Phi_Jf$ vanishes at the $I$-cusp. 
Since this holds at every cusp of ${\HJ}$, we see that $\Phi_{J}f$ is a cusp form.   
\end{proof}

Let $\Gamma_{J}$ be the image of ${\GJZ}$ in ${\SL}(J)\simeq {\SL}(2, {\Z})$. 
We call the map
\begin{equation}\label{eqn: Siegel operator}
{\MG}\to S\!_{\lambda_{1}+k}(\Gamma_{J}) \otimes V(J)_{\lambda'}, \quad f\mapsto \Phi_{J}f, 
\end{equation}
the \textit{Siegel operator} at the $J$-cusp. 

We look at some examples. 
We use the same notation as in the proof of Proposition \ref{lem: V(I)lJ weight}. 

\begin{example}\label{ex: Siegel wedge}
Let $\lambda=(1^{d})$ for $0<d<n$, namely $V_{\lambda}=\wedge^{d}V$. 
Then 
\begin{equation*}
(\wedge^{d}V)^{U} = {\C}e_{1}\wedge (\wedge^{d-1} \langle e_{1}, \cdots, e_{n-1} \rangle) 
\simeq {\C}e_{1}\otimes \wedge^{d-1} \langle e_{2}, \cdots, e_{n-1} \rangle. 
\end{equation*}
In this case, $\Phi_{J}f$ is a $\binom{n-2}{d-1}$-tuple of scalar-valued cusp forms of weight $k+1$. 
\end{example}

\begin{example}\label{ex: Siegel sym}
Let $\lambda=(d)$, namely $V_{(d)}$ is the main irreducible component of ${\rm Sym}^{d}V$ 
(see Example \ref{ex: wedge sym} (2)). 
We have 
\begin{equation*}
V_{(d)}^{U} = {\C}e_{1}^{\otimes d} \subset {\rm Sym}^{d}V. 
\end{equation*}
In this case, $\Phi_{J}f$ is a single scalar-valued cusp form of weight $k+d$. 
\end{example}

The Siegel operator for vector-valued Siegel modular forms is studied in \cite{We} \S 2. 
The case of genus $2$ is also studied in \cite{Ar} \S 1.  
Let us observe that the weight calculation in Example \ref{ex: Siegel sym} in the case $n=3$ agrees with 
the results of \cite{Ar} and \cite{We} for Siegel modular forms of genus $2$. 

\begin{example}\label{ex: compare SMF}
Let $n=3$. 
In \cite{Ar} and \cite{We}, it is proved that the Siegel operator for a Siegel modular form of genus $2$ and weight $({\rm Sym}^{j}, \det^{l})$ 
produces a scalar-valued cusp form of weight $j+l$ on the $1$-dimensional cusp. 

On the other hand, when $j=2d$ is even, we saw in Example \ref{ex: weight n=3} that 
the Siegel weight $({\rm Sym}^{2d}, \det^{l})$ corresponds to the orthogonal weight $(\lambda, k)=((d), d+l)$.   
According to Example \ref{ex: Siegel sym}, $\Phi_{J}f$ is a cusp form of weight $d+(d+l)=j+l$. 
This agrees with the above results of \cite{Ar} and \cite{We}. 
\end{example}

In general, the Siegel operator in the form of \eqref{eqn: Siegel operator} is still not surjective for the following obvious reason. 
Let $\Gamma(J)_{{\Z}}^{\ast}$ be the stabilizer of $J$ in ${\G}$, and 
let $\Gamma_{J}^{\ast}$ be the image of $\Gamma(J)_{{\Z}}^{\ast}$ in ${\rm SL}(J)\times {\rm O}(J^{\perp}/J)$. 
Then $\Gamma_{J}=\Gamma_{J}^{\ast}\cap {\rm SL}(J)$ is of finite index and normal in $\Gamma_{J}^{\ast}$. 
Let 
\begin{equation*}
G = \Gamma_{J}^{\ast}/\Gamma_{J} \simeq \Gamma(J)_{{\Z}}^{\ast}/ {\GJZ}. 
\end{equation*}
The modular forms are not only ${\GJZ}$-invariant but also $\Gamma(J)_{{\Z}}^{\ast}$-invariant. 
Therefore, in view of Remark \ref{rmk: PhiJEl O(V(J))}, 
we see that the image of the map \eqref{eqn: Siegel operator} is contained in 
the $G$-invariant part of $S\!_{\lambda_{1}+k}(\Gamma_{J})\otimes V(J)_{\lambda'}$.


\chapter{Fourier-Jacobi expansion}\label{sec: FJ}

Let $L$ be a lattice of signature $(2, n)$ with $n\geq 3$ and ${\G}$ be a finite-index subgroup of ${\OL}$. 
We fix a rank $2$ primitive isotropic sublattice $J$ of $L$. 
In this chapter we study the Fourier-Jacobi expansion of vector-valued modular forms at the $J$-cusp. 
From a geometric point of view, the Fourier-Jacobi expansion is the Taylor expansion along the boundary divisor $\Delta_{J}$ 
of the partial toroidal compactification ${\XJcpt}$. 
The $m$-th Fourier-Jacobi coefficient is the $m$-th Taylor coefficient, 
and is essentially a section of the vector bundle ${\Elk}\otimes \Theta_{J}^{\otimes m}$ 
over $\Delta_{J}$ where $\Theta_{J}$ is the conormal bundle of $\Delta_{J}$. 
Here we have some special properties beyond general Taylor expansion: 
\begin{itemize}
\item existence of the projection $\pi_{1}\colon {\XJcpt}\to \Delta_{J}$ and 
the isomorphism ${\Elk}\simeq \pi_{1}^{\ast}({\Elk}|_{\Delta_{J}})$ (Proposition \ref{prop: L and E as GJR bundle}), and 
\item existence of a special generator $\omega_{J}$ of the ideal sheaf of $\Delta_{J}$ which is a linear map on each fiber of $\pi_{1}$. 
\end{itemize}
These properties ensure that the $m$-th Fourier-Jacobi coefficient as a section of ${\Elk}\otimes \Theta_{J}^{\otimes m}$ over $\Delta_{J}$ is 
canonically defined (Corollary \ref{cor: FJ coefficient well-defined}) and is invariant under ${\GJZbar}$. 
If we take the $(I, \omega_{J})$-trivialization for $I\subset J$, 
we can pass to a more familiar definition of the Fourier-Jacobi coefficient as a slice in the Fourier expansion at $I$. 

In general, we define vector-valued Jacobi forms as 
${\GJZbar}$-invariant sections of ${\Elk}\otimes \Theta_{J}^{\otimes m}$ over $\Delta_{J}$ with cusp condition (Definition \ref{def: Jacobi form}). 
Thus the Fourier-Jacobi coefficients are vector-valued Jacobi forms (Proposition \ref{prop: FJ expansion = Taylor expansion II}). 
Although our approach is geometric, our Jacobi forms in the scalar-valued case are indeed 
classical Jacobi forms in the sense of Skoruppa \cite{Sk} 
if we introduce suitable coordinates and the $(I, \omega_{J})$-trivialization (\S \ref{ssec: classical Jacobi form}). 
When $n=3$, our vector-valued Jacobi forms essentially agree with 
those considered by Ibukiyama-Kyomura \cite{IK} for Siegel modular forms of genus $2$. 

When $J$ comes from an integral embedding $2U\hookrightarrow L$ and ${\G}$ is the so-called stable orthogonal group, 
the Fourier-Jacobi expansion of scalar-valued modular forms is well-understood through the work of Gritsenko \cite{Gr}. 
A large part of this chapter can be regarded as a geometric reformulation and a generalization of the calculation in \cite{Gr} \S 2. 
A lot of effort will be paid for keeping introduction of coordinates as minimal as possible (though never zero), 
or in other words, for describing what is canonical in a canonical way. 
We believe that this style would be suitable even in the scalar-valued case when working with general $({\G}, J)$, 
for which simple expression by coordinates is no longer available.

\section{Fourier-Jacobi and Fourier expansion}\label{ssec: FJ expansion} 

We begin with the familiar (but non-canonical) way to define Fourier-Jacobi expansion: 
slicing the Fourier expansion. 
The passage to a canonical formulation will be given in \S \ref{ssec: geometry FJ}. 

We choose a rank $1$ primitive sublattice $I$ of $J$, and also a rank $1$ sublattice $I'\subset L$ with $(I, I')\ne 0$. 
Recall from \S \ref{sssec: Siegel vs tube} that ${\UJR}=\wedge^2J_{{\R}}$ is identified with the isotropic line 
$(J/I)_{{\R}}\otimes I_{{\R}}$ in ${\UIR}=(I^{\perp}/I)_{{\R}}\otimes I_{{\R}}$, 
and that the Siegel domain realization of ${\D}$ with respect to $J$ can be identified with the restriction of the projection 
\begin{equation*}
{\UIC} \to {\UIC}/{\UJC} \to {\UIC}/U(J)_{{\C}}^{\perp} 
\end{equation*}
to the tube domain ${\DI}\subset {\UIC}$ 
after the tube domain realization ${\D}\simeq {\DI}$. 
The orientation of $J$ determines the nonnegative part $\sigma_{J}=({\UJR})_{\geq 0}$ of ${\UJR}$. 
Let $v_{J,{\G}}$ be the positive generator of ${\UJZ}={\UJQ}\cap {\G}$. 
We choose a rational isotropic vector $l_{J,{\G}}\in {\UIQ}$ such that $(v_{J,{\G}}, l_{J,{\G}})=1$. 
Then $v_{J,{\G}}, l_{J,{\G}}$ span a rational hyperbolic plane in ${\UIQ}$. 
We put 
\begin{equation*}
\omega_{J} = q^{l_{J,{\G}}} = e((l_{J,{\G}}, Z)), \qquad Z\in {\UIC}. 
\end{equation*}
This is a holomorphic function on ${\UIC}$ invariant under the translation by ${\UJZ}$. 
Thus we have chosen the auxiliary datum $I$, $I'$, $l_{J,{\G}}$. 
These will be fixed until Lemma \ref{lem: Taylor expansion canonical}. 

Let $f$ be a ${\G}$-modular form of weight $(\lambda, k)$. 
We identify $f$ with a ${\VIlk}$-valued holomorphic function on ${\DI}$ via the $I$-trivialization and the tube domain realization, 
and let $f(Z)=\sum_{l}a(l)q^l$ be its Fourier expansion. 
Like the calculation in \S \ref{sssec: Fourier revisit} (see also Remark \ref{rmk: when l_0 not in UIZZ}),  
we can rewrite the Fourier expansion as 
\begin{equation}\label{eqn: FJ expansion I}
f(Z) = \sum_{m\geq 0} \left( \sum_{l\in U(J)_{{\Q}}^{\perp}} a(l+ml_{J,{\G}}) q^{l} \right) \omega_{J}^{m}. 
\end{equation}
Here $l$ ranges over vectors in $U(J)_{{\Q}}^{\perp}$ such that $l+ml_{J,{\G}}\in {\UIZZ}$. 
They form a translation of a full lattice in $U(J)_{{\Q}}^{\perp}$. 
Although $l_{J,{\G}}$ is not necessarily a vector in ${\UIZZ}$, this expression still makes sense over the tube domain ${\DI}$. 
We call \eqref{eqn: FJ expansion I} the \textit{Fourier-Jacobi expansion} of $f$ at the $J$-cusp relative to $I$, $I'$, $l_{J,{\G}}$,  
and usually write it as 
\begin{equation}\label{eqn: FJ expansion II}
f=\sum_{m\geq 0}\phi_{m}\omega_{J}^{m} 
\end{equation}
with 
\begin{equation}\label{eqn: FJ coefficient}
\phi_{m} = \sum_{l\in U(J)_{{\Q}}^{\perp}} a(l+ml_{J,{\G}})q^{l}. 
\end{equation}
We call $\phi_{m}$ the $m$-th \textit{Fourier-Jacobi coefficient} of $f$ at the $J$-cusp relative to $I$, $I'$, $l_{J,{\G}}$. 
This is a ${\VIlk}$-valued function on ${\DI}$. 
Since $l\in U(J)_{{\Q}}^{\perp}$ in \eqref{eqn: FJ coefficient}, 
$\phi_{m}$ actually descends to a ${\VIlk}$-valued function on $\Delta_{J}\subset {\UIC}/{\UJC}$. 
We often do not specify the precise index lattice in \eqref{eqn: FJ coefficient}; 
it is convenient to allow enlarging it as necessary by putting $a(l+ml_{J,{\G}})=0$ when $l+ml_{J,{\G}}\not\in {\UIZZ}$. 
When $m=0$, $\phi_{0}$ is the restriction of $f$ to $\Delta_{J}$ and was studied in \S \ref{sec: Siegel}. 
In this chapter we study the case $m>0$.

\section{Geometric approach to Fourier-Jacobi expansion}\label{ssec: geometry FJ}

In \S \ref{ssec: geometry FJ} and \S \ref{ssec: Jacobi form} 
we give a geometric reformulation of the Fourier-Jacobi expansion \eqref{eqn: FJ expansion II}. 
Our starting observation is (compare with \S \ref{sssec: Fourier revisit}): 

\begin{lemma}\label{lem: FJ = Taylor expansion function}
The Fourier-Jacobi expansion \eqref{eqn: FJ expansion II} 
gives the Taylor expansion of the ${\VIlk}$-valued holomorphic function $f$ on ${\XJcpt}$ 
along the boundary divisor $\Delta_{J}$ with respect to the normal parameter $\omega_{J}$, 
where $\phi_{m}$ is the $m$-th Taylor coefficient as a ${\VIlk}$-valued function on $\Delta_{J}$. 
\end{lemma}

\begin{proof}
Since the function $f$ is invariant under the translation by ${\UJZ}\subset {\UIZ}$, 
it descends to a function on ${\XJ}\simeq \mathcal{D}_{I}/{\UJZ}$. 
Since $(l_{J,{\G}}, v_{J,{\G}})=1$ for the positive generator $v_{J,{\G}}$ of ${\UJZ}$, 
the function $\omega_{J}=e((l_{J,{\G}}, Z))$ descends to a function on ${\XJ}$ 
and extends holomorphically over ${\XJcpt}$, with the boundary divisor $\Delta_{J}$ defined by $\omega_{J}=0$. 
In particular, $\omega_{J}$ generates the ideal sheaf of $\Delta_{J}$. 
On the other hand, as explained above, the Fourier-Jacobi coefficient $\phi_{m}$ is 
the pullback of a ${\VIlk}$-valued function on $\Delta_{J} \subset {\UIC}/{\UJC}$ (again denoted by $\phi_{m}$). 
Thus $f=\sum_{m}(\pi_{1}^{\ast}\phi_{m})\omega_{J}^{m}$ gives the Taylor expansion of $f$ along $\Delta_{J}$ 
with respect to the normal parameter $\omega_{J}$, 
in which the ${\VIlk}$-valued function $\phi_m$ on $\Delta_{J}$ is the $m$-th Taylor coefficient. 
\end{proof}
 
Recall from \S \ref{ssec: partial compact} that ${\XJcpt}$ is an open set of the relative torus embedding 
$\overline{\mathcal{T}(J)}=\mathcal{T}(J)\times_{T(J)}\overline{T(J)}$ 
which has the structure of a line bundle on $\Delta_{J}$. 
Since ${\D}(J)\subset Q(I) \simeq {\UIC}$, 
the function $\omega_{J}$ on ${\XJcpt}$ extends over $\overline{\mathcal{T}(J)}$ naturally. 
It is a linear map on each fiber of $\overline{\mathcal{T}(J)}\to \Delta_{J}$. 
Indeed, the fact that $\omega_{J}$ preserves the scalar multiplication follows from the equality 
\begin{equation*}
e((l_{J,{\G}}, \, \alpha v_{J,{\G}}+Z)) = e(\alpha) \cdot e((l_{J,{\G}}, Z)), \quad \alpha\in {\C}, 
\end{equation*}
and similarly for the sum. 
The following property will be used in \S \ref{ssec: Jacobi form}. 

\begin{lemma}\label{claim: omegaJ}
For each $\gamma \in {\GJZbar}$ we have 
$\gamma^{\ast}\omega_{J}=(\pi_{1}^{\ast}j_{\gamma}) \cdot \omega_{J}$ 
for a nowhere vanishing function $j_{\gamma}$ on $\Delta_{J}$. 
\end{lemma}

\begin{proof}
Since $\gamma$ acts on $\overline{\mathcal{T}(J)}\to \Delta_{J}$ as an equivariant action on the line bundle (see \S \ref{ssec: partial compact}), 
$\gamma^{\ast}\omega_{J}$ is also linear on each fiber. 
Therefore $\gamma^{\ast}\omega_{J}/\omega_{J}$ is the pullback of a function on $\Delta_{J}$. 
See also Corollary \ref{cor: f.a. omegaJ} for a computational proof. 
\end{proof}

Let us reformulate Lemma \ref{lem: FJ = Taylor expansion function} 
by passing from vector-valued functions to sections of vector bundles. 
Let $\mathcal{I}=\mathcal{I}_{\Delta_{J}}$ be the ideal sheaf of $\Delta_{J}$ and 
$\Theta_{J}=\mathcal{I}/\mathcal{I}^2$ be the conormal bundle of $\Delta_{J}$. 
As explained above, $\omega_{J}$ generates  $\mathcal{I}$ over ${\XJcpt}$. 
In particular, it generates $\Theta_{J}$ over $\Delta_{J}$. 
We have 
\begin{equation*}
\mathcal{I}^{m}/\mathcal{I}^{m+1} \simeq \Theta_{J}^{\otimes m} = \mathcal{O}_{\Delta_{J}}\omega_{J}^{\otimes m} 
\end{equation*}
for every $m\geq 0$. 
In what follows, we write ${\Elk}|_{\Delta_{J}}\otimes \Theta_{J}^{\otimes m}={\Elk}\otimes \Theta_{J}^{\otimes m}$ for simplicity. 
The $I$-trivialization ${\Elk}|_{\Delta_{J}}\simeq {\VIlk}\otimes \mathcal{O}_{\Delta_{J}}$ of ${\Elk}|_{\Delta_{J}}$ 
and the trivialization of $\Theta_{J}^{\otimes m}$ by $\omega_{J}^{\otimes m}$ define an isomorphism 
${\Elk}\otimes \Theta_{J}^{\otimes m}\simeq {\VIlk}\otimes \mathcal{O}_{\Delta_{J}}$. 
We call it the \textit{$(I, \omega_{J})$-trivialization} of ${\Elk}\otimes \Theta_{J}^{\otimes m}$. 
Via this isomorphism, we regard the ${\VIlk}$-valued function $\phi_{m}$ over $\Delta_{J}$ 
as a section of ${\Elk}\otimes \Theta_{J}^{\otimes m}$ over $\Delta_{J}$. 
Specifically, the process is to multiply the function $\phi_{m}$ by $\omega_{J}^{\otimes m}$, and then 
regard $\phi_{m}\otimes \omega_{J}^{\otimes m}$ as a section of ${\Elk}\otimes \Theta_{J}^{\otimes m}$ by the $I$-trivialization.  

\begin{proposition}\label{prop: FJ expansion = Taylor expansion I}
The Taylor expansion of sections of ${\Elk}$ over ${\XJcpt}$ along the boundary divisor $\Delta_{J}$ 
with respect to the normal parameter $\omega_{J}$ and with the pullback 
$\pi_{1}^{\ast}\colon H^0(\Delta_{J}, {\Elk}|_{\Delta_{J}})\hookrightarrow H^0({\XJcpt}, {\Elk})$ 
defines an embedding 
\begin{equation}\label{eqn: FJ expansion = Taylor expansion I}
H^0({\XJcpt}, {\Elk}) \hookrightarrow \prod_{m\geq 0} H^0(\Delta_{J}, {\Elk}\otimes \Theta_{J}^{\otimes m}), \quad 
f\mapsto (\phi_{m}\otimes \omega_{J}^{\otimes m})_{m},  
\end{equation}
where $\phi_{m}$ are the sections of ${\Elk}|_{\Delta_{J}}$ with $f=\sum_{m}(\pi_{1}^{\ast}\phi_{m}) \omega_{J}^m$. 
If we send a modular form $f\in M_{\lambda,k}({\G})$ as a section of ${\Elk}$ by this map, 
its image is the Fourier-Jacobi coefficients of $f$ regarded as sections of ${\Elk}\otimes \Theta_{J}^{\otimes m}$ 
via the $(I, \omega_{J})$-trivialization. 
\end{proposition}

Here the pullback $\pi_{1}^{\ast} \colon H^0(\Delta_{J}, {\Elk}|_{\Delta_{J}})\to H^0({\XJcpt}, {\Elk})$ is defined by 
the isomorphism ${\Elk}\simeq \pi_{1}^{\ast}({\Elk}|_{\Delta_{J}})$ in Proposition \ref{prop: L and E as GJR bundle}. 
Via the $I$-trivialization, this is just the pullback of ${\VIlk}$-valued functions by $\pi_{1}\colon {\XJcpt}\to \Delta_{J}$ 
(see Remark \ref{rmk: pullback and I-trivialization}). 
The existence of this pullback map is one of key properties in the Fourier-Jacobi expansion. 

\begin{proof}
The exact sequence of sheaves 
\begin{equation*}
0 \to \mathcal{I}^{m+1}{\Elk} \to  \mathcal{I}^{m}{\Elk} \to {\Elk}\otimes \Theta_{J}^{\otimes m} \to 0 
\end{equation*}
on ${\XJcpt}$ defines the canonical exact sequence 
\begin{equation}\label{eqn: ElkThetaJm}
0 \to H^{0}({\XJcpt}, \, \mathcal{I}^{m+1}{\Elk}) \to H^{0}({\XJcpt}, \, \mathcal{I}^{m}{\Elk}) 
\to H^0(\Delta_{J}, \, {\Elk}\otimes \Theta_{J}^{\otimes m}). 
\end{equation}
The generator $\omega_{J}^{m}$ of $\mathcal{I}^{m}$ and the pullback map 
\begin{equation*}
\pi_{1}^{\ast}\colon H^0(\Delta_{J}, \, {\Elk}|_{\Delta_{J}})\to H^0({\XJcpt}, \, {\Elk}) 
\end{equation*} 
define the splitting map 
\begin{equation}\label{eqn: split Taylor expansion}
H^0(\Delta_{J}, \, {\Elk}\otimes \Theta_{J}^{\otimes m}) \hookrightarrow H^{0}({\XJcpt}, \, \mathcal{I}^{m}{\Elk}), \quad 
\phi\otimes \omega_{J}^{\otimes m}\mapsto \omega_{J}^{m}\cdot \pi_{1}^{\ast}\phi 
\end{equation}
of \eqref{eqn: ElkThetaJm}. 
Here $\omega_{J}^{\otimes m}$ in the source is a section of $\Theta_{J}^{\otimes m}$ over $\Delta_{J}$, 
while $\omega_{J}^{m}$ in the target is a section of $\mathcal{I}^{m}$ over ${\XJcpt}$.  
This defines a splitting of the filtration $(H^{0}({\XJcpt}, \mathcal{I}^{m}{\Elk}))_{m}$ on $H^{0}({\XJcpt}, {\Elk})$ and thus an embedding 
\begin{equation*}
H^0({\XJcpt}, {\Elk}) \hookrightarrow \prod_{m\geq 0} H^0(\Delta_{J}, {\Elk}\otimes \Theta_{J}^{\otimes m}).  
\end{equation*}
Explicitly, this is given by writing a section $f$ of ${\Elk}$ over ${\XJcpt}$ as 
$f=\sum_{m}(\pi_{1}^{\ast}\phi_{m})\omega_{J}^{m}$ with $\phi_{m}$ a section of ${\Elk}|_{\Delta_{J}}$, 
and sending $f$ to the collection $(\phi_{m}\otimes \omega_{J}^{\otimes m})_{m}$ of sections.  

Since $\pi_{1}^{\ast}$ is just the ordinary pullback after the $I$-trivialization, 
the equation $f=\sum_{m}(\pi_{1}^{\ast}\phi_{m})\omega_{J}^{m}$ when $f$ is a modular form 
coincides with the Fourier-Jacobi expansion \eqref{eqn: FJ expansion II} of $f$ after the $I$-trivialization. 
Thus the $I$-trivialization of $\phi_{m}$ is the $m$-th Fourier-Jacobi coefficient \eqref{eqn: FJ coefficient}. 
It follows that the section $\phi_{m}\otimes \omega_{J}^{\otimes m}$ is identified with 
the Fourier-Jacobi coefficient by the $(I, \omega_{J})$-trivialization. 
\end{proof} 
 
At first glance, the Taylor expansion \eqref{eqn: FJ expansion = Taylor expansion I} may seem non-canonical 
because the lifting map \eqref{eqn: split Taylor expansion} uses the special normal parameter $\omega_{J}$, 
which as a function on ${\XJcpt}$ depends on the choice of $l_{J,{\G}}, I', I$. 
In fact, it \textit{is} canonical: 

\begin{lemma}\label{lem: Taylor expansion canonical}
The map \eqref{eqn: split Taylor expansion}, and hence the Taylor expansion \eqref{eqn: FJ expansion = Taylor expansion I}, 
does not depend on the choice of $l_{J,{\G}}, I', I$. 
\end{lemma}

\begin{proof}
Let $\tilde{\omega}_{J}$ be the special normal parameter constructed from another such data $(\tilde{I}, \tilde{I}', \tilde{l}_{J,{\G}})$. 
Both $\omega_{J}$ and $\tilde{\omega}_{J}$ extend over $\overline{\mathcal{T}(J)}$ 
and are linear at each fiber of 
$\pi_{1}\colon \overline{\mathcal{T}(J)} \to \Delta_{J}$. 
Therefore we have 
$\tilde{\omega}_{J}/\omega_{J}=\pi_{1}^{\ast}\xi$ 
for a nowhere vanishing holomorphic function $\xi$ on $\Delta_{J}$. 
Then the map \eqref{eqn: split Taylor expansion} defined by using $\tilde{\omega}_{J}$ in place of $\omega_{J}$ 
sends $\phi\otimes \omega_{J}^{\otimes m}$ as 
\begin{equation*}
\phi\otimes \omega_{J}^{\otimes m} = (\xi^{-m} \phi) \otimes \tilde{\omega}_{J}^{\otimes m} \: \mapsto \: 
\tilde{\omega}_{J}^{m} \cdot \pi_{1}^{\ast}(\xi^{-m} \phi) = \omega_{J}^{m}\cdot \pi_{1}^{\ast} \phi. 
\end{equation*}
This coincides with the map using $\omega_{J}$. 
\end{proof}

This in particular implies the following. 

\begin{corollary}\label{cor: FJ coefficient well-defined}
The $m$-th Fourier-Jacobi coefficient of a modular form, 
viewed as a section of ${\Elk}\otimes \Theta_{J}^{\otimes m}$ over $\Delta_{J}$ via the $(I, \omega_{J})$-trivialization, 
does not depend on the choice of $l_{J,{\G}}, I', I$. 
\end{corollary}

This means that we obtain the same section of ${\Elk}\otimes \Theta_{J}^{\otimes m}$ 
even if we start from the Fourier expansion at another $0$-dimensional cusp $\tilde{I}\subset J$. 

To summarize, the Fourier-Jacobi expansion of a modular form $f$ as a section of ${\Elk}$ 
is a canonical Taylor expansion along $\Delta_{J}$ 
which uses but does not depend on the choice of a special normal parameter $\omega_{J}$. 
The $m$-th Fourier-Jacobi coefficient is canonically determined as a section of ${\Elk}\otimes \Theta_{J}^{\otimes m}$. 
If we take the $(I, \omega_{J})$-trivialization, this section is identified with 
the ${\VIlk}$-valued function  \eqref{eqn: FJ coefficient} defined as a slice in the Fourier expansion of $f$ at the $I$-cusp.

\section{Vector-valued Jacobi forms}\label{ssec: Jacobi form}

We want to refine Proposition \ref{prop: FJ expansion = Taylor expansion I} 
by taking the invariant part for the integral Jacobi group ${\GJZ}$ and imposing cusp condition.  
This leads us to define vector-valued Jacobi forms in a geometric style. 
In what follows, we let $m>0$ and consider the vector bundle ${\Elk}\otimes \Theta_{J}^{\otimes m}$ over $\Delta_{J}$, 
leaving modular forms on ${\D}$ for a while. 

As in \S \ref{ssec: FJ expansion} and \S \ref{ssec: geometry FJ}, 
we choose a rank $1$ primitive sublattice $I$ of $J$, 
a rank $1$ sublattice $I'\subset L$ with $(I, I')\ne 0$, 
and an isotropic vector $l_{J,{\G}}\in {\UIQ}$ with $(l_{J,{\G}}, v_{J, {\G}})=1$. 
($I$ will be fixed until Definition \ref{def: Jacobi form}, and $I', l_{J,{\G}}$ will be fixed until Lemma \ref{lem: Jacobi form cusp well-defined}.) 
We keep the same notation as in \S \ref{ssec: FJ expansion}. 
Since ${\UIZ}\subset {\GJZ}$ by \eqref{eqn: U(J) U(I) G(J)}, 
the group ${\GJRbar}$ contains ${\UIZ}/{\UJZ}$ as a subgroup. 
As recalled in \S \ref{ssec: FJ expansion}, 
$I'$ determines an embedding  
$\Delta_{J}\hookrightarrow {\UIC}/{\UJC}$. 
The action of ${\UIZ}/{\UJZ}$ on $\Delta_{J}$ is given by the translation on ${\UIC}/{\UJC}$. 

We consider the action of ${\UIZ}/{\UJZ}$ on the vector bundle ${\Elk}\otimes \Theta_{J}^{\otimes m}$. 
The $I$-trivialization ${\Elk}|_{\Delta_{J}}\simeq {\VIlk}\otimes \mathcal{O}_{\Delta_{J}}$ over $\Delta_{J}$ 
is equivariant with respect to the subgroup $({\GIR}\cap {\GJR})/{\UJZ}$ of ${\GJRbar}$. 
In particular, it is equivariant with respect to ${\UIZ}/{\UJZ}$. 
Since ${\UIZ}/{\UJZ}$ acts trivially on ${\VIlk}$, 
the factor of automorphy for the $I$-trivialization of ${\Elk}|_{\Delta_{J}}$ is trivial on this group. 
On the other hand, as for the $\omega_{J}$-trivialization of $\Theta_{J}$, we note the following. 

\begin{lemma}\label{lem: UIZ omegaJ}
There exists a finite-index sublattice $\Lambda_{0}$ of ${\UIZ}/{\UJZ}$ such that 
$\gamma^{\ast}\omega_{J}=\omega_{J}$ for every $\gamma \in \Lambda_{0}$. 
In particular, the factor of automorphy for the $(I, \omega_{J})$-trivialization 
${\Elk}\otimes \Theta_{J}^{\otimes m}\simeq {\VIlk}\otimes \mathcal{O}_{\Delta_{J}}$ 
of ${\Elk}\otimes \Theta_{J}^{\otimes m}$ 
is trivial on the group $\Lambda_{0}$. 
\end{lemma}

\begin{proof}
Recall that $v\in {\UIR}$ acts on the tube domain ${\DI}\simeq {\D}$ as the translation by $v$, say $t_{v}$. 
Then 
\begin{equation*}
t_{v}^{\ast}\omega_{J} = e((l_{J,{\G}}, Z+v)) = e((l_{J,{\G}}, v))\cdot \omega_{J}. 
\end{equation*}
Therefore, if we put 
\begin{equation}\label{eqn: Lambda}
\Lambda_{0} = \{ \: v\in {\UIZ}/{\UJZ} \: | \: (l_{J,{\G}}, v+{\UJZ})\subset {\Z} \: \}, 
\end{equation}
we have $t_{v}^{\ast}\omega_{J}=\omega_{J}$ for every $v\in \Lambda_{0}$. 
Since $({\UIZ}, l_{J,{\G}})\subset {\Q}$ and ${\UIZ}$ is finitely generated, 
we have $({\UIZ}, l_{J,{\G}})\subset N^{-1}{\Z}$ for some natural number $N$. 
This shows that $\Lambda_{0}$ is of finite index in ${\UIZ}/{\UJZ}$.   
\end{proof}

Let $\phi$ be a ${\GJZbar}$-invariant section of ${\Elk}\otimes \Theta_{J}^{\otimes m}$ over $\Delta_{J}$. 
By the $(I, \omega_{J})$-trivialization of ${\Elk}\otimes \Theta_{J}^{\otimes m}$, 
we regard $\phi$ as a ${\VIlk}$-valued holomorphic function on $\Delta_{J}$. 
By Lemma \ref{lem: UIZ omegaJ}, the function $\phi$ is invariant under the translation by the lattice $\Lambda_{0}$. 
Therefore it admits a Fourier expansion of the form 
\begin{equation}\label{eqn: Fourier expansion Jacobi form}
\phi(Z) = \sum_{l\in \Lambda} a(l) q^{l}, \qquad Z \in \Delta_{J} \subset {\UIC}/{\UJC}, 
\end{equation}
where $a(l)\in {\VIlk}$, $q^{l}=e((l, Z))$, and $\Lambda$ is a full lattice in $U(J)_{{\Q}}^{\perp}$ 
(which is the dual space of ${\UIQ}/{\UJQ}$). 

At this point, $\Lambda$ can be taken to be the dual lattice of $\Lambda_{0}$ defined by \eqref{eqn: Lambda}, 
but we can replace $\Lambda$ by its arbitrary overlattice (or even the whole $U(J)_{{\Q}}^{\perp}$) 
by setting $a(l)=0$ if $l\not\in \Lambda_{0}^{\vee}$. 
It is sometimes convenient to enlarge $\Lambda$ in this way. 
For this reason, we do not specify the lattice $\Lambda$ in \eqref{eqn: Fourier expansion Jacobi form}. 

\begin{remark}
The dual lattice of $\Lambda_{0}$ in $U(J)_{{\Q}}^{\perp}$ can be explicitly written as 
\begin{equation*}\label{eqn: Lambda initial}
\Lambda_{0}^{\vee} = \langle {\UIZZ}, \, {\Z}l_{J,{\G}} \rangle \cap U(J)_{{\Q}}^{\perp}. 
\end{equation*}
We do not use this information. 
\end{remark}

Replacing $\Lambda$ by its overlattice, we assume that $\Lambda$ is of the split form 
\begin{equation*}
\Lambda = {\Z} (\beta_{1}v_{J,{\G}}) \oplus K, 
\end{equation*}
where $\beta_{1}>0$ is a rational number and $K$ is a full lattice in $l_{J,{\G}}^{\perp}\cap U(J)_{{\Q}}^{\perp}$. 
Note that $K$ is negative-definite. 
Accordingly, we can rewrite the Fourier expansion of $\phi$ as 
\begin{equation}\label{eqn: Fourier expansion Jacobi form I}
\phi(Z) = \sum_{n\in \beta_1{\Z}} \sum_{l\in K} a(n, l)  q^{l} q_{J,{\G}}^{n}, 
\qquad q_{J,{\G}}=e((v_{J,{\G}}, Z)), 
\end{equation}
for $Z \in \Delta_{J} \subset {\UIC}/{\UJC}$. 

\begin{definition}\label{def: Jacobi form cusp condition}
We say that $\phi$ is \textit{holomorphic at the $I$-cusp} of ${\HJ}$ if $a(n, l)\ne 0$ only when $2nm\geq |(l, l)|$. 
We say that $\phi$ \textit{vanishes at the $I$-cusp} if $a(n, l)\ne 0$ only when $2nm> |(l, l)|$. 
\end{definition}

The expression \eqref{eqn: Fourier expansion Jacobi form I} of the Fourier expansion of $\phi$ 
depends on the choice of $I'$, $l_{J,{\G}}$, $\Lambda$. 
Specifically, 
\begin{itemize}
\item $I'$ determines the embedding $\Delta_{J} \hookrightarrow {\UIC}/{\UJC}$. 
\item $l_{J,{\G}}$ determines the normal parameter $\omega_{J}$ which determines the trivialization of $\Theta_{J}^{\otimes m}$. 
The vector $l_{J,{\G}}$ also determines the splitting $U(J)_{{\Q}}^{\perp}={\UJQ}\oplus K_{{\Q}}$ of the index space $U(J)_{{\Q}}^{\perp}$. 
\item $\Lambda$ is the index lattice in the Fourier expansion which is taken to be a split form. 
\end{itemize}
However, we can prove the following. 
 
\begin{lemma}\label{lem: Jacobi form cusp well-defined} 
Definition \ref{def: Jacobi form cusp condition} does not depend on the choice of $I'$, $l_{J,{\G}}$, $\Lambda$. 
\end{lemma} 

\begin{proof}
We verify this for the holomorphicity condition. 
The case of vanishing condition is similar. 

(1) 
If we change $I'$, its effect is the translation on $\Delta_{J}\subset {\UIC}/{\UJC}$ by a vector of ${\UIQ}/{\UJQ}$. 
This multiplies each Fourier coefficient $a(n, l)$ by a nonzero constant, 
so its vanishing/nonvanishing  does not change. 

(2) 
The condition $2nm \geq |(l, l)|$ is the same as the condition 
\begin{equation}\label{eqn: Jacobi form holomorphic positive index}
(ml_{J,{\G}}+v, \: ml_{J,{\G}}+v) \geq 0 
\end{equation} 
for the vector $v=nv_{J,{\G}}+l$ of $U(J)_{{\Q}}^{\perp}$ which corresponds to the index $(n, l)$. 
With $l_{J,{\G}}$ fixed, this condition does not depend on the lattice $\Lambda$. 

(3) 
Finally, if we change $l_{J,{\G}}$, the new vector can be written as 
\begin{equation*}
l_{J,{\G}}' = l_{J,{\G}} + l_{0} - 2^{-1} (l_{0}, l_{0})v_{J,{\G}} 
\end{equation*}
for some vector $l_{0}\in K_{{\Q}}$. 
Since the normal parameter $\omega_{J}=e((l_{J,{\G}}, Z))$ is replaced by 
\begin{equation*}
\omega_{J}' = e((l_{J,{\G}}', Z)) = q^{l_{0}} \cdot q_{J,{\G}}^{-(l_{0}, l_{0})/2} \cdot \omega_{J}, 
\end{equation*} 
we have to multiply the function $\phi$ by $q^{-ml_{0}} \cdot q_{J,{\G}}^{m(l_{0}, l_{0})/2}$ 
when passing from the $\omega_{J}$-trivialization to the $\omega_{J}'$-trivialization of $\Theta_{J}^{\otimes m}$. 
Also $K_{{\Q}}=l_{J,{\G}}^{\perp}\cap U(J)_{{\Q}}^{\perp}$ is replaced by 
$K_{{\Q}}'=(l_{J,{\G}}')^{\perp}\cap U(J)_{{\Q}}^{\perp}$, 
for which we have the natural isometry 
\begin{equation*}
K_{{\Q}}\to K_{{\Q}}', \qquad l\mapsto l':=l-(l, l_{0})v_{J,{\G}}. 
\end{equation*}
Therefore the new Fourier expansion is  
\begin{eqnarray*}
\phi' & := & \phi \cdot q^{-ml_{0}} \cdot q_{J,{\G}}^{m(l_{0}, l_{0})/2} \\ 
& = & \sum_{n\in {\Q}} \sum_{l\in K_{{\Q}}} a(n, l)q^{l-ml_{0}} q_{J,{\G}}^{n+m(l_{0}, l_{0})/2} \\ 
& = & \sum_{n\in {\Q}} \sum_{l\in K_{{\Q}}} a(n, l)q^{l'-ml_{0}'} q_{J,{\G}}^{n+(l, l_{0})-m(l_{0}, l_{0})/2}. 
\end{eqnarray*}
In the last equality we used 
\begin{equation*}
l-ml_{0} = (l-ml_{0})' + (l-ml_{0}, l_{0}) v_{J,{\G}}. 
\end{equation*}
This means that $a(n, l)$ is equal to the Fourier coefficient of $\phi'$ of index 
\begin{equation*}
(n+(l, l_0)-m(l_0, l_0)/2, \: l'-ml_{0}') \: \: \: \in {\Q}\oplus K_{{\Q}}'. 
\end{equation*}
The holomorphicity condition $2nm \geq -(l, l)$ for $\phi$ can be rewritten as 
\begin{equation*}
2m(n+(l, l_{0})-m(l_{0}, l_{0})/2) \geq -(l'-ml_{0}', l'-ml_{0}'). 
\end{equation*}
This is the holomorphicity condition for $\phi'$. 
\end{proof}

Lemma \ref{lem: Jacobi form cusp well-defined} ensures that Definition \ref{def: Jacobi form cusp condition} is well-defined 
for a ${\GJZbar}$-invariant section of ${\Elk}\otimes \Theta_{J}^{\otimes m}$. 

\begin{definition}\label{def: Jacobi form}
We denote by 
\begin{equation*}
J_{\lambda, k, m}({\GJZ}) \; \subset \; H^{0}(\Delta_{J}, {\Elk}\otimes \Theta_{J}^{\otimes m}) 
\end{equation*}
the space of ${\GJZbar}$-invariant sections $\phi$ of ${\Elk}\otimes \Theta_{J}^{\otimes m}$ over $\Delta_{J}$ 
which are holomorphic at every cusp $I\subset J$ of ${\HJ}$ in the sense of Definition \ref{def: Jacobi form cusp condition}. 
We call such a section $\phi$ a \textit{Jacobi form} of weight $(\lambda, k)$ and index $m$ 
for the integral Jacobi group ${\GJZ}$. 
We call $\phi$ a \textit{Jacobi cusp form} if it vanishes at every cusp $I\subset J$. 
When $\lambda=0$, we especially write $J_{0,k,m}({\GJZ})=J_{k,m}({\GJZ})$. 
\end{definition}

For later use (\S \ref{ssec: classical Jacobi form}), we note the following. 

\begin{lemma}\label{lem: holomorphic at gamma(I)}
Let $\gamma$ be an element of ${\GJQ}$ which stabilizes $J$. 
A ${\GJZbar}$-invariant section $\phi$ of ${\Elk}\otimes \Theta_{J}^{\otimes m}$ over $\Delta_{J}$ 
is holomorphic at the $\gamma(I)$-cusp of ${\HJ}$ if and only if 
the $\gamma^{-1} {\GJZbar}\gamma$-invariant section $\gamma^{\ast}\phi$ of ${\Elk}\otimes \Theta_{J}^{\otimes m}$ 
is holomorphic at the $I$-cusp of ${\HJ}$. 
\end{lemma}

\begin{proof}
This holds because the pullback of a Fourier expansion of $\phi$ at the $\gamma(I)$-cusp by the $\gamma$-action 
\begin{equation*}
\gamma : {\UIC}/{\UJC} \to U(\gamma I)_{{\C}}/{\UJC} 
\end{equation*}
and the isomorphism 
$\gamma\colon {\VIlk}\to V(\gamma I)_{\lambda,k}$ 
gives a Fourier expansion of $\gamma^{\ast}\phi$ at the $I$-cusp. 
\end{proof}

Now we go back to modular forms on ${\D}$ and refine Proposition \ref{prop: FJ expansion = Taylor expansion I} for ${\MG}$. 
Recall that the $m$-th Fourier-Jacobi coefficient of a modular form was initially defined 
as a ${\VIlk}$-valued function on $\Delta_{J}$ by \eqref{eqn: FJ coefficient}, 
and then regarded as a section of ${\Elk}\otimes \Theta_{J}^{\otimes m}$ by the $(I, \omega_{J})$-trivialization. 
By Corollary \ref{cor: FJ coefficient well-defined}, this section is independent of $I$. 

\begin{proposition}\label{prop: FJ expansion = Taylor expansion II}
For $m>0$ the $m$-th Fourier-Jacobi coefficient of a modular form $f\in {\MG}$ 
as a section of ${\Elk}\otimes \Theta_{J}^{\otimes m}$ is a Jacobi form of weight $(\lambda, k)$ and index $m$ 
in the sense of Definition \ref{def: Jacobi form}. 
When $f$ is a cusp form, the Fourier-Jacobi coefficient is a Jacobi cusp form.  
\end{proposition}

\begin{proof}
In what follows, $\tilde{\phi}_{m}$ stands for the $m$-th Fourier-Jacobi coefficient of $f$ 
as a section of ${\Elk}\otimes \Theta_{J}^{\otimes m}$. 
What has to be shown is that $\tilde{\phi}_{m}$ is ${\GJZbar}$-invariant and is holomorphic at every cusp of ${\HJ}$. 
We first check the cusp condition. 
Let $I\subset J$ be an arbitrary cusp (not necessarily the initial one). 
Corollary \ref{cor: FJ coefficient well-defined} ensures that 
the Fourier expansion of $\tilde{\phi}_{m}$ at the $I$-cusp of ${\HJ}$ is given by the series \eqref{eqn: FJ coefficient} 
obtained from the Fourier expansion of $f$ at the $I$-cusp of ${\D}$. 
Then the holomorphicity condition for $\tilde{\phi}_{m}$ at $I$, written in the form \eqref{eqn: Jacobi form holomorphic positive index}, 
follows from the cusp condition in the Fourier expansion of $f$ at $I$. 
The assertion for cusp forms follows similarly. 

It remains to check the ${\GJZbar}$-invariance of $\tilde{\phi}_{m}$. 
Let $\phi_{m}=\tilde{\phi}_{m}\otimes \omega_{J}^{\otimes -m}$. 
This is a section of ${\Elk}|_{\Delta_{J}}$ whose $I$-trivialization is 
the $(I, \omega_{J})$-trivialized form \eqref{eqn: FJ coefficient} of $\tilde{\phi}_{m}$. 
By Proposition \ref{prop: FJ expansion = Taylor expansion I}, we have the expansion 
\begin{equation}\label{eqn: FJ expansion in proof I}
f=\sum_{m}(\pi_{1}^{\ast}\phi_{m})\omega_{J}^{m} 
\end{equation}
as a section of ${\Elk}$, 
where we view $\omega_{J}$ as a generator of the ideal sheaf $\mathcal{I}$ of $\Delta_{J}$. 
We let $\gamma\in {\GJZbar}$ act on this equality. 
Then we have  
\begin{equation}\label{eqn: invariance of FJ coefficient proof}
\gamma^{\ast}f  =  \sum_{m} \gamma^{\ast}(\pi_{1}^{\ast}\phi_{m})(\gamma^{\ast}\omega_{J})^{m}  
 =  \sum_{m} \pi_{1}^{\ast}(\gamma^{\ast}\phi_{m})(\gamma^{\ast}\omega_{J})^{m} 
\end{equation}
by Proposition \ref{prop: L and E as GJR bundle}. 
By Lemma \ref{claim: omegaJ}, we have $\gamma^{\ast}\omega_{J}=(\pi_{1}^{\ast}j_{\gamma}) \cdot \omega_{J}$ 
for a holomorphic function $j_{\gamma}$ on $\Delta_{J}$. 
Therefore we have 
\begin{equation}\label{eqn: FJ expansion in proof II}
\gamma^{\ast}f = \sum_{m} \pi_{1}^{\ast}(j_{\gamma}^{m} \cdot \gamma^{\ast}\phi_{m})\omega_{J}^{m}. 
\end{equation}
Since $f$ is ${\G}$-invariant, we have $\gamma^{\ast}f=f$. 
Comparing \eqref{eqn: FJ expansion in proof I} and \eqref{eqn: FJ expansion in proof II}, 
we obtain 
$\phi_{m}=j_{\gamma}^{m}\cdot \gamma^{\ast}\phi_{m}$ for every $m$. 
This means that $\tilde{\phi}_{m}=\phi_{m}\otimes \omega_{J}^{\otimes m}$ is $\gamma$-invariant. 
This proves Proposition \ref{prop: FJ expansion = Taylor expansion II}. 
\end{proof}

For the sake of completeness, for $m=0$ with $\lambda\ne \det$, 
let us denote by $J_{\lambda,k,0}({\GJZ})$ the space of ${\GJZbar}$-invariant sections of 
${\ElJ}\otimes {\LL}^{\otimes k}|_{\Delta_{J}} \simeq \pi_{2}^{\ast}\mathcal{L}_{J}^{\otimes k+\lambda_{1}}\otimes V(J)_{\lambda'}$ 
over $\Delta_{J}$ which is holomorphic at every cusp of ${\HJ}$. 
By the result of \S \ref{sec: Siegel}, the $0$-th Fourier-Jacobi coefficient 
$\phi_{0}=f|_{\Delta_{J}}$ of a modular form $f\in {\MG}$ belongs to this space (cuspidal when $\lambda \ne 1$). 
Then, as a refinement of Proposition \ref{prop: FJ expansion = Taylor expansion I} for ${\MG}$, 
we see that the Fourier-Jacobi expansion gives the embedding 
\begin{equation*}
M_{\lambda,k}({\G}) \hookrightarrow \prod_{m\geq 0}J_{\lambda, k, m}({\GJZ}), \quad 
f=\sum_{m} (\pi_{1}^{\ast}\phi_{m})\omega_{J}^{m} \mapsto (\phi_{m}\otimes \omega_{J}^{m})_{m}, 
\end{equation*}
which is canonically determined by $J$.

\section{Classical Jacobi forms}\label{ssec: classical Jacobi form}

In this section we introduce coordinates and translate Jacobi forms with $\lambda=0$ in the sense of Definition \ref{def: Jacobi form} 
to classical scalar-valued Jacobi forms \`a la \cite{Gr} and \cite{Sk}. 
The result is stated in Proposition \ref{prop: translate Jacobi form}. 
Our purpose is to deduce a vanishing theorem in the present setting (Proposition \ref{prop: Jacobi form vanish}) 
from the one for classical Jacobi forms.

\subsection{Coordinates}\label{sssec: Siegel domain coordinate}

We begin by setting some notations. 
In ${\UJQ}\simeq \wedge^{2}J_{{\Q}}$ we have two natural lattices: $\wedge^{2}J$ and ${\UJZ}$. 
The former depends on $L$, and the latter depends on ${\G}$. 
Recall that the positive generator of ${\UJZ}$ is denoted by $v_{J,{\G}}$ (\S \ref{ssec: FJ expansion}),  
and the positive generator of $\wedge^{2}J$ is denoted by $v_{J}$ (\S \ref{sssec: Siegel vs tube}). 
Then $v_{J}=\beta_{0}v_{J,{\G}}$ for some rational number $\beta_{0}>0$. 
This constant $\beta_{0}$ depends only on $L$ and ${\G}$. 
We choose an isotropic plane in $L_{{\Q}}$ whose pairing with $J_{{\Q}}$ is nondegenerate, 
and denote it by $J_{{\Q}}^{\vee}$ for the obvious reason. 
This is fixed throughout \S \ref{ssec: classical Jacobi form}. 
We identify $V(J)_{{\Q}}=(J^{\perp}/J)_{{\Q}}$ with 
the subspace $(J_{{\Q}}\oplus J_{{\Q}}^{\vee})^{\perp}$ of $L_{{\Q}}$. 

Next we choose a rank $1$ primitive sublattice $I$ of $J$. 
Let $e_{1}, f_{1}, e_{2}, f_{2}$ be the standard hyperbolic basis of $2U$. 
We take an embedding $2U_{{\Q}}\hookrightarrow L_{{\Q}}$ which sends 
\begin{equation*}
{\Z}e_{1}\oplus {\Z}e_{2} \to J, \quad {\Z}e_{1} \to I, \quad {\Q}f_{1}\oplus {\Q}f_{2} \to J_{{\Q}}^{\vee} 
\end{equation*}  
isomorphically. 
Thus it is compatible with $I\subset J$ in the sense of \S \ref{sec: cano exte}. 
We identify $e_{1}, f_{1}, e_{2}, f_{2}$ with their image in $L_{{\Q}}$. 
Then $v_{J}=e_{2}\otimes e_{1}$. 
We define vectors $l_{J}, l_{J,{\G}}\in {\UIQ}$ (as in \S \ref{sssec: Siegel vs tube} and \S \ref{ssec: FJ expansion}) by 
$l_{J}=f_{2}\otimes e_{1}$ and $l_{J,{\G}}=\beta_{0}l_{J}$. 
We also put $I'={\Z}f_{1}$. 
The choice of these data has two effects: 
it introduces coordinates on ${\D}$ and on the Jacobi group. 

The coordinates on ${\D}$ are introduced following \S \ref{sssec: Siegel vs tube}. 
The choice of $I'={\Z}f_{1}$ defines the tube domain realization ${\D}\to {\DI}\subset {\UIC}$. 
According to the decomposition  
\begin{equation*}
{\UIC} = (U_{{\C}}\oplus V(J)_{{\C}}) \otimes I_{{\C}} = {\C}l_{J} \times (V(J)\otimes {\C}e_{1}) \times {\C}v_{J}, 
\end{equation*}
we express a point of ${\UIC}$ as 
\begin{equation*}
Z=\tau l_{J}+z\otimes e_{1} + w v_{J} = (\tau, z, w), \qquad \tau, w\in {\C}, \; z\in V(J)_{{\C}}. 
\end{equation*} 
These are the same coordinates as in \eqref{eqn: coordinate tube Siegel} except that 
$z$ in \eqref{eqn: coordinate tube Siegel} is $z\otimes e_{1}$ here. 
When $Z\in {\DI}$, the corresponding point of ${\D}$ is ${\C}\omega(Z)$ where 
\begin{equation}\label{eqn: omega(Z)}
\omega(Z) = f_{1} + \tau f_{2} + z + w e_{2} - ((z, z)/2 + \tau w)e_{1} \; \; \in L_{{\C}}. 
\end{equation}
Note that this vector is normalized so as to have pairing $1$ with $e_{1}$. 
In this coordinates, the Siegel domain realization ${\D}\to{\VJ}\to{\HJ}$ with respect to $J$ is the restriction of the projection 
\begin{equation*}
{\C}l_{J}\times V(J) \times {\C}v_{J}  \to {\C}l_{J}\times V(J) \to {\C}l_{J}, 
\qquad (\tau, z, w)\mapsto (\tau, z)\mapsto \tau 
\end{equation*}
to the tube domain ${\DI}$. 
The coordinates introduced on ${\HJ}\subset {\proj}(L/J^{\perp})_{{\C}}$ and ${\VJ}\subset {\proj}(L/J)_{{\C}}$ are written as 
\begin{equation}\label{eqn: HJ coordinate}
{\HH} \stackrel{\simeq}{\to} {\HJ}, \quad \tau \mapsto \tau l_{J}={\C}(f_{1}+\tau f_{2}), 
\end{equation}
\begin{equation}\label{eqn: VJ coordinate}
{\HH}\times V(J) \stackrel{\simeq}{\to} \mathcal{V}_{J}, \quad (\tau, z)\mapsto \tau l_{J}+z\otimes e_{1}={\C}(f_{1}+\tau f_{2}+z). 
\end{equation}
Note that the isomorphism \eqref{eqn: HJ coordinate} 
maps the cusps ${\proj}^{1}_{{\Q}}=\{ i\infty \} \cup {\Q}$ of ${\HH}\subset {\proj}^{1}$ 
to the cusps ${\proj}J_{{\Q}}^{\vee}$ of ${\HJ}\subset {\proj}J_{{\C}}^{\vee}$, 
and especially maps the cusp $i\infty$ to the $I$-cusp $I^{\perp}\cap {\proj}J_{{\C}}^{\vee}$ of ${\HJ}$. 

Next we consider the Jacobi group ${\GJF}$, $F={\Q}, {\R}$. 
Recall from \eqref{eqn: Jacobi group split} that the splitting $L_{F}=(J_{F}\oplus J_{F}^{\vee})\oplus V(J)_{F}$ defines an isomorphism  
\begin{equation}\label{eqn: Jacobi group split again}
{\GJF} \simeq {\SL}(J_{F}) \ltimes W(J)_{F}, 
\end{equation}
which we fix below. 
(This splitting depends on $J_{F}^{\vee}$, but not on $I$.) 
We identify 
\begin{equation*}
{\SL}(J_{F}) = {\SL}(J_{F}^{\vee}) = {\SL}(2, F) 
\end{equation*}
by the basis $f_{2}, f_{1}$ of $J_{F}^{\vee}$, or equivalently, by the basis $e_1, -e_2$ of $J_{F}$. 
Thus an element 
$\begin{pmatrix} a & b \\ c & d \end{pmatrix}\in {\SL}(2, F)$ acts on $J_{F}\oplus J_{F}^{\vee}$ by 
\begin{equation*}
e_{1} \mapsto ae_{1}-ce_{2}, \quad e_{2} \mapsto -be_{1}+de_{2}, \quad f_{1} \mapsto df_{1}+bf_{2}, \quad f_{2} \mapsto cf_{1}+af_{2}. 
\end{equation*} 
Finally, we have a splitting of the Heisenberg group $W(J)_{F}$ \textit{as a set}: 
\begin{eqnarray}\label{eqn: Heisenberg split}
W(J)_{F} & \simeq & 
{\UJF} \times (V(J)_{F}\otimes Fe_{1}) \times (V(J)_{F}\otimes Fe_{2}) \\ 
& \simeq & F\times V(J)_{F} \times V(J)_{F}, \nonumber 
\end{eqnarray}
where we take $v_{J}$ as the basis of $U(J)_{F}$. 
Accordingly, we write an element of $W(J)_{F}$ as 
$(\alpha, v_1, v_2)$ where $\alpha\in F$ and $v_1, v_2\in V(J)_{F}\subset L_{F}$. 
In this expression, $(\alpha, 0, 0)=\alpha v_{J}$ corresponds to $E_{\alpha e_{2}\wedge e_{1}}\in {\UJF}$, 
$(0, v_1, 0)$ to $E_{v_{1}\otimes e_{1}}$, and $(0, 0, v_2)$ to $E_{v_{2}\otimes e_{2}}$. 
Note that each $V(J)_{F}\otimes Fe_{1}$ and $V(J)_{F}\otimes Fe_{2}$ are respectively subgroups of $W(J)_{F}$, but they do not commute. 

\begin{proposition}\label{prop: action of Jacobi group} 
The action of ${\GJF}$ on ${\D}$ is described as follows. 

(1) $(\alpha, 0, 0)\in {\UJF}$ acts by 
\begin{equation*}
(\tau, \, z, \, w)\mapsto (\tau, \, z, \, w+\alpha). 
\end{equation*}

(2) $(0, v_1, 0)\in W(J)_{F}$ acts by 
\begin{equation*}
(\tau, \, z, \, w)\mapsto (\tau, \, z+v_{1}, \, w). 
\end{equation*}

(3) $(0, 0, v_2)\in W(J)_{F}$ acts by 
\begin{equation*}
(\tau, \, z, \, w)\mapsto (\tau, \: z+\tau v_2, \: w-(v_2, z)-2^{-1}(v_2, v_2)\tau). 
\end{equation*}

(4) $\begin{pmatrix} a & b \\ c & d \end{pmatrix}\in {\SL}(2, F)$ acts by 
\begin{equation*}
(\tau, \, z, \, w)\mapsto \left( \frac{a\tau+b}{c\tau+b},\:  \frac{z}{c\tau+d}, \: w+\frac{c(z, z)}{2(c\tau+d)} \right). 
\end{equation*}
\end{proposition}

\begin{proof}
Let $\omega(Z)\in L_{{\C}}$ be as in \eqref{eqn: omega(Z)}. 
By direct calculation using the definition \eqref{eqn: Eichler} of Eichler transvections, we see that 
\begin{eqnarray*}
E_{\alpha e_2\wedge e_1}(\omega(Z)) & = & 
f_{1}+\tau f_2 +z + (w+\alpha)e_2+Ae_1 \\ 
& = & \omega(Z+(0, 0, \alpha)),  
\end{eqnarray*}
\begin{eqnarray*}
E_{v_1\otimes e_1}(\omega(Z)) & = & 
f_{1}+\tau f_2 + (z+v_1) + we_2+Ae_1 \\ 
& = & \omega(Z+(0, v_1, 0)), 
\end{eqnarray*}
\begin{eqnarray*}
E_{v_2\otimes e_2}(\omega(Z)) & = & 
f_{1}+\tau f_2 + (z+\tau v_2) + (w-(z, v_2)-(\tau/2)(v_2, v_2))e_2+Ae_1 \\ 
& = & \omega(Z+(0, \tau v_2, -(z, v_2)-(\tau/2)(v_2, v_2))), 
\end{eqnarray*}
\begin{equation*}
\begin{pmatrix} a & b \\ c & d \end{pmatrix}(\omega(Z)) = 
(c\tau+d)f_{1} + (a\tau+b)f_2 + z + ((c\tau+d)w+(c/2)(z, z))e_2+Ae_1. 
\end{equation*}
Here the constant $A$ in each equation is an unspecified constant determined by the isotropicity condition. 
This proves (1) -- (4). 
\end{proof}

Proposition \ref{prop: action of Jacobi group} agrees with the classical description of the action of Jacobi group in \cite{Gr} p.1185. 
($\alpha, v_{1}, v_{2}$ correspond to $r, y, x$ in \cite{Gr} respectively.) 
We note two consequences of the calculation in Proposition \ref{prop: action of Jacobi group}. 

\begin{corollary}\label{cor: ctau+d}
Let $\gamma \in {\GJR}$ and $\begin{pmatrix} a & b \\ c & d \end{pmatrix}$ be its image in ${\SL}(2, {\R})$. 
The factor of automorphy of the $\gamma$-action on ${\LL}$ with respect to the $I$-trivialization 
${\LL}\simeq {\ICv}\otimes {\OD}$ is $c\tau+d$. 
\end{corollary}

\begin{proof}
In view of \eqref{eqn: f.a. L}, 
this follows by looking at the coefficients of $f_{1}$ in the equations in the proof of Proposition \ref{prop: action of Jacobi group}. 
\end{proof}

This gives a computational explanation of the ${\GJR}$-equivariant isomorphism 
${\LL}\simeq \pi^{\ast}{\LJ}$ in Lemma \ref{lem: L=LJ}. 
We also provide a computational proof of Lemma \ref{claim: omegaJ}. 
 
\begin{corollary}[cf.~Lemma \ref{claim: omegaJ}]\label{cor: f.a. omegaJ}
Let $\gamma\in {\GJR}$ and  $\omega_{J}=e((l_{J,{\G}}, Z))$ be as in \S \ref{ssec: FJ expansion}. 
Then $\gamma^{\ast}\omega_{J}=j_{\gamma}(\tau, z)\omega_{J}$ for a function $j_{\gamma}(\tau, z)$ of $(\tau, z)$ 
which does not depend on the $w$-component. 
\end{corollary}

\begin{proof}
Since $l_{J,{\G}}=\beta_{0}l_{J}$, if we express $Z=(\tau, z, w)$, we have 
\begin{equation*}
\omega_{J} = e((l_{J,{\G}}, Z)) = e((\beta_{0}l_{J}, wv_{J})) = e(\beta_{0}w). 
\end{equation*}
Therefore, if we denote by $\gamma^{\ast}w$ the $w$-component of $\gamma(Z)$, we have  
\begin{equation}\label{eqn: f.a. omega_J pre}
(\gamma^{\ast}\omega_{J})/\omega_{J} = 
e (\beta_{0}( \gamma^{\ast}w- w)). 
\end{equation}
It remains to observe from Proposition \ref{prop: action of Jacobi group} that 
$\gamma^{\ast}w- w$ depends only on $(\tau, z)$. 
\end{proof}

The function $j_{\gamma}(\tau, z)$ is the inverse of the factor of automorphy of the $\gamma$-action 
($=$ pullback by $\gamma^{-1}$) on the conormal bundle $\Theta_{J}$ of $\Delta_{J}$ with respect to the $\omega_{J}$-trivialization. 
Thus $j_{\gamma}(\tau, z)$ is the multiplier in the slash operator by $\gamma$ 
on $\Theta_{J}$ with respect to the $\omega_{J}$-trivialization. 
By \eqref{eqn: f.a. omega_J pre}, $j_{\gamma}(\tau, z)$ is explicitly written as follows. 
\begin{equation}\label{eqn: f.a. omega_J}
j_{\gamma}(\tau, z) = 
\begin{cases}
e(\beta_{0}\alpha) & \gamma=(\alpha, 0, 0) \\ 
1 & \gamma=(0, v_1, 0) \\ 
e(-\beta_{0}(v_{2}, z)-2^{-1}\beta_{0}(v_2, v_2)\tau) & \gamma=(0, 0, v_2) \\ 
e\left( \frac{\beta_{0}c(z, z)}{2(c\tau+d)} \right) & \gamma=\begin{pmatrix} a & b \\ c & d \end{pmatrix} 
\end{cases}
\end{equation}
If we divide ${\GJR}$ by ${\UJR}$, these coincide with the multipliers in the slash operator in 
\cite{Sk} p.248 with $k=0$ and the quadratic space $V(J)_{{\Q}}(-\beta_{0})$. 
(We identify the half-integral matrix $F$ in \cite{Sk} with the even lattice with Gram matrix $2F$, 
and this lattice tensored with ${\Q}$ corresponds to our $V(J)_{{\Q}}(-\beta_{0})$.)

\subsection{Translation to classical Jacobi forms}

Now, using the coordinates prepared in \S \ref{sssec: Siegel domain coordinate}, 
we describe Jacobi forms with $\lambda=0$ in a more classical manner. 
We identify $\Delta_{J}\simeq {\HH}\times V(J)$ by \eqref{eqn: VJ coordinate} and 
accordingly use the coordinates $(\tau, z)$ on $\Delta_{J}$. 
We put $q_{J}=e(\tau)=e((v_{J}, Z))$ and 
$q_{J,{\G}}=e((v_{J,{\G}}, Z))$ (as in \eqref{eqn: Fourier expansion Jacobi form I}) for $Z=(\tau, z)\in \Delta_{J}$. 
Since $v_{J,{\G}}=\beta_{0}^{-1}v_{J}$, then $q_{J,{\G}}=e(\beta_{0}^{-1}\tau)=(q_{J})^{\beta_{0}^{-1}}$. 
We also write $\beta_{2}= \beta_{0}^{-1}\beta_{1}$. 

Let $\phi\in J_{k,m}({\GJZ})$ be a Jacobi form of weight $(0, k)$ and index $m$ in the sense of Definition \ref{def: Jacobi form}. 
Via the $(I, \omega_{J})$-trivialization and the basis $e_{1}$ of $I$, 
\begin{equation*}
{\LL}^{\otimes k}\otimes \Theta_{J}^{\otimes m} \simeq 
({\ICv})^{\otimes k} \otimes \mathcal{O}_{\Delta_{J}} \simeq \mathcal{O}_{\Delta_{J}}, 
\end{equation*}
we regard $\phi$ as a scalar-valued function on $\Delta_{J}$. 
Let $V(J)(\beta_{0}m)$ be the scaling of the quadratic space $V(J)$ by $\beta_{0}m$. 

\begin{lemma}\label{lem: Jacobi form translate 1}
We identify $V(J)=V(J)(\beta_{0}m)$ as a ${\C}$-linear space naturally 
and regard $\phi$ as a function on $\Delta_{J}\simeq {\HH}\times V(J)(\beta_{0}m)$. 
Then $\phi$ has a Fourier expansion of the form 
\begin{equation*}
\phi(\tau, z) = \sum_{n\in \beta_{2}{\Z}} \sum_{l\in K_{I}(\beta_{0}m)^{\vee}} a(n, l) \, q^{l} \, q_{J}^{n}, 
\qquad  \tau\in {\HH}, \; z\in V(J)(\beta_{0}m). 
\end{equation*}
Here $q^{l}=e((l, z))$ with $(l, z)$ being the pairing in $V(J)(\beta_{0}m)$, 
and $K_{I}$ is some full lattice in $V(J)_{{\Q}}$ such that $K_{I}(\beta_{0})$ is an even lattice. 
The holomorphicity condition at the $I$-cusp is $2n\geq |(l, l)|$. 
\end{lemma}

\begin{proof}
Recall from \eqref{eqn: Fourier expansion Jacobi form I} that $\phi$ as a function on ${\HH}\times V(J)$ has 
a Fourier expansion of the form 
\begin{equation*}
\phi(\tau, z) = \sum_{n\in \beta_{1}{\Z}} \sum_{l\in K'} a(n, l) \, q^{l} \, q_{J,{\G}}^{n}, 
\qquad \tau \in {\HH}, \; z\in V(J),
\end{equation*}
where $K'$ is some full lattice in $V(J)_{{\Q}}$ and $q^{l}=e((l, z))$. 
(The vectors $l$ in \eqref{eqn: Fourier expansion Jacobi form I} are $l\otimes e_{1}$ here.) 
The $I$-cusp condition is $2nm\geq |(l, l)|$. 
We substitute $q_{J,{\G}}=(q_{J})^{\beta_{0}^{-1}}$ and rewrite $\beta_{0}^{-1}n$ as $n$. 
Then this expression is rewritten as 
\begin{equation*}
\phi(\tau, z) = \sum_{n\in \beta_{2}{\Z}} \sum_{l\in K'} a(n, l) \, q^{l} \, q_{J}^{n}, 
\qquad \tau\in {\HH}, \; z\in V(J), 
\end{equation*}
with the $I$-cusp condition being $2n\beta_{0}m\geq |(l, l)|$. 
By enlarging $K'$, we may assume that $K'=K_{I}^{\vee}$ for a lattice $K_{I}\subset V(J)_{{\Q}}$ such that $K_{I}(\beta_{0})$ is even. 

Next we identify $V(J)=V(J)(\beta_{0}m)$ as a ${\C}$-linear space, 
which multiplies the quadratic form by $\beta_{0}m$. 
This identification maps the lattice $K_{I}^{\vee}\subset V(J)$ 
to the lattice $\beta_{0}m K_{I}(\beta_{0}m)^{\vee} \subset V(J)(\beta_{0}m)$. 
Then, by multiplying the index lattice $K_{I}^{\vee}$ by $(\beta_{0}m)^{-1}$ and 
identifying it with $K_{I}(\beta_{0}m)^{\vee}$ by this scaling, 
the Fourier expansion of $\phi$ as a function on ${\HH}\times V(J)(m\beta_{0})$ is written as 
\begin{equation*}
\phi(\tau, z) = \sum_{n\in \beta_{2}{\Z}} \sum_{l\in K_{I}(\beta_{0}m)^{\vee}} a(n, l) \, q^{l} \, q_{J}^{n}, 
\qquad \tau\in {\HH}, \: \:  z\in V(J)(\beta_{0}m), 
\end{equation*}
where $(l, z)$ in $q^{l}=e((l, z))$ is the pairing in $V(J)(\beta_{0}m)$.  
The $I$-cusp condition is then rewritten as $2n\geq |(l, l)|$ for $l\in K_{I}(\beta_{0}m)^{\vee}$. 
\end{proof}

Here we passed from $q_{J,{\G}}$ to $q_{J}$ because the latter does not depend on ${\G}$, 
and passed from $V(J)$ to $V(J)(\beta_{0}m)$ in order to match our holomorphicity condition at the $I$-cusp 
to the holomorphicity condition at $i\infty$ of Skoruppa \cite{Sk} p.249. 

Next we shrink the integral Jacobi group ${\GJZbar}$ to a subgroup of simpler form. 
We let ${\G}_{J}\subset {\SL}(J)$ be the intersection of ${\GJZ}$ with the lifted group ${\SL}(J_{{\Q}})\subset {\GJQ}$. 
(This is different from the notation in \S \ref{ssec: Siegel operator} in general.) 
Note that ${\G}_{J}$ does not depend on $I$ (but on $J_{{\Q}}^{\vee}\subset L_{{\Q}}$). 
The splitting \eqref{eqn: Jacobi group split again} defines an isomorphism  
\begin{equation*}
{\GJQ}/{\UJQ} \simeq {\SL}(J_{{\Q}})\ltimes (V(J)_{{\Q}}\otimes J_{{\Q}}), 
\end{equation*}
where ${\SL}(J_{{\Q}})$ acts on $V(J)_{{\Q}}\otimes J_{{\Q}}$ by its natural action on $J_{{\Q}}$. 
We fix this splitting of ${\GJQ}/{\UJQ}$. 
The inclusion ${\GJZ}\subset {\GJQ}$ defines a canonical injective map ${\GJZbar}\to {\GJQ}/{\UJQ}$. 
Its image is not necessarily a semi-product. 
Elements in the intersection ${\GJZbar}\cap (V(J)_{{\Q}}\otimes {\Q}e_{i})$ are 
images of elements of ${\GJZ}$ of the form $E_{\alpha e_{1}\wedge e_{2}} \circ E_{v\otimes e_{i}}$, $v\in V(J)_{{\Q}}$, 
but $\alpha e_{1}\wedge e_{2}\in {\UJQ}$ is not necessarily contained in ${\UJZ}$ in general. 
We remedy these two subtle problems by passing to a subgroup of ${\GJZbar}$ as follows.  

\begin{lemma}\label{lem: Jacobi form translate 2}
There exists a full lattice $K_{I}'$ in $V(J)_{{\Q}}$ such that 
\begin{equation}\label{eqn: split subgroup GJZbar}
{\G}_{J}\ltimes (K_{I}'\otimes_{{\Z}} J) \; \subset \; {\GJZbar} 
\end{equation}
as subgroups of ${\GJQ}/{\UJQ}$, 
and for each $i=1, 2$, 
the subgroup $K_{I}'\otimes_{{\Z}} {\Z}e_{i}$ of this semi-product is contained in 
the image of $W(J)_{{\Z}}\cap (V(J)_{{\Q}}\otimes {\Q}e_{i})$ in ${\GJQ}/{\UJQ}$, 
where $V(J)_{{\Q}}\otimes {\Q}e_{i}$ is the component of $W(J)_{{\Q}}$ in \eqref{eqn: Heisenberg split}. 
\end{lemma} 

\begin{proof}
The intersection of $W(J)_{{\Z}}$ with the component $V(J)_{{\Q}}\otimes {\Q}e_{i}$ in \eqref{eqn: Heisenberg split} 
is a full lattice in $V(J)_{{\Q}}\otimes {\Q}e_{i}$ and hence can be written as $K_{i}\otimes_{{\Z}} {\Z}e_{i}$ 
for some full lattice $K_{i}$ in $V(J)_{{\Q}}$. 
We put $K_{I}'=K_{1}\cap K_{2}$. 
Then the second property holds by construction. 
Since $J={\Z}e_{1}\oplus {\Z}e_{2}$, it follows that 
\begin{equation*}
K_{I}'\otimes_{{\Z}} J \: \:  \subset \: \: {\GJZbar} \cap (V(J)_{{\Q}}\otimes J_{{\Q}}). 
\end{equation*}
Since we also have 
${\G}_{J}\subset {\GJZbar}\cap {\SL}(J_{{\Q}})$ by construction, 
the inclusion \eqref{eqn: split subgroup GJZbar} is verified. 
\end{proof}

The second property in Lemma \ref{lem: Jacobi form translate 2} means that 
$E_{v\otimes e_{1}}, E_{v\otimes e_{2}}\in W(J)_{{\Z}}$ for $v\in K_{I}'$, 
and their images in ${\GJQ}/{\UJQ}$ form the subgroups 
$K_{I}'\otimes_{{\Z}}{\Z}e_{1}$, $K_{I}'\otimes_{{\Z}}{\Z}e_{2}$ in \eqref{eqn: split subgroup GJZbar} respectively. 
Their factors of automorphy on $\Theta_{J}$ are given by the second and the third line in \eqref{eqn: f.a. omega_J} respectively. 
This is why we require the second property in Lemma \ref{lem: Jacobi form translate 2}.

We can now state the translation of Jacobi forms in a precise form. 
For an even negative-definite lattice $K'$, 
let $J_{k,K'}({\G}_{J})$ be the space of Jacobi forms of weight $k$ and index lattice $K'(-1)$ for the group 
${\G}_{J}< {\SL}(J)\simeq {\SL}(2, {\Z})$ in the sense of Skoruppa \cite{Sk} p.249. 
(In the notation of \cite{Sk}, $K'(-1)$ is the positive-definite even lattice with Gram matrix $2F$, 
and corresponds to the ${\Z}^{n}$ in the Heisenberg group in \cite{Sk} p.248. 
The dual lattice of $K'(-1)$ corresponds to the index ${\Z}^{n}$ in the Fourier expansion in \cite{Sk} p.249.) 

\begin{proposition}\label{prop: translate Jacobi form}
There exists a full lattice $K$ in $V(J)_{{\Q}}$ such that $K(\beta_{0})$ is an even lattice and we have an embedding 
\begin{equation*}
J_{k,m}({\GJZ}) \hookrightarrow J_{k,K(\beta_{0}m)}({\G}_{J}) 
\end{equation*}
for every $m>0$ and $k\in {\Z}$. 
\end{proposition}

\begin{proof}
The correspondence is summarized as follows: 
\begin{enumerate}
\item Start from a section $\phi$ of ${\LL}^{\otimes k}\otimes \Theta_{J}^{\otimes m}$ over $\Delta_{J}$. 
\item Choose a rank $1$ primitive sublattice $I\subset J$ and 
identify $\phi$ with a holomorphic function on $\Delta_{J}$ 
by the $(I, \omega_{J})$-trivialization of ${\LL}^{\otimes k}\otimes \Theta_{J}^{\otimes m}$. 
\item Identify $\Delta_{J}\simeq {\HH}\times V(J)(\beta_{0}m)$ by the coordinates in \S \ref{sssec: Siegel domain coordinate} 
and the scaling $V(J)\simeq V(J)(\beta_{0}m)$. 
\item In this way $\phi$ is identified with a holomorphic function on ${\HH}\times V(J)(\beta_{0}m)$. 
\end{enumerate}
We shall show that this defines a well-defined map from 
$J_{k,m}({\GJZ})$ to $J_{k,K(\beta_{0}m)}({\G}_{J})$ 
for a suitable lattice $K\subset V(J)_{{\Q}}$.  

We replace $K_{I}$ in Lemma \ref{lem: Jacobi form translate 1} and $K_{I}'$ in Lemma \ref{lem: Jacobi form translate 2} 
by their intersection $K_{I}\cap K_{I}'$ and rewrite it as $K_{I}$. 
Then Lemma \ref{lem: Jacobi form translate 1} says that 
our Jacobi form $\phi$ viewed as a function on ${\HH}\times V(J)(\beta_{0}m)$ by the above procedure 
has the same shape of Fourier expansion as that of 
Jacobi forms of weight $k$ and index lattice $K_{I}(\beta_{0}m)$ at $i\infty$ in the sense of \cite{Sk} p.249.  
Our $I$-cusp condition $2n\geq |(l, l)|$ agrees with the holomorphicity condition at $i\infty$ in \cite{Sk}. 
By Corollary \ref{cor: ctau+d} and \eqref{eqn: f.a. omega_J}, we see that 
the factor of automorphy for the action of ${\G}_{J}\ltimes (K_{I}\otimes J)$ on ${\LL}^{\otimes k}\otimes \Theta_{J}^{\otimes m}$ 
with respect to the $(I, \omega_{J})$-trivialization agrees with 
the factor of automorphy for the slash operator $|_{k, V(J)(\beta_{0}m)}$ in \cite{Sk} p.248. 
Therefore the function $\phi$ satisfies the transformation rule of \cite{Sk} p.249 (Definition (i)) for the group ${\G}_{J}<{\SL}(J)$ 
with weight $k$ and index lattice $K_{I}(\beta_{0}m)$. 
In particular, the function $\phi$ is also holomorphic (in the sense of \cite{Sk}) at the cusps equivalent to $I$ under ${\G}_{J}$. 

It remains to cover all cusps. 
The coincidence of the automorphy factors on ${\SL}(J_{{\R}})$ implies that 
the function $\phi|_{k, V(J)(\beta_{0}m)}\gamma$ for $\gamma\in {\SL}(J)$ is identified with the section $\gamma^{\ast}\phi$ 
via the $(I, \omega_{J})$-trivialization. 
Then we have 
\begin{eqnarray*}
& & \textrm{the section $\phi$ is holomorphic at the $\gamma I$-cusp in our sense} \\ 
& \Leftrightarrow & \textrm{the section $\gamma^{\ast}\phi$ is holomorphic at the $I$-cusp in our sense} \\ 
& \Leftrightarrow & \textrm{the function $\phi|_{k, V(J)(\beta_{0}m)}\gamma$ is holomorphic at $i\infty$ in the sense of \cite{Sk}.} 
\end{eqnarray*}
The first equivalence follows from Lemma \ref{lem: holomorphic at gamma(I)}, 
and the second equivalence follows by applying the argument so far 
to the Jacobi form $\gamma^{\ast}\phi$ for $\gamma^{-1}{\GJZ}\gamma \subset {\GJQ}$ 
(with $J=\gamma(J)$ and ${\UJZ}$ unchanged). 
Here the index lattice for $\gamma^{\ast}\phi$ is determined from the Fourier expansion of $\gamma^{\ast}\phi$ at the $I$-cusp 
with the group $\gamma^{-1}{\GJZ}\gamma$, by the procedure in Lemma \ref{lem: Jacobi form translate 1}. 
We denote it by $K_{\gamma I}(\beta_{0}m)$, with $K_{\gamma I}$ a full lattice in $V(J)_{{\Q}}$. 
This may be in general different from $K_{I}$. 

Then we take representatives $I_{1}=I, I_{2}, \cdots, I_{N}$ of 
${\G}_{J}$-equivalence classes of rank $1$ primitive sublattices of $J$ and put 
\begin{equation*}
K=\bigcap_{i}K_{I_{i}} \: \: \subset V(J)_{{\Q}}. 
\end{equation*}
As a function on ${\HH}\times V(J)(\beta_{0}m)$,  
$\phi$ satisfies the transformation rule of Jacobi forms of weight $k$ and index lattice $K(\beta_{0}m)$ for ${\G}_{J}<{\SL}(J)$, 
and is holomorphic at the cusps $I_{1}=i\infty, I_{2}, \cdots, I_{N}$ of ${\HJ}\simeq {\HH}$ in the sense of \cite{Sk}. 
If $\gamma(I_{i})$, $\gamma\in \Gamma_{J}$, is an arbitrary cusp of ${\HJ}$, 
the holomorphicity of $\phi=\phi|_{k, V(J)(\beta_{0}m)}\gamma$ at $I_{i}$ 
implies that of $\phi$ at $\gamma(I_{i})$. 
Thus $\phi$ is holomorphic at all cusps, 
namely $\phi\in J_{k,K(\beta_{0}m)}({\G}_{J})$. 
\end{proof}

Proposition \ref{prop: translate Jacobi form} implies the following. 

\begin{proposition}\label{prop: Jacobi form vanish}
We have $J_{k,m}({\GJZ})=0$ when $k<n/2-1$. 
\end{proposition}

\begin{proof}
This holds because $J_{k,K'}({\G}_{J})=0$ when $k< {\rm rk}(K')/2=n/2-1$ 
(see \cite{Sk} p.251). 
\end{proof}


\chapter{Filtrations associated to $1$-dimensional cusps}\label{sec: filtration}

Let $L$, ${\G}$, $J$ be as in \S \ref{sec: FJ}. 
In this chapter we introduce filtrations on the automorphic vector bundles canonically associated to the $J$-cusp, 
and study its basic properties. 
These filtrations will play a fundamental role in the study of the Fourier-Jacobi expansion. 
Our geometric approach will be effective here. 
In \S \ref{ssec: J-filtration E} we define the filtration on the second Hodge bundle ${\E}$.  
This induces filtrations on general automorphic vector bundles ${\Elk}$ (\S \ref{ssec: J-filtration Elk}). 
In \S \ref{ssec: J-filtration and representation} we study these filtrations from the viewpoint of representations of a parabolic subgroup. 
In \S \ref{ssec: decomp Jacobi form}, as the first application of our filtration, 
we prove that vector-valued Jacobi forms decompose, in a certain sense, into scalar-valued Jacobi forms of various weights. 
The second application will be given in \S \ref{sec: VT I}.

\section{$J$-filtration on ${\E}$}\label{ssec: J-filtration E}

In this section we define a filtration on ${\E}$ canonically associated to $J$. 
For $[\omega]\in {\D}$ we consider the filtration 
\begin{equation}\label{eqn: pre J-filtration fiber} 
0 \; \subset \; \omega^{\perp}\cap J_{{\C}} \; \subset \; \omega^{\perp}\cap J_{{\C}}^{\perp} \; \subset \; \omega^{\perp} 
\end{equation}
on $\omega^{\perp}=\omega^{\perp}\cap L_{{\C}}$. 

\begin{lemma}
Let $p\colon \omega^{\perp}\to \omega^{\perp}/{\C}\omega$ be the projection. 
Then $p(\omega^{\perp}\cap J_{{\C}})$ has dimension $1$ and 
$p(\omega^{\perp}\cap J_{{\C}}^{\perp})=p(\omega^{\perp}\cap J_{{\C}})^{\perp}$ in $\omega^{\perp}/{\C}\omega$. 
\end{lemma}

\begin{proof}
Since $(\omega, J)\not\equiv 0$, we have $\dim(\omega^{\perp}\cap J_{{\C}})=1$. 
The fact that ${\C}\omega\not\subset J_{{\C}}$ then implies that $p(\omega^{\perp}\cap J_{{\C}})$ has dimension $1$. 
Next we prove the second assertion. 
It is clear that 
$p(\omega^{\perp}\cap J_{{\C}}^{\perp}) \subset p(\omega^{\perp}\cap J_{{\C}})^{\perp}$. 
Since $p(\omega^{\perp}\cap J_{{\C}})^{\perp}$ is of codimension $1$ in $\omega^{\perp}/{\C}\omega$ by the first assertion, 
it is sufficient to show that $p(\omega^{\perp}\cap J_{{\C}}^{\perp})$ is of codimension $1$ too. 
Since ${\C}\omega \not\subset J_{{\C}}$, we have $(\omega, J^{\perp})\not\equiv 0$. 
This implies that $\omega^{\perp}\cap J_{{\C}}^{\perp}$ is of codimension $1$ in $J_{{\C}}^{\perp}$, 
and so of codimension $2$ in $\omega^{\perp}$. 
The fact that ${\C}\omega\not\subset J_{{\C}}^{\perp}$ implies that the projection 
$\omega^{\perp}\cap J_{{\C}}^{\perp}\to \omega^{\perp}/{\C}\omega$ 
is injective. 
Hence $p(\omega^{\perp}\cap J_{{\C}}^{\perp})$ is of codimension $1$ in $\omega^{\perp}/{\C}\omega$. 
\end{proof}

Let ${\EJ}$ be the sub line bundle of ${\E}$ whose fiber over $[\omega]\in {\D}$ is 
the image of $\omega^{\perp}\cap J_{{\C}}$ in ${\omega}^{\perp}/{\C}\omega$. 
This is an isotropic sub line bundle of ${\E}$. 
Taking the image of \eqref{eqn: pre J-filtration fiber} in $\omega^{\perp}/{\C}\omega$ and varying $[\omega]\in {\D}$, 
we obtain the filtration 
\begin{equation}\label{eqn: J-filtration}
0 \; \subset \; {\EJ} \; \subset \; {\EJp} \; \subset \; {\E} 
\end{equation}
on ${\E}$. 
We call it the \textit{$J$-filtration} on ${\E}$. 
By construction, this is ${\GJR}$-invariant. 

We calculate the graded quotients of the $J$-filtration. 
Let $\pi\colon {\D}\to {\HJ}$ be the projection to the $J$-cusp and ${\LJ}$ be the Hodge bundle on ${\HJ}$. 
We write $V(J)=(J^{\perp}/J)_{{\C}}$ as before. 

\begin{proposition}\label{prop: J-filtration Gr}
We have ${\GJR}$-equivariant isomorphisms 
\begin{equation}\label{eqn: J-filtration Gr}
{\EJ}\simeq \pi^{\ast}{\LJ}, \qquad {\EJp}/{\EJ} \simeq V(J)\otimes {\OD}, \qquad {\E}/{\EJp}\simeq \pi^{\ast}{\LL}_{J}^{-1}. 
\end{equation}
\end{proposition}

\begin{proof}
We begin with ${\EJ}$. 
Let $[\omega]\in {\D}$. 
The fiber of ${\EJ}$ over $[\omega]$ is the line $\omega^{\perp}\cap J_{{\C}}\subset J_{{\C}}$, 
while that of $\pi^{\ast}{\LJ}$ is the image of ${\C}\omega$ in $(L/J^{\perp})_{{\C}}$. 
In order to compare these two lines, we consider the canonical isomorphisms 
\begin{equation}\label{eqn: L/Jp vs J}
(L/J^{\perp})_{{\C}} \to J_{{\C}}^{\vee} \leftarrow J_{{\C}}. 
\end{equation}
Here the first map is induced by the pairing on $L$, 
and the second map is induced by the canonical symplectic form $J\times J \to \wedge^{2}J \simeq {\Z}$ on $J$. 
The second map sends a line in $J_{{\C}}$ to its annihilator in $J_{{\C}}^{\vee}$. 
In \eqref{eqn: L/Jp vs J}, the above two lines are both sent to the line $({\C}\omega, \cdot)|_{J_{{\C}}}$ in $J_{{\C}}^{\vee}$ 
(the pairing of $J_{{\C}}$ with ${\C}\omega$). 
This gives the canonical isomorphism 
\begin{equation*}
(\pi^{\ast}{\LJ})_{[\omega]} = {\rm Im}({\C}\omega \to (L/J^{\perp})_{{\C}}) \: \: \to \: \: 
\omega^{\perp}\cap J_{{\C}} = ({\EJ})_{[\omega]}. 
\end{equation*}
Varying $[\omega]$, we obtain a ${\GJR}$-equivariant isomorphism $\pi^{\ast}{\LJ}\simeq {\EJ}$. 

Consequently, we obtain the description of the last graded quotient 
\begin{equation*}
{\E}/{\EJp} \; \simeq \; {\E}_{J}^{\vee} \; \simeq \; \pi^{\ast}{\LL}_{J}^{-1}, 
\end{equation*}
where the first map is induced by the quadratic form on ${\E}$. 

Finally, we consider the middle graded quotient ${\EJp}/{\EJ}$. 
The fiber of this vector bundle over $[\omega]\in {\D}$ is 
$(\omega^{\perp}\cap J_{{\C}}^{\perp})/(\omega^{\perp}\cap J_{{\C}})$. 
We have a natural map 
\begin{equation}\label{eqn: J-filtration middle gr}
(\omega^{\perp}\cap J_{{\C}}^{\perp}) / (\omega^{\perp}\cap J_{{\C}}) \to J_{{\C}}^{\perp}/J_{{\C}} = V(J). 
\end{equation}
This is clearly injective. 
Since the source and the target have the same dimension, this map is an isomorphism. 
Varying $[\omega]$, we obtain a ${\GJR}$-equivariant isomorphism 
${\EJp}/{\EJ} \to V(J)\otimes {\OD}$. 
\end{proof}
 
Next we choose a rank $1$ primitive sublattice $I$ of $J$ and describe the $J$-filtration under the $I$-trivialization. 

\begin{proposition}\label{prop: J-filtration under I-trivialization}
The $I$-trivialization ${\E}\simeq V(I)\otimes {\OD}$ sends the $J$-filtration \eqref{eqn: J-filtration} on ${\E}$ to the filtration 
\begin{equation*}
(0 \: \subset \: J/I  \: \subset \: J^{\perp}/I  \: \subset \: I^{\perp}/I)_{{\C}} \otimes {\OD} 
\end{equation*}
on $V(I)\otimes {\OD}$. 
\end{proposition}

\begin{proof}
Since the $I$-trivialization $V(I)\otimes {\OD} \to {\E}$ preserves the quadratic forms, 
it suffices to check that this sends $(J/I)_{{\C}}\otimes {\OD}$ to ${\EJ}$. 
Recall that the $I$-trivialization at $[\omega]\in {\D}$ is the composition map 
\begin{equation}\label{eqn: recall I-trivialization filtration}
{\ICp}/I_{{\C}} \to \omega^{\perp}\cap {\ICp} \to \omega^{\perp}/{\C}\omega. 
\end{equation}
The inverse of the first map sends 
the line $\omega^{\perp}\cap J_{{\C}}$ in $\omega^{\perp}\cap {\ICp}$ 
to the line $J_{{\C}}/I_{{\C}}$ in ${\ICp}/I_{{\C}}$, 
and the second map sends $\omega^{\perp}\cap J_{{\C}}$ to $({\EJ})_{[\omega]}$ by definition. 
Therefore \eqref{eqn: recall I-trivialization filtration} sends $J_{{\C}}/I_{{\C}}$ to $({\EJ})_{[\omega]}$. 
This proves our assertion. 
\end{proof}

The $J$-filtration descends to a filtration on the descent of ${\E}$ to ${\XJ}={\D}/{\UJZ}$. 
We consider the canonical extension over the partial toroidal compactification ${\XJcpt}$. 

\begin{proposition}\label{prop: extend J-filtration}
The $J$-filtration on ${\E}$ over ${\XJ}$ extends to a filtration on the canonical extension of ${\E}$ over ${\XJcpt}$ 
by ${\GJRbar}$-invariant sub vector bundles. 
The isomorphisms \eqref{eqn: J-filtration Gr} for the graded quotients on ${\XJ}$ 
extend to isomorphisms between the canonical extensions of both sides over ${\XJcpt}$. 
\end{proposition}

\begin{proof}
We choose a rank $1$ primitive sublattice $I$ of $J$. 
Recall that the canonical extension of ${\E}$ is defined via the $I$-trivialization ${\E}\to V(I)\otimes \mathcal{O}_{{\XJ}}$. 
By Proposition \ref{prop: J-filtration under I-trivialization}, 
the $I$-trivialization sends the sub vector bundles ${\EJ}$, ${\EJp}$ of ${\E}$ to the sub vector bundles 
$(J/I)_{{\C}}\otimes \mathcal{O}_{{\XJ}}$, $(J^{\perp}/I)_{{\C}}\otimes \mathcal{O}_{{\XJ}}$ of 
$V(I)\otimes \mathcal{O}_{{\XJ}}$ respectively. 
The latter clearly extend to the sub vector bundles   
$(J/I)_{{\C}}\otimes \mathcal{O}_{{\XJcpt}}$, $(J^{\perp}/I)_{{\C}}\otimes \mathcal{O}_{{\XJcpt}}$ of 
$V(I)\otimes \mathcal{O}_{{\XJcpt}}$ respectively. 
This means that ${\EJ}$, ${\EJp}$ extend to sub vector bundles of the canonical extension of ${\E}$. 
They are still ${\GJRbar}$-invariant by continuity. 

We prove that the isomorphisms \eqref{eqn: J-filtration Gr} extend over ${\XJcpt}$. 
We begin with ${\EJ}\simeq \pi^{\ast}{\LJ}$. 
For each $[\omega]\in {\D}$ we have the following commutative diagram of isomorphisms between $1$-dimensional linear spaces: 
\begin{equation*}
\xymatrix{
\omega^{\perp}\cap J_{{\C}} \ar[r]^{p_1} \ar[d]_{p_3} & ({\C}\omega, \cdot)|_{J_{{\C}}} \ar[d]^{p_4} \\ 
J_{{\C}}/I_{{\C}} \ar[r]_{p_2} & {\ICv}. 
}
\end{equation*}
Here $p_1$ is restriction of the second isomorphism $J_{{\C}}\to J_{{\C}}^{\vee}$ in \eqref{eqn: L/Jp vs J}, 
$p_2$ is the map induced from this $J_{{\C}}\to J_{{\C}}^{\vee}$, 
$p_3$ is the natural projection, and 
$p_4$ is the restriction of the natural map $J_{{\C}}^{\vee}\to {\ICv}$ to the line $({\C}\omega, \cdot)|_{J_{{\C}}}$ of  $J_{{\C}}^{\vee}$. 
Recall from the proof of Proposition \ref{prop: J-filtration Gr} that 
$p_{1}$ is identified with the isomorphism ${\EJ}\to \pi^{\ast}{\LJ}$ at $[\omega]$ 
after the canonical isomorphism $J_{{\C}}^{\vee}\simeq (L/J^{\perp})_{{\C}}$. 
Varying $[\omega]$, we obtain the following commutative diagram of isomorphisms between line bundles on ${\XJ}$: 
\begin{equation*}
\xymatrix{
{\EJ} \ar[r]^{p_1} \ar[d]_{p_3} & \pi^{\ast}{\LJ} \ar[d]^{p_4} \\ 
(J/I)_{{\C}}\otimes \mathcal{O}_{{\XJ}} \ar[r]_{p_2} & {\ICv}\otimes \mathcal{O}_{{\XJ}}. 
}
\end{equation*}
Here $p_1$ is the isomorphism we want to extend, 
$p_{2}$ is the constant homomorphism, 
$p_{3}$ is the $I$-trivialization of ${\EJ}$, and 
$p_{4}$ is the pullback of the $I$-trivialization of ${\LJ}$ (cf.~Remark \ref{rmk: I-trivialization LJ}).  
By construction, the canonical extension of ${\EJ}$ is given via $p_{3}$. 
Similarly, by the proof of Proposition \ref{prop: L at 1dim cusp}, 
the canonical extension of $\pi^{\ast}{\LJ}$ is given via $p_{4}$. 
Since $p_{2}$ is constant, it extends over ${\XJcpt}$. 
Then this commutative diagram shows that 
$p_1$ extends to an isomorphism between the canonical extensions of ${\EJ}$ and $\pi^{\ast}{\LJ}$. 

Next we consider ${\EJp}/{\EJ}\to V(J)\otimes \mathcal{O}_{{\XJ}}$. 
We observe that for each $[\omega]\in {\D}$, the natural composition 
\begin{equation*}\label{eqn: extend J-filtration middle gr}
(\omega^{\perp}\cap J_{{\C}}^{\perp}) / (\omega^{\perp}\cap J_{{\C}}) \to 
(J^{\perp}_{{\C}}/I_{{\C}}) / (J_{{\C}}/I_{{\C}}) \to J^{\perp}_{{\C}}/J_{{\C}}, 
\end{equation*}
where the first isomorphism comes from $\omega^{\perp}\cap {\ICp} \to {\ICp}/I_{{\C}}$, 
coincides with the isomorphism \eqref{eqn: J-filtration middle gr} 
defining ${\EJp}/{\EJ}\to V(J)\otimes \mathcal{O}_{{\XJ}}$ at $[\omega]$. 
Therefore the isomorphism 
${\E}_{J}^{\perp}/{\EJ}\to V(J)\otimes \mathcal{O}_{{\XJ}}$ 
in \eqref{eqn: J-filtration Gr} factorizes as 
\begin{equation*}
{\E}_{J}^{\perp}/{\EJ} \to (J^{\perp}/I)_{{\C}}\otimes \mathcal{O}_{{\XJ}} / (J/I)_{{\C}}\otimes \mathcal{O}_{{\XJ}} 
\to V(J)\otimes \mathcal{O}_{{\XJ}}, 
\end{equation*}
where the first isomorphism is induced by the $I$-trivialization and hence gives the canonical extension of ${\EJp}/{\EJ}$, 
and the second isomorphism is the constant homomorphism. 
The constancy of the second isomorphism ensures that it extends over ${\XJcpt}$. 
This shows that the isomorphism ${\E}_{J}^{\perp}/{\EJ}\to V(J)\otimes \mathcal{O}_{{\XJ}}$ in \eqref{eqn: J-filtration Gr} 
extends to an isomorphism between the canonical extensions. 

Finally, the extendability of ${\E}/{\EJp}\simeq \pi^{\ast}{\LL}_{J}^{-1}$ follows from 
the extendability of ${\EJ}\simeq \pi^{\ast}{\LJ}$ 
and the fact that the quadratic form on ${\E}$ extends over the canonical extension (by construction). 
\end{proof}

\section{$J$-filtration on ${\Elk}$}\label{ssec: J-filtration Elk}

In this section we use the $J$-filtration on ${\E}$ to define a filtration on a general automorphic vector bundle ${\Elk}$. 

We begin with a recollection from linear algebra. 
Let $V$ be a ${\C}$-linear space of finite dimension endowed with a decreasing filtration of length $3$: 
\begin{equation*}
0 \: \subset \: F^{1}V \: \subset \: F^{0}V \: \subset \: F^{-1}V =V. 
\end{equation*}
We denote by ${\rm Gr}^{r} V=F^{r}V/F^{r+1}V$ the $r$-th graded quotient. 
(By convention, $F^{2}V=0$.) 
Let $d>0$. 
On the tensor product $V^{\otimes d}$ we have a decreasing filtration of length $2d+1$ defined by 
\begin{equation}\label{eqn: basic filtration LA}
F^{r}V^{\otimes d} = \sum_{|\vec{i}|=r} F^{i_{1}}V\otimes F^{i_{2}}V\otimes \cdots \otimes F^{i_{d}}V, 
\qquad -d \leq r \leq d, 
\end{equation}
where $\vec{i}=(i_1, \cdots, i_{d})$ run over all multi-indices such that $|\vec{i}|=i_{1}+\cdots + i_{d}$ is equal to $r$. 
The graded quotient 
${\rm Gr}^{r} V^{\otimes d} =F^{r}V^{\otimes d}/F^{r+1}V^{\otimes d}$ 
is canonically isomorphic to 
\begin{equation}\label{eqn: basic filtration LA Gr}
{\rm Gr}^{r} V^{\otimes d}  \simeq \bigoplus_{|\vec{i}|=r} 
{\rm Gr}^{i_{1}} V  \otimes {\rm Gr}^{i_{2}} V \otimes \cdots \otimes {\rm Gr}^{i_{d}} V. 
\end{equation}
This construction of filtration is well-known in the case $d=2$; the construction for general $d$ is obtained inductively. 

We apply this construction relatively to the $J$-filtration on the second Hodge bundle ${\E}$. 
We write $F^{1}{\E}={\EJ}$, $F^{0}{\E}={\EJp}$, $F^{-1}{\E}={\E}$, and define a decreasing filtration 
\begin{equation*}
0 \: \subset \: F^{d}{\E}^{\otimes d} \: \subset \: F^{d-1}{\E}^{\otimes d} \: \subset  \cdots \subset \: F^{-d}{\E}^{\otimes d} = {\E}^{\otimes d} 
\end{equation*}
of length $2d+1$ on ${\E}^{\otimes d}$ by 
\begin{equation*}
F^{r}{\E}^{\otimes d} = \sum_{|\vec{i}|=r} F^{i_{1}}{\E}\otimes F^{i_{2}}{\E}\otimes \cdots \otimes F^{i_{d}}{\E}, \qquad -d\leq r \leq d.   
\end{equation*}
This is a filtration by ${\GJR}$-invariant sub vector bundles. 

\begin{lemma}\label{lem: Gr J-filtration Ed}
We have a ${\GJR}$-equivariant isomorphism 
\begin{equation}\label{eqn: GrEd explicit}
{\rm Gr}^{r} {\E}^{\otimes d} \simeq 
\pi^{\ast}\mathcal{L}_{J}^{\otimes r} \otimes \bigoplus_{|\vec{i}|=r} V(J)^{\otimes b(\vec{i})}, 
\end{equation}
where $b(\vec{i})\geq 0$ is the number of components $i_{\ast}$ of $\vec{i}=(i_{1}, \cdots, i_{d})$ equal to $0$. 
\end{lemma}

\begin{proof}
By \eqref{eqn: basic filtration LA Gr} we have 
\begin{equation}\label{eqn: GrEd}
{\rm Gr}^{r} {\E}^{\otimes d} \simeq \bigoplus_{|\vec{i}|=r} {\rm Gr}^{i_{1}} {\E} \otimes \cdots \otimes {\rm Gr}^{i_{d}} {\E}. 
\end{equation}
By Proposition \ref{prop: J-filtration Gr}, each factor ${\rm Gr}^{i_{\ast}} {\E}$ is isomorphic to 
$\pi^{\ast}{\LJ}$, $V(J)\otimes {\OD}$, $\pi^{\ast}{\LL}_{J}^{-1}$ 
according to $i_{\ast}=1, 0, -1$ respectively. 
Let $a(\vec{i}), b(\vec{i}), c(\vec{i})$ be the number of components $i_{\ast}$ of $\vec{i}=(i_{1}, \cdots, i_{d})$ equal to 
$1, 0, -1$ respectively. 
Then \eqref{eqn: GrEd} can be written more explicitly as 
\begin{equation*}
{\rm Gr}^{r} {\E}^{\otimes d} \simeq \bigoplus_{|\vec{i}|=r} 
V(J)^{\otimes b(\vec{i})} \otimes \pi^{\ast}\mathcal{L}_{J}^{\otimes a(\vec{i})-c(\vec{i})}. 
\end{equation*}
We have $a(\vec{i})-c(\vec{i}) = |\vec{i}| =r$. 
\end{proof}

Since ${\rm Gr}^{-i} {\E} \simeq ({\rm Gr}^{i} {\E})^{\vee}$, 
the expression \eqref{eqn: GrEd} shows that we have the duality 
\begin{equation*}
{\rm Gr}^{-r} {\E}^{\otimes d} \simeq ({\rm Gr}^{r} {\E}^{\otimes d})^{\vee}, 
\end{equation*}
by sending an index $\vec{i}=(i_{1}, \cdots, i_{d})$ to its dual index $(-i_{1}, \cdots, -i_{d})$.  

By Proposition \ref{prop: J-filtration under I-trivialization}, 
the $I$-trivialization ${\E}^{\otimes d}\simeq V(I)^{\otimes d}\otimes {\OD}$ 
sends the sub vector bundle $F^{r}{\E}^{\otimes d}$ of ${\E}^{\otimes d}$ 
to the sub vector bundle $F^{r}V(I)^{\otimes d}\otimes {\OD}$ of $V(I)^{\otimes d}\otimes {\OD}$, 
where $F^{r}V(I)^{\otimes d}$ is the filtration \eqref{eqn: basic filtration LA} applied to 
$V=V(I)$, $F^{1}V=(J/I)_{{\C}}$ and $F^{0}V=(J^{\perp}/I)_{{\C}}$. 
This implies that the filtration $F^{\bullet}{\E}^{\otimes d}$ on ${\E}^{\otimes d}$ over ${\XJ}$ 
extends to a filtration on the canonical extension of ${\E}^{\otimes d}$ over ${\XJcpt}$ by sub vector bundles. 
(We use the same notation.)  

Now we consider a general automorphic vector bundle ${\Elk}={\El}\otimes {\LL}^{\otimes k}$. 
Let $d=|\lambda|$. 
Recall from \S \ref{ssec: automorphic VB} that 
${\El}=c_{\lambda}\cdot {\E}^{[d]}$ is defined as an ${\OLR}$-invariant sub vector bundle of ${\E}^{\otimes d}$, 
where $c_{\lambda}=b_{\lambda}a_{\lambda}$ is the Young symmetrizer for $\lambda$. 
We define a decreasing filtration on ${\El}$ by taking the intersection with $F^{r}{\E}^{\otimes d}$ inside ${\E}^{\otimes d}$: 
\begin{equation*}
F^{r}{\El} = {\El}\cap F^{r}{\E}^{\otimes d}, \qquad -d \leq r \leq d. 
\end{equation*}
Then we take the twist by ${\LL}^{\otimes k}$: 
\begin{equation*}
F^{r}{\Elk} = F^{r}{\El} \otimes {\LL}^{\otimes k}. 
\end{equation*}
This is a ${\GJR}$-invariant filtration on ${\Elk}$. 
We call it the \textit{$J$-filtration} on ${\Elk}$. 
This is a standard filtration on ${\Elk}$ 
that can be induced from the $J$-filtration on ${\E}$. 
In Proposition \ref{cor: J-filtration and representation}, 
we will prove that the range of the level $r$ reduces to $-\lambda_{1} \leq r \leq \lambda_{1}$. 

\begin{remark}
We also have the following natural expressions of $F^{r}{\El}$: 
\begin{equation*}
F^{r}{\El} = c_{\lambda}({\E}^{[d]}\cap F^{r}{\E}^{\otimes d}) = {\E}^{[d]} \cap c_{\lambda}(F^{r}{\E}^{\otimes d}). 
\end{equation*}
These equalities hold because 
we have $c_{\lambda}(F^{r}{\E}^{\otimes d}) \subset F^{r}{\E}^{\otimes d}$ 
by the $\frak{S}_{d}$-invariance of $F^{r}{\E}^{\otimes d}$ and 
$c_{\lambda}$ is an idempotent up to scalar multiplication. 
\end{remark}

Let 
\begin{equation}\label{eqn: filtration VIl}
F^{r} {\VIl} = {\VIl} \cap F^{r}V(I)^{\otimes d}, \qquad -d \leq r \leq d, 
\end{equation}
be the similar filtration on ${\VIl}$. 
The $I$-trivialization ${\El}\simeq {\VIl}\otimes {\OD}$ sends the $J$-filtration  $F^{\bullet}{\El}$ on ${\El}$ 
to the filtration $F^{\bullet} {\VIl}\otimes {\OD}$ on ${\VIl}\otimes {\OD}$. 
This implies that the $J$-filtration on ${\Elk}$, after descending to ${\XJ}$,  
extends to a filtration on the canonical extension of ${\Elk}$ over ${\XJcpt}$ by ${\GJRbar}$-invariant sub vector bundles. 

\begin{proposition}\label{prop: basic filtration Elk}
At the boundary divisor $\Delta_{J}$ of ${\XJcpt}$, we have a ${\GJRbar}$-equivariant isomorphism 
\begin{equation}\label{eqn: J-filtration Elk Gr}
{\rm Gr}^{r}({\Elk}|_{\Delta_{J}}) \simeq (\pi_{2}^{\ast}{\LL}_{J}^{\otimes r+k})^{\oplus \alpha(r)},  
\end{equation}
where $\alpha(r) \geq 0$ is the rank of ${\rm Gr}^{r}{\El}$ 
and $\pi_{2}$ is the projection $\Delta_{J} \to {\HJ}$. 
\end{proposition}

\begin{proof}
Since ${\LL}|_{\Delta_{J}}\simeq \pi_{2}^{\ast}{\LJ}$ by Proposition \ref{prop: L at 1dim cusp}, 
it suffices to prove this assertion in the case $k=0$. 
By Lemma \ref{lem: Gr J-filtration Ed}, we have a ${\GJRbar}$-equivariant embedding 
\begin{equation*}
{\rm Gr}^{r} {\El} \hookrightarrow {\rm Gr}^{r} {\E}^{\otimes |\lambda|} \simeq 
(\pi^{\ast}{\LL}_{J}^{\otimes r})^{\oplus b}  
\end{equation*}
over ${\XJ}$ for some $b>0$. 
By Proposition \ref{prop: extend J-filtration}, this embedding extends over ${\XJcpt}$. 
By restricting it to $\Delta_{J}$, we obtain a ${\GJRbar}$-equivariant embedding 
\begin{equation*}
{\rm Gr}^{r}({\El}|_{{\Delta}_{J}}) \hookrightarrow (\pi_{2}^{\ast}{\LL}_{J}^{\otimes r})^{\oplus b}.  
\end{equation*}
The image of this embedding is a ${\GJRbar}$-invariant sub vector bundle of $(\pi_{2}^{\ast}{\LL}_{J}^{\otimes r})^{\oplus b}$. 
Since the Heisenberg group $W(J)_{{\R}}\subset {\GJR}$ acts on each fiber of $\pi_{2}\colon \Delta_{J}\to {\HJ}$ transitively, 
this image can be written as 
$\pi_{2}^{\ast}\mathcal{F}$ for some ${\SL}(J_{{\R}})$-invariant sub vector bundle $\mathcal{F}$ of $({\LL}_{J}^{\otimes r})^{\oplus b}$. 
By the ${\SL}(J_{{\R}})$-invariance, $\mathcal{F}$ is isomorphic to a direct sum of copies of ${\LL}_{J}^{\otimes r}$. 
\end{proof}

Before finishing this section, we look at two typical examples. 

\begin{example}\label{ex: J-filtration exterior product}
Let $\lambda = (1^{d})$ with $0<d<n$, namely $V_{\lambda}=\wedge^{d}V$. 
We have $\wedge^{i}{\EJ}=0$ if $i>1$ and 
$(\wedge^{i}{\EJp})\wedge (\wedge^{j}{\E}) = \wedge^{i+j}{\E}$ if $j>0$. 
This shows that the $J$-filtration on $\wedge^{d}{\E}$ reduces to the following filtration of length $3$: 
\begin{equation*}
0 \; \subset \; {\EJ}\wedge (\wedge^{d-1}{\EJp}) \; \subset \;
 \wedge^{d}{\EJp} + {\EJ}\wedge (\wedge^{d-1}{\E}) \; \subset \;  \wedge^{d}{\E}.  
\end{equation*}
These three subspaces have level $1$, $0$, $-1$ respectively. 
(Note that $\wedge^{d-1}{\E}=(\wedge^{d-2}{\EJp})\wedge {\E}$ in the second term and 
$\wedge^{d}{\E}=(\wedge^{d-1}{\EJp})\wedge {\E}$ in the last term.) 
The three graded quotients are respectively isomorphic to 
\begin{equation*}
{\EJ}\otimes \wedge^{d-1}({\EJp}/{\EJ}) \simeq  \wedge^{d-1}V(J) \otimes \pi^{\ast}{\LJ}, 
\end{equation*}
\begin{equation*}
\wedge^{d}({\EJp}/{\EJ}) \oplus \wedge^{d-2}({\EJp}/{\EJ})  
\simeq (\wedge^{d}V(J) \oplus \wedge^{d-2}V(J)) \otimes {\OD}, 
\end{equation*}
\begin{equation*}
({\E}/{\EJp})\otimes \wedge^{d-1}({\EJp}/{\EJ}) \simeq \wedge^{d-1}V(J)\otimes \pi^{\ast}{\LL}_{J}^{-1}. 
\end{equation*}
Here $\wedge^{d-2}V(J)=0$ when $d=1$, and $\wedge^{d}V(J)=0$ when $d=n-1$. 
\end{example}

\begin{example}\label{ex: J-filtration symmetric product}
The $J$-filtration on ${\rm Sym}^{d}{\E}$ has length $2d+1$, 
with subspaces  
\begin{equation*}
F^{r}{\rm Sym}^{d}{\E}= \sum_{\substack{a+b+c=d \\ a-c=r}} 
{\rm Sym}^{a}{\EJ} \cdot {\rm Sym}^{b}{\EJp} \cdot {\rm Sym}^{c}{\E}, \qquad -d \leq r \leq d. 
\end{equation*}
The graded quotient ${\rm Gr}^{r} {\rm Sym}^{d}{\E}$ is isomorphic to 
\begin{equation*}
\pi^{\ast}{\LL}_{J}^{\otimes r} \otimes 
({\rm Sym}^{d-|r|}V(J) \oplus {\rm Sym}^{d-|r|-2}V(J) \oplus \cdots \oplus {\rm Sym}^{0 \, or \, 1}V(J)). 
\end{equation*}
This shows that the $J$-filtration on the main irreducible component ${\E}_{(d)}$ of ${\rm Sym}^{d}{\E}$ 
has length $2d+1$ with graded quotient 
\begin{equation}\label{eqn: Gr J-filtration (d)}
{\rm Gr}^{r} {\E}_{(d)} \simeq \pi^{\ast}{\LL}_{J}^{\otimes r} \otimes {\rm Sym}^{d-|r|}V(J), 
\qquad -d \leq r \leq d.  
\end{equation}
\end{example}

\section{$J$-filtration and representations}\label{ssec: J-filtration and representation}

In this section we study the $J$-filtration, in its $I$-trivialized form, 
from the viewpoint of representations of a parabolic subgroup. 
As consequences, 
we determine the range of possible levels, and 
also relate the Siegel operator (\S \ref{sec: Siegel}) to the $J$-filtration. 

We choose a rank $1$ primitive sublattice $I\subset J$. 
Let $P(J/I)_{{\C}}$ be the stabilizer of the isotropic line $(J/I)_{{\C}}\subset V(I)$ in ${\rm O}(V(I))$. 
As in \eqref{eqn: P(I,J)}, $P(J/I)_{{\C}}$ sits in the exact sequence 
\begin{equation}\label{eqn: P(J/I) exact sequence}
0 \to U(J/I)_{{\C}} \to P(J/I)_{{\C}} \to {\rm GL}((J/I)_{{\C}})\times {\rm O}(V(J)) \to 1, 
\end{equation}
where $U(J/I)_{{\C}}\simeq V(J)\otimes (J/I)_{{\C}}$ is the unipotent radical of $P(J/I)_{{\C}}$  
consisting of the Eichler transvections of $V(I)$ with respect to $(J/I)_{{\C}}$. 
The filtration 
\begin{equation*}
(F^{r}V(I))_{-1\leq r \leq 1} \: \: = \: \: (0\subset (J/I)_{{\C}} \subset (J^{\perp}/I)_{{\C}} \subset V(I)) 
\end{equation*}
on $V(I)$ is $P(J/I)_{{\C}}$-invariant. 
The unipotent radical $U(J/I)_{{\C}}$ acts on the graded quotients trivially, so they are representations of 
\begin{equation*}
{\rm GL}((J/I)_{{\C}}) \times {\rm O}(V(J)) \: \simeq \: {\C}^{\ast} \times {\rm O}(n-2, {\C}). 
\end{equation*}
Specifically, 
\begin{itemize}
\item ${\rm Gr}^{1}V(I) = (J/I)_{{\C}}$ is the weight $1$ character of ${\C}^{\ast}$. 
\item ${\rm Gr}^{0}V(I) = V(J)$ is the standard representation of ${\rm O}(V(J))$. 
\item ${\rm Gr}^{-1}V(I) = (J/I)_{{\C}}^{\vee}$ is the weight $-1$ character of ${\C}^{\ast}$. 
\end{itemize}

Let $d>0$. 
As in \eqref{eqn: basic filtration LA}, let 
\begin{equation*}
F^{r}V(I)^{\otimes d} = \sum_{|\vec{i}|=r} F^{i_{1}}V(I)\otimes \cdots \otimes F^{i_{d}}V(I), \qquad -d \leq r \leq d,  
\end{equation*}
be the induced filtration on $V(I)^{\otimes d}$. 
This is $P(J/I)_{{\C}}$-invariant. 
By \eqref{eqn: basic filtration LA Gr}, the unipotent radical $U(J/I)_{{\C}}$ acts on the graded quotients ${\rm Gr}^{r}V(I)^{\otimes d}$ trivially. 
Hence ${\rm Gr}^{r}V(I)^{\otimes d}$ is a representation of ${\C}^{\ast}\times {\rm O}(V(J))$. 
Specifically, by the same calculation as in Lemma \ref{lem: Gr J-filtration Ed}, we have 
\begin{equation}\label{eqn: Gr weight d-r}
{\rm Gr}^{r}V(I)^{\otimes d} \simeq \chi_{r} \boxtimes \bigoplus_{|\vec{i}|=r} V(J)^{\otimes b(\vec{i})}, 
\end{equation}
where $\chi_{r}$ is the weight $r$ character of ${\C}^{\ast}$. 
If we take a lift of ${\C}^{\ast}\times {\rm O}(V(J))$ in \eqref{eqn: P(J/I) exact sequence}, 
we have a decomposition 
\begin{equation}\label{eqn: weight GrVd} 
V(I)^{\otimes d} \simeq \bigoplus_{r=-d}^{d} {\rm Gr}^{r}V(I)^{\otimes d} 
\end{equation}
as a representation of ${\C}^{\ast}\times {\rm O}(V(J))$ 
because ${\C}^{\ast}\times {\rm O}(V(J))$ is reductive. 
By \eqref{eqn: Gr weight d-r}, this is the weight decomposition with respect to ${\C}^{\ast}$. 

Now let $\lambda=(\lambda_{1}\geq \cdots \geq \lambda_{n})$ be a partition 
expressing an irreducible representation of ${\rm O}(V(I))\simeq {\rm O}(n, {\C})$. 
As in \eqref{eqn: filtration VIl}, let 
\begin{equation*}
F^{r}V(I)_{\lambda} =  V(I)_{\lambda} \cap F^{r}V(I)^{\otimes |\lambda|}
\end{equation*}
be the filtration induced on $V(I)_{\lambda}$. 
This is a $P(J/I)_{{\C}}$-invariant filtration, 
and $U(J/I)_{{\C}}$ acts on the graded quotients trivially. 
By the above argument, if we take a lift of ${\C}^{\ast}\times {\rm O}(V(J))$ in \eqref{eqn: P(J/I) exact sequence}, 
we have a decomposition 
\begin{equation}\label{eqn: weight GrVl} 
{\VIl} \simeq \bigoplus_{r} {\rm Gr}^{r}{\VIl}  
\end{equation}
as a representation of ${\C}^{\ast}\times {\rm O}(V(J))$, 
and this agrees with the weight decomposition for ${\C}^{\ast}$ 
with ${\rm Gr}^{r}{\VIl}$ being the weight $r$ subspace.  


\begin{proposition}\label{prop: J-filtration and representation}
Let $\lambda\ne \det$. 
We have 
\begin{equation}\label{eqn: J-filtration level lambda1}
F^{\lambda_{1}+1}V(I)_{\lambda} = 0, \qquad 
F^{-\lambda_{1}}V(I)_{\lambda} = V(I)_{\lambda}. 
\end{equation}
Thus the filtration $F^{\bullet}V(I)_{\lambda}$ has length $\leq 2\lambda_{1}+1$, 
from level $-\lambda_{1}$ to $\lambda_{1}$. 
Moreover, we have 
\begin{equation}\label{eqn: J-filtration U-invariant}
F^{\lambda_{1}}V(I)_{\lambda} = V(I)_{\lambda}^{U(J/I)_{{\C}}}. 
\end{equation}
\end{proposition}

\begin{proof}
This is purely a representation-theoretic calculation. 
We write $V=V(I)$ and take a basis $e_{1}, \cdots, e_{n}$ of $V$ such that 
$(J/I)_{{\C}}={\C}e_{1}$, 
$(e_{i}, e_{j})=1$ if $i+j=n+1$, and 
$(e_{i}, e_{j})=0$ otherwise. 
We also write $P=P(J/I)_{{\C}}$ and $U=U(J/I)_{{\C}}$. 
(The same notation as in the proof of Proposition \ref{lem: V(I)lJ weight}.) 
We identify $V(J)$ with $V'=\langle e_{2}, \cdots, e_{n-1} \rangle$. 
This defines a lift ${\C}^{\ast}\times {\rm O}(V')\hookrightarrow P$. 
Then ${\C}^{\ast}$ acts on ${\C}e_{1}$ by weight $1$, 
on $V'$ by weight $0$, and on ${\C}e_{n}$ by weight $-1$. 

We first prove \eqref{eqn: J-filtration level lambda1}. 
Recall from \eqref{eqn: Vlambda in wedge tensor} that 
\begin{equation}\label{eqn: Vlambda in wedge J-filtration}
V_{\lambda} \subset \wedge^{{}^t \lambda_{1}}V \otimes \cdots \otimes \wedge^{{}^t \lambda_{\lambda_{1}}}V. 
\end{equation}
Since the weights of ${\C}^{\ast}$ on each space $\wedge^{i}V$ are only $-1, 0, 1$, 
the weights of ${\C}^{\ast}$ on the right hand side of \eqref{eqn: Vlambda in wedge J-filtration} 
are contained in the range $[-\lambda_{1}, \lambda_{1}]$. 
Therefore the weights of ${\C}^{\ast}$ on $V_{\lambda}$ 
are contained in $[-\lambda_{1}, \lambda_{1}]$. 
Since ${\rm Gr}^{r}V_{\lambda}$ is the weight $r$ subspace for the action of ${\C}^{\ast}$, 
this shows that ${\rm Gr}^{r}V_{\lambda}\ne 0$ 
only when $-\lambda_{1} \leq r \leq \lambda_{1}$. 
This implies \eqref{eqn: J-filtration level lambda1}. 

Next we prove \eqref{eqn: J-filtration U-invariant}. 
In Proposition \ref{lem: V(I)lJ weight}, we proved that 
$V_{\lambda}^{U} \simeq \chi_{\lambda_{1}} \boxtimes W$ 
as a representation of ${\C}^{\ast}\times {\rm O}(V')$ 
where $W$ is a representation of ${\rm O}(V')\simeq {\rm O}(n-2, {\C})$. 
(We do not use precise information on $W$.) 
In particular, ${\C}^{\ast}$ acts on $V_{\lambda}^{U}$ by weight $\lambda_{1}$. 
This means that $V_{\lambda}^{U}\subset F^{\lambda_{1}}V_{\lambda}$. 
On the other hand, since $U$ acts trivially on 
\begin{equation*}
{\rm Gr}^{\lambda_{1}}V_{\lambda} \: = \: 
F^{\lambda_{1}}V_{\lambda} / F^{\lambda_{1}+1}V_{\lambda} 
\: = \: F^{\lambda_{1}}V_{\lambda}, 
\end{equation*}
we also see that 
$F^{\lambda_{1}}V_{\lambda} \subset V_{\lambda}^{U}$. 
Therefore $F^{\lambda_{1}}V_{\lambda} = V_{\lambda}^{U}$. 
\end{proof}

We have the following duality between the graded quotients. 

\begin{lemma}\label{prop: duality J-filtration}
We have ${\rm Gr}^{r}{\VIl} \simeq {\rm Gr}^{-r}{\VIl}$ as representations of ${\rm O}(V(J))$. 
\end{lemma}

\begin{proof}
We keep the notation as in the proof of Proposition \ref{prop: J-filtration and representation} and 
take the ${\C}^{\ast}\times {\rm O}(V')$-decomposition \eqref{eqn: weight GrVl}  of $V_{\lambda}$. 
Let $\iota$ be the involution of $V$ which exchanges $e_{1}$ and $e_{n}$ and acts on $V'=\langle e_{2}, \cdots, e_{n-1} \rangle$ trivially. 
Thus $\iota$ and ${\C}^{\ast}={\rm SO}(\langle e_{1}, e_{n} \rangle)$ generate ${\rm O}(\langle e_{1}, e_{n} \rangle)$. 
The involution $\iota$ normalizes ${\C}^{\ast}\times {\rm O}(V')$. 
Its adjoint action acts on ${\C}^{\ast}$ by $\alpha\mapsto \alpha^{-1}$, and acts on ${\rm O}(V')$ trivially. 
Therefore the action of $\iota$ on $V_{\lambda}$ maps the weight $r$ subspace ${\rm Gr}^{r}V_{\lambda}$ 
to the weight $-r$ subspace ${\rm Gr}^{-r}V_{\lambda}$, 
and this map is ${\rm O}(V')$-equivariant. 
\end{proof}

It will be useful to know that the graded quotients in level $-\lambda_{1}$ and $\lambda_{1}$ are indeed nontrivial. 

\begin{lemma}\label{eqn: first and last Gr}
Let $\lambda\ne \det$. 
We have ${\rm Gr}^{\lambda_{1}}{\VIl}\ne 0$ and ${\rm Gr}^{-\lambda_{1}}{\VIl}\ne 0$. 
\end{lemma}

\begin{proof}
We keep the notation as in the proof of Proposition \ref{prop: J-filtration and representation}. 
Recall from \eqref{eqn: Vlambda highest weight vector} that $V_{\lambda}$ contains the vector 
\begin{equation*}
(e_{1}\wedge \cdots \wedge e_{{}^t \lambda_{1}}) \otimes 
(e_{1}\wedge \cdots \wedge e_{{}^t \lambda_{2}}) \otimes \cdots \otimes (e_{1}\wedge \cdots \wedge e_{{}^t \lambda_{\lambda_{1}}}). 
\end{equation*}
Since ${}^{t}\lambda_{1}<n$ by $\lambda \ne \det$, 
this vector is contained in the weight $\lambda_{1}$ subspace for the ${\C}^{\ast}$-action. 
Therefore ${\rm Gr}^{\lambda_{1}}V_{\lambda} \ne 0$. 
The nontriviality of ${\rm Gr}^{-\lambda_{1}}V_{\lambda}$ then follows from Lemma \ref{prop: duality J-filtration}. 
\end{proof}

Since \eqref{eqn: weight GrVl} is the weight decomposition for ${\C}^{\ast}$, 
we can write 
\begin{equation*}
{\rm Gr}^{r} {\VIl} \simeq \chi_{r} \boxtimes V(J)_{\lambda'(r)} 
\end{equation*}
as a representation of ${\C}^{\ast}\times {\rm O}(V(J))$, 
where $V(J)_{\lambda'(r)}$ is some (in general reducible) representation of ${\rm O}(V(J))\simeq {\rm O}(n-2, {\C})$. 
The representation $V(J)_{\lambda'(r)}$ can be understood through the restriction rule of $V_{\lambda}$ for 
${\rm SO}(2, {\C}) \times {\rm O}(n-2, {\C}) \subset {\rm O}(n, {\C})$. 
See \cite{Ki} and \cite{KT} for a description of this restriction rule in terms of the Littlewood-Richardson numbers. 

By translating the conclusions of Proposition \ref{prop: J-filtration and representation} and 
Lemmas \ref{prop: duality J-filtration} and \ref{eqn: first and last Gr} by the $I$-trivialization, 
we obtain the following consequence for the $J$-filtration on ${\El}$. 

\begin{proposition}\label{cor: J-filtration and representation}
Let $\lambda\ne \det$. 
The $J$-filtration $F^{\bullet}{\El}$ on ${\El}$ satisfies 
\begin{equation*}
F^{\lambda_{1}+1}{\El}=0, \qquad 
F^{-\lambda_{1}}{\El}={\El} 
\end{equation*}
and 
\begin{equation*}
F^{\lambda_{1}}{\El} = {\rm Gr}^{\lambda_{1}}{\El}\ne 0, \qquad 
{\rm Gr}^{-\lambda_{1}}{\El}\ne 0. 
\end{equation*}
Thus $F^{\bullet}{\El}$ has length $\leq 2\lambda_{1}+1$, 
from level $-\lambda_{1}$ to $\lambda_{1}$. 
The graded quotients ${\rm Gr}^{r}{\El}$ and ${\rm Gr}^{-r}{\El}$ have the same rank. 
Moreover, $F^{\lambda_{1}}{\El}$ coincides with 
the sub vector bundle ${\ElJ}$ of ${\El}$ defined in \S \ref{ssec: reduce Elk}. 
\end{proposition}

\begin{remark}
(1) By this description of ${\ElJ}$, 
some of the results of \S \ref{ssec: reduce Elk} also follow from the results of \S \ref{ssec: J-filtration Elk}. 

(2) The isomorphism \eqref{eqn: J-filtration Elk Gr} can be written better as 
\begin{equation*}
{\rm Gr}^{r}({\Elk}|_{\Delta_{J}}) \simeq \pi_{2}^{\ast}{\LL}_{J}^{\otimes r+k} \otimes V(J)_{\lambda'(r)}.   
\end{equation*}
\end{remark}

\section{Decomposition of Jacobi forms}\label{ssec: decomp Jacobi form}

In this section we use the $J$-filtration on ${\Elk}$ to show that vector-valued Jacobi forms 
decompose, in a sense, into some tuples of scalar-valued Jacobi forms. 

\begin{proposition}\label{prop: decompose Jacobi form}
Let $\lambda\ne \det$. 
There exists an injective map 
\begin{equation}\label{eqn: decompose Jaco bi form}
J_{\lambda,k,m}({\GJZ}) \hookrightarrow 
\bigoplus_{r=-\lambda_{1}}^{\lambda_{1}}J_{k+r,m}({\GJZ})^{\oplus \alpha(r)}, 
\end{equation}
where $\alpha(r)$ is the rank of ${\rm Gr}^{r}{\El}$. 
\end{proposition}

\begin{proof}
We use the notation in \S \ref{sec: FJ}. 
Let $F^{r}J_{\lambda,k,m}({\GJZ})$ be the subspace of $J_{\lambda,k,m}({\GJZ})$ 
consisting of Jacobi forms which take values in the sub vector bundle 
$F^{r}{\Elk}\otimes \Theta_{J}^{\otimes m}$ of ${\Elk}\otimes \Theta_{J}^{\otimes m}$. 
This defines a filtration on $J_{\lambda,k,m}({\GJZ})$ from level $r=-\lambda_{1}$ to $\lambda_{1}$.  
By the exact sequence 
\begin{equation*}
0 \to F^{r+1}{\Elk} \otimes \Theta_{J}^{\otimes m} \to F^{r}{\Elk} \otimes \Theta_{J}^{\otimes m} \to 
{\rm Gr}^{r}{\Elk}\otimes \Theta_{J}^{\otimes m} \to 0 
\end{equation*}
and Proposition \ref{prop: basic filtration Elk}, we obtain an embedding 
\begin{eqnarray*}
{\rm Gr}^{r}(J_{\lambda,k,m}({\GJZ})) & \hookrightarrow & 
H^{0}(\Delta_{J}, {\rm Gr}^{r}{\Elk} \otimes \Theta_{J}^{\otimes m} )^{{\GJZbar}} \\ 
& \simeq & 
H^{0}(\Delta_{J}, (\pi_{2}^{\ast}{\LL}_{J}^{\otimes r+k})^{\oplus \alpha(r)} \otimes \Theta_{J}^{\otimes m} )^{{\GJZbar}}.  
\end{eqnarray*}
The image of this embedding is contained in 
$J_{r+k,m}({\GJZ})^{\oplus \alpha(r)}$, 
namely holomorphic at the cusps of ${\HJ}$. 
Indeed, if we take the $(I, \omega_{J})$-trivialization at $I\subset J$, 
the quotient homomorphism 
$F^{r}{\Elk} \otimes \Theta_{J}^{\otimes m} \to {\rm Gr}^{r}{\Elk}\otimes \Theta_{J}^{\otimes m}$ 
is identified with the quotient homomorphism 
\begin{equation*}
F^{r}{\VIl}\otimes ({\ICv})^{\otimes k}\otimes \mathcal{O}_{\Delta_{J}} \to 
{\rm Gr}^{r}{\VIl}\otimes ({\ICv})^{\otimes k}\otimes \mathcal{O}_{\Delta_{J}}. 
\end{equation*}
Since this is constant over $\Delta_{J}$, 
its effect on the Fourier expansion of a Jacobi form is just reducing each Fourier coefficient from 
$F^{r}{\VIl}\otimes ({\ICv})^{\otimes k}$ to ${\rm Gr}^{r}{\VIl}\otimes ({\ICv})^{\otimes k}$, 
so the Fourier coefficients still satisfy the holomorphicity condition at the $I$-cusp. 

Therefore we obtain a canonical embedding 
\begin{equation}\label{eqn: Jacobi form J-filtration Gr embed}
{\rm Gr}^{r}(J_{\lambda,k,m}({\GJZ})) \hookrightarrow J_{r+k,m}({\GJZ})^{\oplus \alpha(r)}.  
\end{equation}
Finally, if we choose a splitting of the filtration $F^{\bullet}J_{k,\lambda, m}({\GJZ})$, we obtain a (non-canonical) isomorphism 
\begin{equation*}
J_{k,\lambda, m}({\GJZ}) \simeq 
\bigoplus_{r=-\lambda_{1}}^{\lambda_{1}} {\rm Gr}^{r}(J_{k, \lambda, m}({\GJZ})). 
\end{equation*}
This defines an embedding as claimed.  
\end{proof}

As the proof shows, the embedding \eqref{eqn: decompose Jaco bi form} is not canonical: 
it requires a choice of a splitting of the filtration $F^{\bullet}J_{\lambda,k,m}({\GJZ})$. 
But at least the last subspace is canonically determined: 

\begin{corollary}
Let $\lambda\ne \det$. 
We have a canonical embedding 
\begin{equation*}
J_{k+\lambda_{1},m}({\GJZ})\otimes V(J)_{\lambda'} \hookrightarrow J_{\lambda,k,m}({\GJZ}) 
\end{equation*}
where $\lambda'=(\lambda_{2}\geq \cdots \geq \lambda_{n-1})$. 
\end{corollary}

\begin{proof}
The last (= level $\lambda_{1}$) subspace $F^{\lambda_{1}}J_{\lambda,k,m}({\GJZ})$ 
is the space of Jacobi forms with values in $F^{\lambda_{1}}{\Elk}\otimes \Theta_{J}^{\otimes m}$. 
By Proposition \ref{cor: J-filtration and representation} and Theorem \ref{thm: Siegel operator}, 
this sub vector bundle is isomorphic to 
$\pi_{2}^{\ast}{\LL}_{J}^{\otimes k+\lambda_1} \otimes V(J)_{\lambda'}\otimes \Theta_{J}^{\otimes m}$. 
\end{proof}

\begin{example}
Let $n=3$ and $\lambda=(d)$. 
In this case, in view of \eqref{eqn: Gr J-filtration (d)}, 
the embedding \eqref{eqn: decompose Jaco bi form} takes the form 
\begin{equation*}
J_{(d),k,m}({\GJZ}) \hookrightarrow \bigoplus_{r=-d}^{d} J_{k+r,m}({\GJZ}). 
\end{equation*}
In the context of Siegel modular forms of genus $2$, 
Ibukiyama-Kyomura \cite{IK} found an \textit{isomorphism} of the same shape for a certain type of integral Jacobi groups. 
(In our notation, $L=2U\oplus \langle -2 \rangle$, $K= \langle -2 \rangle$, $J\subset 2U$ the standard one, 
${\UJZ}=\wedge^2J$, and ${\GJZbar}=\Gamma_{J}\ltimes (K\otimes J)$.) 
The method of Ibukiyama and Kyomura is different, based on differential operators. 
It might be plausible that their decomposition essentially agrees with that of us. 
\end{example}

Proposition \ref{prop: decompose Jacobi form} and Proposition \ref{prop: translate Jacobi form} 
enable us to deduce some basic results for vector-valued Jacobi forms from 
those for scalar-valued Jacobi forms. 
We present two such consequences. 

\begin{corollary}\label{cor: Jacobi form vanish}
Let $\lambda\ne \det$. 
We have $J_{\lambda,k,m}({\GJZ})=0$ 
when $k+\lambda_{1}<n/2-1$. 
\end{corollary}

\begin{proof}
In this case, all weights $k+r$ in \eqref{eqn: decompose Jaco bi form} satisfy 
$k+r \leq k+\lambda_{1} <n/2-1$. 
Then we have $J_{k+r,m}({\GJZ})=0$ by Proposition \ref{prop: Jacobi form vanish}. 
\end{proof}

\begin{corollary}
$J_{\lambda,k,m}({\GJZ})$ has finite dimension. 
Moreover, we have the following asymptotic estimates:  
\begin{equation*}
\dim J_{\lambda,k,m}({\GJZ}) = O(k) \qquad (k\to \infty), 
\end{equation*} 
\begin{equation*}
\dim J_{\lambda,k,m}({\GJZ}) = O(m^{n-2}) \qquad (m\to \infty). 
\end{equation*} 
\end{corollary}

\begin{proof}
By Proposition \ref{prop: decompose Jacobi form} and Proposition \ref{prop: translate Jacobi form}, we have 
\begin{eqnarray*}
\dim J_{\lambda,k,m}({\GJZ}) 
& \leq & 
\sum_{r=-\lambda_{1}}^{\lambda_{1}} \alpha(r) \cdot \dim J_{k+r, m}({\GJZ}) \\ 
& \leq & 
\sum_{r=-\lambda_{1}}^{\lambda_{1}} \alpha(r) \cdot \dim J_{k+r, K(\beta_{0}m)}(\Gamma_{J}), 
\end{eqnarray*}
where $K, \beta_{0}, \Gamma_{J}$ do not depend on $\lambda, k, m$. 
By the dimension formula of Skoruppa (\cite{Sk} Theorem 6), we see that 
each $J_{k+r, K(\beta_{0}m)}(\Gamma_{J})$ is finite-dimensional and 
\begin{equation*}
\dim J_{k+r, K(\beta_{0}m)}(\Gamma_{J}) = O(k) \qquad (k\to \infty), 
\end{equation*}
\begin{equation*}
\dim J_{k+r, K(\beta_{0}m)}(\Gamma_{J}) = O(\det K(\beta_{0}m)) = O(m^{n-2}) \qquad (m\to \infty). 
\end{equation*}
These imply the asymptotic estimates for $\dim J_{\lambda,k,m}({\GJZ})$. 
\end{proof}

\begin{remark}
From Proposition \ref{prop: J-filtration and representation}, 
we have imposed the assumption $\lambda\ne \det$. 
This was necessary in our representation-theoretic calculation. 
Indeed, \eqref{eqn: J-filtration U-invariant} and Lemma \ref{eqn: first and last Gr} do not hold for $\lambda = \det$. 
On the other hand, since ${\GJZ}\subset {\rm SO}^+(L)$, Jacobi forms with $\lambda = \det$ are the same as 
those with $\lambda=1$ (scalar-valued Jacobi forms) as far as ${\GJZ}$ is concerned. 
The difference arises when we consider the action by the full stabilizer $\Gamma(J)_{{\Z}}^{\ast}$, 
which may contain an element of determinant $-1$.  
\end{remark}


\chapter{Vanishing theorem I}\label{sec: VT I}

Let $L$ be a lattice of signature $(2, n)$ with $n\geq 3$. 
We assume that $L$ has Witt index $2$, i.e., has a rank $2$ isotropic sublattice. 
This is always satisfied when $n\geq 5$. 
Let ${\G}$ be a finite-index subgroup of ${\OL}$. 
Let $\lambda=(\lambda_{1}\geq \cdots \geq \lambda_{n})$ be a partition with ${}^t \lambda_{1}+{}^t \lambda_{2}\leq n$ 
which expresses an irreducible representation of ${\On}$. 
We assume $\lambda\ne 1, \det$. 
In this chapter, as an application of the $J$-filtration, we prove the following vanishing theorem. 

\begin{theorem}\label{thm: VT I}
Let $\lambda\ne 1, \det$. 
If $k<\lambda_{1}+n/2-1$, then ${\MG}=0$. 
In particular, we have ${\MG}=0$ whenever $k<n/2$. 
\end{theorem}

This generalizes the well-known vanishing theorem $M_{k}({\G})=0$ for $0<k<n/2-1$ in the scalar-valued case. 
This classical fact can be deduced from the vanishing of scalar-valued Jacobi forms (Fourier-Jacobi coefficients) of weight $<n/2-1$. 
Our proof of Theorem \ref{thm: VT I} is a natural generalization of this approach. 
The outline is as follows. 

The first step is to take the projection ${\Elk}\to {\rm Gr}^{-\lambda_{1}}{\Elk}$ 
to the first graded quotient of the $J$-filtration for each $1$-dimensional cusp $J$. 
Then we apply the classical vanishing theorem of scalar-valued Jacobi forms (Proposition \ref{prop: Jacobi form vanish}) 
to ${\rm Gr}^{-\lambda_{1}}{\Elk}$. 
This tells us that when $k-\lambda_{1}<n/2-1$, the Fourier coefficients of a modular form 
at a $0$-dimensional cusp $I\subset J$ are contained in a proper subspace of $V(I)_{\lambda,k}$. 
Finally, running $J$ over all $1$-dimensional cusps containing $I$, we find that the Fourier coefficients are zero. 

The second step of this argument (and hence the bound in Theorem \ref{thm: VT I}) 
could be improved for some specific $({\G}, L)$ if a stronger vanishing theorem of classical Jacobi forms is available 
(cf.~Remark \ref{rmk: 2nd proof VT I}). 
Theorem \ref{thm: VT I} would be a prototype in this direction. 

Let us look at Theorem \ref{thm: VT I} in the cases $n=3, 4$ under the accidental isomorphisms. 

\begin{example}
Let $n=3$. 
Recall from Example \ref{ex: weight n=3} that the orthogonal weight $(\lambda, k)=((d), k)$ corresponds to 
the ${\rm GL}(2, {\C})$-weight $(\rho_{1}, \rho_{2})=(k+d, k-d)$ for Siegel modular forms of genus $2$. 
In this case, the bound in Theorem \ref{thm: VT I} is $k<d+1/2$, namely $k\leq d$. 
This is rewritten as $\rho_{2}\leq 0$. 
This is the same bound as the vanishing theorem of Freitag \cite{Fr79} and Weissauer \cite{We} 
for Siegel modular forms of genus $2$. 

In the case of Siegel modular forms of genus $2$, 
the idea to use Jacobi forms to derive a vanishing theorem of vector-valued modular forms seems to go back to Ibukiyama. 
See \cite{Ibu0} Section 6 (and also \cite{Ib} p.54).  
Our proof of Theorem \ref{thm: VT I} can be regarded as a generalization of the argument of Ibukiyama. 
\end{example}

\begin{example}
Let $n=4$. 
Recall from Example \ref{ex: weight n=4} that the orthogonal weight $(\lambda, k)=((d), k)$ corresponds to 
the weight $(r, \rho\boxtimes \rho)$ with $r=k-d$ and $\rho={\rm Sym}^d$ for Hermitian modular forms of degree $2$. 
In this case, the bound in Theorem \ref{thm: VT I} is 
$k<d+1$, i.e., $k\leq d$. 
Thus Theorem \ref{thm: VT I} says that 
there is no nonzero Hermitian modular form of degree $2$ and weight $(r, \, \rho\boxtimes \rho)$ with $\rho={\rm Sym}^{d}\ne 1$ 
when $r\leq 0$. 
Furthermore, our second vanishing theorem (Theorem \ref{thm: VT II} (1)) says that 
there is no nonzero cusp form when $r\leq 1$. 
\end{example}

The rest of this chapter is as follows. 
In \S \ref{ssec: VT I proof} we prove Theorem \ref{thm: VT I}. 
In \S \ref{ssec: hol tensor} we give an application of Theorem \ref{thm: VT I} 
to the vanishing of holomorphic tensors of small degree on the modular variety ${\FG}$.

\section{Proof of Theorem \ref{thm: VT I}}\label{ssec: VT I proof}

In this section we prove Theorem \ref{thm: VT I}. 
Let $\lambda\ne 1, \det$ and assume that $k-\lambda_{1}<n/2-1$. 
For a rank $2$ primitive isotropic sublattice $J$ of $L$, we denote by 
$F_{J}{\Elk}=F_{J}^{-\lambda_{1}+1}{\Elk}$ 
the level $-\lambda_{1}+1$ ($=$ the first) sub vector bundle of ${\Elk}$ in the $J$-filtration. 
Here we add $J$ in the notation in order to indicate the cusp. 

\begin{step++}\label{step++0}
Every Jacobi form in $J_{\lambda,k,m}({\GJZ})$ takes values in the sub vector bundle 
$F_{J}{\Elk}\otimes \Theta_{J}^{\otimes m}$ of ${\Elk}\otimes \Theta_{J}^{\otimes m}$. 
\end{step++}

\begin{proof}
Recall from \eqref{eqn: Jacobi form J-filtration Gr embed} that 
we have an embedding 
\begin{equation*}
{\rm Gr}^{-\lambda_{1}}(J_{\lambda,k,m}({\GJZ})) \hookrightarrow J_{k-\lambda_{1},m}({\GJZ})^{\oplus \alpha(-\lambda_{1})}.  
\end{equation*}
Since $k-\lambda_{1}<n/2-1$, we have $J_{k-\lambda_{1},m}({\GJZ})=0$ by Proposition \ref{prop: Jacobi form vanish}. 
Therefore ${\rm Gr}^{-\lambda_{1}}(J_{\lambda,k,m}({\GJZ}))=0$, 
which means that every Jacobi form in $J_{\lambda,k,m}({\GJZ})$ takes values in $F_{J}{\Elk} \otimes \Theta_{J}^{\otimes m}$. 
\end{proof}

Now let $f\in {\MG}$. 
We want to prove that $f=0$. 
We fix a rank $1$ primitive isotropic sublattice $I$ of $L$ and let 
$f=\sum_{l}a(l)q^{l}$ be the Fourier expansion of $f$ at the $I$-cusp, where $a(l)\in V(I)_{\lambda,k}$. 
For a rank $2$ primitive isotropic sublattice $J$ of $L$ containing $I$, 
we denote by $F_{J}{\VIl}=F_{J}^{-\lambda_{1}+1}{\VIl}$ the level $-\lambda_{1}+1$ subspace 
in the $J$-filtration \eqref{eqn: filtration VIl} on ${\VIl}$ and write 
\begin{equation*}
F_{J}{\VIlk}= F_{J}{\VIl}\otimes ({\ICv})^{\otimes k} \; \subset V(I)_{\lambda,k}. 
\end{equation*}

\begin{step++}\label{step++1}
Every Fourier coefficient $a(l)$ is contained in the subspace $F_{J}{\VIlk}$ of ${\VIlk}$. 
\end{step++}

\begin{proof}
Let $\sigma_{J}$ be the isotropic ray in ${\UIR}$ corresponding to $J$. 
If $l\in \sigma_{J}$, then $a(l)$ appears as a Fourier coefficient of the restriction $f|_{\Delta_{J}}$ of $f$ to $\Delta_{J}$. 
By Lemma \ref{lem: reduce a(l) 1-dim cusp} and Proposition \ref{prop: J-filtration and representation}, 
we see that $a(l)$ is contained in $F_{J}^{\lambda_{1}}{\VIlk} \subset F_{J}{\VIlk}$. 

Next let $l\not\in \sigma_{J}$. 
Then $a(l)$ appears as a Fourier coefficient of the $m$-th Fourier-Jacobi coefficient $\phi_{m}$ 
of $f$ for some $m>0$ along the $J$-cusp (see \S \ref{ssec: FJ expansion}). 
By Proposition \ref{prop: FJ expansion = Taylor expansion II}, $\phi_{m}$ is a Jacobi form of weight $(\lambda, k)$ and index $m$. 
By Step \ref{step++0}, $\phi_{m}$ as a section of ${\Elk}\otimes \Theta_{J}^{\otimes m}$ 
takes values in the sub vector bundle 
$F_{J}{\Elk} \otimes \Theta_{J}^{\otimes m}$. 
Since the $I$-trivialization over ${\XJcpt}$ sends $F_{J}{\Elk}$ to $F_{J}{\VIlk}\otimes \mathcal{O}_{{\XJcpt}}$, 
this implies that the Jacobi form $\phi_{m}$, 
regarded as a $V(I)_{\lambda,k}$-valued function on $\Delta_{J}$ via the $(I, \omega_{J})$-trivialization, 
takes values in the subspace $F_{J}V(I)_{\lambda,k}$ of $V(I)_{\lambda,k}$. 
It follows that its Fourier coefficients $a(l)$ are contained in $F_{J}V(I)_{\lambda,k}$. 
\end{proof}



\begin{step++}\label{step++3}
Every Fourier coefficient $a(l)$ is zero. 
\end{step++}

\begin{proof}
Let $W=\bigcap_{J\supset I} F_{J}V(I)_{\lambda}$. 
By applying Step \ref{step++1} to all $J\supset I$, we find that $a(l)$ is contained in $W\otimes ({\ICv})^{\otimes k}$. 
We shall prove that $W=0$. 
Since $(J/I)_{{\Q}}$ runs over all isotropic lines in $V(I)_{{\Q}}$ in the definition of $W$ and  
\begin{equation*}
F_{\gamma J}{\VIl}=\gamma(F_{J}{\VIl}) 
\end{equation*}
for $\gamma\in {\rm O}(V(I)_{{\Q}})$, 
we see that $W$ is an ${\rm O}(V(I)_{{\Q}})$-invariant subspace of ${\VIl}$. 
Since ${\rm O}(V(I)_{{\Q}})$ is Zariski dense in ${\rm O}(V(I))$, we find that $W$ is ${\rm O}(V(I))$-invariant. 
But $V(I)_{\lambda}$ is irreducible as a representation of ${\rm O}(V(I))$, 
so we have either $W=0$ or $W=V(I)_{\lambda}$. 
Since $F_{J}V(I)_{\lambda}\ne V(I)_{\lambda}$ by Lemma \ref{eqn: first and last Gr}, we have $W\ne V(I)_{\lambda}$. 
Therefore $W=0$. 
This finishes the proof of Theorem \ref{thm: VT I}. 
\end{proof}

\begin{remark}\label{rmk: 2nd proof VT I}
At least when $V_{\lambda}$ remains irreducible as a representation of ${\SOn}$, 
it is also possible to replace the argument in Step \ref{step++3} by an argument using 
the symmetry of the Fourier coefficients in Proposition \ref{prop: Fourier coeff} and 
the Zariski density of ${\GIZbar}$ as in the proof of Proposition \ref{cor: a(0)=0}. 
This approach allows improvement of Theorem \ref{thm: VT I} when a stronger vanishing theorem of scalar-valued Jacobi forms 
holds for ${\GJZ}$. 
\end{remark}

\section{Vanishing of holomorphic tensors}\label{ssec: hol tensor}

In this section, as an application of Theorem \ref{thm: VT I}, 
we deduce vanishing of holomorphic tensors of small degree on the modular variety ${\FG}={\G}\backslash {\D}$. 
To be more precise, let $X$ be the regular locus of ${\FG}$. 
Sections of $(\Omega_{X}^{1})^{\otimes k}$ are called \textit{holomorphic tensors} on $X$. 
Among them, those which extend holomorphically over a smooth projective compactification of $X$ are a birational invariant of ${\FG}$. 

\begin{theorem}\label{thm: hol tensor}
When $0<k<n/2-1$, we have $H^{0}(X, (\Omega_{X}^{1})^{\otimes k})=0$. 
In particular, $H^{0}(\tilde{X}, (\Omega_{\tilde{X}}^{1})^{\otimes k})=0$ for any smooth projective model $\tilde{X}$ of ${\FG}$. 
\end{theorem}

\begin{proof}
Let $\pi\colon {\D}\to {\FG}$ be the projection. 
We can pullback sections of $(\Omega_{X}^{1})^{\otimes k}$ to ${\G}$-invariant sections of 
$(\Omega_{\pi^{-1}(X)}^{1})^{\otimes k}$. 
They extend holomorphically over ${\D}$  
because the complement of $\pi^{-1}(X)$ in ${\D}$ is of codimension $\geq 2$. 
Hence we have an embedding 
\begin{equation}\label{eqn: pullback hol tensor}
H^{0}(X, (\Omega_{X}^{1})^{\otimes k}) \hookrightarrow H^{0}({\D}, (\Omega_{{\D}}^{1})^{\otimes k})^{{\G}}. 
\end{equation}
Recall from \eqref{eqn: TD} that $\Omega_{{\D}}^{1}\simeq {\E}\otimes {\LL}$. 
If we denote by ${\rm St}^{\otimes k}=\bigoplus_{\alpha}V_{\lambda(\alpha)}$ 
the irreducible decomposition of ${\rm St}^{\otimes k}$, 
we thus obtain an embedding   
\begin{equation}\label{eqn: pullback hol tensor II}
H^{0}(X, (\Omega_{X}^{1})^{\otimes k}) \hookrightarrow \bigoplus_{\alpha}M_{\lambda(\alpha),k}({\G}). 
\end{equation}
When $\lambda(\alpha)\ne 1, \det$, we have $M_{\lambda(\alpha),k}({\G})=0$ for $k<n/2$ by Theorem \ref{thm: VT I}. 
The determinant character does not appear in the irreducible decomposition of ${\rm St}^{\otimes k}$ if $k<n$ (\cite{Ok} Theorem 8.21). 
Finally, when $\lambda(\alpha)=1$, we have $M_{k}({\G})=0$ for $0<k<n/2-1$ as it is classically known. 
Therefore $H^{0}(X, (\Omega_{X}^{1})^{\otimes k})=0$ when $0<k<n/2-1$. 
\end{proof}

We can also classify possible types of holomorphic tensors on $X$ in the next few degrees $n/2-1 \leq k \leq n/2$. 

\begin{proposition}\label{prop: hol tensor smallest}
We write $N(k)=k!/2^{k/2}(k/2)!$ when $k$ is even. 

(1) Let $k=[(n-1)/2]$. 
Then we have an embedding 
\begin{equation*}
H^{0}(X, (\Omega_{X}^{1})^{\otimes k}) \hookrightarrow 
\begin{cases}
0 & n \equiv 0, 3 \mod 4,  \\ 
M_{k}({\G})^{\oplus N(k)} & n \equiv 1, 2 \mod 4. 
\end{cases}
\end{equation*}


(2) Let $k=n/2$ with $n$ even. 
Then we have an embedding 
\begin{equation*}
H^{0}(X, (\Omega_{X}^{1})^{\otimes k}) \hookrightarrow 
\begin{cases}
M_{\wedge^{k},k}({\G}) & n \equiv 2 \mod 4,  \\ 
M_{\wedge^{k},k}({\G})\oplus M_{k}({\G})^{\oplus N(k)} & n \equiv 0 \mod 4. 
\end{cases}
\end{equation*}

The component $M_{\wedge^{k},k}({\G})$ in (2) gives the holomorphic differential forms of degree $k=n/2$. 
The component $M_{k}({\G})^{\oplus N(k)}$ in both (1) and (2) corresponds to the trivial summands in ${\rm St}^{\otimes k}$. 
In both (1) and (2), the embedding is an isomorphism when $\langle {\G}, -{\rm id} \rangle$ contains no reflection. 
\end{proposition}

\begin{proof}
We keep the same notation as in the proof of Theorem \ref{thm: hol tensor}. 

(1) 
When $\lambda(\alpha)\ne 1, \det$, we still have $M_{\lambda(\alpha),k}({\G})=0$ for $k<n/2$ by Theorem \ref{thm: VT I}. 
The determinant character does not appear too. 
By \cite{Ok} Exercise 12.2, 
${\rm St}^{\otimes k}$ does not contain the trivial representation when $k$ is odd, 
while it occurs with multiplicity $N(k)$ when $k$ is even. 

(2) 
When $\lambda(\alpha)\ne \wedge^{d}$ with $0\leq d \leq n$, 
we have $\lambda_{1}\geq 2$ and so $M_{\lambda(\alpha),n/2}({\G})=0$ by Theorem \ref{thm: VT I}. 
By \cite{Ok} Theorem 8.21, the representations $\wedge^{d}$ with $d>n/2$ or $d\not\equiv n/2$ mod $2$ 
do not appear in ${\rm St}^{\otimes n/2}$, and $\wedge^{n/2}$ occurs with multiplicity $1$. 
The multiplicity of the trivial summand is as before. 
It remains to consider $\wedge^{d}$ with $0<d<n/2$ and $d\equiv n/2$ mod $2$. 
We apply our second vanishing theorem (Theorem \ref{thm: VT II} (2)). 
This says that $M_{\wedge^{d},n/2}({\G})=0$ when $n/2\leq n-d-2$, namely $d\leq n/2-2$. 

Finally, when $\langle {\G}, -{\rm id} \rangle$ contains no reflection, 
the projection ${\D}\to {\FG}$ is unramified in codimension $1$ by \cite{GHS07}. 
Then \eqref{eqn: pullback hol tensor} and \eqref{eqn: pullback hol tensor II} are isomorphisms, 
and so the above embeddings are isomorphisms. 
\end{proof}

\begin{remark}
(1) The weight $k=[(n-1)/2]$ in Proposition \ref{prop: hol tensor smallest} (1) is the so-called \textit{singular weight} when $n$ is even, 
and the \textit{critical weight} when $n$ is odd, for scalar-valued modular forms. 
Since $M_k({\G})\ne 0$ in general for these weights, 
the bound in Theorem \ref{thm: hol tensor} is optimal as a general bound. 


(2) Theorem \ref{thm: hol tensor} and Proposition \ref{prop: hol tensor smallest} imply in particular  
vanishing of holomorphic differential forms of degree $<n/2$ on $X$. 
Via the extension theorem of Pommerening \cite{Po}, 
this can also be deduced from the vanishing of the corresponding Hodge components in the $L^2$-cohomology (cf.~\cite{BLMM}). 
%
\end{remark}


\chapter{Square integrability}\label{sec: L2}

Let $L$ be a lattice of signature $(2, n)$ with $n\geq 3$ and ${\G}$ be a finite-index subgroup of ${\OL}$. 
In this chapter we study convergence of the Petersson inner product 
\begin{equation*}
\int_{{\FG}}(f, g)_{\lambda,k}{\volD} 
\end{equation*} 
for $f, g\in {\MG}$, 
where $(\; , \;)_{\lambda,k}$ is the Petersson metric on the vector bundle ${\Elk}$ and ${\volD}$ is the invariant volume form on ${\D}$. 

For $\lambda=(\lambda_{1}\geq \cdots \geq \lambda_{n})$ let  
$\bar{\lambda} = (\lambda_{1}-\lambda_{n}, \cdots, \lambda_{[n/2]}-\lambda_{n+1-[n/2]})$ 
be the associated highest weight for ${\SOn}$ (see \S \ref{sssec: SO rep}). 
We denote by $|\bar{\lambda}|$ 
the sum of all components of $\bar{\lambda}$. 
Our results are summarized as follows. 

\begin{theorem}\label{thm: L2}
Let $f, g \in {\MG}$ with $\lambda\ne 1, \det$. 

(1) If $f$ is a cusp form, then $\int_{{\FG}}(f, g)_{\lambda,k}{\volD}<\infty$. 

(2) When $k\geq n+|\bar{\lambda}|-1$, $f$ is a cusp form if and only if $\int_{{\FG}}(f, f)_{\lambda,k}{\volD}<\infty$. 

(3) When $k\leq n-|\bar{\lambda}|-2$, we always have $\int_{{\FG}}(f, g)_{\lambda,k}{\volD}<\infty$. 
\end{theorem}

See Remark \ref{rmk: L2 scalar-valued} for the scalar-valued case. 
The assertion (1) should be more or less standard. 
The assertions (2) and (3) give a characterization of square integrability except in the range 
\begin{equation}\label{eqn: remaining range L2}
n-|\bar{\lambda}|-1 \, \leq \, k \, \leq \, n+|\bar{\lambda}|-2. 
\end{equation}
The assertion (3) is in fact an intermediate step in the proof of our second vanishing theorem (Theorem \ref{thm: VT II}), 
where we eventually prove that ${\MG}=0$ when $k\leq n-|\bar{\lambda}|-2$. 

This chapter starts with defining the Petersson metrics on the Hodge bundles explicitly (\S \ref{ssec: Petersson}) 
and calculating them over the tube domain (\S \ref{ssec: Pete tube domain}). 
In \S \ref{ssec: asymptotic Petersson tube domain} we give some asymptotic estimates needed in the proof of Theorem \ref{thm: L2}. 
In \S \ref{ssec: proof L2} we prove Theorem \ref{thm: L2}.

\section{Petersson metrics}\label{ssec: Petersson}

In this section we explicitly define the Petersson metrics on the Hodge bundles ${\LL}$ and ${\E}$, 
and hence on the automorphic vector bundles ${\Elk}$. 

We begin with ${\LL}$. 
By the definition of ${\D}$, the Hermitian form $(\cdot, \bar{\cdot})$ on $L_{{\C}}$ is positive on the lines parametrized by ${\D}$. 
Thus restriction of this Hermitian form defines a Hermitian metric on each fiber of ${\LL}$, 
and hence an ${\OLR}$-invariant Hermitian metric on ${\LL}$. 
We call it the \textit{Petersson metric} on ${\LL}$ and denote it by $(\:, \:)_{\LL}$. 

Next we consider ${\E}$. 
We first define the real part of ${\E}$. 
We write $\underline{L_{{\R}}}$ for the product real vector bundle $L_{{\R}}\times {\D}$, 
which we regard as a sub real vector bundle of $L_{{\C}}\otimes{\OD}$ in the natural way. 
Then we define a sub real vector bundle of $\underline{L_{{\R}}}$ by  
\begin{equation*}
{\E}_{{\R}} := {\LL}^{\perp}\cap \underline{L_{{\R}}} = ({\LL}\oplus\bar{{\LL}})^{\perp}\cap \underline{L_{{\R}}}. 
\end{equation*}
This is a real vector bundle of rank $n$. 
By the second expression, 
the fiber of ${\E}_{{\R}}$ over $[\omega]\in {\D}$ is the negative-definite subspace 
\begin{equation}\label{eqn: fiber ER}
\langle {\rm Re}(\omega),  {\rm Im}(\omega)\rangle^{\perp} \cap L_{{\R}} 
\end{equation}
of $L_{{\R}}$ (cf.~\S \ref{ssec: domain}). 
The ${\OLR}$-action on $\underline{L_{{\R}}}$ preserves the sub vector bundle ${\E}_{{\R}}$. 
The natural homomorphism 
\begin{equation*}
{\E}_{{\R}}\otimes_{{\R}}{\C} \hookrightarrow {\LL}^{\perp} \to {\E} 
\end{equation*}
gives an ${\OLR}$-equivariant $C^{\infty}$-isomorphism between ${\E}_{{\R}}\otimes_{{\R}}{\C}$ and ${\E}$. 
This defines a real structure of ${\E}$. 

By the description \eqref{eqn: fiber ER} of the fibers, 
the real vector bundle ${\E}_{{\R}}$ is naturally endowed with an ${\OLR}$-invariant negative-definite quadratic form. 
We take the $(-1)$-scaling to turn it to positive-definite. 
This is a Riemannian metric on ${\E}_{{\R}}$. 
It extends to a Hermitian metric on ${\E}_{{\R}}\otimes_{{\R}}{\C}$ in the usual way. 
(Explicitly, the Hermitian pairing between two vectors $v, w$ is the quadratic pairing between $v$ and $\bar{w}$.) 
Via the $C^{\infty}$-isomorphism ${\E}_{{\R}}\otimes_{{\R}}{\C}\to {\E}$, 
we obtain an ${\OLR}$-invariant Hermitian metric on ${\E}$. 
We call it the \textit{Petersson metric} on ${\E}$ and denote it by $(\; , \; )_{{\E}}$. 

The Petersson metric on ${\E}$ induces an ${\OLR}$-invariant Hermitian metric on ${\E}^{\otimes d}$, 
and hence by restriction an ${\OLR}$-invariant Hermitian metric on ${\El}$ with $|\lambda|=d$. 
Taking the tensor product with the Petersson metric on ${\LL}^{\otimes k}$, we obtain an ${\OLR}$-invariant Hermitian metric on ${\Elk}$. 
We call it the \textit{Petersson metric} on ${\Elk}$ and denote it by $(\; , \; )_{\lambda,k}$. 

\begin{remark}
When $L$ is the primitive integral cohomology of a lattice-polarized $K3$ surface $X$ with period $[\omega]\in {\D}$, 
we have the identifications ${\LL}_{[\omega]}=H^{2,0}(X)$, ${\E}_{{\R}, [\omega]}=H_{\textrm{prim}}^{1,1}(X, {\R})$, and 
${\E}_{{\R}, [\omega]}\otimes_{{\R}}{\C}\to {\E}_{[\omega]}$ is identified with 
$H^{1,1}_{\textrm{prim}}(X, {\C}) \to H^{2,0}(X)^{\perp}/H^{2,0}(X)$. 
On $H^{2,0}(X)$ and $H^{1,1}_{\textrm{prim}}(X, {\C})$ we have the \textit{Hodge metrics} defined by 
$\int_{X}\alpha \wedge \bar{\beta}$ and $-\int_{X}\alpha \wedge \bar{\beta}$ respectively (see \cite{Vo2} \S 6.3.2). 
Thus the Petersson metrics on ${\LL}$ and ${\E}$ are essentially the Hodge metrics in this geometric setting. 
\end{remark}

Let $I$ be a rank $1$ primitive isotropic sublattice of $L$. 
For a vector $v$ of $V(I)_{{\R}}=(I^{\perp}/I)_{{\R}}$, let $s_{v}$ be the section of ${\E}$ which corresponds to 
the constant section $v$ of $V(I)\otimes {\OD}$ by the $I$-trivialization $V(I)\otimes {\OD} \simeq {\E}$. 
We compute the Hermitian pairing between these distinguished sections. 
We choose and fix a lift $V(I)_{{\R}}\hookrightarrow I^{\perp}_{{\R}}$ of $V(I)_{{\R}}$ 
and regard vectors of $V(I)_{{\R}}$ as vectors of $I^{\perp}_{\R}\subset L_{{\R}}$ in this way. 

\begin{lemma}\label{lem: Pete}
Let $v_1, v_2\in V(I)_{{\R}}$. 
The pairing of the sections $s_{v_1}$, $s_{v_2}$ of ${\E}$ with respect to the Petersson metric $(\: , \: )_{{\E}}$ is given by 
\begin{equation*}
(s_{v_1}([\omega]), s_{v_2}([\omega]))_{{\E}} \: = \: -(v_1, v_2) + 
\frac{2\cdot (v_1, {\rm Im}(\omega)) \cdot (v_2, {\rm Im}(\omega))}{({\rm Im}(\omega), {\rm Im}(\omega))} 
\end{equation*}
for $[\omega]\in {\D}$. 
In the right hand side, $( \: , \: )$ is the quadratic form on $L_{{\R}}$, 
and $\omega$ is normalized so as to have real pairing with $I_{{\R}}$. 
In particular, $(s_{v_1}, s_{v_2})_{{\E}}$ is ${\R}$-valued. 
\end{lemma}

\begin{proof}
Let $[\omega]\in {\D}$. 
We choose a nonzero vector $l\in I$. 
We may normalize $\omega$ so that $(l, \omega)=1$. 
For $v\in V(I)_{{\R}}\subset I_{{\R}}^{\perp}$ we write 
\begin{equation*}
\alpha(v) = \frac{(v, {\rm Im}(\omega))}{({\rm Im}(\omega), {\rm Im}(\omega))} 
= \frac{(v, {\rm Im}(\omega))}{({\rm Re}(\omega), {\rm Re}(\omega))} \quad  \in {\R} 
\end{equation*}
and define a vector of $L_{{\C}}$ by  
\begin{equation}\label{eqn: sv'}
s_{v}'([\omega]) = v - (v, \omega)l + \sqrt{-1}\alpha(v) \omega. 
\end{equation}

\begin{claim}\label{claim: sv}
$s_{v}'$ is a section of ${\E}_{{\R}}\otimes_{{\R}} {\C}$ and is the image of $s_v$ 
under the $C^{\infty}$-isomorphism ${\E}\to {\E}_{{\R}}\otimes_{{\R}}{\C}$. 
\end{claim}

We prove Claim \ref{claim: sv}. 
The conditions to be checked are 
\begin{equation*}\label{eqn: condition ER}
({\rm Re}(s_{v}'([\omega])), \omega)=0, \quad ({\rm Im}(s_{v}'([\omega])), \omega)=0, \quad 
s_{v}'([\omega])\in s_{v}([\omega])+{\C}\omega. 
\end{equation*}  
Since $s_v([\omega])=v-(v, \omega)l+{\C}\omega$ by Lemma \ref{lem: basic section E}, 
the last condition follows from the definition of $s_{v}'$. 
We check the first equality. 
Since 
\begin{equation*}
{\rm Re}(s_{v}'([\omega])) = v - (v, {\rm Re}(\omega))l - \alpha(v) \cdot {\rm Im}(\omega), 
\end{equation*}
we see that 
\begin{eqnarray*}
({\rm Re}(s_{v}'([\omega])), \omega) 
& = & (v, \omega) - (v, {\rm Re}(\omega)) - \sqrt{-1}\alpha(v)({\rm Im}(\omega), {\rm Im}(\omega)) \\ 
& = & (v, \omega) - (v, {\rm Re}(\omega)) - \sqrt{-1}(v, {\rm Im}(\omega)) \\ 
& = & 0. 
\end{eqnarray*}
In the first equality we used $({\rm Re}(\omega), {\rm Im}(\omega))=0$. 
The equality 
$({\rm Im}(s_{v}'([\omega])), \omega) = 0$ 
can be verified similarly. 
This proves Claim \ref{claim: sv}. 

\medskip 

We return to the proof of Lemma \ref{lem: Pete}. 
We take two vectors $v_1, v_2\in V(I)_{{\R}}$. 
By definition, $(s_{v_1}([\omega]), s_{v_2}(\omega]))_{{\E}}$ is the pairing of 
$s_{v_1}'([\omega])$ and $s_{v_2}'([\omega])$ with respect to the Hermitian form on ${\E}_{{\R}}\otimes_{{\R}}{\C}$. 
This in turn is the pairing of the vectors $s_{v_1}'([\omega])$ and $\overline{s_{v_2}'([\omega])}$ of $L_{{\C}}$ with respect to 
the $(-1)$-scaling of the quadratic form on $L_{{\C}}$. 
By the expression \eqref{eqn: sv'} of $s_{v}'([\omega])$, we can calculate 
\begin{eqnarray*}
& & 
- (s_{v_1}([\omega]), s_{v_2}([\omega]))_{{\E}} \\
& = & 
(v_1-(v_1, \omega)l + \sqrt{-1}\alpha(v_{1}) \omega, \;  v_2-(v_2, \bar{\omega})l - \sqrt{-1}\alpha(v_{2})\bar{\omega} ) \\ 
& = & 
(v_1, v_2) + \alpha(v_{1})\alpha(v_{2})(\omega, \bar{\omega}) - 2 \alpha(v_{1})({\rm Im}(\omega), v_2) - 2 \alpha(v_{2})({\rm Im}(\omega), v_1).  
\end{eqnarray*}
Since we have  
\begin{eqnarray*} 
& & \alpha(v_{1})\alpha(v_{2})(\omega, \bar{\omega}) 
\: = \: 2 \alpha(v_{1})({\rm Im}(\omega), v_2) \: = \: 2 \alpha(v_{2})({\rm Im}(\omega), v_1) \\ 
& = & \frac{2 \: (v_1, {\rm Im}(\omega)) \: (v_2, {\rm Im}(\omega))}{({\rm Im}(\omega), {\rm Im}(\omega))}, 
\end{eqnarray*}
this proves Lemma \ref{lem: Pete}. 
\end{proof}

\begin{remark}
By the expression \eqref{eqn: sv'}, the imaginary part of $s_{v}'([\omega])$ is nonzero for general $[\omega]$. 
This shows that the real structure on ${\E}\simeq V(I)\otimes {\OD}$ given by ${\E}_{{\R}}$ is different from that given by $V(I)_{{\R}}$. 
Nevertheless, the Petersson metric on the real part given by $V(I)_{{\R}}$ is ${\R}$-valued by Lemma \ref{lem: Pete}. 
\end{remark}

Let ${\volD}$ be the invariant volume form on ${\D}$. 
The Petersson metric $(\, , \, )_{\det, n}$ of weight $(\lambda, k)=(\det, n)$ 
gives an invariant metric on the canonical bundle $K_{{\D}}\simeq {\LL}^{\otimes n}\otimes \det$, 
where $\det$ stands for the determinant character (cf.~Example \ref{ex: det}). 
This can be used to express ${\volD}$ as follows. 
If $\Omega$ is an arbitrary nonzero vector of $(K_{{\D}})_{[\omega]}$ over a point $[\omega]$ of ${\D}$, 
the volume form ${\volD}$ at $[\omega]$ is written as 
\begin{equation}\label{eqn: volD}
{\volD}([\omega])= \frac{\Omega\wedge \bar{\Omega}}{(\Omega, \Omega)_{\det, n}} 
\end{equation}
up to a constant independent of $[\omega]$.  
Indeed, the right hand side does not depend on the choice of $\Omega$, 
and the differential form of degree $(n, n)$ on ${\D}$ defined by the right hand side is clearly ${\OLR}$-invariant, 
so it should coincide with ${\volD}$ up to constant.

\section{Petersson metrics on the tube domain}\label{ssec: Pete tube domain}

Let $I$ be a rank $1$ primitive isotropic sublattice of $L$. 
We calculate the Petersson metrics on ${\LL}$, ${\E}$ over the tube domain ${\DI}\subset {\UIC}$. 
We choose a rank $1$ isotropic sublattice $I'\subset L$ with $(I, I')\ne 0$. 
Recall that the choice of $I'$ determines a tube domain realization ${\D}\to {\DI}$. 
We take a generator $l$ of $I$ and identify $U(I)_{{\Q}}\simeq V(I)_{{\Q}}$ accordingly. 
  
\begin{lemma}\label{lem: Petersson tube domain}
On the tube domain ${\DI}$ we have 
\begin{equation}\label{eqn: Petersson L tube}
(s_l(Z), s_l(Z))_{{\LL}} = 2({\ImZ}, {\ImZ}), 
\end{equation}
\begin{equation}\label{eqn: Petersson E tube}
(s_{v_1}(Z), s_{v_2}(Z))_{{\E}} = -(v_1, v_2) + 
\frac{2\cdot (v_1, {\ImZ})\cdot (v_2, {\ImZ})}{({\ImZ}, {\ImZ})}, 
\end{equation}
for $Z\in {\DI}$. 
Here $s_{l}$ is the section of ${\LL}$ corresponding to the dual vector of $l$, 
$v_1, v_2$ are vectors of $V(I)_{{\R}}$, 
and $(\, , \, )$ in the right hand sides are the natural quadratic form on $V(I)_{{\R}}\simeq {\UIR}$. 
\end{lemma}

\begin{proof}
We begin with $(\, , \, )_{{\LL}}$. 
We can view the section $s_l$ over ${\DI}$ as a function ${\DI}\to L_{{\C}}$ which lifts the inverse ${\DI}\to {\D}$ 
of the tube domain realization and satisfies $(s_{l}, l)\equiv 1$. 
Let $l'$ be the vector of $I'_{{\Q}}$ with $(l, l')=1$, 
and we identify $V(I)_{{\Q}}$ with $(I_{{\Q}}\oplus I_{{\Q}}')^{\perp}$. 
Then we can explicitly write $s_{l}$ as 
\begin{equation*}
s_l(Z) = l'+Z-2^{-1}(Z, Z)l \quad \in L_{{\C}}  
\end{equation*}
for $Z\in {\DI}\subset V(I)$. 
Thus we have 
\begin{eqnarray*}
(s_l(Z), s_l(Z))_{{\LL}} & = & 
(s_l(Z), \overline{s_l(Z)}) 
= (Z, \bar{Z}) - (Z, Z)/2 - \overline{(Z, Z)}/2 \\ 
 &= & 2({\ImZ}, {\ImZ}). 
\end{eqnarray*}

Next we calculate $(\, , \, )_{{\E}}$. 
By Lemma \ref{lem: Pete}, we have 
\begin{equation*}
(s_{v_1}(Z), s_{v_2}(Z))_{{\E}} = -(v_1, v_2) + 
\frac{2\cdot (v_{1}, {\rm Im}(s_l(Z)))\cdot (v_{2}, {\rm Im}(s_l(Z)))}{({\rm Im}(s_l(Z)), {\rm Im}(s_l(Z)))}. 
\end{equation*}
Since 
\begin{equation*}
{\rm Im}(s_l(Z)) = {\ImZ} - 2^{-1}{\rm Im}((Z, Z))l, 
\end{equation*}
we see that 
\begin{equation*}
({\rm Im}(s_l(Z)), {\rm Im}(s_l(Z))) = ({\ImZ}, {\ImZ}), \quad (v_{i}, {\rm Im}(s_l(Z)))=(v_{i}, {\ImZ}). 
\end{equation*}
This proves \eqref{eqn: Petersson E tube}. 
\end{proof}

At each point $Z\in {\DI}$, the Petersson metric on ${\E}$ can be understood as follows. 
We take an ${\R}$-basis $v_{1}, \cdots, v_{n}$ of $V(I)_{{\R}}$ such that 
$v_{1}\in {\R}{\rm Im}(Z)$ and $(v_{i}, {\rm Im}(Z))=0$ for $i>1$. 
Then, by \eqref{eqn: Petersson E tube}, we have  
\begin{equation*}
(s_{v_{i}}(Z), s_{v_{j}}(Z))_{{\E}} = 
\begin{cases}
(v_{1}, v_{1}) & i=j=1 \\ 
-(v_{i}, v_{j}) & i, j>1 \\ 
0 & i=1, j>1 
\end{cases}
\end{equation*}
The right hand side can be seen as the positive-definite modification of the hyperbolic quadratic form on $V(I)_{{\R}}$ 
given by taking the $(-1)$-scaling of the negative-definite subspace ${\rm Im}(Z)^{\perp}$. 
The Petersson metric on $\mathcal{E}_{Z}\simeq V(I)$ is 
the Hermitian extension of this modified real metric on $V(I)_{{\R}}$ to $V(I)$. 

Finally, we recall the expression of ${\volD}$ over ${\DI}$. 
Let ${\rm vol}_{I}$ be a flat volume form on ${\DI}\subset {\UIC}$. 
Then, as it is well-known, we have 
\begin{equation}\label{eqn: volD tube}
{\volD}=({\ImZ}, {\ImZ})^{-n}{\rm vol}_{I}.
\end{equation}
This can be seen by substituting $\Omega = s_{l}^{\otimes n}\otimes v_0$ in \eqref{eqn: volD} 
and using \eqref{eqn: Petersson L tube}, 
where $v_0$ is a nonzero vector of $\det$. 
The section $s_{l}^{\otimes n}\otimes v_0$ of ${\LL}^{\otimes n}\otimes \det$ corresponds to 
a flat canonical form on ${\DI} \subset {\UIC}$ by its ${\UIC}$-invariance.

\section{Asymptotic estimates on the tube domain}\label{ssec: asymptotic Petersson tube domain}

In this section we prepare some estimates of the Petersson metrics on ${\Elk}$ over the tube domain ${\DI}$. 
This will be a main ingredient in the proof of Theorem \ref{thm: L2}. 
We keep the setting of \S \ref{ssec: Pete tube domain}. 

We choose an ${\R}$-basis $\{ v_{i} \}_{i}$ of the real part $(V(I)_{{\R}})_{\lambda}$ of ${\VIl}$. 
Then $\{ v_{i} \}_{i}$ is also a ${\C}$-basis of ${\VIl}$. 
Let $s_{i}'$ be the section of ${\El}$ corresponding to $v_{i}$ via the $I$-trivialization ${\El}\simeq {\VIl}\otimes {\OD}$ 
and let $s_{i}=s_{i}'\otimes s_{l}^{\otimes k}$. 
Then $\{ s_{i} \}_{i}$ is a frame of ${\Elk}$ corresponding to a basis of ${\VIlk}$ by the $I$-trivialization. 
Accordingly, we express a section $f$ of ${\Elk}$ over ${\D}\simeq {\DI}$ as $f=\sum_{i}f_{i}s_{i}$  
with $f_{i}$ a scalar-valued holomorphic function on ${\DI}$. 

\begin{lemma}\label{lem: (f,g)vol tube domain}
There exist real homogeneous polynomials $\{ P_{ij} \}_{i,j}$ on ${\UIR}$ of degree $\leq 2|\lambda|$ 
determined by the basis $\{ v_{i} \}_{i}$ of $(V(I)_{{\R}})_{\lambda}$ such that 
\begin{equation}\label{eqn: (f,g)vol tube domain}
(f, g)_{\lambda,k} {\volD} = \sum_{i,j} f_{i} \: \bar{g_{j}} \cdot P_{ij}({\ImZ}) \cdot ({\ImZ}, {\ImZ})^{k-n-|\lambda|} \: {\rm vol}_{I} 
\end{equation}
for all sections $f=\sum_{i}f_{i}s_{i}$, $g=\sum_{i}g_{i}s_{i}$ of ${\Elk}$ over ${\DI}$. 
The matrix $(P_{ij}({\ImZ}))_{i,j}$ is symmetric and positive-definite for $Z\in {\DI}$.  
\end{lemma}
 
\begin{proof} 
The section $s_{i}'$ is an ${\R}$-linear combination of $|\lambda|$-fold tensor products of the distinguished sections $s_{v}$ of ${\E}$ 
associated to $v\in V(I)_{{\R}}$. 
(Recall that $V_{\lambda}\subset V^{\otimes |\lambda|}$.) 
The equation \eqref{eqn: Petersson E tube} can be written as  
\begin{equation*}
(s_{v_1}(Z), s_{v_2}(Z))_{{\E}} =  
\frac{-(v_1, v_2)({\ImZ}, {\ImZ}) + 2 (v_1, {\ImZ})(v_2, {\ImZ})}{({\ImZ}, {\ImZ})}. 
\end{equation*}
The numerator is a real homogeneous polynomial of ${\rm Im}(Z)$ of degree $\leq 2$. 
Therefore the Petersson paring between $s_{i}'$ and $s_{j}'$ can be written as  
\begin{equation}\label{eqn: Pij}
(s_{i}'(Z), s_{j}'(Z))_{\lambda} = P_{ij}({\ImZ}) \cdot ({\ImZ}, {\ImZ})^{-|\lambda|} 
\end{equation}
for a real homogeneous polynomial $P_{ij}$ of ${\ImZ}$ of degree $\leq 2|\lambda|$. 
Together with \eqref{eqn: Petersson L tube} and \eqref{eqn: volD tube}, we obtain 
\begin{equation*}
(s_{i}(Z), s_{j}(Z))_{\lambda,k} {\volD} = P_{ij}({\ImZ}) \cdot ({\ImZ}, {\ImZ})^{k-n-|\lambda|} \: {\rm vol}_{I}. 
\end{equation*}
This proves \eqref{eqn: (f,g)vol tube domain}. 
Since the matrix $((s_{i}'(Z), s_{j}'(Z))_{\lambda})_{i,j}$ is real symmetric and positive-definite, 
so is $(P_{ij}({\rm Im}(Z)))_{i,j}$ by \eqref{eqn: Pij}. 
\end{proof}

Let ${\G}$ be a finite-index subgroup of ${\OL}$ and let ${\XI}={\DI}/{\UIZ}$. 
We take a regular ${\GIZ}$-admissible cone decomposition $\Sigma_{I}$ of 
$\mathcal{C}_{I}^{+}\subset {\UIR}$ in the sense of \S \ref{sssec: partial compact}. 
Let ${\XI}^{\Sigma_{I}}$ be the associated partial toroidal compactification of ${\XI}$. 
Let $\sigma$ be a cone in $\Sigma_{I}$ of dimension $c$. 
By the regularity of $\Sigma_{I}$, we can write 
$\sigma={\R}_{\geq0}v_{1} + \cdots + {\R}_{\geq0}v_{c}$ 
such that $v_{1}, \cdots, v_{c}$ is a part of a ${\Z}$-basis of ${\UIZ}$, say $v_{1}, \cdots, v_{n}$. 
Let $l_{1}, \cdots, l_{n}\in {\UIZZ}$ be the dual basis of $v_{1}, \cdots, v_{n}$. 
Then $z_{i}=(l_{i}, Z)$, $1\leq i \leq n$, are flat coordinates on ${\UIC}$. 
We have 
\begin{equation*}
{\rm vol}_{I} = dz_{1}\wedge \cdots \wedge dz_{n} \wedge  d\bar{z}_{1}\wedge \cdots \wedge d\bar{z}_{n} 
\end{equation*}
up to constant. 
We write $q_{i}=e(z_{i})$ for $1\leq i \leq c$. 
Let $\Delta_{\sigma}$ be the boundary stratum of $\mathcal{X}(I)^{\Sigma_{I}}$ corresponding to the cone $\sigma$, 
and $\Delta_{i}=\Delta_{v_{i}}$ be the boundary divisor corresponding to the ray ${\R}_{\geq 0}v_{i}$. 
Then $q_{1}, \cdots, q_{c}, z_{c+1}, \cdots, z_{n}$ give local coordinates around $\Delta_{\sigma}$. 
The divisor $\Delta_{i}$ is defined by $q_{i}=0$, 
and $\Delta_{\sigma}$ is defined by $q_{1}= \cdots =q_{c}=0$. 
We write $q_{i}=r_{i}e(\theta_{i})$ with $r_{i}=|q_{i}|$ and $0\leq \theta_{i} < 1$. 
Then 
\begin{eqnarray}\label{eqn volI boundary}
{\rm vol}_{I} & = & 
\frac{dq_{1}}{q_{1}} \wedge \frac{d\bar{q}_{1}}{\bar{q}_{1}} \wedge \cdots \wedge \frac{dq_{c}}{q_{c}} \wedge \frac{d\bar{q}_{c}}{\bar{q}_{c}} 
\wedge dz_{c+1} \wedge \cdots \wedge  d\bar{z}_{n} \nonumber  \nonumber \\ 
& = & (r_{1} \cdots r_{c})^{-1} dr_{1}\wedge d\theta_{1} \wedge \cdots \wedge dr_{c} \wedge d\theta_{c} 
\wedge dz_{c+1} \wedge \cdots \wedge d\bar{z}_{n} 
\end{eqnarray}
up to constant. 

We want to give an asymptotic estimate of the right hand side of \eqref{eqn: (f,g)vol tube domain} as $q_{1}, \cdots, q_{c}\to 0$. 
We take an arbitrary base point $Z_{0}\in {\DI}$ and consider a flow of points of the form 
\begin{equation}\label{eqn: flow}
Z = Z(t_{1}, \cdots , t_{c}) = Z_{0} + \sqrt{-1}(t_{1}v_{1}+ \cdots +t_{c}v_{c}), \qquad t_{1}, \cdots, t_{c} \to \infty. 
\end{equation}
This flow converges to a point of $\Delta_{\sigma}$ as $t_{1}, \cdots, t_{c}\to \infty$, 
and every point of $\Delta_{\sigma}$ can be obtained in this way. 
Let $v_{0}={\rm Im}(Z_{0})$. 
This is a vector in the positive cone $\mathcal{C}_{I}$. 

\begin{lemma}\label{lem: estimate (ImZ, ImZ) pre}
The following asymptotic estimates hold as $t_{1}, \cdots, t_{c}\to \infty$. 
\begin{equation}\label{eqn: Pij asymptotic pre}
P_{ij}({\ImZ}) = O((t_{1}+ \cdots +t_{c})^{2|\lambda|}), 
\end{equation}
\begin{equation}\label{eqn: (ImZ, ImZ) asymptotic pre}
({\ImZ}, {\ImZ}) = O((t_{1}+ \cdots +t_{c})^{2}), 
\end{equation}
\begin{equation}\label{eqn: (ImZ, ImZ)^-1 asymptotic pre}
({\ImZ}, {\ImZ})^{-1} = O((t_{1}+ \cdots +t_{c})^{-1}).  
\end{equation}
\end{lemma}

\begin{proof}
We have ${\ImZ}=v_{0}+\sum_{i}t_{i}v_{i}$. 
Since $P_{ij}$ is a real homogeneous polynomial of degree $\leq 2|\lambda|$ on ${\UIR}$, 
we see that $P_{ij}(v_{0}+\sum_{i}t_{i}v_{i})$ is a real inhomogeneous polynomial of $t_{1}, \cdots, t_{c}$ of degree $\leq 2|\lambda|$. 
This implies \eqref{eqn: Pij asymptotic pre}.  
Next we have  
\begin{equation*}\label{eqn: (ImZ, ImZ) expand}
({\ImZ}, {\ImZ}) = (v_{0}, v_{0}) + 2 \sum_{i}(v_{0}, v_{i})t_{i} + 2 \sum_{i\ne j} (v_{i}, v_{j})t_{i}t_{j} + \sum_{i}(v_i, v_i)t_{i}^{2}. 
\end{equation*}
The estimate \eqref{eqn: (ImZ, ImZ) asymptotic pre} is obvious from this expression. 
Since $v_{0}\in \mathcal{C}_{I}$ and $v_{1}, \cdots, v_{c}\in \overline{\mathcal{C}_{I}}$, 
all coefficients in the right hand side are nonnegative; 
possibly except for $(v_{i}, v_{i})$ with $i\geq 1$, they are furthermore positive. 
Therefore we have 
\begin{equation*}
({\ImZ}, {\ImZ}) > 2 \sum_{i} (v_{0}, v_{i})t_{i} > C \cdot \sum_{i} t_{i} 
\end{equation*}
for some constant $C>0$. 
This implies \eqref{eqn: (ImZ, ImZ)^-1 asymptotic pre}. 
\end{proof}

\begin{lemma}\label{lem: estimate (ImZ, ImZ)}
In a small neighborhood of an arbitrary point of $\Delta_{\sigma}$, we have 
\begin{equation}\label{eqn: Pij asymptotic}
P_{ij}({\ImZ}) = O((-\log r_1\cdots r_{c})^{2|\lambda|}), 
\end{equation}
\begin{equation}\label{eqn: (ImZ, ImZ) asymptotic}
({\ImZ}, {\ImZ}) = O((-\log r_1\cdots r_{c})^{2}), 
\end{equation}
\begin{equation}\label{eqn: (ImZ, ImZ)^-1 asymptotic}
({\ImZ}, {\ImZ})^{-1} = O((-\log r_1\cdots r_{c})^{-1}),  
\end{equation}
as $q_{1}, \cdots, q_{c}\to 0$. 
\end{lemma}

\begin{proof}
We consider the flow \eqref{eqn: flow} with $Z_{0}$ varying over the range 
${\rm Re}(Z_{0})\in {\UIR}/{\UIZ}$ and 
$v_{0}={\rm Im}(Z_{0})$ in a small neighborhood of an arbitrary point of $\mathcal{C}_{I}$. 
Since 
\begin{equation*}
r_{i} = |q_{i}|  =  \exp(-2\pi(l_{i}, {\ImZ}))  =  \exp (-2\pi(l_{i}, v_{0})-2\pi t_{i}), 
\end{equation*}
we have 
\begin{equation}\label{eqn: t=log(r)}
t_{i} = -(2\pi)^{-1} \log r_{i} - (l_{i}, v_{0}).  
\end{equation}
The constant term $-(l_{i}, v_{0})$ depends on $v_{0}={\rm Im}(Z_{0})$ continuously. 
Therefore our assertions follow by substituting $t_{i}\sim -(2\pi)^{-1} \log r_{i}$ 
in the estimates in Lemma \ref{lem: estimate (ImZ, ImZ) pre} and 
using $\log r_{1} + \cdots + \log r_{c} = \log r_{1}\cdots r_{c}$. 
\end{proof}

Summing up the calculations so far, we obtain the following asymptotic estimate of $(f, g)_{\lambda,k} {\volD}$. 

\begin{proposition}\label{prop: fgvol tube domain asymtotic}
Let $f=\sum_{i}f_{i}s_{i}$ and $g=\sum_{i}g_{i}s_{i}$ be as in Lemma \ref{lem: (f,g)vol tube domain}. 
In a small neighborhood of an arbitrary point of $\Delta_{\sigma}$, we have 
\begin{eqnarray*}
 (f, g)_{\lambda,k} {\volD}  
& = & 
\sum_{i,j}f_{i}\: \bar{g}_{j} \cdot O((-\log r_{1} \cdots r_{c})^{\alpha}) \cdot (r_{1}\cdots r_{c})^{-1} \\ 
& & \times \; \; 
dr_{1}\wedge \cdots \wedge dr_{c} \wedge d\theta_{1}\wedge \cdots \wedge d\theta_{c} \wedge dz_{c+1}\wedge \cdots \wedge d\bar{z}_{n}   
\end{eqnarray*}
as $q_{1}, \cdots, q_{c}\to 0$, where 
\begin{equation*}
\alpha = \begin{cases} 
\: 2k-2n & \; k\geq n+|\lambda| \\ 
\: k-n+|\lambda| & \; k< n+|\lambda|. 
\end{cases} 
\end{equation*}
\end{proposition}

\begin{proof}
By substituting \eqref{eqn: Pij asymptotic} and \eqref{eqn volI boundary} 
in the right hand side of \eqref{eqn: (f,g)vol tube domain}, we obtain  
\begin{eqnarray*}
& & (f, g)_{\lambda,k} {\volD} \\ 
& = & 
\sum_{i,j}f_{i}\: \bar{g}_{j} \cdot O((-\log r_{1} \cdots r_{c})^{2|\lambda|}) \cdot ({\rm Im}(Z), {\rm Im}(Z))^{k-n-|\lambda|} \\ 
& & \times \; \; (r_{1}\cdots r_{c})^{-1} \cdot  
dr_{1}\wedge \cdots \wedge dr_{c} \wedge d\theta_{1}\wedge \cdots \wedge d\theta_{c} \wedge dz_{c+1}\wedge \cdots \wedge d\bar{z}_{n}.    
\end{eqnarray*}
Then, according to whether the power degree $k-n-|\lambda|$ of $({\ImZ}, {\ImZ})$ is nonnegative or negative, 
we use \eqref{eqn: (ImZ, ImZ) asymptotic} and \eqref{eqn: (ImZ, ImZ)^-1 asymptotic} respectively. 
\end{proof}

Before going to \S \ref{ssec: proof L2}, we recall the following exercise in calculus. 

\begin{lemma}\label{lem: integral log(r)r} 
Let $m\in {\Z}$. The integral 
\begin{equation*}
\lim_{\varepsilon \to 0} \int_{\varepsilon}^{1/2} \frac{1}{(\log r)^{m} \cdot r} dr 
\end{equation*} 
converges if $m\geq 2$, and diverges if $m\leq 1$. 
\end{lemma}

\begin{proof}
This can be seen from 
\begin{equation*}
\left( \frac{1}{(\log r)^{m-1}} \right)' = \frac{1-m}{(\log r)^m \cdot r} 
\end{equation*}
when $m\ne 1$, and 
$(\log (-\log r))'=((\log r) \cdot r)^{-1}$ 
when $m=1$. 
\end{proof}

\section{Proof of Theorem 10.1}\label{ssec: proof L2}

Now we prove Theorem \ref{thm: L2}. 
We begin with some reductions. 
For the proof of Theorem \ref{thm: L2}, there is no loss of generality even if we replace the given group ${\G}$ 
by a subgroup of finite index. 
Thus we may assume that ${\G}$ is neat. 
In particular, ${\G}$ is contained in ${\rm SO}^{+}(L)$. 
By Proposition \ref{prop: El SO} (1), when ${}^t \lambda_{1}>n/2$, 
we have ${\El}\simeq {\E}_{\bar{\lambda}}$ as ${\SOLR}$-equivariant vector bundles. 
This isomorphism preserves the Petersson metrics up to constant 
by their uniqueness as ${\SOLR}$-invariant Hermitian metrics. 
Thus we have a natural isomorphism ${\MG}\simeq M_{\bar{\lambda},k}({\G})$ 
which preserves the Petersson inner product up to constant. 
Since the highest weight for the partition $\bar{\lambda}$ is $\bar{\lambda}$ itself, 
the assertions of Theorem \ref{thm: L2} for weight $(\lambda, k)$ 
follow from those for weight $(\bar{\lambda}, k)$. 
Thus 
we may assume that ${}^t \lambda_{1}\leq n/2$. 

We take a smooth toroidal compactification ${\FGcpt}$ of ${\FG}$ where the fans $\Sigma_{I}$ are regular.  
We take a subdivision of $\Sigma_{I}$ as follows. 

\begin{lemma}
There exists a ${\GIZ}$-admissible and regular subdivision $\Sigma_{I}'$ of $\Sigma_{I}$ such that 
every cone in $\Sigma_{I}'$ contains at most one isotropic ray. 
\end{lemma}

\begin{proof}
We take representatives $\tau_{1}, \cdots, \tau_{N}$ of ${\GIZbar}$-equivalence classes of 
$2$-dimensional cones spanned by two isotropic rays. 
For each $\tau_{a}$, we choose a rational vector from the interior of $\tau_{a}$. 
This vector has positive norm, and the ray it generates divides $\tau_{a}$. 
Letting ${\GIZbar}$ act, we obtain a division of 
every $2$-dimensional cone $\tau$ spanned by two isotropic rays. 
This is well-defined because ${\GIZbar}$ is torsion-free and so acts on the set of such cones freely. 
The collection of these divisions is ${\GIZbar}$-invariant. 

The division of $\tau$ uniquely induces a division of every cone $\sigma$ having $\tau$ as a face, because $\sigma$ is simplicial. 
Explicitly, if 
$\sigma={\R}_{\geq0}v_{1}+\cdots+{\R}_{\geq0}v_{c}$, 
$\tau={\R}_{\geq0}v_{1}+{\R}_{\geq0}v_{2}$ and 
$v_{0}\in \tau$ is the division vector, 
we add the wall ${\R}_{\geq0}v_{0}+{\R}_{\geq0}v_{3}+\cdots+{\R}_{\geq0}v_{c}$. 
The collection of these new walls defines a ${\GIZbar}$-invariant subdivision of the fan $\Sigma_{I}$ 
such that every cone contains at most one isotropic ray. 
Taking its regular subdivision (\cite{AMRT} p.186), 
we obtain a desired subdivision. 
\end{proof} 

Thus our reduced situation is: 
${\G}$ is neat, ${}^t \lambda_{1}\leq n/2$ so that $\bar{\lambda}=\lambda$, 
and every cone in $\Sigma_{I}$ contains at most one isotropic ray. 
(The last property will be used only in the proof of the assertion (3).) 

The integral $\int_{{\FG}}(f, g)_{\lambda,k}{\volD}$ converges if for every boundary point $x$ of ${\FGcpt}$ 
there exists a neighborhood $U=U_{x}$ of $x$ such that $\int_{U}(f, g)_{\lambda,k}{\volD}$ converges. 
Thus, for the proof of (1) and (3) of Theorem \ref{thm: L2}, 
it suffices to verify the convergence of the integral over $U$. 
Conversely, when $f=g$, if $\int_{U}(f, f)_{\lambda,k}{\volD}$ diverges around some boundary point $x$, 
then $\int_{{\FG}}(f, f)_{\lambda,k}{\volD}$ diverges because $(f, f)_{\lambda,k}$ is nonnegative, real-valued. 
Therefore, for the proof of (2) of Theorem \ref{thm: L2}, it suffices to show that 
the integral $\int_{U}(f, f)_{\lambda,k}{\volD}$ diverges at some $U$ when $f$ is not a cusp form. 

Recall that we have \'etale maps 
$\mathcal{X}(I)^{\Sigma_{I}}\to {\FGcpt}$ and ${\XJcpt}\to {\FGcpt}$ 
which give local charts around the boundary points of ${\FGcpt}$. 
Moreover, we have an \'etale gluing map ${\XJcpt} \to \mathcal{X}(I)^{\Sigma_{I}}$ for $I\subset J$. 
Thus the problem is reduced to estimating $\int_{U}(f, g)_{\lambda,k}{\volD}$ 
for a small neighborhood $U$ of a boundary point of $\mathcal{X}(I)^{\Sigma_{I}}$ 
over a $0$-dimensional cusp $I$. 
We are thus in the situation of \S \ref{ssec: asymptotic Petersson tube domain}. 
In what follows, we use the notation in \S \ref{ssec: asymptotic Petersson tube domain}. 

(1) We first prove the assertion (1) of Theorem \ref{thm: L2}. 
By Proposition \ref{prop: fgvol tube domain asymtotic}, 
the local integral $\int_{U}(f, g)_{\lambda,k}{\volD}$ can be bounded from above by  
\begin{eqnarray*}\label{eqn: tube domain integral asymptotic}
& & \lim_{\varepsilon_{1}, \cdots, \varepsilon_{c}\to 0} 
\int_{\varepsilon_{1}}^{a_{1}} \cdots \int_{\varepsilon_{c}}^{a_{c}} 
\int_{0}^{1}  \cdots \int_{0}^{1}  \int_{U'}  \\  
& & \qquad \quad  
\sum_{i,j}f_{i}\: \bar{g}_{j} \cdot O((-\log r_1 \cdots r_c)^{N}) \cdot (r_1 \cdots r_c)^{-1} \\ 
& & \qquad \quad \times \: dr_{1} \wedge \cdots \wedge dr_{c} \wedge d\theta_{1} \wedge \cdots \wedge d\theta_{c}\wedge 
dz_{c+1}\wedge \cdots \wedge d\bar{z}_{n} 
\end{eqnarray*}
for some integer $N>0$, where $a_{1}, \cdots, a_{c}>0$ are small constants and 
$U'$ is a small open set in $\Delta_{\sigma}$ with coordinates $z_{c+1}, \cdots, z_{n}$. 
If $f$ is a cusp form, its components $f_{i}$ vanish at the boundary divisors 
$\Delta_{1}, \cdots, \Delta_{c}$ by Lemma \ref{lem: modular form extend}. 
Therefore we have $f_{i}=q_{1}\cdots q_{c} \cdot O(1)$. 
Similarly we have $g_{j}= O(1)$. 
We also have $- \log r_1 \cdots r_{c} \leq \prod_{l=1}^{c} (-\log r_{l}).$ 
Then the above integral can be bounded from above by    
\begin{equation*}
\lim_{\varepsilon_{1}, \cdots, \varepsilon_{c}\to 0} 
\int_{\varepsilon_{1}}^{a_{1}} \cdots \int_{\varepsilon_{c}}^{a_{c}} 
\prod_{l=1}^{c} O((-\log r_{l})^{N}) \: dr_{1} \wedge \cdots \wedge dr_{c}. 
\end{equation*} 
This integral converges because $\int_{\varepsilon}^{a}(\log r)^{N}dr$ converges in $\varepsilon \to 0$. 
Hence $\int_{U}(f, g)_{\lambda,k}{\volD}$ converges if $f$ is a cusp form. 
This proves the assertion (1) of Theorem \ref{thm: L2}. 

(3) Next we prove the assertion (3) of Theorem \ref{thm: L2}. 
Let $k\leq n-|\lambda|-2$. 
When $\sigma$ has no isotropic ray, 
$f$ and $g$ vanish at the boundary divisors $\Delta_{1}, \cdots, \Delta_{c}$ by Lemma \ref{lem: modular form extend}. 
(Recall our assumption $\lambda\ne 1, \det$.) 
Therefore we can give a similar (actually stronger) estimate as in the case (1) above, 
which implies that $\int_{U}(f, g)_{\lambda,k}{\volD}$ converges. 
We consider the case when $\sigma$ has an isotropic ray, say ${\R}_{\geq0}v_{1}$. 
Since other rays ${\R}_{\geq0}v_{2}, \cdots, {\R}_{\geq0}v_{c}$ are not isotropic by our assumption, 
we see from Lemma \ref{lem: modular form extend} that $f$ and $g$ vanish at $\Delta_{2}, \cdots, \Delta_{c}$. 
Therefore we have $f=q_{2}\cdots q_{c}\cdot O(1)$ and $g=q_{2}\cdots q_{c}\cdot O(1)$. 
By substituting these estimates in the second case of Proposition \ref{prop: fgvol tube domain asymtotic}, 
we see that 
\begin{eqnarray*}
(f, g)_{\lambda,k}{\volD}  
& = &  
(r_{2}\cdots r_{c}) \cdot O(1) \cdot O((-\log r_1 \cdots r_{c})^{k-n+|\lambda|}) \cdot r_{1}^{-1} \\ 
& &   \times \: dr_{1} \wedge \cdots \wedge dr_{c} \wedge d\theta_{1} \wedge \cdots \wedge d\theta_{c}\wedge 
dz_{c+1}\wedge \cdots \wedge d\bar{z}_{n}. 
\end{eqnarray*}
We have $(-\log r_1 \cdots r_{c})^{-1}\leq (-\log r_{1})^{-1}$. 
Therefore $\int_{U}(f, g)_{\lambda,k}{\volD}$ can be bounded from above by 
\begin{equation*}
\lim_{\varepsilon_1 \to 0} 
\int_{\varepsilon_1}^{a_1} O((-\log r_{1})^{k-n+|\lambda|} \cdot r_{1}^{-1}) \: dr_{1}. 
\end{equation*}
Since $k-n+|\lambda|\leq -2$ by the assumption, 
this integral converges by Lemma \ref{lem: integral log(r)r}.

(2) Finally, we prove the assertion (2) of Theorem \ref{thm: L2}. 
When $L$ has Witt index $\leq 1$, we have ${\SG}={\MG}$ by Proposition \ref{cor: a(0)=0}. 
Thus we may assume that $L$ has Witt index $2$. 
Let $k\geq n+|\lambda|-1$ and assume that $f$ is not a cusp form. 
Then $f$ does not vanish identically at a boundary divisor $\Delta=\Delta_{\sigma}$ 
corresponding to an isotropic ray $\sigma={\R}_{\geq 0}v$ for some $0$-dimensional cusp $I$. 
We shall show that $\int_{U}(f, f)_{\lambda,k}{\volD}$ diverges for a general point $x$ of $\Delta$. 
Thus we consider the case $c=1$. 
We rewrite $q_1=r_1e(\theta_1)$ as $q=re(\theta)$, 
and also denote $Z'=(z_{2}, \cdots, z_{c})$ which give local charts on $\Delta$.  

We go back to the flow $Z=Z_{0}+\sqrt{-1}tv$ in \S \ref{ssec: asymptotic Petersson tube domain}. 
Then $P_{ij}({\ImZ})$ is a real polynomial of $t$ whose coefficients depend continuously on $v_{0}={\rm Im}(Z_{0})$. 
Therefore, by substituting \eqref{eqn: t=log(r)}, we see that in a neighborhood of $x$, 
\begin{equation*}
P_{ij}({\ImZ}) =Q_{ij}( \log r) 
\end{equation*}
for a real polynomial $Q_{ij}$ of one variable whose coefficients depend continuously on $Z'$. 
Moreover, as in the proof of Lemma \ref{lem: estimate (ImZ, ImZ) pre}, we have 
\begin{equation*}
({\ImZ}, {\ImZ})  \; = \; (v_{0}, v_{0}) + 2(v_{0}, v)t  \; \sim \;  -\, C \log r 
\end{equation*}
for some constant $C=C(Z')>0$ depending continuously on $Z'$. 
Therefore, by the same calculation as in \S \ref{ssec: asymptotic Petersson tube domain}, we see that 
\begin{equation*}
(f, f)_{\lambda,k}{\volD} \succcurlyeq 
\sum_{i,j} f_{i}\bar{f_{j}} \, Q_{ij}(\log r) \, (-\log r)^{k-n-|\lambda|} \, r^{-1} \, dr \wedge d\theta \wedge \cdots 
\end{equation*}
as $r\to 0$. 

We take the base change of the frame $(s_{i})_{i}$ by a ${\rm GL}_{N}({\C})$-valued holomorphic function $A=A(Z')$ 
of $Z'$ around $x$ so that $f_{1}\to 1$ and $f_i \to 0$ for $i>1$ as $r\to 0$. 
This is possible because $f\ne 0 \in {\VIlk}$ around $x$. 
Then the real symmetric matrix $Q=(Q_{ij})_{i,j}$ is replaced by the Hermitian matrix ${}^t\bar{A} QA$, 
which we denote by $Q'=(Q'_{ij})_{i,j}$. 
Each $Q'_{ij}$ is a ${\C}$-polynomial of $\log r$ whose coefficients depend continuously on $Z'$. 
Since the Hermitian matrix $Q'$ is positive-definite when $r$ is small, we have in particular $Q'_{11}\ne 0$. 
Then 
\begin{equation*}
(f, f)_{\lambda,k}{\volD} \: \succcurlyeq \: Q'_{11}(\log r) \, (-\log r)^{k-n-|\lambda|} \, r^{-1} \, dr \wedge d\theta \wedge \cdots 
\end{equation*}
as $r\to 0$. 
Since $Q'_{11}$ is a nonzero real polynomial and $k-n-|\lambda| \geq -1$ by our assumption, we obtain 
\begin{equation*}
(f, f)_{\lambda,k}{\volD} \: \succcurlyeq \: (-\log r)^{-1} \, r^{-1} \, dr \wedge d\theta \wedge \cdots 
\end{equation*}
as $r\to 0$. 
By Lemma \ref{lem: integral log(r)r}, this implies that the integral $\int_{U}(f, f)_{\lambda,k}{\volD}$ diverges. 
This completes the proof of Theorem \ref{thm: L2}. 
\qed 

\begin{remark}\label{rmk: L2 scalar-valued}
As the proof shows, the assertion (1) of Theorem \ref{thm: L2} holds even when $\lambda=1, \det$. 
Similarly, the assertion (2) holds also for $\lambda=1, \det$ at least when $L$ has Witt index $2$. 
On the other hand, the proof of (3) makes use of Proposition \ref{cor: a(0)=0}, which requires $\lambda\ne 1, \det$. 
\end{remark}


\chapter{Vanishing theorem II}\label{sec: VT II}

Let $L$ be a lattice of signature $(2, n)$ with $n\geq 3$ and ${\G}$ be a  finite-index subgroup of ${\OL}$. 
Let $\lambda=(\lambda_{1}\geq \cdots \geq \lambda_{n})$ be a partition expressing an irreducible representation of ${\On}$. 
We assume $\lambda\ne 1, \det$. 
Therefore $\lambda_{1}>0$ and $\lambda_{n}=0$. 
In this chapter we prove our second type of vanishing theorem. 
We define the \textit{corank} of $\lambda$, denoted by ${\rm corank}(\lambda)$, 
as the maximal index $1\leq i\leq [n/2]$ such that 
\begin{equation*}
\lambda_{1}=\lambda_{2} = \cdots = \lambda_{i} \quad \textrm{and} \quad \lambda_{n}=\lambda_{n-1}= \cdots = \lambda_{n+1-i}=0. 
\end{equation*}
Let 
\begin{equation*}\label{eqn: highest weight SO in VT II}
\bar{\lambda}=(\bar{\lambda}_{1}, \cdots, \bar{\lambda}_{[n/2]}) = 
(\lambda_{1}-\lambda_{n}, \lambda_{2}-\lambda_{n-1}, \cdots, \lambda_{[n/2]}-\lambda_{n+1-[n/2]}) 
\end{equation*}
be the highest weight for ${\SOn}$ associated to $\lambda$. 
Then ${\rm corank}(\lambda)$ is the maximal index $i$ such that 
$\bar{\lambda}_{1}=\bar{\lambda}_{2}=\cdots = \bar{\lambda}_{i}$. 
Let $|\bar{\lambda}|=\sum_{i}\bar{\lambda}_{i}$ be as in \S \ref{sec: L2}. 

Our second vanishing theorem is the following. 

\begin{theorem}\label{thm: VT II}
Let $\lambda\ne 1, \det$. 
If $k<n+\lambda_{1}-{\rm corank}(\lambda)-1$, there is no nonzero square integrable modular form of weight $(\lambda, k)$. 
In particular, 

(1) ${\SG}=0$ when $k<n+\lambda_{1}-{\rm corank}(\lambda)-1$.  

(2) ${\MG}=0$ when $k< n- |\bar{\lambda}|-1$. 
\end{theorem}


We compare Theorem \ref{thm: VT II} and Theorem \ref{thm: VT I}. 
The bound $n/2+\lambda_{1}-1$ in Theorem \ref{thm: VT I} is smaller than 
the main bound $n+\lambda_{1}-{\rm corank}(\lambda)-1$ in Theorem \ref{thm: VT II}  
because ${\rm corank}(\lambda)\leq [n/2]$. 
However, Theorem \ref{thm: VT II} is about square integrable modular forms, while Theorem \ref{thm: VT I} is about the whole ${\MG}$, 
so Theorem \ref{thm: VT II} does not supersede Theorem \ref{thm: VT I}. 
The comparison of Theorem \ref{thm: VT II} (1) and Theorem \ref{thm: VT I} 
raises the question if we have convergent Eisenstein series in the range 
\begin{equation*}
n/2+\lambda_{1}-1 \: \leq \: k \: < \: n+\lambda_{1}-{\rm corank}(\lambda)-1. 
\end{equation*}
As for the comparison of Theorem \ref{thm: VT II} (2) and Theorem \ref{thm: VT I},  
it depends on $\lambda$ which $n-|\bar{\lambda}|-1$ or $n/2-1+\lambda_{1}$ is larger. 
Roughly speaking, Theorem \ref{thm: VT II} (2) is stronger when $|\bar{\lambda}|$ is small, 
while Theorem \ref{thm: VT I} is stronger when $\lambda_{1}$ is large. 
Thus Theorem \ref{thm: VT II} and Theorem \ref{thm: VT I} are rather complementary. 



The proof of Theorem \ref{thm: VT II} follows the same strategy as 
Weissauer's vanishing theorem for vector-valued Siegel modular forms (\cite{We}). 
If we have a square integrable modular form $f\ne 0$, 
we can construct a unitary highest weight module for the Lie algebra of ${\SOLR}$ by a standard procedure 
(cf.~\cite{Har}, \cite{We} for the Siegel case). 
By computing its weight and comparing it with the classification of unitarizable highest weight modules (\cite{EP}, \cite{EHW}, \cite{Ja}), 
we obtain the bound $k \geq n+\lambda_{1}-{\rm corank}(\lambda)-1$. 
The more specific assertions (1), (2) in Theorem \ref{thm: VT II} 
are derived from combination with Theorem \ref{thm: L2}. 

The rest of this chapter is devoted to the proof of Theorem \ref{thm: VT II}. 
The construction of highest weight module occupies \S \ref{ssec: lift to SOLR} and \S \ref{ssec: highest weight module}, 
and the concluding step is done in \S \ref{ssec: proof VT II}.

\section{Lifting to the Lie group}\label{ssec: lift to SOLR}

In this section we work with $G={\SOLR}$. 
We lift a square integrable modular form on ${\D}$ to a square integrable function on $G$ in a standard way. 
We choose a base point $[\omega_{0}]\in {\D}$. 
Let $K\simeq {\rm SO}(2, {\R})\times {\rm SO}(n, {\R})$ be the stabilizer of $[\omega_{0}]$ in $G$. 
We denote by $\frak{g}, \frak{k}$ the Lie algebras of $G, K$ respectively. 
Let $\frak{g}=\frak{k}\oplus \frak{p}$ be the Cartan decomposition of $\frak{g}$ with respect to $\frak{k}$, 
and $\frak{p}_{{\C}}=\frak{p}_{+}\oplus \frak{p}_{-}$ be the eigendecomposition 
for the adjoint action of  $\frak{so}(2, {\R})\subset \frak{k}$ on $\frak{p}$. 
Then $\frak{p}$ is identified with the real tangent space $T_{[\omega_{0}], {\R}}{\D}$ of ${\D}$ at $[\omega_{0}]$, 
and the decomposition $\frak{p}_{{\C}}=\frak{p}_{+}\oplus \frak{p}_{-}$ corresponds to the decomposition 
$T_{[\omega_{0}], {\R}}{\D}\otimes_{{\R}} {\C}= T^{1,0}_{[\omega_{0}]}{\D}\oplus T^{0,1}_{[\omega_{0}]}{\D}$. 
For each point $[\omega]=g([\omega_{0}])$ of ${\D}$, 
the $g$-action gives an isomorphism $\frak{p}_{-}\to T_{[\omega]}^{0,1}{\D}$. 
This isomorphism is unique up to the adjoint action of $K$. 

The Lie group $P_{-}=\exp(\frak{p}_{-})$ is abelian 
and is the unipotent radical of the stabilizer of $[\omega_{0}]$ in ${\rm SO}(L_{{\C}})$ (see, e.g., \cite{AMRT} pp.107--108). 
Therefore, in view of \eqref{eqn: stabilizer isotropic line}, 
$P_{-}$ coincides with the group of Eichler transvections of $L_{{\C}}$ with respect to the isotropic line ${\C}\omega_{0}$. 
In particular, $P_{-}$ acts trivially on ${\C}\omega_{0}={\LL}_{[\omega_{0}]}$ and 
$\omega_{0}^{\perp}/{\C}\omega_{0}={\E}_{[\omega_{0}]}$. 
We will use this property in the proof of Claim \ref{claim: lift} (3) below. 

Now let $\lambda$ be a partition for ${\On}$ 
and $\bar{\lambda}$ be the associated highest weight for ${\SOn}$. 
To start with ${\On}$ is somewhat roundabout here, 
but this is for consistency with the formulation of Theorem \ref{thm: VT II} and eventually with other chapters. 
We first consider the case when $V_{\lambda}$ remains irreducible as a representation of ${\SOn}$ (cf.~\S \ref{sssec: SO rep}). 
Let $W_{\bar{\lambda},k}$ be the finite-dimensional irreducible ${\C}$-representation of 
$K\simeq {\rm SO}(n, {\R})\times {\rm SO}(2, {\R})$ with highest weight $(\bar{\lambda}, k)$. 

\begin{lemma}\label{lem: lift from D to G}
Assume that either $n$ is odd or $n=2m$ is even with ${}^t \lambda_{1}\ne m$. 
Let $f\ne 0$ be a square integrable modular form of weight $(\lambda, k)$ 
for a finite-index subgroup ${\G}$ of ${\rm SO}^{+}(L)$. 
Then there exists a smooth function $\phi_{f}\ne 0$ on $G$ with the following properties. 

(1) $\phi_{f}\in L^{2}({\G}\backslash G)$. 

(2) $\frak{p}_{-}\cdot \phi_{f} = 0$. (Here $\frak{g}$ acts on $\phi_{f}$ as the derivative of the right $G$-translations.) 

(3) The linear subspace of $L^2({\G}\backslash G)$ spanned by the right $K$-translations of $\phi_{f}$ 
is finite-dimensional and is isomorphic to $W_{\bar{\lambda},k}^{\vee}$ as a $K$-representation. 
\end{lemma}

\begin{proof}
We choose a rank $1$ primitive isotropic sublattice $I$ of $L$ 
and let $j(g, [\omega])$ be the factor of automorphy associated to the $I$-trivialization 
${\Elk}\simeq V(I)_{\lambda,k}\otimes {\OD}$. 
The homomorphism 
\begin{equation}\label{eqn: f.a. K}
K \to {\rm End}(V(I)_{\lambda,k}), \qquad k\mapsto j(k, [\omega_{0}]), 
\end{equation}
defines a representation of $K$ on $V(I)_{\lambda,k}\simeq ({\Elk})_{[\omega_{0}]}$. 
This is irreducible of highest weight $(\bar{\lambda}, k)$ by our assumption on $\lambda$. 
The Petersson metric on $({\Elk})_{[\omega_{0}]}$ is $K$-invariant. 
Via the $I$-trivialization at $[\omega_{0}]$, this defines a $K$-invariant Hermitian metric on $V(I)_{\lambda,k}$. 
The induced constant Hermitian metric on the product vector bundle $G\times {\VIlk}$ over $G$ 
corresponds to the Petersson metric on ${\Elk}$ through the isomorphism 
\begin{equation}\label{eqn: ass VB}
{\Elk} \simeq G\times_{K} ({\Elk})_{[\omega_{0}]} \simeq G\times_{K} {\VIlk}. 
\end{equation}

Via the $I$-trivialization we regard the modular form $f$ as a ${\VIlk}$-valued holomorphic function on ${\D}$. 
We define a $V(I)_{\lambda,k}$-valued smooth function $\tilde{f}$ on $G$ by 
\begin{equation*}\label{eqn: lift} 
\tilde{f}(g) = j(g, [\omega_{0}])^{-1}\cdot f(g([\omega_{0}])), \qquad g\in G. 
\end{equation*}
This is the ${\VIlk}$-valued function on $G$ that corresponds to the section $f$ of ${\Elk}$ 
via the $G$-equivariant isomorphism \eqref{eqn: ass VB}. 

\begin{claim}\label{claim: lift}
The $V(I)_{\lambda,k}$-valued function $\tilde{f}$ satisfies the following. 

(1) $\tilde{f}(\gamma g)= \tilde{f}(g)$ for $\gamma\in {\G}$. 

(2) $\tilde{f}(gk)= k^{-1}(\tilde{f}(g))$ for $k\in K$, where $k^{-1}$ acts on ${\VIlk}$ by \eqref{eqn: f.a. K}. 

(3) $\frak{p}_{-} \cdot \tilde{f} = 0$. 

(4) $\tilde{f}$ is square integrable over ${\G}\backslash G$ with respect to the Haar measure on $G$ 
and the Hermitian metric on $V(I)_{\lambda,k}$. 
\end{claim}

All these properties should be standard. 
We supply an argument for the sake of completeness 
(cf.~\cite{Har} for the Siegel modular case). 
The property (1) follows from the ${\G}$-invariance of $f$, 
and the property (2) is just the invariance of $\tilde{f}$ under the $K$-action on $G\times {\VIlk}$. 
Both (1) and (2) can also be checked directly by using the cocycle condition for $j(g, [\omega])$.  

The property (4) holds because we have 
\begin{eqnarray*}
\int_{{\G}\backslash G} (\tilde{f}(g), \tilde{f}(g)) \, d\mu_{G} 
& = & \int_{{\G}\backslash {\D}} {\volD} \int_{K} (\tilde{f}(g), \tilde{f}(g)) \, d\mu_{K} \\ 
& = & \int_{{\G}\backslash {\D}} (f, f)_{\lambda,k} \, {\volD} 
\end{eqnarray*}
up to constant, 
where $d\mu_{G}, d\mu_{K}$ are the Haar measures on $G, K$ respectively, 
and $(\: , \: )$ in the first line is the Hermitian metric on ${\VIlk}$. 

Finally, we check the property (3). 
We have 
\begin{equation*}
X \cdot \tilde{f}(g) = 
(X\cdot j(g, [\omega_{0}])^{-1}) \, f(g([\omega_{0}])) + 
j(g, [\omega_{0}])^{-1} \, (X\cdot f(g([\omega_{0}]))) 
\end{equation*}
for $X\in \frak{p}_{-}$. 
Then $X\cdot f(g([\omega_{0}]))=0$ by the holomorphicity of $f$. 
As for the first term, since $P_{-}$ fixes $[\omega_{0}]$ and acts trivially on $({\Elk})_{[\omega_{0}]}$ as noticed before, 
we have 
\begin{equation*}
j(g \exp(tX), [\omega_{0}]) = j(g, \exp(tX)([\omega_{0}])) \circ j(\exp(tX), [\omega_{0}]) = j(g, [\omega_{0}]). 
\end{equation*}
This shows that $X \cdot j(g, [\omega_{0}])=0$ and so 
\begin{equation*}
X\cdot j(g, [\omega_{0}])^{-1} = -j(g, \omega_{0})^{-1} \circ (X \cdot j(g, [\omega_{0}])) \circ j(g, [\omega_{0}])^{-1} = 0.  
\end{equation*}
Therefore $X\cdot \tilde{f}=0$.  
Claim \ref{claim: lift} is thus verified. 

\medskip 

We go back to the proof of Lemma \ref{lem: lift from D to G}. 
The property (2) in Claim \ref{claim: lift} means that $\tilde{f}$ as a vector of the $K$-representation 
\begin{equation*}
L^2({\G}\backslash G, \, {\VIlk}) \simeq L^2({\G}\backslash G)\otimes {\VIlk} 
\simeq L^2({\G}\backslash G)\otimes W_{\bar{\lambda},k} 
\end{equation*}
is $K$-invariant. 
Therefore it corresponds to a nonzero $K$-homomorphism 
$\Phi_{f}: W_{\bar{\lambda},k}^{\vee} \to L^2({\G}\backslash G)$, 
which must be injective by the irreducibility of $W_{\bar{\lambda},k}^{\vee}$. 
The image of $\Phi_{f}$ consists of the scalar-valued functions 
$L\circ \tilde{f}$ for $L\in V(I)_{\lambda,k}^{\vee}$. 
By the irreducibility, the $K$-orbit of any such nonzero vector generates the image of $\Phi_{f}$. 
Then we put $\phi_{f}=L\circ \tilde{f}$ for an arbitrary $L\ne 0$. 
The property (3) in Claim \ref{claim: lift} implies the property (2) in Lemma \ref{lem: lift from D to G}. 
This finishes the proof of Lemma \ref{lem: lift from D to G}. 
\end{proof}

\section{Highest weight modules}\label{ssec: highest weight module}

In this section we construct from $\phi_{f}$ a unitary highest weight module of $\frak{g}$. 
The result is summarized in Propositions \ref{prop: lift to unitary HWM general} and \ref{prop: lift to unitary HWM exceptional}. 

First we recall the theory of highest weight modules following \cite{EP}, \cite{EHW}, \cite{Hu} 
and specialized to $G={\SOLR}$. 
Let $\frak{k}_{0}=\frak{so}(2, {\R})$ and $\frak{k}_{1}=\frak{so}(n, {\R})$. 
Then $\frak{k}=\frak{k}_{0}\oplus \frak{k}_{1}$, 
$\frak{k}_{0}$ is the center of $\frak{k}$, and 
$\frak{k}_{1}=[\frak{k}, \frak{k}]$ is the semi-simple part of $\frak{k}$. 
We take a maximal abelian subalgebra $\frak{h}$ of $\frak{k}$. 
Then $\frak{h}_{{\C}}$ is a Cartan subalgebra of $\frak{g}_{{\C}}$. 
We have $\frak{h}=\frak{k}_{0}\oplus \frak{h}_{1}$ with 
$\frak{h}_{1}=\frak{h}\cap \frak{k}_{1}$ being a maximal abelian subalgebra of $\frak{k}_{1}$. 
We may take a Borel subalgebra $\frak{b}$ of $\frak{g}_{{\C}}$ constructed from the root data in $\frak{h}_{{\C}}$ 
which is the direct sum of a Borel subalgebra of $\frak{k}_{{\C}}$ and $\frak{p}_{-}$ (rather than $\frak{p}_{+}$). 

Let $\tilde{\rho}\in \frak{h}_{{\C}}^{\vee}$ be a weight which is dominant and integral with respect to $\frak{k}_{{\C}}$ 
(rather than $\frak{g}_{{\C}}$). 
According to the decomposition $\frak{h}=\frak{k}_{0}\oplus \frak{h}_{1}$, we can write 
\begin{equation*}
\tilde{\rho} = (\rho, \alpha), \qquad  \rho\in (\frak{h}_{1})_{{\C}}^{\vee}, \: \: \alpha\in (\frak{k}_{0})_{{\C}}^{\vee} \simeq {\C}, 
\end{equation*}
with $\rho$ a dominant and integral weight for $(\frak{k}_{1})_{{\C}}=\frak{so}(n, {\C})$. 
Here we identify $(\frak{k}_{0})_{{\C}}^{\vee}\simeq {\C}$ 
by the pairing with the unique maximal non-compact positive root. 
(In the notation of \cite{EP} \S 4, $\rho=(m_{2}, \cdots, m_{n})$ and $\alpha=m_{1}$; 
in the notation of \cite{EHW} \S 10 -- \S 11, $\rho=(\lambda_{2}, \cdots, \lambda_{n})$ and $\alpha=\lambda_{1}+z$.) 
We denote by ${\C}_{\rho,\alpha}$ the $1$-dimensional module of $\frak{h}_{{\C}}$ of weight $(\rho, \alpha)$. 
We can regard ${\C}_{\rho,\alpha}$ as a module of $\frak{b}$ naturally. 
We also denote by $W_{\rho,\alpha}$ the finite-dimensional irreducible module of $\frak{k}_{{\C}}$ of 
highest weight $(\rho, \alpha)$. 
This is compatible with the notation in \S \ref{ssec: lift to SOLR}. 

Let $\frak{U}(\frak{g}_{{\C}})$ and $\frak{U}(\frak{b})$ be the universal enveloping algebras of $\frak{g}_{{\C}}$ and $\frak{b}$ respectively. 
Let 
\begin{equation*}
M(\rho, \alpha) = \frak{U}(\frak{g}_{{\C}})\otimes_{\frak{U}(\frak{b})} {\C}_{\rho,\alpha} 
\end{equation*}
be the Verma module of $\frak{g}_{{\C}}$ with highest weight $(\rho, \alpha)$. 
The module $M(\rho, \alpha)$ has a unique irreducible quotient $L(\rho, \alpha)$ (see \cite{Hu} \S 1.3). 
This is called the \textit{irreducible highest weight module} of $\frak{g}_{{\C}}$ with highest weight $(\rho, \alpha)$. 
The module $L(\rho, \alpha)$ is also a unique irreducible quotient of the generalized (or parabolic) Verma module 
\begin{equation*}
N(\rho, \alpha) = \frak{U}(\frak{g}_{{\C}})\otimes_{\frak{U}(\frak{k}_{{\C}}\oplus \frak{p}_{-})} W_{\rho, \alpha},  
\end{equation*}
because $N(\rho, \alpha)$ is also a quotient of $M(\rho, \alpha)$ (see \cite{Hu} \S 9.4). 
The highest weight module $L(\rho, \alpha)$ is said to be \textit{unitarizable} 
if it is isomorphic as a $\frak{g}_{{\C}}$-module to the $K$-finite part of a unitary representation of $G$. 

Now we go back to modular forms on ${\D}$. 

\begin{proposition}\label{prop: lift to unitary HWM general}
Assume that either $n$ is odd or $n=2m$ is even with ${}^t \lambda_{1}\ne m$. 
If we have a square integrable modular form $f\ne0 \in {\MG}$,  
then the irreducible highest weight module $L(\bar{\lambda}^{\vee}, -k)$ is unitarizable. 
\end{proposition}

\begin{proof}
Let $V_{f}$ be the minimal Hilbert subspace of $L^{2}({\G}\backslash G)$ 
which contains the right $G$-translations of the function $\phi_{f}$ in Lemma \ref{lem: lift from D to G}. 
This is a sub unitary representation of $L^{2}({\G}\backslash G)$.  
The $K$-finite part $(V_{f})_{K}$ of $V_{f}$ is a $(\frak{g}, K)$-module. 
Let $V_{0}$ be the subspace of $(V_{f})_{K}$ generated by the right $K$-translations of $\phi_{f}$. 
By Lemma \ref{lem: lift from D to G} (3), $V_{0}$ is isomorphic to 
$W_{\bar{\lambda}, k}^{\vee}=W_{\bar{\lambda}^{\vee}, -k}$ 
as a $K$-representation. 
By Lemma \ref{lem: lift from D to G} (2), $V_{0}$ is annihilated by $\frak{p}_{-}$. 
Indeed, for $X\in \frak{p}_{-}$ and $k\in K$, we have 
\begin{equation*}
k^{-1} \cdot (X\cdot (k\cdot \phi_{f})) = {\rm Ad}_{k^{-1}}(X) \cdot \phi_{f} = 0 
\end{equation*}
because the adjoint action of $K$ preserves $\frak{p}_{-}$. 
Therefore the natural homomorphism 
$\frak{U}(\frak{g}_{{\C}})\otimes_{{\C}}V_{0}\twoheadrightarrow (V_{f})_{K}$ 
descends to a surjective homomorphism 
\begin{equation*}
N(\bar{\lambda}^{\vee}, -k)\simeq 
\frak{U}(\frak{g}_{{\C}})\otimes_{\frak{U}(\frak{k}_{{\C}}\oplus \frak{p}_{-})}V_{0} 
\twoheadrightarrow (V_{f})_{K} 
\end{equation*}
from the generalized Verma module $N(\bar{\lambda}^{\vee}, -k)$. 
By the minimality of the quotient $L(\bar{\lambda}^{\vee}, -k)$, 
this in turn implies that there exists a surjective homomorphism  
\begin{equation*}
(V_{f})_{K} \twoheadrightarrow  L(\bar{\lambda}^{\vee}, -k). 
\end{equation*}
Since $(V_{f})_{K}$ is unitarizable, so is $L(\bar{\lambda}^{\vee}, -k)$.  
\end{proof}

So far we have considered the case when $V_{\lambda}$ remains irreducible as an ${\SOn}$-representation. 
It remains to consider the exceptional case $n=2m$, ${}^t \lambda_{1}=m$ where $V_{\lambda}$ gets reducible. 
In that case, Proposition \ref{prop: lift to unitary HWM general} is modified as follows. 
For a highest weight $\rho=(\rho_{1}, \cdots, \rho_{m})$ for ${\rm SO}(2m, {\C})$, we write 
$\rho^{\dag}=(\rho_{1}, \cdots, \rho_{m-1}, -\rho_{m})$ as in \S \ref{sssec: SO rep}. 

\begin{proposition}\label{prop: lift to unitary HWM exceptional}
Let $n=2m$ be even and ${}^t \lambda_{1}=m$. 
Suppose that we have a square integrable modular form $f\ne 0\in {\MG}$. 
Then either $L(\bar{\lambda}^{\vee}, -k)$ or $L((\bar{\lambda}^{\dag})^{\vee}, -k)$ is unitarizable.  
\end{proposition}

\begin{proof}
According to the decomposition of ${\El}$ in Proposition \ref{prop: El SO} (2), 
we can write $f=(f_{+}, f_{-})$ with $f_{+}$ of weight $(\bar{\lambda}, k)$ and $f_{-}$ of weight $(\bar{\lambda}^{\dag}, k)$ 
with respect to ${\rm SO}(n, {\R})\times {\rm SO}(2, {\R})$. 
We have either $f_{+}\ne 0$ or $f_{-}\ne 0$. 
Then we can do the same construction for the nonzero component $f_{\pm}$ as before, 
by using the component-wise $I$-trivialization \eqref{eqn: I-trivialization SO reducible}. 
\end{proof}

Finally, we recall the classification of unitarizable irreducible highest weight modules (\cite{EHW}, \cite{Ja}, \cite{EP}). 
For our purpose, we restrict ourselves to those weights $(\rho, \alpha)$ such that 
$\alpha\in {\Z}$ and $\rho$ is a highest weight for ${\SOn}$ (rather than $\frak{so}(n, {\C})$). 
In this situation, the version in \cite{EP} is convenient to use. 
For such a weight $\rho=(\rho_{1}, \cdots, \rho_{[n/2]})$, 
we denote by ${\rm corank}(\rho)$ the maximal index $i$ such that $\rho_{1}=\rho_{2}=\cdots =\rho_{i-1}= |\rho_{i}|$. 

\begin{theorem}[\cite{EP}, \cite{EHW}, \cite{Ja}]\label{thm: classify unitary HWM}
Let $\rho=(\rho_{1}, \cdots, \rho_{[n/2]})$ be a highest weight for ${\rm SO}(n, {\C})$. 
Assume that $\rho_{1}\ne 0$, i.e., $\rho$ nontrivial. 
Let $\alpha\in {\Z}$. 
Then the irreducible highest weight module $L(\rho, \alpha)$ is unitarizable if and only if 
$-\alpha \geq n+\rho_{1} - {\rm corank}(\rho)-1$. 
\end{theorem}

Here we follow \cite{EP} Theorem 4.2 and Theorem 4.3, 
with $\alpha=m_{1}$, $\rho=(m_2, \cdots, m_{n})$ and ${\rm corank}(\rho)=i-1$ in the notation there. 
A complete classification of unitary irreducible highest weight modules for general $(\rho, \alpha)$ 
(and also for other Lie groups) is given in \cite{EHW} and \cite{Ja}. 
For the proof of Theorem \ref{thm: VT II}, we just use the ``only if'' part of Theorem \ref{thm: classify unitary HWM}. 

\begin{remark}\label{rmk: HEW}
In fact, the result of \cite{EHW} tells us more than unitarizability. 
Let $\rho_{1}>0$. 
By the calculation of ``the first reduction point'' in \cite{EHW} Lemma 10.3 and Lemma 11.3, 
we see that the generalized Verma module $N(\rho, \alpha)$ is already irreducible when 
$-\alpha>n+\rho_{1}-{\rm corank}(\rho)-1$. 
Thus $L(\rho, \alpha)=N(\rho, \alpha)$ in that case. 
Furthermore, according to \cite{EHW} Theorem 2.4 (b), 
$L(\rho, \alpha)$ belongs to the holomorphic discrete series when $-\alpha > n+\rho_{1}-1$, 
and to the limit of holomorphic discrete series when $-\alpha = n+\rho_{1}-1$.  
Note that $\alpha=\lambda_{1}+z$ in the notation of \cite{EHW} \S 10 -- \S 11, 
and this $\lambda_{1}$ corresponds to $-\rho_{1}-n+1$ in our notation, 
so $z$ in \cite{EHW} is $\alpha+n+\rho_{1}-1$ here. 
\end{remark}

\section{Proof of Theorem 11.1}\label{ssec: proof VT II}

With the preliminaries in \S \ref{ssec: lift to SOLR} and \S \ref{ssec: highest weight module}, 
we can now complete the proof of Theorem \ref{thm: VT II}. 
Let $n\geq 3$ and $\lambda\ne 1, \det$. 
We first consider the case when either $n$ is odd or $n=2m$ is even with ${}^t \lambda_{1}\ne m$. 
Suppose that we have a square integrable modular form $f\ne 0 \in {\MG}$. 
Then the highest weight module $L(\bar{\lambda}^{\vee}, -k)$ is unitarizable by Proposition \ref{prop: lift to unitary HWM general}. 
By applying Theorem \ref{thm: classify unitary HWM} to $(\rho, \alpha)=(\bar{\lambda}^{\vee}, -k)$, 
we see that $(\lambda, k)$ must satisfy  
\begin{equation*}
k \geq n+ (\bar{\lambda}^{\vee})_{1} - {\rm corank}(\bar{\lambda}^{\vee})-1. 
\end{equation*}
Recall from \S \ref{sssec: SO rep} that $\bar{\lambda}^{\vee}=\bar{\lambda}$ in the case $n\not\equiv 2$ mod $4$ and 
$\bar{\lambda}^{\vee}=\bar{\lambda}^{\dag}$ in the case $n\equiv 2$ mod $4$. 
Since $n\geq 3$, we have $\bar{\lambda}_{1}=(\bar{\lambda}^{\dag})_{1}$ and so 
\begin{equation}\label{eqn: lambda1}
(\bar{\lambda}^{\vee})_{1}=\bar{\lambda}_{1}=\lambda_{1}  
\end{equation}
in both cases.  
Since ${\rm corank}(\bar{\lambda}^{\dag}) = {\rm corank}(\bar{\lambda})$, we also have 
\begin{equation}\label{eqn: corank}
{\rm corank}(\bar{\lambda}^{\vee}) = {\rm corank}(\bar{\lambda}) = {\rm corank}(\lambda) 
\end{equation}
by the definition of ${\rm corank}(\lambda)$. 
(Note that all components of $\bar{\lambda}$ are nonnegative.) 
Hence $(\lambda, k)$ satisfies the bound 
\begin{equation}\label{eqn: weight bound L2}
k \geq n+ \lambda_{1} - {\rm corank}(\lambda) - 1. 
\end{equation}
This proves the main assertion of Theorem \ref{thm: VT II}. 
The assertion (1) for ${\SG}$ is then a consequence of Theorem \ref{thm: L2} (1). 
As for the assertion (2), we note that the inequality 
\begin{equation*}
n-|\bar{\lambda}|-1 < n+\lambda_{1} - {\rm corank}(\lambda) -1 
\end{equation*}
holds, because ${\rm corank}(\bar{\lambda})\leq |\bar{\lambda}|$ and $\lambda_{1}>0$. 
Therefore, when $k<n-|\bar{\lambda}|-1$, any modular form of weight $(\lambda, k)$ is square integrable by Theorem \ref{thm: L2} (3), 
but at the same time its weight violates the bound \eqref{eqn: weight bound L2}. 
This implies that ${\MG}=0$ in this case. 

Next we consider the exceptional case when $n=2m$ is even and ${}^t \lambda_{1}=m$. 
Note that $\bar{\lambda}=\lambda$ in this case. 
If we have a square integrable modular form $f\ne 0 \in {\MG}$, 
then either $L(\bar{\lambda}^{\vee}, -k)$ or $L((\bar{\lambda}^{\dag})^{\vee}, -k)$ is unitarizable 
by Proposition \ref{prop: lift to unitary HWM exceptional}. 
Since $\lambda^{\vee}=\lambda$ or $\lambda^{\dag}$, this means that 
either $L(\lambda, -k)$ or $L(\lambda^{\dag}, -k)$ is unitarizable. 
By Theorem \ref{thm: classify unitary HWM}, we obtain the bound 
\begin{equation*}
k \: \geq \: \min (n+\lambda_{1}-{\rm corank}(\lambda)-1, \; n+(\lambda^{\dag})_{1}-{\rm corank}(\lambda^{\dag})-1). 
\end{equation*}
Since $(\lambda^{\dag})_{1}=\lambda_{1}$ and ${\rm corank}(\lambda^{\dag})={\rm corank}(\lambda)$ as before, 
this reduces to the same bound as \eqref{eqn: weight bound L2}. 
The rest of the argument is similar to the non-exceptional case. 
This completes the proof of Theorem \ref{thm: VT II}. 
\qed

\begin{remark}
Since Theorem \ref{thm: VT II} (2) is derived from Theorem \ref{thm: L2} (3), 
this part could be improved if we could improve the characterization of square integrability 
in the remaining range \eqref{eqn: remaining range L2}. 
\end{remark}

\begin{remark}
Let $V_{f}\subset L^{2}({\G}\backslash G)$ be the unitary representation attached to 
a square integrable modular form $f\in {\MG}$, say in the non-exceptional case. 
Recall from the proof of Proposition \ref{prop: lift to unitary HWM general} that 
\begin{equation*}
N(\bar{\lambda}^{\vee}, -k) \twoheadrightarrow (V_{f})_{K} \twoheadrightarrow L(\bar{\lambda}^{\vee}, -k). 
\end{equation*}
If we apply Remark \ref{rmk: HEW} to $(\rho, \alpha)=(\bar{\lambda}^{\vee}, -k)$ and use 
\eqref{eqn: lambda1} and \eqref{eqn: corank}, we find that 
\begin{equation*}\label{eqn: Vf irreducible}
(V_{f})_{K}\simeq L(\bar{\lambda}^{\vee}, -k) \simeq N(\bar{\lambda}^{\vee}, -k) 
\end{equation*}
when $k\geq n+\lambda_{1}-{\rm corank}(\lambda)$. 
The unitary representation $V_{f}$ belongs to the holomorphic discrete series when $k\geq n+\lambda_{1}$, 
and to the limit of holomorphic discrete series when $k=n+\lambda_{1}-1$. 
\end{remark}


\end{document}